\pdfoutput=1
\documentclass[11pt]{article}
\usepackage{vmargin}
\usepackage{times}

\usepackage[utf8]{inputenc}
\usepackage[T1]{fontenc}

\usepackage[whole]{bxcjkjatype}
\usepackage{CJK}

\usepackage{comment}

\usepackage{epsfig}
\usepackage{graphicx}
\usepackage{amssymb}
\usepackage{amsmath}
\usepackage[utf8]{inputenc}
\usepackage{amsfonts}
\usepackage{amsthm}
\usepackage{paralist}
\usepackage{stmaryrd}
\usepackage[toc,page]{appendix}
\usepackage{bbm}
\usepackage{bm}

\usepackage{mathrsfs}

\numberwithin{equation}{section}

\usepackage{longtable}
\usepackage[labelfont=bf,font=it]{caption}

\usepackage[colorlinks=true,linkcolor=blue,citecolor=blue,urlcolor=blue,pdfborder={0 0 0}]{hyperref}
\usepackage{bookmark}

\usepackage{color}

\usepackage[dvipsnames]{xcolor}

\usepackage{centernot}

\usepackage{cite}
\usepackage{cleveref}
 \usepackage{crossreftools}

  \usepackage{microtype}

\usepackage{mathtools} %Uggly sums

\theoremstyle{plain}
\newtheorem{corollary}{Corollary}[section]
\newtheorem{theorem}[corollary]{Theorem}
\newtheorem{lemma}[corollary]{Lemma}

\newtheorem*{conjecture*}{Conjecture}
\newtheorem{proposition}[corollary]{Proposition}
\newtheorem*{theorem*}{Theorem}
\newtheorem*{lemma*}{Lemma}
\newtheorem*{definition*}{Definition}
\newtheorem*{corollary*}{Corollary}
\theoremstyle{definition}

\theoremstyle{remark}

\newtheorem{remark}[corollary]{Remark}

\newcommand{\R}{\mathbb{R}}

\newcommand{\Z}{\mathbb{Z}}
\newcommand{\proba}{\mathbb{P}}
\newcommand{\N}{\mathbb{N}}

\newcommand{\E}{\mathbb{E}}
\newcommand{\C}{\mathcal{C}}

\renewcommand{\P}{\mathbf{P}}

\newcommand{\Mc}{\mathcal{M}}
\newcommand{\Comp}{\mathfrak{C}}

\newcommand{\Gauss}{\mathcal{G}}

\newcommand{\Hc}{\mathcal{H}}

\newcommand{\B}{\mathscr{B}}

\newcommand{\X}{\chi}
\newcommand{\1}{\mathbf{1}}

\newcommand{\vpg}{\mathbf{v}}

\renewcommand{\i}{\mathbf{i}}

\newcommand{\offx}[1]{\, {\hbox{\rm off}}_{#1} \,}

\newcommand{\be}{\begin{eqnarray}}
\newcommand{\ee}{\end{eqnarray}}

\newcommand{\lmix}{\ell_{\text{mix}}}
\newcommand{\lblock}{\ell_{\text{block}}}

\newcommand{\eps}{\varepsilon}

\newcommand{\A}{\mathcal{A}}
\newcommand{\abf}{\mathbf{a}}
\newcommand{\vA}{\vec{\mathcal{A}}}
\newcommand{\Df}{\mathfrak{D}}
\newcommand{\Dcr}{\mathscr{D}}
\newcommand{\Dbf}{\Df}

\newcommand{\Abf}{\mathbf{A}}

\newcommand{\Cgood}{\vec{\mathfrak C}_{\text{good}}}

\newcommand{\Ccut}{\mathfrak C_{\text{cut}}}
\newcommand{\Ckill}{\vec{\mathfrak C}_{\partial}}
\newcommand{\Clive}{\vec{\mathfrak C}^*}
\newcommand{\Plive}{\varrho^*}
\newcommand{\Tcut}{T_{\text{cut}}}

\newcommand{\SA}{\mathbf{S}}
\newcommand{\T}{\mathbf{T}}
\newcommand{\ball}{\mathcal{B}}
\newcommand{\Zb}{\mathbf{Z}}
\newcommand{\Fc}{\mathcal{V}}
\newcommand{\Four}{\mathfrak{F}}
\newcommand{\Lap}{\mathcal{L}}
\newcommand{\Ft}[1]{\Four[#1]}%{\hat{#1}}

\newcommand{\twosquare}{\boxminus}

\newcommand{\Amp}{\Upsilon}
\newcommand{\Rot}{\mathcal{R}}

\DeclareMathOperator*{\lr}{\longleftrightarrow}
\DeclareMathOperator*{\nlr}{\centernot\longleftrightarrow}

\newcommand{\edges}{\mathcal{E}}
\newcommand{\vedges}{\vec{\mathcal{E}}}

\DeclareMathOperator*{\slr}{\leftrightarrow}
\DeclareMathOperator*{\snlr}{\nleftrightarrow}

\begin{document}
\title{Markovian renormalisation for percolation in high-dimension: Semi-decidability of mean field behavior}
\author{Arthur Blanc-Renaudie\footnote{CNRS and University of Rouen Normandy, arthur.r.blanc-renaudie@cnrs.fr}}
\maketitle
\begin{abstract}
We develop a new approach to study Bernoulli percolation in dimensions $d>6$.
The key idea is to approximate open paths at probability $p'>p$ by a Markov chain of pointed $p$-open clusters.
This allows us to transfer sharp information on the two point function from $p$ to $p'$.
This inductively gives good asymptotic estimates on the two point function as $p\uparrow p_c$.
As a main application, we show that having critical meanfield behavior is a semi-decidable problem.
Along the way, we also prove sharp asymptotics for the susceptibility and the average radius of gyration, and prove a local central limit theorem for the slightly subcritical two-point function.

\end{abstract}
\section{Introduction} \label{sec:intro}
\subsection{Motivation and main results} \label{subsec:intro-main-results}
Percolation is a central model for large-scale connectivity in random media, with deep connections to computer science, probability theory, and statistical physics.
In Bernoulli bond percolation on $\Z^d$, each edge (or bond) is independently open with probability $p$ and closed otherwise.
We denote the associated probability measure by $\proba_p$ and expectation by $\E_p$.
We consider the edge set $\edges=\{\{x,y\}:0<\|x-y\|_1\leq L\}$, where $L=1$ corresponds to the nearest-neighbor model, and $L>1$ to the spread-out model.
We write $\edges_p$ for the set of open edges.
We work under the usual coupling such that for every $p<p'$, almost surely $\edges_p\subset \edges_{p'}$.
Two vertices $x,y\in \Z^d$ are said to be connected, denoted $x\slr y$, if there exists a path of open edges joining them.
The cluster of a vertex $x\in \Z^d$, denoted by $\C(x)=\C_p(x)$, is the largest connected subgraph of $(\Z^d,\edges_p)$ containing $x$.

The model undergoes a phase transition at a critical probability $0<p_c<1$, separating the subcritical phase $p<p_c$, where all open clusters are finite, from the supercritical phase $p>p_c$, where an infinite cluster emerges.
The main question is to understand the critical regime $p\approx p_c$.
A central observation is that above the critical dimension $6$, the influence of the lattice structure fades near $p_c$, and the critical exponents match those of percolation on Cayley trees.
This mean-field behavior is a fundamental phenomenon in statistical physics.
The main rigorous tool to establish it in percolation is the lace expansion.
Originally introduced by Brydges and Spencer \cite{MR782962} for weakly self-avoiding walk,
this method was adapted to percolation by Hara and Slade \cite{MR1043524},
who proved mean-field behavior for the nearest-neighbor model in sufficiently high dimensions,
and for the spread-out model when $d>6$ and $L$ is large enough.
Since then the lace expansion has been refined and adapted for many other models (see e.g. \cite{MR1171762,HSG_Animals,MR1851386,HSA_Contact,sakai2007lace}).
A major advance was obtained by Fitzner and van der Hofstad \cite{fitzner2015nearest}, who proved that $d\geq 11$ suffices for nearest-neighbor percolation, leaving the dimensions $7\leq d\leq 10$ open.
For more details on the subject, we refer to the books by Slade \cite{MR2239599}, and by Heydenreich and van der Hofstad \cite{heydenreich2015progress}.
More recently, Duminil-Copin and Panis \cite{DCPwalk,DCPperco} developed a new approach to recover the mean-field exponents in the spread-out case.

At the heart of the proofs lies a precise estimate for the two-point function $\tau_p:x\mapsto \proba_p(0\leftrightarrow x).$
Usually one starts from some a priori bounds on $\tau_p$, obtained either under assumptions or with numerical input, and proves better ones to deduce the critical exponents.
This often involves differential inequalities and bootstrap bounds for $\tau_p$.
A central source of error in those estimates comes from the correlation between adjacent clusters, which is controlled by the triangle condition
$\nabla(p)=\tau_p\ast \tau_p\ast \tau_p(0)-1,$
where $\ast$ denotes the usual convolution operator.
The condition $\nabla(p_c)<\infty$ was first introduced by Aizenman and Newman \cite{MR762034}, and is known to be equivalent to mean-field behavior \cite{MR1127713}.
When $\nabla(p_c)\ll 1$, the method succeeds; however, when $\nabla(p_c)\gg 1$, the estimates degrade significantly, causing the approach to fail.
This explains why the lace expansion is not expected to resolve the nearest-neighbor case in dimensions $7\leq d \leq 10$.

In this paper, we introduce a new approach to estimate $\tau_p$ and prove mean-field behavior.
Our method is based on an inductive computation of $\tau_{p}$ as $p\uparrow p_c$, and becomes more accurate near $p_c$.
For this reason, the main requirement is no longer to have $\nabla(p_c)\ll 1$,
but rather the \emph{initialization}:
estimating $\tau_{p_0}$ for some fixed $p_0<p_c$ with large enough susceptibility $\X(p_0):=\sum_{x\in \Z^d} \tau_{p_0}(x)$.
For every $x\in \Z^d$, let $\langle x\rangle:=\max(1,\|x\|_2)$.

\begin{theorem} \label{thm:Main}
Let $d>6$, $L\geq 1$. There exists a computable function $f:\R_+ \mapsto \R_+$, depending on $d,L$, such that the following two properties are equivalent:
\begin{itemize}
\item[(a)] There exists $C>0$ and some $p_0<p_c$ with $\chi(p_0)\geq f(C)$ such that
\begin{equation}
\forall x\in \Z^d, \quad \tau_{p_0}(x) \leq C \langle x\rangle^{2-d}. \label{eq:initialization}
\tag{C}
\end{equation}
\item[(b)] There exists $A>0$ such that as $x\to \infty$,
\begin{equation}
\tau_{p_c}(x)\sim A \langle x\rangle^{2-d}.\label{eq:eta}
\tag{A}
\end{equation}
\end{itemize}
\end{theorem}
It is easy to check that (b) implies that (a) holds whenever $p_0<p_c$ is close enough to $p_c$, since $\chi(p)$ diverges as $p\to p_c$, and since $\tau_p(x)$ is increasing in $p$.
Moreover, whenever (a) holds, it may be checked numerically in finite time, since one can, for instance, approximate $\tau_p$ by the two-point function of percolation restricted to a large but finite box. Indeed, outside a large box, one can use sharpness bounds (see e.g. \cite{duminil2015new}), to bound $\tau_{p_0}(x)$ when $\|x\|_2$ is large.
Thus, by Theorem \ref{thm:Main}, it is possible to prove Theorem \ref{thm:Main} (b) in finite time whenever it holds.
In other words, we have:
\begin{corollary} \label{cor:SemiDecidable}
For every $d>6$ and $L\geq 1$, Theorem \ref{thm:Main} (b) is a semi-decidable problem.
\end{corollary}

Theorem \ref{thm:Main} (b) is already proved whenever $d\geq 11$ \cite{fitzner2015nearest}, and in sufficiently spread-out models \cite{MR1043524}.
It is also conjectured to always hold in high dimensions.
It is noteworthy that even assuming \eqref{eq:eta}, the lace expansion still requires the triangle condition to be small enough, so it could not be used to prove Corollary \ref{cor:SemiDecidable}.
Moreover, the method developed in this paper yields better estimates on the two-point function, with errors that often depend on the square diagram,
$
\square(p):= \max_{x\in \Z^d} \tau_p\ast \tau_p \ast \tau_p\ast \tau_p(x),
$
which captures correlations at the macroscopic scale.
Let
\begin{equation}
\psi_d:x\in \R_+\mapsto  \begin{cases}
x^{(8-d)/2} & \text{if $d<8$,} \\
\log(x) & \text{if $d=8$,} \\
1 & \text{if $d>8$.} \\
\end{cases}
\label{eq:def-psi-d}
\end{equation}
\begin{proposition} \label{pro:Square}
Let $d>6$, $L\geq 1$. If \eqref{eq:eta} holds, then as $p\uparrow p_c$, we have $\square(p)=\Theta(\psi_d(\X(p)))$.%, where
\end{proposition}
\begin{remark}
The upper bound on $\square(p)$ only requires that \eqref{eq:initialization} hold at $p$. (See Lemma \ref{lem:TwoTriangleSquare}.)
\end{remark}
\begin{theorem} \label{thm:susceptibility}
Assume \eqref{eq:eta} holds. There exists $C_\X>0$ such that for all $\eps>0$, we have as $p\uparrow p_c$,
\[
\X(p)=C_\X(p_c-p)^{-1}+O((p_c-p)^{-\eps}\psi_d(1/(p_c-p))).
\]
\end{theorem}
In this paper, we are also interested in the average radius of gyration given by
\begin{equation}
\xi_2(p) := \sqrt{\frac{1}{\X(p)} \sum_{x\in \Z^d} \|x\|_2^2 \tau_p(x)}.
\label{eq:def-xi2}
\end{equation}
\begin{theorem} \label{thm:gyration}
Assume \eqref{eq:eta} holds. There exists $C_\xi>0$ such that for every $\eps>0$ we have as $p\uparrow p_c$,
\[
\xi_2(p)=C_\xi(p_c-p)^{-1/2}+O((p_c-p)^{1/2-\eps} \psi_d(1/(p_c-p))).
\]
\end{theorem}
The equivalence for $\X(p)$ was only proved recently in a paper by the author and Nachmias \cite{ScalingTorus}, using a different approach based on multi-arm infinite incipient clusters, and $\xi_2(p)$ was only known up to a constant.
We believe that the error factors given by Theorems \ref{thm:susceptibility} and \ref{thm:gyration} are near optimal.

Finally, as we will see why later, proving the next result is central to our proof of Theorem \ref{thm:Main}.
\begin{theorem} \label{thm:tau_local_CLT}
Assume \eqref{eq:eta} holds. There exist $\eps,C_\Psi, C_\Psi'>0$ such that as $p\uparrow p_c$ and $x\to \infty$,
\[
\tau_p(x) =C_\Psi \int_0^\infty e^{-t/\X(p)} e^{-C_\Psi'\|x\|^2_2/t}t^{-d/2}dt+o(\|x\|_2^{2-d-\eps}).
\]
\end{theorem}
\begin{remark}
Although we did not optimize the error term, the proof shows that any $\eps<1/8$ works.
\end{remark}
Following the intuition that $\tau_p$ behaves as the Green's function of a random walk stopped at an exponential time, Theorem \ref{thm:tau_local_CLT} can be interpreted as a local central limit theorem (LCLT).
This intuition will be made fully rigorous in our proofs.

In the same setting, the previous best subcritical bound for $\tau_p$ is due to Hutchcroft, Michta, and Slade \cite[Theorem 1.1]{hutchcroft2023high}, who proved that for some $c,C>0$, for every $p<p_c$ and $x\in \Z^d$,
\begin{equation}
\tau_p(x)\leq C \| x\|_2^{2-d} \exp(-c\|x\|_2(p_c-p)^{1/2}). \label{eq:Plateau}
\end{equation}
This bound is less accurate than Theorem \ref{thm:tau_local_CLT} when $\tau_p(x)\geq \langle x\rangle^{2-d-\eps}$, and becomes more accurate when $\|x\|_2/\log(\|x\|_2)\gg (p_c-p)^{-1/2}$ due to the exponential tail.
For more details, we refer to \cite{hutchcroft2023high} for a historical account of slightly subcritical estimates for $\tau$.

\paragraph{Other notations}
Let $\vedges$ be the set of oriented edges on $(\Z^d,\edges)$. By convention, for $\vec e\in \vedges$ we write $e^-,e^+$ for the vertices such that $\vec e=(e^-,e^+)$.
Let $\Omega$ be the degree of a vertex in $(\Z^d,\edges)$.
For every set or graph $S$ and $x\in \Z^d$, let $S+x$ be the natural translation of $S$ by $x$.
For $x=(x_i)_{1\leq i\leq d},y=(y_i)_{1\leq i\leq d}\in \R^d$, we write $\langle x,y\rangle=\sum_{i=1}^d x_i y_i$ for the usual scalar product.
Also, for every function $f:\Z^d\mapsto \R$, let $\|f\|_1=\sum_{x\in \Z^d} |f(x)|$ and let $\|f\|_\infty=\max_{x\in \Z^d} |f(x)|$.
We also write $\|f\|_2^2:= \sum_{x\in \Z^d} |f(x)| \|x\|_2^2$ for the weighted second moment of $f$, and let $\Ft{f}:\kappa \mapsto \sum_{x\in \Z^d} f(x) e^{\i\langle \kappa,x\rangle}$ be the Fourier transform of $f$.

For a graph $G=(V,E)$, we write $x\in G$ to mean $x\in V$. Let $|G|=|V|=\#V$. For any two graphs $G=(V,E), G'=(V',E')$, let $G\cup G'=(V\cup V',E\cup E')$, and let $G\cap G'=(V\cap V',E\cap E')$.
Let $\Comp(x)$ be the set of finite connected subgraphs of $\Z^d$ containing $x$.
Also, let $\vec \Comp(x)$ be the set of pointed clusters $(C,\vec e)$ such that $C\in \Comp(x)$ and $e^-\in C$.
Let $\Comp=\bigcup_x \Comp(x)$, and let $\vec \Comp=\bigcup_x \vec \Comp(x)$.
For every $x\in \Z^d, \vec C=(C,\vec e)\in \vec \Comp$, we write $x\in \vec C$ to mean $x\in C$.
For every set $S$, let $G\backslash S$ denote the subgraph $(V\backslash S, E\backslash (V\times S))$. Let $G\offx{x} S$ denote the cluster (i.e. largest connected subgraph) of $G\backslash S$ containing $x$.
By convention, if $x\notin G\backslash S$, then $G\offx{x} S$ is the empty graph written $\emptyset$. Also, if $G'=(V',E')$ is another graph, we write $G\backslash G'=G\backslash V'$ and $G\offx{x} G'=G\offx{x} V'$.

We write $\circ$ for the disjoint occurrence of events, meaning that they hold on edge-disjoint sets.
This notation is used both in the classical sense, notably with the BK inequality \cite{MR799280}, and for events restricted to some fixed graphs.
For instance, "$(0\slr x)\circ (x\slr y) \text{ in }G$" means that there exist two edge-disjoint paths from $0$ to $x$ and from $x$ to $y$ in the graph $G$.
\subsection{About the proofs} \label{subsec:about-proofs}
The bulk of the proofs is based on a new inductive approximation for the two-point function $\tau$.
Very often for some $p<p'$ with $\X(p)\ll \X(p')$, we assume some proper bounds on $\tau_{p}$, and we aim to approach $\tau_{p'}$.
To this end, we decompose a $p'$-open path from $0$ to $x$ as a sequence of $p$-open clusters $(A_i)_{0\leq i\leq k}$ with some oriented edges $\vec e_1, \vec e_2,\dots, \vec e_k$ in between that are $p'$-open but $p$-closed.
Naturally, those clusters are correlated, and one would usually neglect this correlation.
The main novelty of our approach is that we now keep the correlation between adjacent clusters.
To do so, we use Aizenman's off method, which in its simplest form states the following:
\begin{lemma}
Let $x\neq y \in \Z^d$. For every $X\in \Comp(x)$ with $y\notin X$, we have equality between (a) the conditional law of $\C_p(y)$ on the event $\{\C_p(x)=X\}$, and (b) the law of $(\C_p(y) \offx y X)$.
\end{lemma}
In other words, when we keep the correlation between $A_i$ and $A_{i-1}$, we may see $A_i$ as an independent $p$-open cluster of $e_i^+$ where the edges adjacent to $A_{i-1}$ are closed.
When closing those edges, we have to verify that we do not cut the path from $0$ to $x$.
For that reason, we need to check that for every $1\leq i \leq k$, we can still go from $e_i^+$ to $e_{i+1}^-$ in $A_i$ even after removing the edges adjacent to $A_{i-1}$.
Simply put, we must have $e_{i+1}^- \in A_i\offx {e_i^+} A_{i-1}$.

All this motivates the following approximation.
We look at the probability $\tau_{p,p',k}(x)$ that there exists a $p'$-open path between $0$ and $x$ such that $k\geq 0$ is the minimum number of edges on this path that are $p'$-open and $p$-closed.
We approximate $\tau_{p,p',k}(x)/((p'-p)/(1-p))^{k}$ by
\begin{equation}
\tau^{\square,k}_p(x):=
\hspace{-1em} \sum_{\substack{ \vec e_1,\dots \vec e_k\in \vedges \\  e_0^+=0, \, \, e_{k+1}^-=x}}  \hspace{-0.5em}
\sum_{\substack{ A_0,A_1,\dots, A_k\in \Comp \\ A_{-1}=\emptyset}}
\prod_{i=0}^k \proba_p(\C(e_i^+)=A_i)
\prod_{i=0}^k \1(e_{i+1}^-\in (A_i\offx{e_i^+} A_{i-1})).  \label{eq:def_TauSquare}
\end{equation}
Thus we approximate $\tau_{p'}=\sum_{k\geq 0} \tau_{p,p',k}$ by
$\tau^\square_{p,p'}:=\sum_{k\geq 0} ((p'-p)/(1-p))^k\tau^{\square,k}_p$.

The proofs are organized as follows.
In Section \ref{sec:Diagrams}, we show an upper bound and a lower bound on $\tau_{p'}-\tau^\square_{p,p'}$.
The errors correspond to several sums that can be represented as "square diagrams".
Those errors appear for three reasons.
The first is that we forget the correlation between $A_i$ and $A_j$ when $i\leq j-2$, and $A_i$ may cut the path from $0$ to $x$ in $A_j$.
The second is that the choice of the edges $(\vec e_1,\dots, \vec e_k)$ may not be unique.
The last concerns the upper bound and is more technical to understand: when taking into account the correlations between the different clusters, $A_{i-1}$ is not really the cluster of $e_{i-1}^+$, but only an upper bound for the inclusion.
So the fact that $A_{i-1}$ cuts the path from $e_i^+$ to $e_{i+1}^-$ in $A_i$ does not necessarily mean that $\C_p(e_{i-1}^+)$ cuts this path.

In Section \ref{sec:MarkovChain}, which is independent of Section \ref{sec:Diagrams}, we interpret $((A_i-e_i^+,\vec e_{i+1}-e_i^+))_{0\leq i \leq k}$ as a Markov chain on $\vec \Comp(0)$.
To this end, we first construct a preliminary transition kernel, based on \eqref{eq:def_TauSquare},
then we study its spectral properties, and we perform a Doob $h$-transform.
We then show, under \eqref{eq:initialization}, that this Markov chain mixes exponentially fast.
It is noteworthy that in this interpretation $\tau_p^{\square,k}(x)$ is related to the probability that the Markov chain of pointed $p$-open clusters ends at $x$,
that is the probability that we have $x=\sum_{i=0}^k (e_{i+1}^+-e_i^+)-(e_{k+1}^+-e_{k+1}^-)$.

In Section \ref{sec:FirstApplications}, we upper-bound the errors that appear in Section \ref{sec:Diagrams}.
We first focus on the case where \eqref{eq:initialization} is satisfied at $p$ with nearly no information at $p'$.
We then show exponential tails for $\tau_p$ under \eqref{eq:initialization}.
Finally we upper-bound the errors in the case where \eqref{eq:initialization} is satisfied both at $p$ and $p'$.

We then prove Theorems \ref{thm:susceptibility} and \ref{thm:gyration} in Section \ref{Sec:SharpAsymptotics}.
The main idea is to use the fast mixing of the Markov chain of $p$-open clusters defined in Section \ref{sec:MarkovChain}
to approach $p'\mapsto \|\tau_{p,p'}^\square\|_1$ and $p'\mapsto \|\tau_{p,p'}^\square\|^2_2$
by rational functions whose coefficients depend on $p$,
and whose degree is independent of $p$.
Then, by using the bounds of Section \ref{sec:FirstApplications}, we may approach $p'\mapsto \X(p')$ and $p'\mapsto \|\tau_{p'}\|_2^2$ with the same functions in some interval depending on $p$.
Finally we apply and invert some well-chosen linear operators to estimate $\X(p')$ and $\|\tau_{p'}\|_2^2$ as $p'\uparrow p_c$ without knowing any of the coefficients.

In Section \ref{sec:TCL} we perform the induction for Theorem \ref{thm:Main}, proving Theorem \ref{thm:tau_local_CLT} along the way. To this end, we first take $p<p_0$ such that the estimates of Section \ref{sec:FirstApplications} hold,
then take $p'>p_0$ such that $\X(p_0)\ll \X(p')\ll \X(p)^{2}/\square(p)$.
We then perform Fourier-transform calculus to show an LCLT for $\sum_{i=0}^k (e_{i+1}^+-e_i^+)$ where the edges correspond to the Markov chain of pointed $p$-open clusters defined in Section \ref{sec:MarkovChain}.
This in turn implies a LTCL for $\tau^{\square,k}_{p}$, and hence, by summing, for $\tau^{\square}_{p,p'}$.

At this stage, we only know that $\tau^\square_{p,p'}$ is close to a Laplace distribution whose parameters are related to some quantities associated with the Markov chain of $p$-open clusters.
To estimate those quantities, we similarly approach $\tau^\square_{p,p_0}$ by another Laplace distribution.
Using the induction hypothesis \eqref{eq:initialization}, and the bounds of Section \ref{sec:FirstApplications}, we can instead estimate those quantities using $\tau_{p_0}$.
In other words, we manage to approximate $\tau^\square_{p,p'}$ by a Laplace distribution whose parameters are described by $\X(p_0),\X(p')$, and the parameters of another Laplace distribution approximating $\tau_{p_0}$.
To conclude the induction, we use the bounds of Section \ref{sec:FirstApplications} to approach $\tau_{p'}$ by $\tau^\square_{p,p'}$.
\subsection{Further remarks} \label{sec:FurtherRemarks}

\textbf{Comparison with \cite{RWforMeanfield}} Independently, Duminil-Copin, Markar, Panis and Slade recently developed a similar renormalisation approach to show mean-field behavior in high dimensions.
As in this paper they inductively estimate $\tau_{p'}$ using $\tau_p$ by comparing with a random walk, but they use independent increments instead of Markovian increments. This diverging point comes from different objectives.
In \cite{RWforMeanfield}, the aim is to give a simple probabilistic proof that applies to many models (self avoiding walks, percolation, spins, and lattice trees). 
Whereas, our goal is to  bypass the barrier posed by local correlations that prevented a numerical proof of mean-field behavior in dimensions $7\leq d\leq 10$ and prevented sharp estimates on the two-point function. This explains why the proofs are totally different. The reader should thus think of the two papers not as overlapping, but as complementary.

\textbf{Numerical computations} 
By Theorem \ref{thm:Main}, mean-field behavior in dimensions $7\leq d\leq 10$ can be proved in finite time.
However, we do not determine whether this can be done within a reasonable human lifetime.
In fact, we do not even compute explicitly the function in Theorem \ref{thm:Main}.
Its computability follows simply from the fact that all the arguments in this paper are quantitative: every constant in the results can be replaced by an explicit function of the previous constants.

Besides keeping the presentation concise, we avoid doing so because this would result in absurdly large and useless constants.
The main reason is that the local modifications in Section \ref{sec:MarkovChain} are extremely costly.
Nonetheless, it seems that the triangle condition is small enough, even in dimension $d=7$, to bypass those local modifications entirely through numerical estimates.

There are two main reasons why we do not yet have a numerical proof of mean-field behavior for percolation in high dimensions.
First, there appears to be a serious computational barrier: numerically estimating $\tau_{p_0}$ seems very hard when $\X(p_0)\gtrsim 100$ and impossible when $\X(p)\gtrsim 1000$.
This forces us to consider Markov chains at best of length at most $\X(p_0)/\psi_d(\X(p_0))$, which is quite short.
Second, we do not have good numerical non-asymptotic LTCL estimates for Markov chains with short mixing times.
These issues are amplified in dimension $d=7$, where $\psi_d(\X(p))=\sqrt{\X(p)}$: this requires taking $\X(p_0)$ significantly larger, while $\tau_{p_0}$ is harder to compute for the same $\X(p_0)$.

To address these issues, one needs to substantially reduce the errors in these approximations.
One possibility would be to use the local behavior of near-critical clusters to better understand the Markov chain and its increments.
A second possibility would be for the Markov chain to retain more than the last cluster in its memory.
This seems much more effective, but would also require significant new theoretical ideas.
The errors would then no longer correspond to square diagrams, but to pentagonal, hexagonal, or higher-order diagrams.
We emphasize that simply using a Markov chain with memory two or three would not suffice.
Indeed, some errors come from the non-uniqueness of the choice of the edges $\vec e_1,\dots, \vec e_k$ and require a separate treatment. Notably the error $\Diamond_{p,p'}$ of Section \ref{sec:Diagrams} would not disapear by naively looking at longer memory Markov chain. 

Another way to significantly reduce the errors would be to develop a different renormalisation scheme to estimate inductively, as $L\to \infty$,
$\proba_{p_c}(0\slr x, d_{\text{chemical}}(0,x)=L).$
This would also provide new interesting asymptotic results and might be useful for checking the initialization.
Both the longer-memory construction and this alternative scheme should significantly reduce the numerical errors, while making the diagrams much more complex in theory.

\textbf{AI statement.} AI was used during this project to make the table of notations, revise english, and  proofread mathematics. Some part of proofs in Sections \ref{Sec:SharpAsymptotics} and \ref{sec:TCL} was also written using AI and were always personally rewritten. (Proposition \ref{pro:MixingTimeFunc}, Lemma \ref{lem:Bulk_L2} adapted from a previous version, Proposition \ref{pro:MiddleRotation}, Lemma \ref{lem:NotAligned}, final proof of Theorem \ref{thm:FT_Low}, Lemmas \ref{lem:V_p,infty_accurate} and \ref{lem:Compare_Green}, final proof of Proposition \ref{pro:Square}.) I also used AI to search for numerical estimates in dimensions $7\leq d\leq 10$, as well as for ways to obtain good non-asymptotic numerical LTCL, without success.
The setup has been VSCode+internal AI, with an external chat notably to finetune AI prompts and skills. 

\textbf{Aknowledgements} I am particularly grateful to Asaf Nachmias, Romain Panis, Tom Hutchcroft, and S\'ebastien Martineau for the fruitfull discussions we got about this project. Without them, I would have stayed stuck doing numerical computations, and this paper would not have been made. 

\subsection{Table of notation} \label{sec:Notation}
\small
\begin{longtable}{p{0.2\textwidth}p{0.7\textwidth}p{0.15\textwidth}}
\caption{Notation used throughout the paper.}\label{tab}\\
\hline
Symbol & \makebox[\linewidth][c]{Explanation} & Location\\
\hline
\endfirsthead
\hline
Symbol & \makebox[\linewidth][c]{Explanation} & Location\\
\hline
\endhead

\multicolumn{3}{c}{\textbf{Percolation model and main quantities}}\\
\hline
$\edges$ & Range-$L$ edge set, $\edges=\{\{x,y\}:0<\|x-y\|_1\leq L\}$. & Section~\ref{subsec:intro-main-results}\\
$\edges_p$ & Random set of $p$-open edges. & Section~\ref{subsec:intro-main-results}\\
$\proba_p,\E_p$ & Bernoulli bond-percolation law and expectation at density $p$. & Section~\ref{subsec:intro-main-results}\\
$\proba,\E$ & Probability, and expectation for the $\sigma$-algebra generated by $(\edges_p)_{0\leq p\leq 1}$. & Section~\ref{sec:Diagrams}\\
$x\slr y$ ; $x\lr y$ & Event that $x$ and $y$ are joined by a $p$-open path. & Section~\ref{subsec:intro-main-results}\\
$x\snlr y$ ; $x\nlr y$ & Complement of the event $x\lr y$ & Section~\ref{subsec:intro-main-results}\\
$\C(x)=\C_p(x)$ & Open cluster of $x$ at parameter $p$. & Section~\ref{subsec:intro-main-results}\\
$\tau_p(x)$ & Two-point function, $\tau_p(x)=\proba_p(0\leftrightarrow x)$. & Section~\ref{subsec:intro-main-results}\\
$\X(p)=\chi(p)$ & Expected cluster size, $\X(p)=\sum_x\tau_p(x)$. & Section~\ref{subsec:intro-main-results}\\
$\nabla(p)$ & Triangle diagram, $\tau_p\ast\tau_p\ast\tau_p(0)-1$. & Section~\ref{subsec:intro-main-results}\\
$\square(p)$ & Square diagram, $\max_x \tau_p\ast \tau_p\ast \tau_p\ast \tau_p(x)$. & Section~\ref{subsec:intro-main-results}\\
$\psi_d$ & Asymptotic order of $\square(p)$ as a function of $\X(p)$. & \eqref{eq:def-psi-d}\\
$\xi_2(p)$ & Radius of gyration, $\xi_2(p)^2=\X(p)^{-1}\sum_x\|x\|_2^2\tau_p(x)$. & \eqref{eq:def-xi2}\\
$\langle x\rangle$ & Regularized norm, $\max(1,\|x\|_2)$. & Section~\ref{subsec:intro-main-results}\\
\hline
\multicolumn{3}{c}{\textbf{Graphs, clusters, and transforms}}\\
\hline
$\vedges$ & Oriented edge set corresponding to $\edges$. & Section~\ref{subsec:intro-main-results}\\
$\vec e=(e^-,e^+)$ & Oriented edge with tail $e^-$ and head $e^+$. & Section~\ref{subsec:intro-main-results}\\
$\Omega$ & Vertex degree of the range-$L$ lattice. & Section~\ref{subsec:intro-main-results}\\
$D$ & Edge kernel, $D(x)=\1_{\{0,x\}\in\edges}$. & Section~\ref{subsec:diagram-main-results}\\
$E \circ F$ & Disjoint occurrence of events $E$ and $F$. & Section~\ref{subsec:intro-main-results}\\
$G\offx{x}S$ & Connected component of $x$ in $G$ after deleting $S$. & Section~\ref{subsec:intro-main-results}\\
$\Comp(x),\Comp$ & Set of finite connected subgraphs containing $x$; $\Comp=\bigcup_x \Comp(x)$. & Section~\ref{subsec:intro-main-results}\\
$\vec\Comp(x),\vec\Comp$ & Set of pointed clusters $(C,\vec e)$ with $C\in\Comp(x)$ and $e^-\in C$. $\vec\Comp=\bigcup_x \vec \Comp(x)$. & Section~\ref{subsec:intro-main-results}\\
$\|f\|_1,\|f\|_\infty,\|f\|_2^2$ & Norms $\sum_x|f(x)|$, $\max_x|f(x)|$, and $\sum_x|f(x)|\|x\|_2^2$. & Section~\ref{subsec:intro-main-results}\\
$\Ft f(\kappa)$& Fourier transform, $\sum_x f(x)e^{\i\langle\kappa,x\rangle}$. & Section~\ref{subsec:intro-main-results}\\
$\Lap_\kappa(f)$ & Laplace transform, $\sum_x e^{\langle\kappa,x\rangle}f(x)$. & \eqref{eq:def-Laplace-transform}\\
\hline
\multicolumn{3}{c}{\textbf{Square approximation and diagrammatic errors}}\\
\hline
$\tau_{p,p',k} (x)$ & Probability that $0$ and $x$ are connected by a $p'$-open path using exactly $k$ edges open at $p'$ but closed at $p$. & Section~\ref{subsec:about-proofs}\\
$\rho=\rho(p,p')$ & $\rho=(p'-p)/(1-p)$. & Section~\ref{sec:DiagramLower}\\
$\tau_p^{\square,k}$ & Approximation for $\tau_{p,p',k}/\rho^k$. & \eqref{eq:def_TauSquare}\\
$\tau^\square_{p,p'}$& Approximation for $\tau_{p'}$ defined by $\tau^\square_{p,p'}=\sum_{k\ge0}\rho^k\tau_p^{\square,k}$. & \eqref{eq:def_TauSquare}\\
$\tilde p_c(p)$ & Approximation for $p_c$ defined by $\tilde p_c(p):= \inf\{p',\|\tau_{p,p'}^\square\|_1=\infty\}$. & Section~\ref{subsec:diagram-main-results}\\
$\bar\tau_{p,p'}$ & $\1_{\{0\}}+\rho(p,p')\tau_{p'}\ast D$. & \eqref{eq:def-bar-tau}\\
$\bar\tau^\square_{p,p'}$&  $\1_{\{0\}}+\rho(p,p')\tau^\square_{p,p'}\ast D$. & \eqref{eq:def-bar-tau}\\
$\square_{p,p'},\Diamond_{p,p'},\twosquare_{p,p'} $ & Square diagram errors. & \eqref{eq:defsquare}--\eqref{eq:deftwosquare}\\
\hline
\multicolumn{3}{c}{\textbf{Off method and Markov chain}}\\
\hline
$(\edges_p^i)_{i\geq0}$ & Independent copies of the $p$-open edge configuration. & Section~\ref{subsec:off-off-method}\\
$\proba_p^i,\E_p^i$ & Law and expectation for the $i$th percolation copy. & Section~\ref{subsec:off-off-method}\\
$\proba_p^\circledast,\E_p^\circledast$ & Product law and expectation for the independent copies. & Section~\ref{subsec:off-off-method}\\
$\C_p^i(x)$ & $p$-open cluster of $x$ in copy $i$. & Section~\ref{subsec:off-off-method}\\
$\B_j((A_i,x_i)_{0\leq i \leq j})$ & Retained part of \(A_j\): $A_j\offx{x_j}\bigcup_{0\leq i<j}\B_i((A_\ell,x_\ell)_{0\leq \ell\leq i})$. & \eqref{eq:defOffOff}\\
$\vec A=(A,\vec a)$ & Pointed cluster state with cluster $A$ and oriented edge $\vec a$. & Section~\ref{sec:Rewriting}\\
$\delta(\vec A)$ & Edge increment, $\delta(\vec A)=a^+-a^-$. & \eqref{eq:def-pointed-increments}\\
$\Delta(\vec A)$ & Spatial increment, $\Delta(\vec A)=a^+$. & \eqref{eq:def-pointed-increments}\\
$\mu_p(\vec A)$ & Pointed-cluster weight, $\mu_p(\vec A)=\proba_p(\C(0)=A)$. & \eqref{eq:def-pointed-increments}\\
$\Ccut$ & Forbidden clusters transitions: \((\vec A,\vec B)\in\Ccut\iff b^-\notin B\offx{0}(A-a^+)\). & \eqref{eq:def-Ccut}\\
$\proba_p^\otimes,\E_p^\otimes$ & Product law and expectation for i.i.d. pointed clusters. & \eqref{eq:def-pointed-cluster-law}\\
$(\vec\A_i)_{i\geq0}$ & I.i.d. random pointed clusters under \(\proba_p^\otimes\), with law \(\mu_p/(\Omega\X(p))\). & \eqref{eq:def-pointed-cluster-law}\\
$\SA_k$ & Markov-chain displacement, $\SA_k=\sum_{i=0}^k\Delta(\vec\A_i)$. & \eqref{eq:def-SA}\\
$\Tcut$ & First killing time $t$ where $(\vec\A_{t-1},\vec\A_t)\in\Ccut$. & \eqref{eq:def-Tcut}\\
$\Ckill,\Clive$ & \(\Ckill=\{(A,\vec a)\in\vec\Comp(0):a^+\in A\}\) and \(\Clive=\vec\Comp(0)\setminus\Ckill\). & Section~\ref{sec:Rewriting}\\
$\Plive_p$ & Live initial mass, $\proba_p^\otimes(\vec\A_0\in\Clive)$. & Section~\ref{sec:Rewriting}\\
$\T_p$ & Killed transition operator: \(\T_p g(\vec A)=\E_p^\otimes[g(\vec\A_2)\1_{(\vec A,\vec\A_2)\notin\Ccut}]\). & \eqref{eq:def-Tp}\\
$\lambda_p,h_p$ & Perron eigenvalue and positive eigenfunction of $\T_p$. & Lemma~\ref{lem:vpd}\\
$\proba_p^h,\E_p^h$ & Law and expectation of the Doob-transformed chain with respect to $\T_p$. & \eqref{eq:DoobTransform}\\
$\P_p,\P_p^k$ & One-step and $k$-step kernels of the transformed chain. & \eqref{eq:MarkovTransition}, \eqref{eq:def-Ppk}\\
$\vpg_p$ & Stationary law of the transformed chain. & Lemma~\ref{lem:StationaryMeasure}\\
$\proba_p^\vpg,\E_p^\vpg$ & Stationary analogue of $\proba_p^h,\E_p^h$ & \eqref{eq:Proba^vpg} \\
$\Gamma_p(\vec A)$ & Mass to deal with the step not handled by the Doob $h_p$-transform & \eqref{eq:def_Gamma_p} \\
\hline
\multicolumn{3}{c}{\textbf{Mixing, tails, and asymptotic estimates}}\\
\hline
$f_r,\Ccut^r,\Tcut^r$ & Local modification, modified cut relation, and its killing time. & \eqref{eq:def_fr}--\eqref{eq:def-Tcutr}\\
$\|f\|_{\infty,r}$, $\|f\|_{1,r}$ & Truncated norms used to separate near and far contributions. & Section~\ref{Sec:1+1=2}\\
$\Xi(r,p,p')$ & Pointwise error factor for comparing $\tau_{p'}$ with $\tau^\square_{p,p'}$. & \eqref{eq:def-Xi-error}\\
$E_{p,k},E_{p,\infty}$ & Contribution of cluster chain of length $k$ to $\|\tau^\square_{p,\tilde p_c(p)}\|_1$, and limit as $k\to \infty$. & Section \ref{sec:p_c-p_c}\\
$V_{i,j,k}(p)$ &Conditional covariance $V_{i,j,k}(p)=\E_p^\otimes[\langle \Delta(\vA_i),\Delta(\vA_j)\rangle |\Tcut>k]$  & Section \ref{sec:SecondMoment}\\
$V_i(p)$ & Stationary covariance $V_{j-i}(p)=\E_p^\vpg[\langle \Delta(\vA_i),\Delta(\vA_j)\rangle]$.  & Section \ref{sec:SecondMoment} \\
$V(p)$ & Average $L_2$ increment per cluster step. $V(p)=\sum_{i\in \Z} V_i$ & Section \ref{sec:LowFrequencies}\\
$\Gauss_{v}(x)$ & Isotropic Gaussian density $\Gauss_{v}(x)=(2\pi v/d)^{-d/2}\exp(-(\|x\|^2_2 d)/(2v))$. & Section~\ref{seq:TCL_square} \\
$\Gauss_{v,\beta}(x)$ & Green function of a killed Brownian motion. $\Gauss_{v,\beta}(x)=\int_0^\infty \beta^t \Gauss_{vt}(x)dt.$ & Section~\ref{seq:TCL_square_2} \\
$C_v^\Gauss$ & Constant such that $\Gauss_{v,1}(x)=C_v^\Gauss\langle x\rangle^{2-d} $ for every $x\neq 0$, & Section~\ref{seq:TCL_square_2}
\end{longtable}
\normalsize
\begingroup
\sloppy
\enlargethispage{1\baselineskip}
\pagebreak 
\tableofcontents
\endgroup
\pagebreak 
\section{Diagrammatic bounds} \label{sec:Diagrams}
\subsection{Main results} \label{subsec:diagram-main-results}
Let $D:x\mapsto \1_{\{0,x\}\in \edges}$, and let $\1_{\{0\}}:x\mapsto \1_{x=0}$. Then let
\begin{equation}
\bar \tau_{p,p'}= \1_{\{0\}}+\frac{p'-p}{1-p}\tau_{p'}\ast D\quad \text{and} \quad  \bar \tau_{p,p'}^\square= \1_{\{0\}}+\frac{p'-p}{1-p} \tau_{p,p'}^\square\ast D.
\label{eq:def-bar-tau}
\end{equation}
We bound the errors in the approximation $\tau_{p'}\approx \tau^\square_{p,p'}$ using the following diagrammatic sums. Their definitions are represented in Figure \ref{fig:SquareDiagram}.
The reader may also look at Figures \ref{fig:SquareDiagramSimpleLower} and \ref{fig:SquareDiagramSimpleUpper}, together with the introduction, to see how those errors can appear. Let
\begin{equation}
\square_{p,p'}:= \tau_p\cdot(\tau_p\ast D\ast \tau^\square_{p,p'} \ast D \ast \tau_p), \label{eq:defsquare}
\end{equation}
\begin{equation}
\Diamond_{p,p'}:= (\tau_p\ast D\ast \bar \tau^\square_{p,p'} \ast \tau_p) \cdot (\tau_p\ast D\ast \bar \tau_{p,p'} \ast \tau_p), \label{eq:deftimessquare}
\end{equation}
\begin{equation}
\twosquare_{p,p'}:z \mapsto \sum_{v,y\in \Z^d} \tau_p(v)  \tau_p(z-v)\tau_p(y-v)[\tau_p \ast D\ast \bar \tau^\square_{p,p'} \ast \tau_p] (y)[\tau_p\ast D\ast\tau_p](z-y). \label{eq:deftwosquare}
\end{equation}
\begin{figure}[!h]
\centering
\includegraphics[scale=0.8]{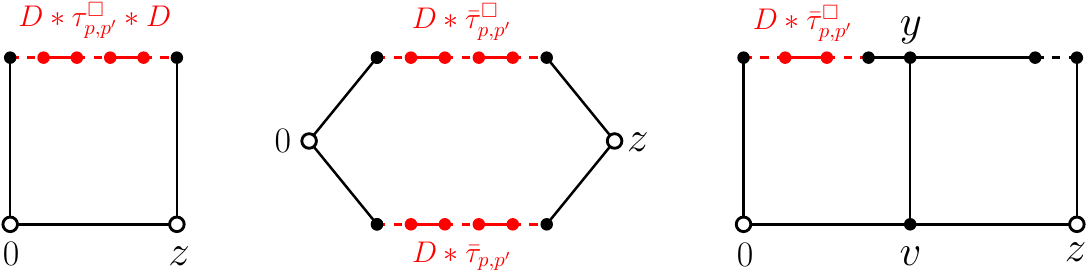}
\caption{\label{Fig:squareDiagram}
The diagrams corresponding to $\square_{p,p'}(z), \Diamond_{p,p'}(z), \twosquare_{p,p'}(z)$ respectively.
Each full segment corresponds to a $p$-open path.
Dashed segments represent edges that are closed at $p$ but open at $p'$.
The red part may consist of any number $0,1,2,\dots$ of $p$-open paths with the corresponding ($p'$-open and $p$-closed) edges in between.
The red part for $\square_{p,p'}$ cannot consist of a single edge.
Consecutive red paths cannot intersect.
} \label{fig:SquareDiagram}
\end{figure}

For $p<p_c$, we set $\tilde p_c(p):= \inf\{p',\|\tau_{p,p'}^\square\|_1=\infty\}$ (see Proposition \ref{pro:PrelimBootstrap} for details).
\begin{theorem} \label{thm:LowerBoundTau}
For every $p\leq p'<\inf(p_c,\tilde p_c(p))$, we have
\[
\tau_{p'}\geq \tau^\square_{p,p'}- ((p'-p)/(1-p))^2\bar \tau_{p,p'}^\square\ast  \tau_p\ast (\square_{p,p'}+\Diamond_{p,p'})\ast \tau_p \ast \bar \tau_{p,p'}^\square.
\]
\end{theorem}
\begin{theorem} \label{thm:UpperBoundTau}
For every $p\leq p'<\inf(p_c,\tilde p_c(p))$, we have:
\[
\tau_{p'}\leq \tau^\square_{p,p'}+((p'-p)/(1-p))^2\bar \tau^\square_{p,p'}\ast  \tau_p \ast  \twosquare_{p,p'} \ast \tau_p \ast \bar \tau_{p,p'}.
\]
\end{theorem}
\begin{remark}
One may notice that the errors in the upper and lower bounds both involve $\tau^\square_{p,p'}$ and $\tau_{p'}$.
However, in many cases, we will only have access to proper bounds on either $\tau_{p'}$ or $\tau^\square_{p,p'}$ but not on the other.
As we will see later, we can actually avoid this issue by using a bootstrap argument.
\end{remark}
\begin{remark}
Contrary to previous works using the lace expansion, we do not need to use an inclusion-exclusion principle, as the errors described in Theorems \ref{thm:LowerBoundTau} and \ref{thm:UpperBoundTau} are already negligible.
\end{remark}
\begin{remark}
The computations show that the order of magnitude of the errors is given by the case where the red parts in Figure \ref{fig:SquareDiagram} contain a minimal number of $p$-open paths (1,0,0,0).
This is the main reason we prefer to look at the diagrams as "square diagrams" instead of larger polygons.
\end{remark}
 \begin{figure}[!h]
\centering
\includegraphics[scale=0.8]{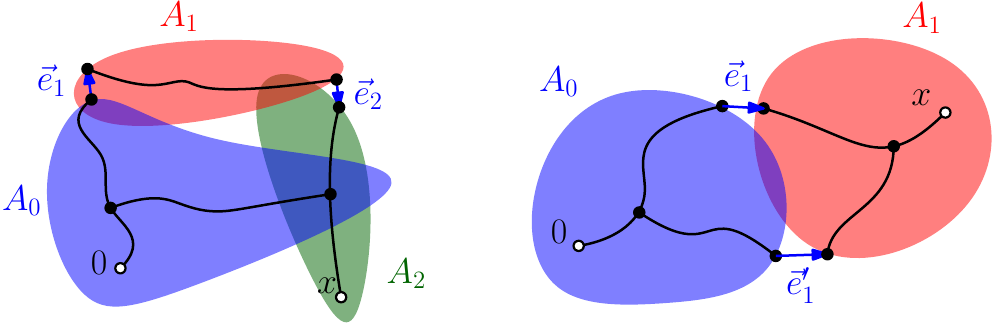}
\caption{We represent the two errors of Theorem \ref{thm:LowerBoundTau} in their simplest case.
The left figure corresponds to $\square_{p,p'}$, and the right corresponds to $\Diamond_{p,p'}$.
The black paths are $p$-open.
The edges $\vec e_1,\vec e_1',\vec e_2,\vec e_3$ are blue and are $p'$-open but $p$-closed.
$0,x$ are white.
$A_0$ is blue, $A_1$ is red, $A_2$ is green.
On the left, $A_0$ cuts the path from $\vec e_2$ to $x$ in $A_2$.
On the right, the choice of $\vec e_1$ is not unique. }
 \label{fig:SquareDiagramSimpleLower}
\includegraphics[scale=1]{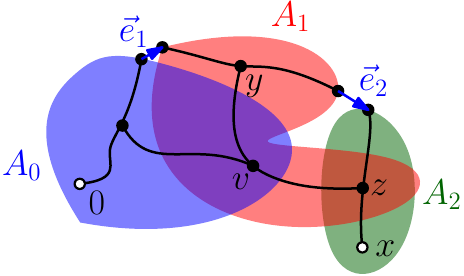}
\caption{We represent the error of Theorem \ref{thm:UpperBoundTau} in its simplest case.
The conventions are the same as in Figure \ref{fig:SquareDiagramSimpleLower}.
$A_1$ should cut the path from $\vec e_2$ to $x$ in $A_2$,
but after $A_0$ cuts $A_1$ it no longer does.}
 \label{fig:SquareDiagramSimpleUpper}
\end{figure}
\subsection{The off-off method} \label{subsec:off-off-method}
We develop here a general version of Aizenman's off method, which will be used in the proofs of Theorems~\ref{thm:LowerBoundTau} and~\ref{thm:UpperBoundTau} in Subsections~\ref{sec:DiagramLower} and~\ref{sec:UpperLower}, respectively.
We first introduce some notation that will be used throughout the section.
Beforehand, recall that we defined clusters as subgraphs of $(\Z^d,\edges)$ and not as simple subsets of $\Z^d$.
This is absolutely crucial in this whole section.

Let $(\edges^i_p)_{i\geq 0}$ be i.i.d. copies of $\edges_p$.
We denote the probability and expectation associated with $\edges^i_p$ by $\proba^i_p$ and $\E^i_p$ respectively.
We also write $\proba^\circledast_p$ for the product measure of $(\proba_p^i)_{i\geq 0}$, and $\E^\circledast_p$ for the corresponding expectation.
We will always specify which set of edges we are working with when we use $\proba^\circledast_p$ and $\E^\circledast_p$.
To avoid any ambiguity, the notations $\proba_p$ and $\E_p$ will only be used when we are working with $\edges_p$. The notations $\proba$ and $\E$ will be reserved for when we work with the sigma-algebra generated by $(\edges_p)_{0\leq p \leq 1}$, with the usual natural increasing coupling.

Also, for every vertex $x\in \Z^d$, let $\C^i(x)=\C_p^i(x)$ denote the cluster (largest connected subgraph) of $x$ in $(\Z^d,\edges^i_p)$.
For every finite family of connected graphs $A_0,A_1,\dots A_k$, and every $(x_i)_{0\leq i \leq k}$ in $\Z^d$, we define inductively $\B_j((A_i,x_i)_{0\leq i \leq k})=\B_j((A_i,x_i)_{0\leq i \leq j})$ such that for every $0\leq j \leq k$,
\begin{equation}
\label{eq:defOffOff} \B_j((A_i,x_i)_{0\leq i \leq j}) = A_j \offx{x_j} \left (\bigcup_{j'=0}^{j-1} \B_{j'}((A_i,x_i)_{0\leq i \leq j'}) \right),
\end{equation}
with the usual convention that $\B_0(A_0)=A_0$ if $x_0\in A_0$, and $\B_0(A_0)=\emptyset$ otherwise.
In general, $\B_j(\cdot)$ takes values in $\{\emptyset\}\cup \Comp$.
The following lemma provides a general way to construct a finite family of clusters conditioned to be disjoint:
\begin{lemma} \label{lem:GeneralOff}
Let $p\in [0,1]$.
Let $k\in \N$.
Let $x_0\neq x_1\neq \dots \neq x_k\in \Z^d$.
Let $(X_i)_{0\leq i \leq k}\in \Comp^{k+1}$.
\begin{compactitem}
\item[(a)] If for every $0\leq i<j \leq k$ we have $X_i\cap X_j=\emptyset$, then
\[
\proba_p( \C(x_0)=X_0,\dots, \C(x_k)=X_k) = \prod_{i=0}^k \proba^i_p \left( \left (\C(x_i)\offx{x_i} (X_0\cup X_1\cup \dots\cup X_{i-1} )\right )=X_i \right).
\]
\item[(b)] Let $\mathbf A=(\C_p^i(x_i),x_i)_{0\leq i \leq k}$. If the $(X_i)_{0\leq i \leq k}$ are as in (a), then
\[
\proba_p( \C(x_0)=X_0,\dots, \C(x_k)=X_k) =\proba_p^\circledast (\forall 0\leq i \leq k, \B_i(\mathbf A)=X_i).
\]
\item[(c)] If for some $i<j$ we have $X_i\cap X_j\neq \emptyset$, then the right-hand sides of (a) and (b) are null.
\end{compactitem}
\end{lemma}
\begin{remark}
Note that nothing is specific to $(\Z^d,\edges)$ in the proof, so the result still holds for Bernoulli bond percolation on any given graph.
\end{remark}
\begin{proof}
For (a), first note that if for some $i$ we have $x_i\notin X_i$, then the left-hand side is null.
The right-hand side is also null, since $(\C(x_i)\offx{x_i} (X_0\cup X_1\cup \dots \cup X_{i-1}))$ must either contain $x_i$ or be the empty graph.
So we may assume without loss of generality that for every $i$ we have $x_i\in X_i$.

Then, for every finite subgraph $G\subset (\Z^d,\edges)$, write $E(G)$ for its set of edges, and write $E^\uparrow (G)$ for the set of edges in $(\Z^d,\edges)$ that are not in $G$ and have at least one endpoint in $G$. We have:
\[
\proba_p(\C(x_0)=X_0,\dots, \C(x_k)=X_k) = p^{\# (E(X_0\cup X_1\cup \dots \cup X_k))} (1-p)^{\#(E^\uparrow(X_0\cup X_1\cup \dots \cup X_k))}.
\]
To compute the right-hand side of (a), observe that the event
	$\C_p^i(x_i) \offx{x_i} (X_0\cup X_1\cup\dots \cup X_{i-1})=X_i$
	is equivalent to the event that all the edges in $E(X_i)$ are open,
	and all the edges in $E^{\uparrow}(X_i)\backslash E^{\uparrow}(X_0\cup X_1\cup\dots \cup X_{i-1})$ are closed.
Thus, using that $X_0, X_1,\dots, X_k$ are disjoint,
\begin{align*}
\prod_{i=0}^k \proba^i_p(\C(x_i)\offx{x_i} (X_0\cup \dots & \cup X_{i-1})=X_i)   =\prod_{i=0}^k p^{\#(E(X_i))}(1-p)^{\#(E^{\uparrow}(X_i)\backslash E^{\uparrow}(X_0\cup X_1\cup\dots \cup X_{i-1}))}
\\ & = p^{\# E(X_0 \cup X_1 \cup\dots \cup X_k)} (1-p)^{\#(E^\uparrow(X_0 \cup X_1\cup \dots \cup X_k))},
\end{align*}
which proves (a).

To prove (b), by \eqref{eq:defOffOff}, we may rewrite the probability on the right-hand side as
\[
\E_p^\circledast \left [\prod_{j=0}^k \1\left (\left (\C_p^j(x_j) \offx{x_j} \bigcup_{i=0}^{j-1} \B_i(\mathbf A) \right)=X_j \right ) \right ].
\]
For $j=k$, using that $\B_i(\mathbf A)=X_i$ for every $i<j$ (otherwise the product is null), we may replace the corresponding indicator by $\1 ( (\C_p^k(x_k) \offx{x_k} \bigcup_{i=0}^{k-1} X_i)=X_k  )$.
We can repeat the same operation for $j=k-1, k-2,\dots, 0$ in this order.
We obtain that the above equals
\[
\E_p^\circledast \left [\prod_{j=0}^k \1\left ( \left (\C_p^j(x_j) \offx{x_j} \bigcup_{i=0}^{j-1} X_i\right)=X_j \right ) \right ],
\]
which is equal to the right-hand side of (a). Thus, (b) follows directly from (a).

To see (c), observe that $\C_p^j(x_j)\offx{x_j} (X_0\cup X_1\cup\dots \cup X_{j-1})$ is defined as the cluster of $x_j$ inside the graph $\C_p^j(x_j)$ after removing all the vertices of $X_0\cup X_1\cup\dots \cup X_{j-1}$.
So it cannot contain any of those vertices.
Thus, $\C_p^j(x_j)\offx{x_j} (X_0\cup X_1\cup\dots \cup X_{j-1})\neq X_j$, and the product in (a) is null.
Since the right-hand sides in (a) and (b) are equal, this concludes the proof.
\end{proof}
\subsection{Lower bound on $\tau_p$: proof of Theorem \ref{thm:LowerBoundTau}} \label{sec:DiagramLower}
Let $p\leq p'<\inf(p_c,\tilde p_c(p))$, let $k\geq 1$, and let $x\in \Z^d$.
In this section, we fix $e_0^+=0$ and $A_{-1}=\emptyset$. Let $\rho=(p'-p)/(1-p)$ be the probability that an edge is $p'$-open conditionally on being $p$-closed. For some oriented edges $(\vec e_i)_{1\leq i \leq k}$,
we denote by $E_x((\vec e_i)_{1\leq i \leq k})$ the $(\edges_p,\edges_{p'})$-measurable event on which the following hold:
\begin{compactitem}
\item[(a)] In $(\Z^d,\edges_p)$ we have $0\slr e_1^-$, $e_1^+\slr e_2^-$, \dots , $e_{k-1}^+\slr e_k^-$, $e_k^+\slr x$.
\item[(b)] The edges $\vec e_1,\dots, \vec e_k$ are $p'$-open but $p$-closed.
\item[(c)] The vertices $0$, $e_1^+$, $e_2^+$, \dots, $e_k^+$ are in distinct $p$-open clusters.
\item[(d)] The tuple $(\vec e_1,\dots, \vec e_k)$ is the only tuple of oriented edges satisfying (a), (b), (c). ($k$ may vary.)
\end{compactitem}
We have for every $k\geq 1$,
\begin{align}
\tau_{p,p',k}(x)  & \geq \sum_{\vec e_1,\dots \vec e_k\in \vedges} \proba(E_{x}((\vec e_i)_{1\leq i \leq k})) \notag
\\ & = \hspace{-2.5em} \sum_{\substack{ \vec e_1,\dots \vec e_k\in \vedges, \\  B_0\in \Comp(e_0^+),\dots, B_k\in \Comp(e_k^+) \\ \forall i\neq j,\, B_i\cap B_j=\emptyset}} \hspace{-2em}
\proba\Big (E_x((\vec e_i)_{1\leq i \leq k}) \Big |\forall 0\leq i\leq k,\, \C_p(e_i^+)=B_i \Big)
\proba\Big(\forall i,\, \C_p(e_i^+)=B_i \Big ).  \label{eq:LowerSplitTauk}
\end{align}
We write $f_x((B_i)_{0\leq i \leq k},(\vec e_i)_{1\leq i \leq k})$ for the first probability in \eqref{eq:LowerSplitTauk}. Let
\[
\mathbf A((e_i^+)_{0\leq i \leq k}):=(\C^i_p(e_i^+),e_i^+)_{0\leq i \leq k}.
\]
By Lemma \ref{lem:GeneralOff} (b),
we may rewrite the product of the two probabilities in \eqref{eq:LowerSplitTauk} as
\begin{equation}
f_x((B_i)_{0\leq i \leq k},(\vec e_i)_{1\leq i \leq k})\proba^\circledast_p(\forall 0\leq i \leq k,\, \B_i(\mathbf A((e_j^+)_{0\leq j \leq k}))=B_i ). \label{eq:LowerSplitTaukB}
\end{equation}

Next, observe that for any fixed $B_0,B_1,\dots, B_k$ as in \eqref{eq:LowerSplitTauk}, the conditional probability, under the event "$\C_p(e_i^+)=B_i$ for every $0\leq i \leq k$", that (a), (b), and (c) hold is given by
\begin{equation}
f^+_x((B_i)_{0\leq i \leq k},(\vec e_i)_{1\leq i \leq k}) := \rho^k \prod_{i=0}^k \1(e_{i+1}^-\in B_i), \label{eq:def:f^+_x}
\end{equation}
with the usual convention $e_{k+1}^-=x$.
Indeed, the edges $\vec e_1,\dots, \vec e_k$ must be $p'$-open and are conditioned to be $p$-closed, and the product comes from (a).
Also note that if (a), (b), (c) hold but (d) does not, then there exist $a<b$, and $u\in B_a$, and $v\in B_b$ such that:
\begin{compactitem}
\item Either $(u,v)\in \edges_{p'}\backslash \edges_p$ and is different up to orientation from $\vec e_1,\vec e_2,\dots, \vec e_k$.
\item Or there exist edges $(u,u'),(v,v')\in \edges_{p'}\backslash \edges_p$
different up to orientation from $\vec e_1,\vec e_2,\dots, \vec e_k$,
and a $p'$-open path from $u'$ to $v'$ that does not intersect $B_0\cup B_1\cup \dots \cup B_k$.
\end{compactitem}
So the conditional probability that (a), (b), and (c) hold but (d) does not is bounded by
\begin{equation}
f^-_x((B_i)_{0\leq i \leq k},(\vec e_i)_{1\leq i \leq k}):= \rho^{k+1} \sum_{0\leq a< b \leq k} S^-(B_a,B_b), \label{eq:LowerDef_f-}
\end{equation}
where for every subgraph $U,V\subset (\Z^d,\edges)$,
\[
S^-(U,V):=\sum_{u\in U,v\in V} \left [D+\rho D\ast \tau_{p'}\ast D  \right ](v-u) = \sum_{u\in U,v\in V} D\ast \bar \tau_{p,p'}(v-u) .
\]
The reason we did not add $\prod_{i=0}^k \1(e_{i+1}^-\in B_i)$ in \eqref{eq:LowerDef_f-} is that we need to force the following replacement events anyway. For every $0\leq i\leq k$, we consider the event
\begin{equation}
e_{i+1}^-\in \left (\C_p^i(e_i^+) \offx{e_i^+} \C_p^{i-1}(e_{i-1}^+) \right), \tag{$E_i(e_{i-1}^+,\vec e_i, e_{i+1}^-)$}
\end{equation}
with conventions $e_{-1}^+=e_0^+=0, e_{k+1}^-=x$, and $\C^{-1}_p(e_{-1}^+)=\emptyset$.
To simplify notation, we write $E_i$ instead of $E_i(e_{i-1}^+,\vec e_i,e_{i+1}^-)$, as there will be no risk of ambiguity.
We deduce that \eqref{eq:LowerSplitTaukB} is at least
\begin{equation*}
(f^+_x-f^-_x)((B_i)_{0\leq i \leq k},(\vec e_i)_{1\leq i \leq k})\proba^\circledast_p(\forall 0\leq i \leq k:\, E_i \text{ and } \B_i(\mathbf A((e_j^+)_{0\leq j \leq k}))=B_i ).
\end{equation*}
Next, by summing the above over all possible clusters $(B_i)_{0\leq i \leq k}$
that $(\B_i(\mathbf A((e_j^+)_{0\leq j \leq k})))_{0\leq i \leq k}$ can take,
we obtain, writing
$\mathbf B((\vec e_j)_{1\leq j \leq k})= ((\B_i(\mathbf A((e_j^+)_{0\leq j \leq k})))_{0\leq i \leq k},(\vec e_i)_{1\leq i \leq k})$, that
\[
\E_p^\circledast [(f^+_x-f_x^-)(\mathbf B((\vec e_j)_{1\leq j \leq k}))  \1(\forall 0\leq i \leq k, E_i) ].
\]
Recalling \eqref{eq:LowerSplitTauk} and \eqref{eq:LowerSplitTaukB}, we finally get the following result:
\begin{proposition} \label{pro:DecomposeThmLowerBoundTau}
For every $k\geq 1$, $p\leq p'<\inf(p_c,\tilde p_c(p))$, $x\in \Z^d$, we have
\[
\tau_{p,p',k}(x)\geq \sum_{\vec e_1,\dots \vec e_k\in \vedges, e_{k+1}^-=x} \E_p^\circledast \left [(f^+_x-f_x^-)(\mathbf B((\vec e_j)_{1\leq j \leq k})) \1(\forall 0\leq i \leq k, E_i) \right ].
\]
\end{proposition}
We now deal with $f^-_x$ and $f^+_x$ separately. We first have:
\begin{proposition} \label{pro:Part1ThmLowerBoundTau}
For every $p\leq p'<\inf(p_c,\tilde p_c(p))$ we have, writing $\rho=(p'-p)/(1-p)$,
\[
\hspace{-3em}\sum_{\substack{k\geq 1, e_{k+1}^-=x \\ \vec e_1,\dots \vec e_k\in \vedges }} \hspace{-1em}
\E_p^\circledast \left [f_x^-(\mathbf B((\vec e_j)_{1\leq j \leq k}))\1(\forall 0\leq i \leq k, E_i) \right ]
\leq \rho^2\bar \tau^\square_{p,p'} \ast \tau_p \ast \Diamond_{p,p'} \ast \tau_p\ast \bar \tau^\square_{p,p'}(x).
\]
\end{proposition}
\begin{proof}
Recall \eqref{eq:LowerDef_f-}, and observe that $S^-$ is increasing in the sense that for every pair of subgraphs $U\subset U', V\subset V'$ we have $S^-(U,V)\leq S^-(U',V')$.
Moreover, recall that by definition, for every $0\leq i \leq k$, $\B_i(\mathbf A((e_j^+)_{0\leq j \leq k}))$ is defined as the cluster of $e_i^+$ inside $\C_p^i(e_i^+)$ after removing some vertices.
It follows for every $0\leq a<b\leq k$ that
\[
S^-(\B_a(\mathbf A((e_j^+)_{0\leq j \leq k})),\B_b(\mathbf A((e_j^+)_{0\leq j \leq k})))\leq S^-(\C_p^a(e_a^+),\C_p^b(e_b^+)) .
\]
Hence, by \eqref{eq:LowerDef_f-}, the sum in the proposition is bounded by
\[
\sum_{0\leq a<b\leq k} \rho^{k+1}\sum_{\substack{\vec e_1,\dots \vec e_k\in \vedges\\ \, e_0^+=0, e_{k+1}^-=x}} \E_p^\circledast \left [S^-(\C_p^a(e_a^+),\C_p^b(e_b^+)) \1(\forall 0\leq i \leq k, E_i) \right ].
\]
Let $S_x^{a,b,k}$ denote the above sum after removing the summation over $a,b,k$ and the power term.

To upper-bound $S_x^{a,b,k}$, we may remove the "off" from the events $\1(E_a)$, $\1(E_{a+1})$, $\1(E_{b})$, $\1(E_{b+1})$. Thus, using that $E_i$ only depends on $\edges^i_p$ and $\edges^{i-1}_p$,
we obtain the upper bound
 \begin{align}
 S^{a,b,k}_x\leq \sum_{\vec e_1,\dots \vec e_k\in \vedges} \proba^\circledast_p(\forall 0\leq i<a, E_i) & \proba^\circledast_p(\forall i\in [a+2,b-1], E_i,e_{a+2}^-\in \C_p^{a+1}(e_{a+1}^+)) \notag \\
\proba^\circledast_p(\forall i\in [b+2,k]: E_i, e_{b+2}^-\in \C_p^{b+1}(e_{b+1}^+)) &\E^\circledast_p\left [ S^{-}(\C_p^a(e_a^+),\C_p^b(e_b^+)) \1_{e_{a+1}^-\in \C_p^a(e_a^+) }\1_{e_{b+1}^-\in \C_p^b(e_b^+) } \right ], \label{eq:BeforeRewriteCross}
 \end{align}
valid whenever $0<a\leq a+2 \leq b <k$.
In general, \eqref{eq:BeforeRewriteCross} still holds with the following modifications: when $a=0$, the first probability is removed; when $a=b-1$, the second is removed; and when $b=k$, the third probability is removed.

Let us rewrite and bound \eqref{eq:BeforeRewriteCross}. We focus on the case $1\leq a\leq a+2\leq b<k$, as the other cases can be treated similarly. We first rewrite the first probability in \eqref{eq:BeforeRewriteCross}.
When $a\geq 1$, we have:
\[
\sum_{\vec e_1,\dots \vec e_{a-1}\in \vedges} \proba^\circledast_p(\forall 0\leq i<a, E_i)
= \sum_{\vec e_1,\dots \vec e_{a-1}\in \vedges}
\proba^\circledast_p \left(\forall 0\leq i<a, e_{i+1}^-\in \left (\C_p^i(e_i^+) \offx{e_i^+} \C_p^{i-1}(e_{i-1}^+)\right )\right ).
\]
Summing over $A_0=\C_p^0(e_0^+)$, $A_1=\C^1_p(e_1^+)$,\dots, $A_{a-1}=\C^{a-1}_p(e_{a-1}^+)$, which are by definition independent clusters, we obtain that the last sum equals
\[
\sum_{\vec e_1,\dots \vec e_{a-1}\in \vedges} \sum_{\substack{ A_0,A_1,\dots, A_{a-1} \in \Comp \\ A_{-1}=\emptyset}}
\prod_{i=0}^{a-1} \proba_p(\C(e_i^+)=A_i) \prod_{i=0}^{a-1} \1\left(e_{i+1}^-\in \left (A_i\offx{e_i^+} A_{i-1}\right )\right ),
\]
which we recognize as $\tau_p^{\square,a-1}(e_a^{-})$. By treating the other probability term in \eqref{eq:BeforeRewriteCross} in exactly the same way, we obtain, when $1\leq a\leq a+2\leq b<k$,
\begin{equation}
S_x^{a,b,k} \leq \sum_{\vec e_a, \vec e_{a+1}, \vec e_b, \vec e_{b+1}}
\tau_p^{\square,a-1}(e_a^{-})
\tau_p^{\square,b-a-2}(e_b^{-}-e_{a+1}^+)
\tau_p^{\square,k-b-1}(x-e_{b+1}^{+} )
\E^\circledast_p\left [ \dots \right ], \label{eq:PartlyRewriteCross}
\end{equation}
where $\E^\circledast_p\left [ \dots \right ]$ is the same expectation as in \eqref{eq:BeforeRewriteCross}, which we now bound.

By definition of $S^-$, this expectation equals
\[
\sum_{u,v\in \Z^d} D\ast \bar \tau_{p,p'}(v-u) \E^\circledast_p \left [ \1_{u,e_{a+1}^-\in \C_p^a(e_a^+)} \1_{v,e_{b+1}^-\in \C_p^b(e_b^+)}\right ].
\]
When $u, e_{a+1}^-\in \C_p^a(e_a^+)$, there exists $w_a\in \Z^d$ such that $(e_a^+ \slr w_a)\circ (w_a\slr e_{a+1}^-)\circ (w_a\slr u)$ occur disjointly in $\edges_p^a$.
We proceed similarly for $\edges_p^b$.
We deduce by the BK inequality that the expectation in \eqref{eq:BeforeRewriteCross} is bounded by
\begin{equation}
\sum_{w_a,w_b,u,v\in \Z^d} D\ast \bar \tau_{p,p'}(v-u)
\tau_p(u-w_a)\tau_p(w_a-e_a^+)\tau_p(e_{a+1}^--w_a)
\tau_p(w_b-e_b^+)\tau_p(e_{b+1}^--w_b)\tau_p(v-w_b),
\label{eq:DoubleBK_ForRewriteCross}
\end{equation}
We recognize $\sum_{u,v} \tau_p(u-w_a)\tau_p(v-w_b)[D\ast \bar \tau_{p,p'}](v-u)=\tau_p\ast D\ast \bar \tau_{p,p'}\ast \tau_p(w_b-w_a)$.
Moreover,
\[
\sum_{\vec e_{a+1},\vec e_b} \tau_p(e_{a+1}^--w_a) \tau_p^{\square,b-a-2}(e_b^{-}-e_{a+1}^+)  \tau_p(e_b^+-w_b)
=\tau_p\ast D\ast \tau^{\square,b-a-2}_p \ast D\ast \tau_p(w_b-w_a).
\]
So by writing for $b\geq a+2$:
\[
\Diamond_{p,p'}^{b-a} := (\tau_p\ast D\ast \tau^{\square,b-a-2}_p \ast D\ast \tau_p) (\tau_p\ast D\ast \bar \tau_{p,p'} \ast \tau_p),
\]
we deduce from the previous bound \eqref{eq:DoubleBK_ForRewriteCross} on the expectation that the sum in \eqref{eq:PartlyRewriteCross} is bounded by
\[
\sum_{\vec e_a, \vec e_{b+1},w_a,w_b}
\tau_p^{\square,a-1}(e_a^{-}) \tau_p(w_a-e_a^{+})
\Diamond_{p,p'}^{b-a}(w_b-w_a)
\tau_p(e_{b+1}^--w_b)\tau_p^{\square,k-b-1}(x-e_{b+1}^{+} ).
\]
Therefore, by rewriting the last sum, we obtain, when $1\leq a\leq a+2\leq b<k$, the bound
\[
S_x^{a,b,k}\leq \tau_p^{\square,a-1}\ast D\ast \tau_p\ast \Diamond_{p,p'}^{b-a} \ast \tau_p\ast D\ast \tau_p^{\square,k-b-1}(x).
\]
A similar calculation shows that the bound still holds in general with the following replacements: when $a=0$, one needs to remove the factor $\tau_p^{\square,a-1}\ast D$.
When $b=k$, one needs to remove the factor $D\ast \tau_p^{\square,k-b-1}(x)$.
When $b=a+1$, one needs to set
$ \Diamond_{p,p'}^{1} :=(\tau_p\ast D \ast \tau_p) (\tau_p\ast D\ast \bar \tau_{p,p'} \ast \tau_p).$
Informally, in the expressions, $\tau_{p}^{\square,-1}(x)\ast D$ is replaced by $\1_{\{0\}}:x\mapsto 1_{x=0}$ whenever it appears.
(If $D$ had an inverse for the convolution, we would set $\tau_p^{\square,-1}$ to be that inverse, but it does not.)

After summing the previous bounds over $a,b,k$, the resulting bound takes a much nicer form. Recall $\rho=(p'-p)/(1-p)$.
With the change of variables $i=b-a$ and $j=k-b$, we get
\begin{align*}
\sum_{\substack{ k\in \N \\ 0\leq a<b \leq k}} \rho^{k+1} S_x^{a,b,k} \leq  \rho^2 & \left (\1_{\{0\}}+\sum_{a= 1}^\infty \rho^a\tau_p^{\square,a-1}\ast D \right )  \ast \tau_p \ast \left ( \sum_{i=1}^\infty \rho^{i-1} \Diamond_{p,p'}^{i} \right )
\\ &  \ast \tau_p \ast \left (\1_{\{0\}}+\sum_{j= 1}^\infty \rho^j \tau_p^{\square,j-1}\ast D \right )(x) .
\end{align*}
We recognize the function inside the first and last parentheses as $\1_{\{0\}}+\rho D\ast \tau_{p,p'}^\square=\bar \tau^\square_{p,p'}$. The middle sum equals
\[
\left ( \tau_p\ast D\ast \left (\1_{\{0\}}+\sum_{i= 2}^\infty  \rho^{i-1} D\ast \tau_p^{\square,i-2} \right )\ast \tau_p  \right )(\tau_p\ast D\ast \bar \tau_{p,p'}\ast \tau_p ),
\]
and we recognize again $\bar \tau^\square_{p,p'}$, then $\Diamond_{p,p'}$. This gives the desired bound.
\end{proof}
\begin{proposition} \label{pro:Part2ThmLowerBoundTau}
For every $p\leq p'<\inf(p_c,\tilde p_c(p))$ we have, writing $\rho=(p'-p)/(1-p)$,
\[
\sum_{\substack{k\geq 1, \\ \vec e_1,\dots \vec e_k\in \vedges, \\ e_{k+1}^-=x  }} \hspace{-1em}
\E_p^\circledast \left [f_x^+(\mathbf B((\vec e_j)_{1\leq j \leq k})) \prod_{i=0}^k \1(E_i) \right ]
\geq [\tau_{p,p'}^\square-\tau_p -\rho^2\bar \tau^\square_{p,p'} \ast \tau_p \ast \square_{p,p'} \ast \tau_p\ast \bar \tau^\square_{p,p'}](x).
\]
\end{proposition}
\begin{proof}
We have (see the paragraph containing \eqref{eq:PartlyRewriteCross})
\[
\sum_{\vec e_1,\dots \vec e_k\in \vedges, e_{k+1}^-=x} \E_p^\circledast \left [ \prod_{i=0}^k \1(E_i) \right ]  = \tau_p^{\square,k}(x).
\]
Summing over $k\geq 1$, we get by definition of $f^+_x$ (see \eqref{eq:def:f^+_x}) that the sum in the proposition equals
\begin{equation}
\sum_{k=1}^\infty \rho^k \tau_p^{\square,k}(x)
- \sum_{\substack{k\geq 1, e_{k+1}^-=x\\ \vec e_1,\dots \vec e_k\in \vedges}}  \rho^k
\E^\circledast_p \left [ \prod_{i=0}^k \1(E_i) \left (1 -  \prod_{i=0}^k \1(e_{i+1}^-\in \B_i(\mathbf A((e_j^+)_{0\leq j \leq k}))) \right ) \right ]. \label{eq:bad_outside_f^+}
\end{equation}
We recognize the first sum as $\tau_{p,p'}^\square(x)-\tau_p(x)$, and it suffices to upper-bound the second sum.

To this end, first note that for $b\in \{0,1\}$, the events $E_b$ and $e_{b+1}^-\in \B_b(\mathbf A((e_j^+)_{0\leq j \leq k}))$ are equivalent.
Let $2\leq b \leq k$.
If $E_b$ holds but $ e_{b+1}^-\notin \B_b(\mathbf A((e_j^+)_{0\leq j \leq k}))$,
then after removing  $\C^{b-1}_p(e_{b-1}^+)$,
the vertices $e_b^+, e_{b+1}^-$ are connected in $(\Z^d,\edges_p^b)$,
but after removing $\bigcup_{a<b} \B_a(\mathbf A((e_j^+)_{0\leq j \leq k}))$,
they are no longer connected.
Also, by definition, for every $a$, we have
$\B_a(\mathbf A((e_j^+)_{0\leq j \leq k})) \subset \C_p^a(e_a^+).$
So there exist $a\leq b-2$ and $v\in \C_p^a(e_a^+)$ such that $(e_b^+\slr v)\circ (v\slr e_{b+1}^-)$ in $(\Z^d,\edges_p^b)$. In other words, by a union bound, the second sum in \eqref{eq:bad_outside_f^+} is bounded by
\[
\sum_{\substack{k\geq 1\\ 2\leq b\leq k \\ 0\leq a \leq b-2}} \rho^k
\sum_{\substack{ \vec e_1,\dots \vec e_k\in \vedges \\ e_{k+1}^-=x}}\sum_{v\in \Z^d}
\E^\circledast_p \left [\1(v\in \C_p^a(e_a^+)) \1\left ((e_b^+\slr v)\circ (v\slr e_{b+1}^-) \text{ in } \edges_p^b \right )   \prod_{i=0}^k \1(E_i)  \right ].
\]
We then proceed as in the last proof. Let $S_x^{a,b,k}$ denote the above sum after removing the summation over $a,b,k$ and the power term.

To upper-bound $S_x^{a,b,k}$, we may remove the "off" from the events $\1(E_a)$, $\1(E_{a+1}), \1(E_b), \1(E_{b+1})$. Thus, using that $E_i$ depends only on $\edges^i_p$ and $\edges^{i-1}_p$, we obtain the upper bound
  \begin{align}
  S^{a,b,k}_x\leq \sum_{\vec e_1,\dots \vec e_k\in \vedges} \sum_{v\in \Z^d} \proba^\circledast_p(\forall 0\leq i<a, E_i) & \proba^\circledast_p(\forall i\in [a+2,b-1], E_i,e_{a+2}^-\in \C_p^{a+1}(e_{a+1}^+)) \notag \\
\proba^\circledast_p(\forall i\in [b+2,k]: E_i, e_{b+2}^-\in \C_p^{b+1}(e_{b+1}^+)) &\proba_p^a(e_a^+\slr e_{a+1}^-\slr v)\proba_p^b((e_b^+\slr v)\circ (v\slr e_{b+1}^-)), \label{eq:BeforeRewriteSquare}
  \end{align}
valid whenever $0\leq a<a+2\leq b\leq k$, with the convention that the first probability is $1$ whenever $a=0$, and the third probability is $1$ whenever $b=k$.
Proceeding exactly as in the proof of \eqref{eq:PartlyRewriteCross}, when $0<a<a+2\leq b<k$, we may rewrite the last sum as
\begin{equation}
\sum_{\vec e_a, \vec e_{a+1}, \vec e_b, \vec e_{b+1}} \sum_{v\in \Z^d}
\tau_p^{\square,a-1}(e_a^{-})
\tau_p^{\square,b-a-2}(e_b^{-}-e_{a+1}^+)
\tau_p^{\square,k-b-1}(x-e_{b+1}^{+})
\proba_p^a(\dots )\proba_p^b(\dots), \label{eq:AfterRewriteSquare}
\end{equation}
where $\proba_p^a(\dots),\proba_p^b(\dots)$ are the last two probabilities in \eqref{eq:BeforeRewriteSquare}.

By the tree graph inequality \cite{MR762034} and the BK inequality, we have
\[
\proba_p^a(e_a^+\slr e_{a+1}^-\slr v) \leq \sum_{w\in \Z^d} \tau_p(w-e_a^+) \tau_p(v-w) \tau_p(e_{a+1}^--w),
\]
and
\[
\proba_p^b((e_b^+\slr v)\circ (v\slr e_{b+1}^-))\leq \tau_p(v-e_b^+)\tau_p(e_{b+1}^--v).
\]
We then write
\[
\square_{p}^{b-a} := [\tau_p\ast D\ast \tau_p^{\square,b-a-2} \ast D\ast \tau_p] \cdot \tau_p.
\]
So we recognize
\[
\sum_{\vec e_{a+1}, \vec e_b} \tau_p(v-w) \tau_p(e_{a+1}^--w)\tau_p^{\square,b-a-2}(e_b^--e_{a+1}^+)\tau_p(v-e_{b}^+) =\square_p^{b-a}(v-w).
\]
Hence, by \eqref{eq:BeforeRewriteSquare}--\eqref{eq:AfterRewriteSquare}, we have
\begin{align*}
S_x^{a,b,k} &  \leq \sum_{\vec e_a,\vec e_{b+1},v,w} \tau_p^{\square,a-1}(e_a^-)\tau_p(w-e_a^+) \square_p^{b-a}(v-w)\tau_p(e_{b+1}^--v)\tau_p^{\square,k-b-1}(x-e_{b+1}^+)
\\ & =\tau_p^{\square,a-1}\ast D\ast \tau_p\ast \square_p^{b-a}\ast \tau_p\ast D\ast \tau_p^{\square,k-b-1} (x).
 \end{align*}
By a similar argument, the bound remains true when either $a=0$ or $b=k$, with the following modifications: when $a=0$, $\tau_p^{\square,a-1}\ast D$ is removed, and when $b=k$, $D\ast \tau_p^{\square,k-b-1}$ is removed.

It remains to sum the bound over $0\leq a<a+2\leq b\leq k$ to bound the second sum in \eqref{eq:bad_outside_f^+}. Recall $\rho=(p'-p)/(1-p)$. With the change of variables $i=b-a$ and $j=k-b$, we have
\begin{align*}
\sum_{\substack{k\geq 1\\ 2\leq b\leq k \\ 0\leq a \leq b-2}} \rho^k S_x^{a,b,k} \leq  \rho^2 & \left (\1_{\{0\}}+\sum_{a= 1}^\infty \rho^a\tau_p^{\square,a-1}\ast D \right )  \ast \tau_p \ast \left ( \sum_{i=2}^\infty \rho^{i-2} \square_p^i \right )  \\
&  \ast \tau_p \ast \left (\1_{\{0\}}+\sum_{j= 1}^\infty \rho^j \tau_p^{\square,j-1}\ast D \right )(x) .
\end{align*}
We recognize the function inside the first and last parentheses as $\1_{\{0\}}+\frac{p'-p}{1-p}D\ast \tau_{p,p'}^\square=\bar \tau^\square_{p,p'}$. The middle sum equals
\[
\left [ \tau_p\ast D\ast \left (\sum_{i= 2}^\infty  \rho^{i-2} \tau_p^{\square,i-2} \right )\ast D\ast \tau_p  \right ]\cdot \tau_p,
\]
and we recognize again $\tau^\square_{p,p'}$, then $\square_{p,p'}$. This gives the desired bound.
\end{proof}
\begin{proof}[Proof of Theorem \ref{thm:LowerBoundTau}.]
We have $\tau_{p'}=\tau_p+\sum_{k\geq 1} \tau_{p,p',k}$, so Theorem \ref{thm:LowerBoundTau} follows by summing Propositions \ref{pro:DecomposeThmLowerBoundTau}, \ref{pro:Part1ThmLowerBoundTau}, and \ref{pro:Part2ThmLowerBoundTau}.
\end{proof}

\subsection{Upper bound on $\tau_p$: proof of Theorem \ref{thm:UpperBoundTau}} \label{sec:UpperLower}
For the upper bound, we sample the clusters of $e_0^+,e_1^+,\dots$ inductively and stop whenever we find a bad event.
This allows us to consider fewer error terms.
We first have to introduce some notation.
Recall that we fix $\rho=(p'-p)/(1-p)$.
For every sequence of clusters $(A_i)_{i\geq 0}$ in $\Comp$, and $(a_i)_{i\geq 1}$ in $\Z^d$, let $\Df_0(A_0)=A_0$, $\Df_1(A_0,A_1,a_1)=A_0\cup (A_1\offx{a_1} A_0)$, and then define inductively for $i\geq 2$,
\[
\Df_i(\{A_j\}_{0\leq j\leq i},\{a_j\}_{1\leq j\leq i}) := \Df_1(\Df_{i-1}(\{A_j\}_{0\leq j< i},\{a_j\}_{1\leq j<i}),A_i,a_i).
\]
We consider the following error terms.
First, for every $p\leq p'<p_c$ and $a,w,x\in \Z^d$, let
\[
\tau_p^w (a,x) :=\proba_p((a\slr w)\circ (w\slr x)),
\]
then let
\[
\tau_{p,p'}^w(a,x):=\tau_p^w(a,x)+\rho \sum_{\vec e \in \vedges} \tau_p^w(a,e^-)\tau_{p'}(x-e^+).
\]
Finally, for every pair of finite subgraphs $A_0,A_1$, and $a_1,a_2,x\in \Z^d$, let
\[
\Xi_{a_2,x}(A_0,A_1,a_1):=\sum_{w\in \Z^d} \1(w\in A_1, w\notin A_1\offx{a_1} A_0)\tau_{p,p'}^w(a_2,x).
\]
The next result is the central piece for the induction.
\begin{proposition} \label{pro:HeredityForUpperTau}
For every $p\leq p'<p_c$, finite subgraphs $A_0,A_1$ of $\Z^d$, vertices $a_1,a_2,x\in \Z^d$,
\begin{align*}
& \proba \left (x\in \C_{p'}(a_2) \offx{a_2} \Df_1(A_0,A_1,a_1) \right )\leq \Xi_{a_2,x}(A_0,A_1,a_1) +\proba(x\in \C_p(a_2) \offx{a_2} (A_0\cup A_1))
\\ & + \rho  \sum_{(A_2,\vec e)\in \vec \Comp(a_2)} \proba(\C_p(a_2)=A_2)  \1_{e^-\in A_2\offx {a_2} (A_0\cup A_1)} \proba \left (x\in \C_{p'}(e^+) \offx{e^+}\Df_2(A_0,A_1,A_2,a_1,a_2) \right ).
\end{align*}
\end{proposition}
\begin{proof}
Write $\Dcr=  \Df_1(A_0,A_1,a_1)$. We may assume $a_2\notin \Dcr$, otherwise the left-hand side is null.

We distinguish between $x\in \C_p(a_2)\offx{a_2} \Dcr$ and $x\in (\C_{p'}(a_2)\offx{a_2} \Dcr)\backslash (\C_p(a_2)\offx{a_2} \Dcr)$. Let $P_1$ be the probability of the first case occurring, and $P_2$ the probability of the second case.
For the first case,
if $x\in \C_p(a_2) \offx{a_2} \Dcr$ but $x\notin \C_p(a_2) \offx{a_2} (A_0\cup A_1)$,
then there must exist a self-avoiding path $\gamma$ from $a_2$ to $x$ in $\C_p(a_2)$
that goes through a vertex $w$ satisfying $w\in A_0\cup A_1$ and $w\notin \Dcr$.
By definition of $\Df_1(A_0,A_1,a_1)$, we must have $w\in A_1$, $w\notin A_1 \offx{a_1} A_0$.
Thus, by a union bound,
\begin{equation}
P_1 \leq \proba(x\in \C_p(a_2)\offx{a_2} (A_0\cup A_1))+\sum_{w\in \Z^d} \1(w\in A_1,w\notin A_1\offx{a_1} A_0) \proba_p((a_2\slr w)\circ (w\slr x)). \label{eq:P1bound_for_UpperDiagram}
\end{equation}

We deal similarly with the second case.
Since $x\in \C_{p'}(a_2) \offx{a_2} \Dcr$, there exists a $p'$-open path $\gamma$ from $a_2$ to $x$ such that $\gamma\cap \Dcr=\emptyset$.
By modifying this path, we may assume that once it leaves $\C_p(a_2)\offx{a_2} \Dcr$, it never enters again.
Let $\vec e$ be the oriented edge where $\gamma$ leaves $\C_p(a_2)\offx{a_2} \Dcr$.
Such an edge exists since $x\notin  \C_p(a_2)\offx{a_2} \Dcr$.
Because $\gamma\cap \Dcr=\emptyset$, we have $e^+\notin \Dcr$, and hence the edge $\vec e$ must be $p$-closed.
Then note that $x \in \C_{p'}(e^+) \offx{e^+} \Df_1(\Dcr,\C_p(a_2),a_2)$ since the part of $\gamma$ from $e^+$ to $x$ avoids both $\Dcr$ and $\C_p(a_2)\offx{a_2} \Dcr$.
Therefore, by a union bound, we have
\begin{equation}
P_2 \leq \rho\sum_{(A_2,\vec e)\in \vec \Comp(a_2)}  \hspace{-1em} \proba(\C_p(a_2)=A_2)  \1_{e^-\in A_2\offx {a_2} \Dcr} \proba \left (x\in \C_{p'}(e^+) \offx{e^+} \Df_1(\Dcr,A_2,a_2) \right ),\label{eq:P2bound_for_UpperDiagram}
\end{equation}
where the term $\rho=(p'-p)/(1-p)$ is for the probability that $\vec e$ is $p'$-open given that it is $p$-closed.

It remains to upper-bound the analogous sum in which we also add the factor $\1_{e^-\notin A_2\offx{a_2} (A_0\cup A_1)}$.
Let $S_{\text{bad}}$ stand for this sum.
By increasing the sum, we may replace the rightmost probability by $\proba_{p'} (e^+\slr x)=\tau_{p'}(x-e^+)$.
Then, exactly as in the first part of the proof, there must be $w\in \Z^d$ such that $w\in A_1$, $w\notin A_1\offx{a_1} A_0$, and $(a_2\slr w)\circ (w\slr e^-)$ in $A_2$.
Thus,
\[
S_{\text{bad}} \leq \sum_{(A_2,\vec e)\in \vec \Comp(a_2)} \sum_{w\in \Z^d} \proba(\C_p(a_2)=A_2) \1_{w\in A_1, w\notin A_1\offx{a_1} A_0}\1_{(a_2\slr w)\circ (w\slr e^-) \text{ in } A_2}\tau_{p'}(x-e^+) .
\]
Summing first over $A_2$, we may rewrite the last bound as
\begin{equation}
S_{\text{bad}} \leq \sum_{\vec e \in \vedges} \sum_{w\in \Z^d} \1(w\in A_1, w\notin A_1\offx{a_1} A_0) \tau_p^w(a_2,e^-) \tau_{p'}(x-e^+). \label{eq:Sbadbound_for_UpperDiagram}
\end{equation}
Summing \eqref{eq:P1bound_for_UpperDiagram}, \eqref{eq:P2bound_for_UpperDiagram}, and $\rho\cdot$\eqref{eq:Sbadbound_for_UpperDiagram} yields the desired result.
\end{proof}
To perform the induction for Theorem \ref{thm:UpperBoundTau}, we have to introduce some notation and conventions.
As usual, $e_0^+=0$ and $A_{-1}=\emptyset$.
We will drop the dependence on $x,p,p'$ in some notation, as they are fixed throughout the proof.
For brevity, for $k\geq 0$, we write $\Abf_k=(\{A_i\}_{0\leq i <k}, \{\vec e_i\}_{1\leq i \leq k})$.
Let
\[
\Df_k(\Abf_k,A_k)=\Df_k(\{A_i\}_{0\leq i \leq k},\{e^+_i\}_{1\leq i \leq k}).
\]
By convention, we let $\Df_{-1}(\cdot)=\emptyset$ and $\Df_{-2}(\cdot)=\emptyset$. To simplify the notation, we will write $\Dbf_k$ instead of $\Df_k(\Abf_k,A_k)$, as there will be no risk of ambiguity.
We then let for $k\geq 1$,
 \[
 \Xi_k(\Abf_k)=\Xi_{e_k^+,x}(\Dbf_{k-2},A_{k-1},e_{k-1}^+) .
 \]
By convention, $\Xi_0(\cdot)=0$. Also note that $\Xi_1(\cdot)=0$ since $\Dbf_{-1}=\emptyset$. For $k\geq 0$, let
 \[
 P_k(\Abf_k)=\proba(x\in \C_p(e_k^+)\offx{e_k^+} (\Dbf_{k-2} \cup A_{k-1})) .
 \]
Let $\Pi_0(\cdot)=1$. For $k\geq 1$ let
 \[
 \Pi_k(\Abf_k)= \rho^{k}\prod_{i=0}^{k-1}\proba(\C_p(e_i^+)=A_i)\prod_{i=0}^{k-1}\1(e_{i+1}^-\in A_i \offx{e_i^+} (\Dbf_{i-2} \cup A_{i-1})) .
 \]
Finally, for $k\geq 0$, let
 \[
 R_k(\Abf_k) =\proba \big (x\in \C_{p'}(e_{k}^+) \offx{e_{k}^+ }\Dbf_{k-1} \big).
 \]

By definition of $\Df$, we have $\Df_1(\Dbf_{k-2},A_{k-1},e_{k-1}^+)=\Dbf_{k-1}$. Similarly,
\[
\Df_2(\Dbf_{k-2},A_{k-1},A_k,e_{k-1}^+,e_k^+)= \Df_1(\Df_1(\Dbf_{k-2},A_{k-1},e_{k-1}^+),A_k,e_k^+)=\Dbf_k.
\]
So by Proposition \ref{pro:HeredityForUpperTau}, we have for every $k\geq 1$, $A_0,A_1,\dots, A_{k-1} \in \Comp$, $\vec e_1,\vec e_2,\dots, \vec e_k\in \vedges$,
\begin{align*}
& \proba \left (x\in \C_{p'}(e_k^+) \offx{e_k^+} \Dbf_{k-1} \right )\leq \Xi_{e_k^+,x}(\Dbf_{k-2},A_{k-1},e^+_{k-1}) +\proba(x\in \C_p(e_k^+) \offx{e_k^+} (\Dbf_{k-2}\cup A_{k-1}))
\\ & + \rho  \sum_{\substack{ A_k\in \Comp\\\vec e_{k+1}\in \vedges}} \proba(\C_p(e_k^+)=A_k)  \1_{e_{k+1}^-\in A_k\offx{e_k^+} (\Dbf_{k-2}\cup A_{k-1})} \proba \left (x\in \C_{p'}(e_{k+1}^+) \offx{e^+_{k+1}} \Dbf_{k} \right ),
\end{align*}
We may further rewrite the above inequality as
\[
R_k(\Abf_k) \leq \Xi_k(\Abf_k) + P_k(\Abf_k)+\sum_{(A_k,\vec e_{k+1})\in \vec \Comp} \frac{\Pi_{k+1}(\Abf_{k+1})}{\Pi_k(\Abf_k)} R_{k+1}(\Abf_{k+1}).
\]
By multiplying by $\Pi_k(\Abf_k)$, and summing over $A_0,\dots, A_{k-1}\in \Comp,\vec e_1,\dots, \vec e_k \in \vedges$, which we simply write as a sum over $\Abf_k$, we get, for every $k\geq 1$,
\begin{equation}
\sum_{\Abf_{k}} \Pi_k(\Abf_k) R_k(\Abf_k) \leq \sum_{\Abf_k} \Pi_k(\Abf_k)\Xi_k(\Abf_k) + \sum_{\Abf_{k}} \Pi_k(\Abf_k) P_k(\Abf_k)+\sum_{\Abf_{k+1}} \Pi_{k+1}(\Abf_{k+1}) R_{k+1}(\Abf_{k+1}). \label{eq:HeredityForUpperBound}
\end{equation}
The inequality is still valid for $k=0$ with our conventions. Indeed, by a union bound,
\begin{align*}
\proba_{p'}(0\slr x) & \leq \proba_p(0\slr x)+ \sum_{(A_0,\vec e_1)\in \vec \Comp} \proba(\C_p(0)=A_0,\,  \vec e_1\in \edges_{p'}\backslash \edges_{p},\, x\in \C_{p'}(e_1^+) \offx{e_1^+} A_0).
\\ & =\proba_p(0\slr x)+\sum_{(A_0,\vec e_1)\in \vec \Comp} \proba(\C_p(0)=A_0)\rho\proba(x\in \C_{p'}(e_1^+) \offx{e_1^+} A_0).
\end{align*}
Next, note that \eqref{eq:HeredityForUpperBound} is of the form $u_k\leq v_k+u_{k+1}$, where $(u_k)_{k\geq 0}$ and $(v_k)_{k\geq 0}$ are nonnegative.
For such sequences, we have, by induction, for every $K\geq 0$, $u_0\leq \sum_{0\leq k\leq K} v_k+u_{K+1}$.
In other words,
\[
\proba_{p'}(0\slr x)\leq \sum_{\substack{0\leq k \leq K \\ \Abf_k}} \Pi_k(\Abf_k)\Xi_k(\Abf_k) + \sum_{\substack{0\leq k \leq K \\ \Abf_k}} \Pi_k(\Abf_k) P_k(\Abf_k)+\sum_{\Abf_{K+1}} \Pi_{K+1}(\Abf_{K+1}) R_{K+1}(\Abf_{K+1}).
\]
We write $S_\Xi$ and $S_P$ for the limits (possibly infinite) as $K\to \infty$ of the first two sums.
We can take the limit because we are summing non-negative terms.
We write $S_R(K)$ for the third sum.
The second sum is the main term, corresponding to $\tau_{p,p'}^\square$.
The first sum is the error term.
The last sum is the remainder that goes to $0$ as we perform more steps of the induction.
More precisely, Theorem \ref{thm:UpperBoundTau} will follow by summing the next three results:
\begin{lemma}
If $p\leq p'<\inf(p_c,\tilde p_c(p))$, then as $k\to \infty$ we have $S_R(k)\to 0$.
\end{lemma}
\begin{proof}
$R_{k+1}(\Abf_{k+1})$ is a probability, so we bound it by $1$. Thus, it suffices to bound
\[
\rho^{k+1}\sum_{\substack{ A_0,A_1,\dots, A_k \in \Comp \\ \vec e_1,\dots, \vec e_{k+1}\in \vedges}} \prod_{i=0}^k \proba(\C_p(e_i^+)=A_i)\prod_{i=0}^k\1(e_{i+1}^-\in A_i \offx{e_i^+} (A_{i-1}\cup \Dbf_{i-2})).
\]
By erasing $\Dbf_{i-2}$, we recognize $\tau_p^{\square,k}$ in \eqref{eq:def_TauSquare}. More precisely, the above sum is bounded by
\[
\rho^{k+1} \sum_{\vec e_{k+1}\in \vedges} \tau_{p}^{\square,k}(e_{k+1}^-) =\Omega \rho^{k+1} \|\tau_p^{\square,k}\|_1 .
\]
By definition of $\tilde p_c(p)$, we have $\|\tau_{p,p'}^\square\|_1=\sum_{k=0}^\infty \rho^k \|\tau_p^{\square,k}\|_1<\infty$. So $S_R(k)\to 0$.
\end{proof}
\begin{lemma}
If $p\leq p'<\inf(p_c,\tilde p_c(p))$, then $S_P\leq \tau_{p,p'}^\square(x)$.
\end{lemma}
\begin{proof}
Dropping the $\Dbf_{k-2}$ in the definition of $P_k$, we have
\[
P_k(\Abf_k)\leq \proba(x\in \C_p(e_k^+)\offx{e_k^+} A_{k-1})=\sum_{A_k\in \Comp, e_{k+1}^-=x} \proba(\C_p(e_k^+)=A_k)\1(e_{k+1}^-\in A_k\offx{e_k^+}A_{k-1}).
\]
So
\[
\sum_{\Abf_k}\Pi_k(\Abf_k)  P_k(\Abf_k) \leq \rho^{k}  \sum_{\substack{ A_0,A_1,\dots, A_k \in \Comp \\ \vec e_1,\dots, \vec e_k\in \vedges, e_{k+1}^-=x}}
\prod_{i=0}^{k}\proba(\C_p(e_i^+)=A_i)\prod_{i=0}^{k}\1(e_{i+1}^-\in A_i \offx{e_i^+} A_{i-1}).
\]
We recognize the above sum as $\tau_p^{\square,k}(x)$ in \eqref{eq:def_TauSquare}. By summing over $k$, we get
$ S_P \leq  \sum_{k\geq 0} \rho^k \tau_p^{\square,k}(x) =\tau_{p,p'}^\square(x).$
\end{proof}
\begin{proposition}
If $p\leq p'<\inf(p_c,\tilde p_c(p))$, we have
\[
S_\Xi\leq ((p'-p)/(1-p))^2\bar \tau^\square_{p,p'}\ast  \tau_p \ast  \twosquare_{p,p'} \ast \tau_p \ast \bar \tau_{p,p'}(x).
\]
\end{proposition}
\begin{proof}
We first bound $\Xi_k(\Abf_k)$. Recall that $\Xi_0(\cdot)=0$ and $\Xi_1(\cdot)=0$. We may restrict to the case
\[
e_{k}^-\in A_{k-1} \offx{e^+_{k-1}} (A_{k-2}\cup \Dbf_{k-3})
\]
since otherwise $\Pi_k(\Abf_k)=0$. For $k\geq 2$, by definition,
 \begin{align*}
 \Xi_k(\Abf_k) & =\Xi_{e_k^+,x}(\Dbf_{k-2},A_{k-1},e_{k-1}^+),
 \\ & = \sum_{w\in \Z^d} \1(w\in A_{k-1}, w\notin A_{k-1}\offx{e_{k-1}^+} \Dbf_{k-2})\tau_{p,p'}^w(e_k^+,x)
 \end{align*}
Assume that for some $w\in \Z^d$ we have $w\in A_{k-1}$ but $w\notin A_{k-1}\offx{e_{k-1}^+} \Dbf_{k-2}$.
Then we may consider in $A_{k-1}$ a self-avoiding path $\gamma_1$ from $e_{k-1}^+$ to $e_k^-$ that avoids $\Dbf_{k-2} \subset A_{k-2} \cup \Dbf_{k-3}$.
Then consider another self-avoiding path $\gamma_2$ in $A_{k-1}$
starting from $y\in \gamma_1$ and ending at $w$
and such that $\gamma_2$ intersects $\gamma_1$ only at $y$.
$\gamma_2$ must intersect $\Dbf_{k-2}$ at some vertex $v$
since otherwise we would have $w\in A_{k-1}\offx{e_{k-1}^+} \Dbf_{k-2}$.
Thus, $ \Xi_k(\Abf_k)$ is bounded by
\[
\sum_{y,v,w\in \Z^d} \1(v\in  \Dbf_{k-2}) \1((e_{k-1}^+\slr y)\circ (y\slr e_k^-)\circ (y\slr v)\circ (v\slr w)\text{ in } A_{k-1})\tau_{p,p'}^w(e_k^+,x).
\]
Hence, since, by definition, $ \Dbf_{k-2}\subset \bigcup_{i=0}^{k-2} A_i$, we get $\Xi_k(\Abf_k)\leq \sum_{i=0}^{k-2} \Xi_{i,k}(\Abf_k),$
where
\begin{equation}
\Xi_{i,k}(\Abf_k) := \sum_{y,v,w\in \Z^d} \1(v\in A_i) \1((e_{k-1}^+\slr y)\circ (y\slr e_k^-)\circ (y\slr v)\circ (v\slr w)\text{ in } A_{k-1})\tau_{p,p'}^w(e_k^+,x). \label{eq:def_Xi_i,k}
\end{equation}
Note already that $\Xi_{i,k}(\Abf_k)$ depends only on $A_i,A_{k-1}, e_{k-1}^+$ and $\vec e_k$.

Next, for every $0\leq i \leq k-2$, by dropping the $\Dbf_{\cdot}(\cdot)$ in the definition of $\Pi_k$, and by also dropping the off in the indicator for $j\in \{i, i+1,k-1\}$, we get
\[
\Pi_k(\Abf_k) \Xi_{i,k}(\Abf_k)\leq \rho^k \Pi_{i,k}^1(\Abf_k)\Pi_{i,k}^2(\Abf_k)\Pi_{i,k}^3(\Abf_k),
\]
where the different terms of the product are defined below. Let
\[
\Pi_{i,k}^1(\Abf_k) :=  \prod_{j=0}^{i-1} \proba(\C_p(e_j^+)=A_j) \prod_{j=0}^{i-1} \1(e_{j+1}^-\in A_j \offx{e_j^+} A_{j-1}),
\]
with the convention $\Pi^1_{0,k}(\cdot)=1$. For $i<k-3$ let
\[
\Pi_{i,k}^2(\Abf_k) := \prod_{j=i+1}^{k-2} \proba(\C_p(e_j^+)=A_j) \1(e_{i+2}^-\in A_{i+1})\prod_{j=i+2}^{k-2} \1(e_{j+1}^-\in A_j \offx{e_j^+} A_{j-1}).
\]
Let $\Pi^2_{k-2,k}(\cdot)=1$, and let $\Pi^2_{k-3,k}(\Abf_k)= \proba(\C_p(e_{k-2}^+)=A_{k-2}) \1(e_{k-1}^-\in A_{k-2})$. Then let
\[
\Pi_{i,k}^3(\Abf_k) :=  \Xi_{i,k}(\Abf_k) \proba(\C_p(e_i^+)=A_i) \1(e_{i+1}^-\in A_i)\proba(\C_p(e_{k-1}^+)=A_{k-1})\1(e_{k}^-\in A_{k-1}).
\]
Note the following dependencies:
\begin{compactitem}
\item $\Pi_{i,k}^1(\Abf_k)$ depends only on $\{A_j\}_{j<i}$ and $\{\vec e_j\}_{1\leq j\leq i}$.
\item $\Pi_{i,k}^2(\Abf_k)$ depends only on $\{A_j\}_{i<j<k-1}$ and $\{\vec e_j\}_{i< j \leq k-1}$.
\item $\Pi_{i,k}^3(\Abf_k)$ depends only on $\Abf_{i,k}:=(A_i, A_{k-1}, e_{i}^+, \vec e_{i+1}, e_{k-1}^+, \vec e_{k})$.
\end{compactitem}
These sets of variables are essentially disjoint, so we may deal with each product separately. First, if $i\geq 1$, then for every $\vec e_i$, by simply recognizing \eqref{eq:def_TauSquare},
 \[
 \sum_{\substack{ \{A_j\}_{j<i}, \\ \{\vec e_j\}_{1\leq j < i}}} \Pi^1_{i,k}(\{A_j\}_{j<i},\{\vec e_j\}_{1\leq j\leq i})= \tau_p^{\square,i-1}(e_i^-).
 \]
Similarly, if
$i<k-2$, for every $\vec e_{i+1}, \vec e_{k-1}$,
  \[
  \sum_{\substack{ \{A_j\}_{i<j<k-1},\\ \{\vec e_j\}_{i+1< j<k-1}} } \Pi^2_{i,k}(\{A_j\}_{i<j<k-1},\{\vec e_j\}_{i< j\leq k-1})= \tau_p^{\square,k-i-3}(e_{k-1}^--e_{i+1}^+ ),
  \]
  by recognizing \eqref{eq:def_TauSquare} after changing the summation index and translating everything by $-e_{i+1}^+$. Therefore, whenever $1\leq i<k-2$,
\begin{equation}
\sum_{\Abf_k} \Pi_k(\Abf_k)\Xi_{i,k}(\Abf_k)\leq \rho^k \sum_{\Abf_{i,k}} [D\ast \tau_p^{\square,i-1}] (e_i^+) [D\ast \tau_p^{\square,k-i-3}] (e_{k-1}^+-e_{i+1}^+)  \Pi^3_{i,k} (\Abf_{i,k}).\label{eq:PartialSumTwoSquare}
\end{equation}
The bound \eqref{eq:PartialSumTwoSquare} still holds when $i=0$ or $i=k-2$ by replacing $D\ast \tau_p^{\square,-1}(\cdot)$ by $\1_{\{0\}}$ when it appears.

We then bound the sum in \eqref{eq:PartialSumTwoSquare}. We first sum only over $A_i, A_{k-1}$ while fixing the other variables. We have
\begin{align*}
\sum_{A_i, A_{k-1}} \Pi^3_{i,k}(\Abf_{i,k}) & = \sum_{A_i, A_{k-1}}\sum_{y,v,w\in \Z^d} \proba(\C_p(e_i^+)=A_i) \1( v,e_{i+1}^-\in A_i)  \proba(\C_p(e_{k-1}^+) =A_{k-1} )
\\ &\hspace{3em} *\1((e_{k-1}^+\slr y)\circ (y\slr e_k^-)\circ (y\slr v)\circ (v\slr w)\text{ in } A_{k-1})  \tau_{p,p'}^w(e_k^+,x)
\\ & = \sum_{y,v,w}  \proba_p(e_i^+\slr e_{i+1}^-\slr v) \proba_p(e_{k-1}^+\slr y \circ y\slr e_k^- \circ y\slr v \circ v\slr w) \tau_{p,p'}^w(e_k^+,x).
 \end{align*}
Thus, by the BK inequality and the tree graph inequality \cite{MR762034}, the last sum is bounded by
 \[
 \sum_{y,v,w,w'} \tau_p(w'-e_i^+) \tau_p(e_{i+1}^--w')\tau_p(v-w') \tau_p(e_{k-1}^+-y)\tau_p(e_k^--y) \tau_p(v-y)\tau_p(w-v) \tau_{p,p'}^w(e_k^+,x).
 \]
Also, by the BK inequality, we have for every $x'\in \Z^d$, $\tau_{p}^w(e_k^+,x') \leq \tau_p(w-e_k^+)\tau_p(x'-w)$. So
 \begin{align*}
 \tau_{p,p'}^w(e_k^+,x)&  \leq \tau_p(w-e_k^+) \tau_p(x-w) + \rho \tau_p(w-e_k^+) \sum_{\vec e\in \vedges} \tau_p(e^--w)\tau_{p'}(x-e^+)
 \\ & =\tau_p(w-e_k^+) [\tau_p\ast \bar \tau_{p,p'}](x-w).
 \end{align*}
Reorganizing the terms, we thus get that the sum in \eqref{eq:PartialSumTwoSquare} is bounded by
\begin{align*}
& \sum_{w,w'}  [\tau_p\ast \bar \tau_{p,p'}](x-w) \Bigg( \sum_{e_i^+} [D\ast \tau_p^{\square,i-1}](e_i^+)\tau_p(w'-e_i^+) \Bigg ) \sum_{y,v} \Bigg (\sum_{\vec e_k} \tau_p(e_k^--y)\tau_p(w-e_k^+) \Bigg  ) \\
 &  \Big (\sum_{\vec e_{i+1}, e_{k-1}^+} \tau_p(e_{i+1}^--w') [D\ast \tau_p^{\square,k-i-3}] (e_{k-1}^+-e_{i+1}^+) \tau_p(y-e_{k-1}^+) \Big )
\tau_p(y-v)\tau_p(w-v)\tau_p(v-w'),
\end{align*}
keeping the convention that $D\ast \tau_p^{\square,-1}$ must be replaced by $\1_{\{0\}}$  whenever it appears.
We recognize the first parenthesis as $D\ast \tau_p^{\square,i-1}\ast \tau_p(w')$.
The second parenthesis is $\tau_p\ast D\ast \tau_p(w-y)$.
The third parenthesis is $\tau_p\ast D\ast \tau_p^{\square,k-i-3} \ast D \ast \tau_p(y-w')$.
This motivates us to define, for $a\geq -1$,
\[
\twosquare_p^{a}(w-w'):=\sum_{y,v} [\tau_p\ast D\ast \tau_p^{\square,a} \ast D \ast \tau_p](y-w')[\tau_p\ast D\ast \tau_p](w-y)\tau_p(y-v)\tau_p(w-v)\tau_p(v-w'),
\]
again removing $D\ast \tau_p^{\square,-1}$ when $a=-1$. The sum in \eqref{eq:PartialSumTwoSquare} is then bounded by
\[
\sum_{w,w'} D\ast \tau_p^{\square,i-1}\ast \tau_p(w') \twosquare_p^{k-i-3}(w-w') \tau_p \ast \bar \tau_{p,p'}(x-w) = D\ast \tau_p^{\square,i-1}\ast \tau_p \ast  \twosquare_p^{k-i-3} \ast \tau_p \ast \bar \tau_{p,p'}(x).
\]

It remains to sum the bound over all $k\geq 2$ and $0\leq i \leq k-2$. We have
\begin{align*}
S_\Xi=\sum_{\substack{ k\geq 2,\\  \Abf_k}} \Pi_k(\Abf_k)\Xi_k(\Abf_k) & \leq \sum_{\substack{ k\geq 2, \Abf_k \\  0\leq i \leq k-2 }}  \Pi_k(\Abf_k)\Xi_{i,k} (\Abf_k)
\\ & \leq \sum_{\substack{ k\geq 2,\\  0\leq i \leq k-2}} \rho^k D\ast \tau_p^{\square,i-1}\ast \tau_p \ast  \twosquare_p^{k-i-3} \ast \tau_p \ast \bar \tau_{p,p'}(x).
\end{align*}
With the change of variables $a=k-i-3$, $S_\Xi$ is bounded by
\[
\rho^2 \left ( \1_{\{0\}}+D\ast \sum_{i=1}^\infty   \rho^i \tau_p^{\square,i-1} \right )\ast \tau_p\ast \left (\sum_{a=-1}^\infty \rho^{a+1}  \twosquare_p^{a} \right)  \ast \tau_p \ast \bar \tau_{p,p'}(x).
\]
Recalling $\rho=(p'-p)/(1-p)$, we recognize the first parenthesis as $\1_{\{0\}}+ \rho D\ast \tau_{p,p'}^\square=\bar \tau_{p,p'}^\square$. The function in the second parenthesis sends $w\in \Z^d$ to
\[
\sum_{y,v} \left [\tau_p\ast \left (\1_{\{0\}}+D\ast  \sum_{a\geq 0}^\infty \rho^{a+1} \tau_p^{\square,a} \right )\ast D \ast \tau_p \right ](y)[\tau_p\ast D\ast \tau_p](w-y)\tau_p(y-v)\tau_p(w-v)\tau_p(v).
\]
We recognize again $\bar \tau_{p,p'}^\square$ and then $\twosquare_{p,p'}$. Therefore,
\[
S_\Xi \leq \rho^2 \bar \tau_{p,p'}^\square \ast \tau_p \ast \twosquare_{p,p'}\ast \tau_p\ast \bar \tau_{p,p'}(x).\qedhere
\]
\end{proof}
\section{Markov chain interpretation} \label{sec:MarkovChain}
In Subsection~\ref{sec:Rewriting}, we rewrite $\tau_p^{\square,k}$ using a Markov chain of pointed clusters, leaving some proofs for Subsection~\ref{sec:SpectralConstruction}.
One object that appears in this rewriting and becomes increasingly relevant as the paper progresses is the maximal eigenvalue $\lambda_p$.
The goal of Subsection~\ref{sec:LowerLambda} is to prove a lower bound on $\lambda_p$ under the condition that \eqref{eq:initialization} holds at $p$.
We then show in Subsection~\ref{sec:ExponentialMixing} that, when \eqref{eq:initialization} holds, the Markov chain mixes exponentially fast.
Finally, Subsection~\ref{Sec:HighFrequencies} gives a high-frequency bound for $\Four[\tau_{p}^{\square,k}]$.
This subsection will only be relevant in Section~\ref{sec:TCL} for proving an LTCL for $\tau_{p}^{\square,k}$ and is conveniently placed here because it shares many tools with the rest of the section.
\subsection{Rewriting the definition of $\tau_p^{\square,k}$} \label{sec:Rewriting}
The main goal of this section is to give a Markov chain interpretation of the definition \eqref{eq:def_TauSquare}, and to describe several properties of this Markov chain.
We start by rewriting \eqref{eq:def_TauSquare} in a more suitable form.
First, we add an auxiliary oriented edge $\vec e_{k+1}=(x,e_{k+1}^+)$, and multiply by $1/\Omega$ to compensate.
(Recall that $\Omega$ is the degree of a vertex.) Then we translate each cluster $A_i$ by $-e_i^+$, so that all clusters are now in $\Comp(0)$.
Finally, we perform the change of variables $\vec e_{i-1} \gets \vec e_i-e_{i-1}^+$.
By doing so, the condition $e_{k+1}^-=x$ becomes  $\sum_{i=0}^k e_{i}^+ + e_{k}^- -e_k^+=x$.
Thus, we may rewrite \eqref{eq:def_TauSquare} as
\[
\frac{1}{\Omega} \hspace{-1em}
\sum_{\substack{ \vec e_0,\dots, \vec e_{k}\in \vedges \\  A_0,\dots, A_k\in \Comp(0) \\ e_{-1}^+=0, A_{-1}=\emptyset } } \hspace{-1em}
\1\left (\sum_{i=0}^k e_{i}^+ + e_{k}^- -e_k^+=x\right)
\prod_{i=0}^k \proba_p(\C(0)=A_i)
\prod_{i=0}^k \1(e_i^-\in (A_i\offx{0} (A_{i-1}-e_{i-1}^+))).
\]
Then, for every $\vec A=(A,\vec a) \in\vec  \Comp(0)$, let
\begin{equation}
\delta(\vec A)=a^+-a^- \quad \text{and} \quad \Delta(\vec A)=a^+ \quad \text{and} \quad \mu_p(\vec A):= \proba_p(\C(0)=A) .
\label{eq:def-pointed-increments}
\end{equation}
We also write $\Ccut$ for the set of pairs $\vec A=(A,\vec a)\in \vec \Comp(0)$, $\vec B=(B,\vec b)\in \vec \Comp(0)$ such that
\begin{equation}
(\vec A,\vec B)\in \Ccut \quad \Longleftrightarrow \quad \left (b^-\notin (B \offx{0} (A-a^+)) \right ).
\label{eq:def-Ccut}
\end{equation}
We may now rewrite \eqref{eq:def_TauSquare} as
\[
\frac{1}{\Omega}  \sum_{\vec A_0,\vec A_1,\dots, \vec A_k\in \vec \Comp(0)} \1\left (\sum_{i=0}^k \Delta(\vec A_i) -\delta(\vec A_k)=x\right) \prod_{i=0}^k \mu_p(\vec A_i) \prod_{i=1}^k \1((\vec A_{i-1},\vec A_i)\notin \Ccut ),
\]
where we have dropped the indicator for $i=0$, since $e_0^-\in A_0$ is implied by $(A_0,\vec e_0)\in\vec \Comp(0)$.

We then note that for every $p<p_c$ we have
\begin{align}
\sum_{\vec A} \mu_p(\vec A) =\sum_{A\in \Comp(0)} \sum_{\vec e:e^-\in A} \mu_p(A)
=\Omega \sum_{A\in \Comp(0)} \sum_{e^-\in A} \proba_p(\C(0)=A) = \Omega\sum_{e^-}\proba_p(0\slr e^-)=\Omega \X(p)
. \label{eq:OrientedX}
\end{align}
Thus, we may consider a probability space $\proba_p^\otimes$, with corresponding expectation $\E^\otimes_p$,
and a sequence of i.i.d. random variables $(\vec \A_i)_{i\geq 0}=((\A_i,\vec \abf_i))_{i\geq 0}$  in $\vec \Comp(0)$
such that for every $\vec A\in \vec \Comp(0)$ we have
\begin{equation}
\proba^\otimes_p(\vec \A_i=\vec A)=\mu_p(\vec A)/(\Omega \X(p)).
\label{eq:def-pointed-cluster-law}
\end{equation}
We may also define, for every $k\geq 0$,
\begin{equation}
\SA_k:=\sum_{i=0}^k \Delta(\vec \A_i) =\sum_{i=0}^k \abf_i^+,
\label{eq:def-SA}
\end{equation}
and
\begin{equation}
\Tcut:=\inf\{t\in \N, (\vec \A_{t-1},\vec \A_t)\in \Ccut\},
\label{eq:def-Tcut}
\end{equation}
with the convention $\inf \emptyset=\infty$. With this notation, \eqref{eq:def_TauSquare} becomes
\begin{equation}
\tau_p^{\square,k}(x)=\Omega^k \X(p)^{k+1} \proba_p^\otimes(\SA_k-\delta(\vec \A_k)=x,\Tcut>k) . \label{eq:RewriteA_TauSquare}
\end{equation}

We are now able to recognize a Markov chain of i.i.d. components killed when forbidden transitions are used. To study such a Markov chain, it is usual to introduce a Doob $h$-transform.
First, note that for every $\vec A=(A,\vec a)\in\vec \Comp(0)$ such that $a^+\in A$, we have $\{\vec A\}\times \vec  \Comp(0)\subset \Ccut$.
Let $\Ckill$ be the set of such terminal $\vec A$, and let $\Clive:=\vec \Comp(0)\backslash \Ckill$.
Let $\Plive_p:=\proba_p^\otimes(\vA_0\in \Clive)$.
Note that $\Plive_p>0$.
We consider the operator $\T_p$, acting on the space of bounded functions $g:\vec \Comp(0)\mapsto \R$, defined by
\begin{equation}
\T_p g(\vec A):=\E_p^\otimes[g(\vec \A_2)\1_{(\vec A, \vec \A_2)\notin \Ccut}].
\label{eq:def-Tp}
\end{equation}
\begin{lemma} \label{lem:vpd}
$\T_p$ admits a maximal eigenvalue $0<\lambda_p\leq 1$ and there exists a unique nonzero associated eigenvector (up to scaling) $h_p:\vec \Comp(0) \mapsto \R$, which satisfies
\[
\T_p h_p=\lambda_p h_p \quad ; \quad h_p=0 \text{ on }\Ckill \quad ; \quad \inf_{\vec A\in \Clive} h_p(\vec A)\geq  \|h_p\|_\infty \frac{ (1-p)^\Omega}{\Omega \X(p)}.
\]
\end{lemma}
 \begin{proof}
 See next section.
 \end{proof}
 Given the previous lemma, we can now consider a probability space $\proba^h_p$ with expectation $\E^h_p$ such that for every $k\geq 0$ and bounded function $F:(\vec \Comp(0))^{k+1}\mapsto \R$, we have
 \begin{equation}
 \E^h_p[F(\vec \A_0,\vec \A_1,\dots \vec \A_k)]=\E^\otimes_p \left [\frac{h_p(\vec \A_k)}{h_p(\vec \A_0)} F(\vec \A_0,\vec \A_1,\dots, \vec \A_k)\lambda_p^{-k}\1_{\Tcut>k} \middle | \vA_0\in \Clive \right ].\label{eq:DoobTransform}
 \end{equation}
 This is known as a Doob $h$-transform. The process $(\vec \A_i)_{i\geq 0}$ is now a Markov chain on this probability space
with transition kernel $\P_p$ given, for every $\vec A,\vec B\in \Clive$, by
\begin{equation}
\proba^h_p(\vec \A_{k+1}=\vec B |\vec \A_k=\vec A)=\frac{h_p(\vec B)}{\lambda_p h_p(\vec A)}\frac{\mu_p(\vec B)}{\Omega \X(p)}\1((\vec A,\vec B)\notin \Ccut). \label{eq:MarkovTransition}
\end{equation}
For $k\geq 0$, we also write
\begin{equation}
\P^k_p(\vec B|\vec A):=\proba^h_p(\vA_k=\vec B|\vA_0=\vec A).
\label{eq:def-Ppk}
\end{equation}
 \begin{lemma} \label{lem:StationaryMeasure}
 \label{lem:vpg} $\P_p$ has a unique stationary probability measure $\vpg_p$ on $ \Clive$.
 \end{lemma}
   \begin{proof}
	   The Markov chain is irreducible and recurrent on $\Clive$. See next section.
 \end{proof}

 We now rewrite \eqref{eq:RewriteA_TauSquare} for $k\geq1$ using the Doob $h$-transform \eqref{eq:DoobTransform}.
We have to be careful here because the event in \eqref{eq:RewriteA_TauSquare} can hold with $\vA_k\in \Ckill$, while $\vA_k\notin \Ckill$ under $\proba^h_p$.
	 If we forget this, we only get
 \begin{align}
 \tau_p^{\square,k}(x) & \geq \Omega^k \X(p)^{k+1} \proba_p^\otimes(\SA_k-\delta(\vec \A_k)=x,\Tcut>k,\vec \A_k\notin \Ckill) \notag
 \\ & =\X(p) (\Omega\X(p)\lambda_p)^k \Plive_p\E_p^h\left [\1_{\SA_k-\delta(\vec \A_k)=x} \frac{h_p(\vec \A_0)}{h_p(\vec \A_k)} \right ].  \label{eq:RewriteB_Wrong_TauSquare}
 \end{align}
	 However, we know that $\Tcut>k$ implies $\vec \A_{k-1}\notin \Ckill$. So we can use the $h$-transform up to time $k-1$, and deal with the last step separately. First, by \eqref{eq:RewriteA_TauSquare}, we have
 \[
 \tau_p^{\square,k}(x) =\Omega^{k-1}\X(p)^k \E^\otimes_p \left[\1_{\Tcut\geq k} \sum_{\vec A_k\in \vec \Comp(0)} \mu_p(\vec A_k)\1_{(\vec \A_{k-1},\vec A_k)\notin\Ccut} \1_{\SA_{k-1}+\Delta(\vec A_k)-\delta(\vec A_k)=x} \right ].
 \]
	 Thus, by using the Doob $h$-transform \eqref{eq:DoobTransform}, we get
 \begin{equation}
 \tau_p^{\square,k}(x)=
 \X(p)(\Omega\X(p)\lambda_p)^{k-1} \Plive_p
 \E^h_p \left[\frac{h_p(\vec \A_0)}{h_p(\vec \A_{k-1})} \sum_{\vec A_k\in \vec \Comp(0)}   \mu_p(\vec A_k) \1_{(\vec \A_{k-1},\vec A_k)\notin\Ccut} \1_{\SA_{k-1}+\Delta(\vec A_k)-\delta(\vec A_k)=x} \right ]. \label{eq:RewriteB_TauSquare}
 \end{equation}

To conclude the section, let us show a quick application of \eqref{eq:RewriteB_Wrong_TauSquare} and \eqref{eq:RewriteB_TauSquare}. The next result is the basis of the many bootstrap arguments in Section \ref{sec:FirstApplications}.
It also ensures that the following alternative definition of $\tilde p_c(p)$ is consistent with Section \ref{sec:Diagrams}. For every $p<p_c$, let
\begin{equation}
\tilde p_c(p):= p+\frac{1-p}{\Omega \X(p)\lambda_p}. \label{eq:def_pcp}
\end{equation}
\begin{proposition} \label{pro:PrelimBootstrap}
For every $p<p_c$, $x\in \Z^d$, the map $p'\mapsto \tau_{p,p'}^\square(x)$ is increasing and analytic in $p'\in [p,\tilde p_c(p))$.
Also, the map $p'\mapsto \|\tau^\square_{p,p'}\|_1$ is increasing and continuous in $p'\in [p,\tilde p_c(p))$, and diverges as $p'\to \tilde p_c(p)$.
\end{proposition}
\begin{proof}
The first assertion follows from the second since for every $k\geq 0$, we have $\tau_p^{\square,k}\geq 0$ pointwise, and
$\tau^\square_{p,p'}=\sum_{k\geq 0} (p'-p)^k/(1-p)^k\tau_p^{\square,k}.$

To show the second assertion, it suffices to prove that as $k\to \infty$,
\[
\left (\frac{\tilde p_c(p)-p }{1-p} \right)^k \|\tau_p^{\square,k}\|_1=\Theta(1),
\]
which is equivalent by \eqref{eq:def_pcp} to
\begin{equation}
\|\tau_p^{\square,k}\|_1=\Theta((\Omega\X(p)\lambda_p)^k). \label{eq:AsympBootstrap}
\end{equation}
For the lower bound, we sum \eqref{eq:RewriteB_Wrong_TauSquare} over $x\in \Z^d$ to get
\[
\|\tau_p^{\square,k}\|_1\geq  \X(p) (\Omega\X(p)\lambda_p)^k \Plive_p \E_p^h\left [\frac{h_p(\vec \A_0)}{h_p(\vec \A_k)} \right ].
\]
The expectation is bounded away from $0$ by Lemma \ref{lem:vpd}, since $\proba^h_p (\vec \A_0, \vec \A_k\in\Clive)=1$.
For the upper bound, we can sum \eqref{eq:RewriteB_TauSquare} over $x\in \Z^d$, and drop the indicator $ \1((\vec \A_{k-1},\vec A_k)\notin\Ccut)$ to get
 \[
 \|\tau_p^{\square,k}\|_1\leq  \X(p)(\Omega\X(p)\lambda_p)^{k-1} \Plive_p \E^h_p \left[\frac{h_p(\vec \A_0)}{h_p(\vec \A_{k-1})} \sum_{\vec A_k\in \vec \Comp(0)}   \mu_p(\vec A_k) \right ].
	 \]
	 The sum over $\vec A_k$ is constant (equal to $\Omega \X(p)$). As before, Lemma \ref{lem:vpd} concludes the proof.
\end{proof}
We fix the following terminology for later use.
For every $\vec A\in \vec \Comp(0)$, let $-\vec A$ denote its natural central reflection with respect to the origin.
We say that a function $f$ on $\Clive$ is symmetric (resp. antisymmetric) if for every $\vec A$ we have $f(\vec A)=f(-\vec A)$ (resp. $f(\vec A)=-f(-\vec A)$).
\begin{lemma}\label{lem:Symmetric_hp_vpg} $\vpg_p$ and $h_p$ are symmetric.
\end{lemma}
\begin{proof}
The map $\vec A\mapsto -\vec A$ preserves $\Clive$, $\mu_p$, and $\Ccut$.
Thus $\T_p$ commutes with this reflection, so the uniqueness in Lemma \ref{lem:vpd} gives the symmetry of $h_p$.
The reflected measure $(\vpg_p(-\vec A))_{\vec A\in \Clive}$ is stationary for $\P_p$, so the uniqueness in Lemma \ref{lem:StationaryMeasure} gives the symmetry of $\vpg_p$.
\end{proof}
\subsection{Proof of Lemmas \ref{lem:vpd} and \ref{lem:StationaryMeasure}.} \label{sec:SpectralConstruction}
\textbf{Construction of $\lambda=\lambda_p$ and $h=h_p$:} Let $(S_i)_{i\geq 0}$ be an increasing sequence of finite subsets of $\vec \Comp(0)$ such that $\bigcup_{i\in \N} S_i=\vec \Comp(0)$.
Then, for every $i\in \N$, we define the truncated operator $\tilde \T_{i}$ acting on the space of bounded functions $g:S_i\mapsto \R$ by
 \begin{equation}
 \tilde \T_i g(\vec A):=\E_p^\otimes[g(\vec \A_2)\1(\vec \A_2\in S_i) \1((\vec A,\vec \A_2)\notin \Ccut)]. \label{eq:DefTruncatedForDoob}
 \end{equation}
By the Perron-Frobenius theorem, the operator $\tilde \T_i$ admits a maximal eigenvalue $\tilde \lambda_i\geq 0$
with a non-negative eigenvector $\tilde h_i$ such that $\tilde \T_i \tilde h_i= \tilde \lambda_i \tilde h_i$
and $\|\tilde h_i\|_\infty=1$.
Since $(S_i)$ is increasing, $(\tilde \lambda_i)$ is non-decreasing.
Also note that for every $\vec A\in S_i$, we have
\begin{equation}
\tilde \lambda_i\tilde h_i(\vec A)= \tilde \T_i \tilde h_i(\vec A) \leq
\|\tilde h_i\|_\infty.\label{eq:NaiveBoundEigen}
\end{equation}
So $(\tilde \lambda_i)$ is bounded by $1$ and, by monotonicity, converges to some $0\leq \lambda \leq 1$.

Then let $\tilde h_i=0$ on $\vec \Comp(0)\backslash S_i$.
Since $\|\tilde h_i\|_\infty=1$ for every $i$, a diagonal extraction yields a sequence $(i_n)_{n\in \N}$ such that for every $\vec A \in \vec \Comp(0)$, the sequence $(\tilde h_{i_n}(\vec A))_{n\in \N}$ converges.
Up to redefining the sets $(S_i)_{i\in \N}$ as $(S_{i_n})_{n\in \N}$, we may directly assume that for every $\vec A \in \vec \Comp(0)$, the sequence $(\tilde h_i(\vec A))_{i\in \N}$ converges to some $h(\vec A)\in [0,1]$.
By \eqref{eq:NaiveBoundEigen} and dominated convergence (uniform integrability is easily checked by removing the indicators in \eqref{eq:DefTruncatedForDoob}), we see that $\T_p h=\lambda h$.

It remains to show that $h\neq 0$ and $\lambda>0$.
To this end, let $\mathbb B$ be the set of all $b\in \Z^d$ such that $\{0,b\}\in \edges$.
Let $\mathbf O$ denote the graph $(\{0\},\emptyset)$.
For $b\in \mathbb B$, let $\vec{\mathbf{O}}_b=(\mathbf O,\{0,b\})$.
Note that for every $\vec A=(A,\vec a), \vec B\in \vec \Comp(0)$,
\begin{equation}
\1((\vec A,\vec B)\notin \Ccut)\leq \sum_{b\in\mathbb B} \1((\vec{\mathbf{O}}_{b},\vec B)\notin \Ccut). \label{eq:NulIsBetter}
\end{equation}
Indeed, if $(\vec A,\vec B)\notin \Ccut$, then $(\vec{\mathbf{O}}_{a^+-a^-},\vec B)\notin \Ccut$. In particular, for every $b,b'\in \mathbb B$, we have $(\vec{\mathbf{O}}_{b},\vec{\mathbf{O}}_{b'})\notin \Ccut$. Thus, by the Perron-Frobenius theorem, for every $i$ large enough, $\tilde \lambda_i>0$. So $\lambda>0$.
Then, summing \eqref{eq:NulIsBetter} over $\vec \A_2=\vec B$ according to \eqref{eq:DefTruncatedForDoob}, we get for every $i$ large enough and $\vec A \in \vec \Comp(0)$,
\[
\tilde \T_i \tilde h_i(\vec A) \leq \sum_{b\in \mathbb B} \E_p^\otimes[\tilde h_i(\vec \A_2)\1(\vec \A_2\in S_i) \1((\vec{\mathbf{O}}_b,\vec \A_2)\notin \Ccut)]= \sum_{b\in \mathbb B} \tilde \T_i\tilde h_i(\vec{\mathbf{O}}_b).
\]
By choosing $\vec A$ to maximize $\tilde h_i(\vec A)$ and using $\tilde \T_i \tilde h_i= \tilde \lambda_i \tilde h_i$, we obtain
 $1\leq \sum_{b\in \mathbb B} \tilde h_i(\vec{\mathbf{O}}_b).$
By taking $i\to \infty$, we deduce
\begin{equation}
\sum_{b\in \mathbb B} h(\vec{\mathbf{O}}_b) \geq 1. \label{eq:SumWorstComponents}
\end{equation}

 \textbf{Other properties of $h$:} Note that for every $\vec A\in \Ckill$ and every $\vec B\in \vec \Comp(0)$, we have $(\vec A, \vec B)\in \Ccut$, so $\T_p h(\vec A)=0$. In other words, $h=0$ on $\Ckill$. Then note that
\begin{equation}
\forall \vec A \in \Clive, \forall b\in \mathbb B, \quad (\vec A,\vec{\mathbf{O}}_b)\notin \Ccut. \label{eq:irreductible}
\end{equation}
Indeed, if $a^+\notin A$, then removing $A-a^+$ does not disconnect $0$ and $0$.
Together with \eqref{eq:SumWorstComponents}, we obtain, on $\Clive$,
\[
\T_p h(\vec A)\geq  \sum_{b\in \mathbb B} \proba_p^\otimes(\vec \A_2=\vec{\mathbf{O}}_b)h(\vec{\mathbf{O}}_b)\geq (1-p)^{\Omega}/(\Omega \X(p)).
\]
 Thus,
 \begin{equation}
 \inf_{\vec A \in \Clive}h(\vec A)\geq \|h\|_\infty (1-p)^{\Omega}/(\Omega \X(p))>0. \label{eq:InfVPD}
 \end{equation}

\textbf{Maximality of $\lambda$:}
Suppose for contradiction that there exist $\lambda'>\lambda$ and $h'\neq 0$ such that $\lambda' h'=  \T_p h'$. Then, by setting $h''=|h'|$, we have $\lambda' h''\leq  \T_p h''$.
Proceeding as above, we get $h''=0$ on $\Ckill$ and $\inf_{\vec A \in \Clive}h''(\vec A)>0.$
Also, by dominated convergence, for every $i$ large enough,
\[
\E_p^\otimes [h''(\vec \A_2)\1(\vec \A_2\notin S_i)] \leq \frac{\lambda'-\lambda}{2}\inf_{\vec C\in \Clive} h''(\vec C).
\]
We may rewrite $\lambda' h''(\vec A)\leq  \T_p h''(\vec A)$ as
\[
\lambda' h''(\vec A)-\tilde \T_i h''(\vec A)\leq \E_p^\otimes[h''(\vec \A_2)\1(\vec \A_2\notin S_i)\1((\vec A,\vec \A_2)\notin \Ccut)].
\]
Thus, combining the two previous displays and $h''\geq 0$, we get
\[
(\lambda+\lambda')h''/2\leq \tilde \T_i h'',
\]
which contradicts the maximality of $\tilde \lambda_i\leq \lambda$.

\textbf{Uniqueness of $h$:} Let $h'$ be another eigenfunction for $\lambda$.
Let $a$ be the smallest real such that $a h+h'$ is non-negative on $\Clive$.
By continuity, $\inf_{\vec A\in \Clive} [a h+h'](\vec A)=0$.
The same argument as for \eqref{eq:InfVPD} then implies $a h+h'=0$.
So $h$ and $h'$ are collinear.
This concludes the proof of Lemma \ref{lem:vpd}.

\textbf{Proof of Lemma \ref{lem:StationaryMeasure}:} Simply note that by \eqref{eq:NulIsBetter}
and \eqref{eq:irreductible} the Markov chain is irreducible.
By \eqref{eq:irreductible} and the bound $\inf_{\vec A \in \Clive}h(\vec A)>0$,
the chain has, at each step, a positive probability of returning to some $\vec{\mathbf{O}}_b$,
and so it is positive recurrent.
Lemma \ref{lem:StationaryMeasure} follows from standard results on Markov chains.
\subsection{Lower-bounding $\lambda_p$ under \eqref{eq:initialization}} \label{sec:LowerLambda}
The goal of this section is to prove the following result:
\begin{theorem} \label{thm:LowerLambda}
Let $p<p_c$ and $C>0$ be such that \eqref{eq:initialization} holds. There exist $c,K>0$, depending on $d,L,C$, such that if $\X(p)\geq K$, then $\lambda_p\geq c$.
\end{theorem}
\begin{remark}
Although we do not try to optimize the constants $K,c$, if the triangle condition is small enough,
one may completely skip the local modifications of the clusters in the proof below,
and obtain good constants for $K,c$ that do not depend on $d,L,C$.
\end{remark}
We are led by the intuition that correlations between clusters are local: if a cluster cuts another cluster, i.e. $(\vec A,\vec B)\in \Ccut$, the cut should be close to $0$ in $\vec B$.
By modifying the clusters locally, which costs a multiplicative constant, we may avoid cuts that are at distance at most $r>0$ for some fixed $r$.
If $r$ is large enough, we can deal with cuts that occur at distance at least $r$ apart.

Formally, for every $x\in \R^d$, $r\in \R_+$, let $\ball_r(x):=\{y:\|y-x\|_1\leq r\}$, and let $\ball^\edges_r(x)$ be the set of edges with both endpoints in $\ball_r(x)$.
We consider a function $f_r:\vec \Comp(0)\mapsto \vec \Comp(0)$ such that for every $(A,\vec a)$, $f_r(A,\vec a)$ is of the form $(B,\vec a)$ with
\begin{equation}
\big( A \cup \ball^\edges_r(0) \big ) \offx{0} \ball_{r}(a^+)=B\backslash \ball_{r}(a^+). \label{eq:def_fr}
\end{equation}
We will choose $f_r$ later.  We write $\Ccut^r$ for the set of pairs $\vec A,\vec B$ such that
\begin{equation}
(\vec A,\vec B)\in \Ccut^r  \quad \Longleftrightarrow \quad  (f_r(\vec A),f_r(\vec B))\in \Ccut.
\label{eq:def-Ccutr}
\end{equation}
Using the notations of Section \ref{sec:Rewriting}, let
\begin{equation}
\Tcut^r:=\inf\{t\in \N, (\vA_{t-1},\vA_t)\in \Ccut^r\}.
\label{eq:def-Tcutr}
\end{equation}
Since by \eqref{eq:AsympBootstrap} and \eqref{eq:RewriteA_TauSquare} we have $\proba_p^\otimes(\Tcut>k)=\Theta(\lambda_p^k)$ as $k\to \infty$, the next result allows us to lower-bound $\lambda_p$ using $\Tcut^r$.
\begin{proposition} \label{pro:RelateLocalLambda}
Let $p<p_c$ be such that $\X(p)\geq 4$. For every $r>0$, there exists $c>0$ depending only on $d,L,r$ such that for every $k\geq 0$,
\[
\proba_p^\otimes(\Tcut>k)\geq c^{k+1} \proba_p^\otimes(\Tcut^r>k).
\]
\end{proposition}
\begin{proof}
By definition, the second probability is equal to
\[
\proba_p^\otimes\left (\forall 0\leq i \leq k-1 : \quad f_r(\vec \A_i), f_r(\vec \A_{i+1})\notin \Ccut \right ),
\]
which we can rewrite, by splitting according to the values of $f_r(\vA_0),\dots f_r(\vA_k)$, as
\[
\sum_{\vec B_0,\dots, \vec B_k\in \vec \Comp(0)}\prod_{i=0}^k \proba_p^\otimes(f_r(\vec \A_i)=\vec B_i) \prod_{i=0}^{k-1} \1( (\vec B_i, \vec B_{i+1})\notin \Ccut).
\]
We upper-bound this last sum by
\[
\left (\sup_{\vec B\in \vec \Comp(0)} \frac{\proba_p^\otimes(f_r(\vec \A_1)=\vec B)}{\proba_p^\otimes(\vA_1=\vec B)}\right )^{k+1} \hspace{-1em}
\sum_{\vec B_0,\dots, \vec B_k\in \vec \Comp(0)}
\prod_{i=0}^k \proba_p^\otimes(\vA_i=\vec B_i)
\prod_{i=0}^{k-1} \1( (\vec B_i, \vec B_{i+1})\notin \Ccut).
\]
Removing the supremum, we recognize $ \proba_p^\otimes(\Tcut>k)$. So it suffices to upper-bound this supremum by a constant depending only on $r,d,L$.

To this end, note that by \eqref{eq:def_fr}, for a given $\vec B=(B,\vec b)$, the probability in the supremum equals
\[
\frac{1}{\Omega \X(p)}\proba_p \Bigg (0\slr b^-,\big ( \C(0) \cup \ball^\edges_r(0) \big ) \offx{0} \ball_r(b^+)=B\backslash \ball_r(b^+)\Bigg ).
\]
If the above event holds, then we can get the event $\{\C_p(0)=B\}$ by
 opening/closing at most $N\in \N$ edges, where $N$ is a constant depending only on $r,d,L$.
Indeed, by opening all the edges of $\ball^\edges_r(0)$ and then closing all the edges adjacent to $ \ball_r(b^+)$, one gets $\C(0)=B\backslash \ball_r(b^+)$.
It remains to open all the edges of $B$ which have at least one endpoint in $\ball_r(b^+)$ and to close some of the edges adjacent to those edges.
 We deduce that the last probability is at most
\begin{equation}
\proba_p(\C(0)=B) (p(1-p))^{-N} =\mu_p(\vec B)(p(1-p))^{-N}. \label{eq:CostLocalModif}
\end{equation}

We have almost finished, but we need a bound independent of $p$.
To this end, note that $\X(1/2)=\infty$ since $(\Z^d,\edges)$ contains a copy of the nearest-neighbor lattice on $\Z^2$ and $p_c(\Z^2)=1/2$.
So $p<1/2$.
Also $p>1/(9\Omega)$, since by comparison with a Galton--Watson process, it is easy to get $\X(1/(9\Omega))<4$.
Therefore, we may replace $p(1-p)$ in \eqref{eq:CostLocalModif} by $1/(18\Omega)$.
Hence, we get an upper bound on the supremum as desired, which concludes the proof.
\end{proof}

 We now define, for $r\geq 5L$, the function $f_r$ in order to optimize the probability that, for typical clusters $\vec A=(A,\vec a),\vec B=(B,\vec b)\in \vec \Comp(0)$, we have $(\vec A,\vec B)\notin \Ccut^r$.
To this end, note that if $\vec A$ cuts $\vec B$, then the cuts typically occur locally around $0$ in $B$ (resp. around $\vec a$ in $\vec A$).
Thus, we want to minimize the number of vertices in $A$ around $\vec a$, and maximize the number of edges in $\vec B$ around $0$.
 We define $f_r(\vec A)$ as $(\vec A',\vec a)\in \vec \Comp(0)$, where $A'$ is constructed to satisfy \eqref{eq:def_fr} as follows (see Fig. \ref{fig:LocalModif}):
 \begin{compactitem}
 \item We first add to $A$ all the edges in $\ball^\edges_r(0)$.
\item We then close all the edges adjacent to $\ball_r(a^+)$, and only keep the connected component of $0$.
\item Next, take an arbitrary self-avoiding open path from $0$ to $a^-$ in $A$,
and let $v_r(\vec A)$ be the vertex preceding the first entry in $\ball_r(a^+)$ of this path.
$v_r(\vec A)$ is well defined when $\|a^+\|_\infty\geq 3r$.
(Otherwise, we simply let $\vec A'$ be an arbitrary cluster satisfying \eqref{eq:def_fr}.)
Note that $\|v_r(\vec A)-a^+\|_1\in(r,r+L]$,
and $v_r(\vec A)$ is in the cluster that we are currently constructing.
\item Finally, open a path of minimal length from $v_r(\vec A)$ to $a^-$ that stays in $\{v_r(\vec A)\}\cup \ball_r(a^+)\backslash\{a^+\}$.
\end{compactitem}

Our construction ensures that there is a path in $f_r(\vec A)$ from $0$ to $a^-$ that can be split into three:
\begin{compactitem}
\item A first path $\Gamma^1_r(\vec A)$ in $\ball_r(0)$ that ends up at some $w_r(\vec A)$.
\item A second path $\Gamma^2_r(\vec A)$ that starts at $w_r(\vec A)$
then stays in $A\backslash (\ball_r(0)\cup \ball_r(a^+))$ and ends up at $v_r(\vec A)$,
and is also a subpath of a self-avoiding path from $0$ to $a^-$ in $A$.
\item The path $\Gamma^3_r(\vec A)$ from $v_r(\vec A)$ to $a^-$ that stays in $\{v_r(\vec A)\}\cup \ball_r(a^+)\backslash\{a^+\}$.
\end{compactitem}
This will be particularly important in the proof of the following main proposition.

 \begin{figure}[!h]
\centering
\includegraphics[scale=1.1]{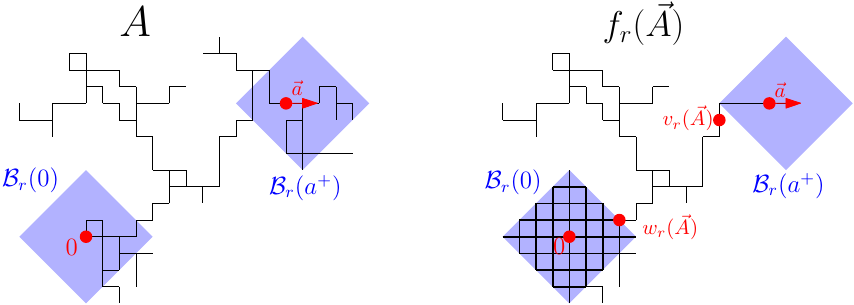}
\caption{The left panel represents a cluster $\vec A$, and the right panel represents $f_r(\vec A)$.
The vertices $0,w_r(\vec A), v_r(\vec A)$ and the oriented edge $\vec a=(a^-,a^+)$ are red.
The balls $\ball_r(0),\ball_r(a^+)$ are blue.
Black segments represent edges.
The added edges are represented as thicker segments.}
 \label{fig:LocalModif}
 \end{figure}
\begin{proposition} \label{pro:OrientedOpenTriangle}
Let $p<p_c$ and $C>0$ be such that \eqref{eq:initialization} holds. If $r>0$ is large enough, depending on $d,L,C$, and if $\X(p)$ is large enough, depending on $d,L,C,r$, we have
\[
\proba^\otimes_p((\vec \A_1,\vec \A_2)\in \Ccut^r) < 1/4.
\]
\end{proposition}
We need to check that $0\slr \abf_2^-$ holds with high probability in $f_r(\vec \A_2)\backslash (f_r(\vec \A_1)-\abf_1^+)$.
To do so, it is enough to prove that, w.h.p.,
(1) $\|\abf_1^+\|_\infty,\|\abf_2^+\|_\infty \geq 3r$,
then, in this case, that inside $f_r(\vec \A_2) \backslash (f_r(\vec \A_1)-\abf_1^+)$
we have (2)  $0$ and $w_r(\vec \A_2)$ are connected,
(3) $\Gamma^2_r(\vec \A_2)$ is open,
(4) $\Gamma^3_r(\vec \A_2)$ is open.
This is the content of the following four lemmas.
\begin{lemma}
In the setting of Proposition \ref{pro:OrientedOpenTriangle}, we have, for every $i\in \N$,
\[
\proba^\otimes_p(\|\abf^+_i\|_\infty<3r) < 1/32.
\]
\end{lemma}
\begin{proof}
The desired probability can be written as
\[
\frac{1}{\Omega \X(p)} \sum_{(A,\vec a)\in \vec \Comp(0)} \proba_p(\C(0)=A) \1(a^-\in A, \|a^+\|_\infty\leq 3r).
\]
Increasing the sum, we may remove the event $a^-\in A$, so we can sum over $\vec a$ and $A$ separately.
The sum over $A$ gives $1$.
The sum over $\vec a$ gives $\Omega O(r^d)$.
So we can indeed take $\X(p)$ large enough that the probability is as small as desired.
\end{proof}
\begin{lemma} \label{lem:NoCutAfterModif_a}
In the setting of Proposition \ref{pro:OrientedOpenTriangle}, we have
\[
\proba^\otimes_p\left (\|\abf^+_1\|_\infty\geq 3r,\|\abf_2^+\|_\infty \geq 3r, \quad w_r(\vec \A_2) \notin \left (f_r(\vec \A_2) \offx{0} (f_r(\vec \A_1)-\abf_1^+) \right ) \right )<1/16.
\]
\end{lemma}

\begin{proof}

Recall that all the edges in $\ball^\edges_r(0)$ are open in $f_r(\vec B)$, and that inside $\ball_r(a^+)$, $f_r(\vec A)$ consists of a minimal path from $v_r(\vec A)$ to $a^-$  that does not go through $a^+$. It follows that if
\[
w_r(\vec B)\notin f_r(\vec B) \offx{0}(f_r(\vec A)-a^+),
\]
then we must have $\|(v_r(\vec A)-a^+)-w_r(\vec B)\|_1\leq 9L$. Thus, the desired probability is bounded by
\[
\proba_p^\otimes \left (
\exists v\in \A_1, \exists w\in \A_2:
\quad \|w\|_1> r-L,
\quad \|v-\abf_1^+-w\|_1\leq 9L,
\quad \substack{  (0\slr v)\circ (v\slr \abf_1^-) \text{ in } \A_1 \\ (0\slr w)\circ (w\slr \abf_2^-) \text{ in } \A_2}  \right ).
\]
By a union bound, the last probability is bounded by
\[
\sum_{v\in \Z^d} \sum_{w\notin \ball_{r-L}(0)} \proba_p^\otimes (\|v-\abf_1^+-w\|_1\leq 9L, (0\slr v)\circ (v\slr \abf_1^-) \text{ in } \A_1) \proba_p^\otimes((0\slr w)\circ (w\slr \abf_2^-) \text{ in } \A_2 ),
\]
which we may rewrite as
\[
\frac{1}{\Omega^2 \X(p)^2} \sum_{v\in \Z^d} \sum_{w\notin \ball_{r-L}(0)} \sum_{\substack{\vec a_1, \vec a_2 \\ \|v-a_1^+-w\|_1\leq 9L }} \proba_p((0\slr v)\circ (v\slr a_1^-)) \proba_p((0\slr w)\circ (w\slr a_2^-)).
\]
By the BK inequality, the last sum is then bounded by
\[
\frac{1}{\Omega^2 \X(p)^2} \sum_{v\in \Z^d} \sum_{w\notin \ball_{r-L}(0)} \sum_{\substack{\vec a_1, \vec a_2 \\ \|v-a_1^+-w\|_1\leq 9L }} \tau_p(v)\tau_p(a_1^--v) \tau_p(w)\tau_p(a_2^--w).
\]
Summing over $\vec a_2$ yields a factor $\X(p)\Omega$. Then, making the change of variables $\vec a_1\gets \vec a_1-v$ and summing over $v$ yields another term $\X(p)$. Thus the last sum equals
\[
\frac{1}{\Omega} \sum_{w\notin \ball_{r-L}(0)} \sum_{\vec a_1: \|-a_1^+-w\|_1\leq 9L} \tau_p(a_1^-) \tau_p(w).
\]
Using \eqref{eq:initialization}, we get that the last sum is bounded by
\[
C^2 \sum_{w\notin \ball_{r-L}(0)} \sum_{a^-_1: \|a_1^-+w\|_1\leq 10L} \|a_1^-\|_2^{2-d} \|w\|_2^{2-d}.
\]
Since the last sum is finite, dominated convergence shows that it tends to $0$ as $r\to \infty$. In other words, by taking $r$ large enough, the probability in the lemma is as small as we want.
\end{proof}
\begin{lemma} \label{lem:NoCutAfterModif_b}
In the setting of Proposition \ref{pro:OrientedOpenTriangle}, we have
\[
\proba^\otimes_p\left (\|\abf^+_1\|_\infty\geq 3r,\|\abf_2^+\|_\infty \geq 3r, \quad \Gamma^2_r(\vec \A_2)\backslash\{w_r(\vec \A_2)\} \not \subset \left (f_r(\vec \A_2) \backslash (f_r(\vec \A_1)-\abf_1^+) \right ) \right )<1/16.
\]
\end{lemma}
\begin{proof}
If a vertex $z\in \Gamma^2_r(\vec \A_2)\backslash\{w_r(\vec \A_2)\}$ is missing, then the following holds:
First, since $z\in \Gamma^2_r(\vec \A_2)\backslash\{w_r(\vec \A_2)\}$,
we must have $\|z\|_1>r$ and $(0\slr z)\circ (z\slr \abf_2^-)$ in $\A_2$.
Also, $z$ is in $f_r(\vec \A_1)-\abf_1^+$, which can be decomposed into three parts: a path in $\ball_r(0)$, a part of $\vec \A_1-a_1^+$, and $\ball^\edges_r(-a_1^+)$.
So, since $\|z\|_1> r$, we either have (a) $z+\abf_1^+ \in \A_1$ or (b) $\|z+\abf_1^+\|_1\leq r$.

By a union bound, the probability of case (a) occurring is bounded by
\[
\sum_{z\notin \ball_r(0)} \proba_p^\otimes (z+\abf_1^+\in \A_1)\proba_p^\otimes((0\slr z)\circ (z\slr \abf_2^-)\text{ in }\A_2),
\]
which we may rewrite as in the previous proof as
\[
\frac{1}{\Omega^2 \X(p)^2} \sum_{z\notin \ball_r(0)} \sum_{\vec a_1, \vec a_2} \proba_p(0\slr (z+a_1^+)\slr a_1^-)\proba_p((0\slr z)\circ (z\slr a_2^-)).
\]
By the BK inequality and the tree graph inequality \cite{MR762034}, the last sum is then bounded by
\[
\frac{1}{\Omega^2 \X(p)^2} \sum_{z\notin \ball_r(0),y\in \Z^d} \sum_{\vec a_1, \vec a_2} \tau_p(y)\tau_p(z+a_1^+-y) \tau_p(y-a_1^-)\tau_p(z)\tau_p(a_2^--z).
\]
As in the previous proof, summing over $\vec a_2$ yields a factor of $\Omega \X(p)$. We then let $\delta=a_1^+-a_1^-$ and make the change of variables $y\gets y+a^+_1$, so the last sum equals
\[
\frac{1}{\Omega\X(p)} \sum_{z\notin \ball_r(0),y\in \Z^d} \sum_{\delta, (0,\delta)\in \edges} \sum_{a_1^+\in \Z^d} \tau_p(y-a_1^+)\tau_p(z-y) \tau_p(y-\delta)\tau_p(z).
\]
Summing over $a_1^+$ yields another factor $\X(p)$. It remains
\[
\frac{1}{\Omega} \sum_{z\notin \ball_r(0),y\in \Z^d} \sum_{\delta\in \ball_L(0)}  \tau_p(z-y) \tau_p(y-\delta)\tau_p(z).
\]
We conclude as in the proof of Lemma \ref{lem:NoCutAfterModif_a} by directly substituting the bound from \eqref{eq:initialization}.

For case (b), $\|z+\abf_1^+\|_1\leq r$, we proceed similarly. By a union bound, the corresponding probability is bounded by
\[
\sum_{z\notin \ball_r(0)} \proba_p^\otimes (\|z+\abf_1^+\|_1\leq r)\proba_p^\otimes((0\slr z)\circ (z\slr \abf_2^-)\text{ in }\A_2),
\]
which we may rewrite as
\[
\frac{1}{\Omega^2 \X(p)^2} \sum_{z\notin \ball_r(0)} \sum_{\substack{ \vec a_1,\vec a_2\\ \|z+a_1^+\|_1\leq r}}\proba_p(0\slr a_1^-)\proba_p((0\slr z)\circ (z\slr a_2^-)).
\]
As before, we can use the BK inequality and then sum over $a_2^-$, resulting in a $\Omega \X(p)$ term. It remains to bound
\[
\frac{1}{\Omega \X(p)} \sum_{z\notin \ball_r(0)} \sum_{\vec a_1: \|z+a_1^+\|_1\leq r}\proba_p(0\slr a_1^-)\proba_p(0\slr z).
\]
Using \eqref{eq:initialization}, we get that the last sum is bounded by
\[
\frac{C^2 }{\Omega \X(p)} \sum_{z\notin \ball_r(0)} \sum_{\vec a_1: \|z+a_1^+\|_1\leq r}\langle a_1^-\rangle^{2-d} \langle z\rangle^{2-d}.
\]
For every $r\geq 0$, the last sum is bounded, so we may take $\X(p)$ large enough (depending in particular on $r$) that the desired probability is as small as we want. This concludes the last case.
\end{proof}
\begin{lemma} \label{lem:NoCutAfterModif_c}
In the setting of Proposition \ref{pro:OrientedOpenTriangle}, we have
\[
\proba^\otimes_p\left (\|\abf^+_1\|_\infty\geq 3r,\|\abf_2^+\|_\infty \geq 3r, \quad \Gamma^3_r(\vec \A_2)\backslash \{v_r(\vec \A_2)\} \not \subset \left (f_r(\vec \A_2) \backslash (f_r(\vec \A_1)-\abf_1^+) \right ) \right )<1/16.
\]
\end{lemma}
\begin{proof}
We proceed as before.
If a vertex $z\in \Gamma^3_r(\vec \A_2)\backslash\{v_r(\vec \A_2)\}$ is missing, then we must have $\|z-\abf_2^+\|_1\leq r$.
Also, $z\in f_r(\vec \A_1)-\abf_1^+$.
Using $\|\abf_2^+\|_1\geq 3r$, we have $\|z\|_1>r$, so either (a) $z+\abf_1^+ \in \A_1$ or (b) $\|z+\abf_1^+\|_1\leq r$.

By a union bound, the probability of case (a) occurring is bounded by
\[
\sum_{z\in \Z^d} \proba_p^\otimes (z+\abf_1^+\in \A_1)\proba_p^\otimes(\|z-\abf_2^+\|_1\leq r),
\]
which we may rewrite as
\[
\frac{1}{\Omega^2  \X(p)^2}\sum_{\vec a_1, \vec a_2} \sum_{z\in \ball_r(a_2^+)} \proba_p(0\slr (z+a_1^+)\slr a_1^-) \tau_p(a_2^-).
\]
Then, by the tree graph inequality \cite{MR762034}, the last sum is bounded by
\[
\frac{1}{\Omega^2  \X(p)^2}  \sum_{\vec a_1, \vec a_2} \sum_{z\in \ball_r(a_2^+)} \sum_{y\in \Z^d} \tau_p(y)\tau_p(z+a_1^+-y)\tau_p(y-a_1^-)\tau_p(a_2^-).
\]
As in the previous proof, we let $\delta=a_1^+-a_1^-$ and make the change of variables $y\gets y+a^+_1$, then sum over $a_1^+$ to get that the last sum equals
\[
\frac{1}{\Omega^2 \X(p)} \sum_{\vec a_2} \sum_{y\in \Z^d}\sum_{z\in \ball_r(a_2^+)}\sum_{\delta\in \ball_L(0)}\tau_p(z-y) \tau_p(y-\delta)\tau_p(a_2^-).
\]
By \eqref{eq:initialization}, the last sum is bounded by
\[
\frac{C^3}{\Omega^2 \X(p)} \sum_{\vec a_2} \sum_{y\in \Z^d}\sum_{z\in \ball_r(a_2^+)}\sum_{\delta\in \ball_L(0)}\langle z-y\rangle^{2-d} \langle y-\delta\rangle^{2-d} \langle a_2^-\rangle^{2-d}.
\]
This last sum is finite for fixed $r$, so we may take $\X(p)$ large enough (depending on $r$) that the corresponding probability is as small as we want.

We upper-bound the probability corresponding to case (b) by
\[
\sum_{z\in \Z^d} \proba_p^\otimes (\|z+\abf_1^+\|_1\leq r,\|z-\abf_2^+\|_1\leq r),
\]
which we can rewrite as
\[
\frac{1}{\Omega^2\X(p)^2} \sum_{z\in \Z^d} \sum_{\vec a_1,\vec a_2} \1(\|z+a_1^+\|_1\leq r) \1(\|z-a_2^+\|_1\leq r) \tau_p(a_1^-)\tau_p(a_2^-).
\]
Using \eqref{eq:initialization} as before, we see that the last sum is uniformly bounded. Thus, we may take $\X(p)$ large enough that the probability corresponding to case (b) is as small as we want.
\end{proof}

Our next step is to use Proposition \ref{pro:OrientedOpenTriangle} to lower-bound $\proba_p^\otimes(\Tcut^r>k)$ as $k\to \infty$.
To this end, we use the following general result:
\begin{lemma} \label{lem:NaiveFrobeniusBound}
Let $S\neq \emptyset$ be a finite set. Let $(v_s)_{s\in S}\in \R_+^S$. Let $(u_{s,s'})_{s,s'\in S}\in[0,1]^{S^2}$. Let
\[
\varepsilon_{\mathrm F}:=\left (\sum_{s,s'\in S} v_s (1-u_{s,s'}) v_{s'} \right )\left (\sum_{s\in S} v_s\right )^{-2}.
\]
If $\varepsilon_{\mathrm F}<1/2$, then the matrix $(u_{s,s'}v_{s'})_{s,s'\in S}$ has a maximal eigenvalue of at least $(1-\sqrt{2\varepsilon_{\mathrm F}}) \sum_{s\in S} v_s$.
\end{lemma}
\begin{proof}
 Using the change of basis matrix $(\sqrt{v_s}\1_{s=s'})_{s,s'\in S}$, we see that $M:=(u_{s,s'}v_{s'})_{s,s'\in S}$ is similar to the matrix
 \[
 \tilde M:=(\sqrt{v_s}u_{s,s'}\sqrt{v_{s'}})_{s,s'\in S}.
 \]
  Since the entries are all non-negative, it is then enough to lower-bound the maximal eigenvalue of
  \[
  M':=(\sqrt{v_s}\inf(u_{s',s},u_{s,s'})\sqrt{v_{s'}})_{s,s'\in S}.
  \]
Let $V:=(\sqrt{v_s}\sqrt{v_{s'}})_{s,s'\in S}$. Note that $V$ has maximal eigenvalue $\sum_{s\in S} v_s$.
Since the largest singular value is an operator norm and, for symmetric non-negative matrices, coincides with the maximal eigenvalue, it suffices to show that $
  M'':=V-M'$
  has a maximal eigenvalue at most $\sqrt{2\varepsilon_{\mathrm F}} \sum_{s\in S} v_s$. To this end, simply note that the Frobenius norm of $M''$ is given by
  \[
\|M''\|_F^2=\sum_{s,s'\in S} v_s(1-\inf(u_{s',s},u_{s,s'}))^2 v_{s'}\leq \sum_{s,s'\in S} v_s((1-u_{s',s})+(1-u_{s,s'})) v_{s'} =2\varepsilon_{\mathrm F} \left (\sum_{s\in S} v_s\right )^2 .
\]
	  Since the square of the Frobenius norm of a symmetric matrix equals the sum of its eigenvalues squared, this concludes the proof.
\end{proof}
\begin{proof}[Proof of Theorem \ref{thm:LowerLambda}]
Let $S$ be a large enough finite subset approximating $\vec \Comp(0)$.
We consider the transition matrix $M_{p,S}=\left (\1_{(\vec A_1,\vec A_2)\notin \Ccut^r}\1_{\vec A_1\in S} \1_{\vec A_2\in S} \mu_p(\vec A_2)\right )_{\vec A_1,\vec A_2}$.
By Proposition \ref{pro:OrientedOpenTriangle},
\[
\varepsilon_{\mathrm F}:=
\left (\sum_{\vec A_1,\vec A_2\in \vec \Comp(0)} \mu_p(\vec A_1) \1_{(\vec A_1,\vec A_2)\in \Ccut^r} \mu_p(\vec A_2)\right )
/ \left(\sum_{\vec A_1\in \vec \Comp(0)} \mu_p(\vec A_1) \right)^2 <1/4.
\]
By Lemma \ref{lem:NaiveFrobeniusBound} and dominated convergence, if $S$ is large enough, then $M_{p,S}$ has a maximal eigenvalue strictly larger than $(1-\sqrt{2/4})\Omega \X(p)$. It follows that, as $k\to \infty$, we have
\[
\sum_{\vec A_0, \vec A_1, \vec A_2,\dots, \vec A_k\in S} \mu_p(\vec A_0) \prod_{i=0}^{k-1} M_{p,S}(\vec A_i, \vec A_{i+1})\gg (1-1/\sqrt{2})^k (\Omega \X(p))^k,
\]
which we can rewrite as
\[
\proba_p^\otimes(\Tcut^r>k,\quad \forall 0\leq i \leq k, \vA_i\in S) \gg (1-1/\sqrt{2})^k .
\]
Thus, as $k\to \infty$,
\[
\proba_p^\otimes(\Tcut^r>k) \gg (1-1/\sqrt{2})^k .
\]
Finally, recall that by \eqref{eq:AsympBootstrap} and \eqref{eq:RewriteA_TauSquare} we have $\proba_p^\otimes(\Tcut>k)=\Theta(\lambda_p^k)$ as $k\to \infty$.
Taking $k\to \infty$ in Proposition \ref{pro:RelateLocalLambda} thus concludes the proof.
\end{proof}
\subsection{Exponentially fast mixing under \eqref{eq:initialization}} \label{sec:ExponentialMixing}
The goal of this section is to prove exponentially fast mixing for $(\vA_i)_{i\in \N}$ under $\proba^h_p$ in total variation:
\begin{theorem} \label{thm:MixingTime}
Let $p<p_c$ and $C>0$ be such that \eqref{eq:initialization} holds. There exist $0<c<1$ and $K>0$, depending on $d,L,C$, such that if $\X(p)\geq K$, then for every $\vec A \in \Clive$, for every $k\in \N$, we have
\[
\sum_{\vec B\in \Clive}\left |  \P_p^k(\vec B|\vec A)-\vpg_p(\vec B)\right | \leq c^k \X(p).
\]
\end{theorem}
First, let us introduce some notation. For every $\vec A$ we write
\begin{equation}
\nu^\uparrow_p (\vec A) :=\proba_p^\otimes(\vec \A_1=\vec A,(\vec \A_1,\vec \A_2)\notin \Ccut)
\quad \text{and} \quad
\nu^\downarrow_p (\vec A) :=\proba_p^\otimes(\vec \A_2=\vec A,(\vec \A_1,\vec \A_2)\notin \Ccut). \label{eq:def_nu_up+down}
\end{equation}
We also let $\Zb_p:=\proba_p^\otimes((\vec \A_1,\vec \A_2)\notin \Ccut)$.
Note that $\|\nu^\uparrow_p\|_1=\|\nu^\downarrow_p\|_1=\Zb_p$.
For Theorem \ref{thm:MixingTime}, we will need to prove that $\Zb_p\geq \lambda_p^2$.
For future reference, we prove the following more general result.
\begin{lemma} \label{lem:LambdaLowMemory}
For every $k\geq 0$, we have $\proba_p^\otimes(\Tcut > k,\vA_k\in \Clive)\geq \lambda_p^{k+1}$.
\end{lemma}
\begin{proof}
Note that for every $n\in \N$, $k\geq 1$,
\begin{align*}
\proba_p^\otimes(\Tcut>kn) & \leq \proba_p^\otimes(\forall 0\leq i< n, \forall 0\leq j<k-1,\, (\vec \A_{ki+j},\vec \A_{ki+j+1})\notin \Ccut,\, \vA_{ki+k-1 }\in \Clive)
\\ & =\proba_p^\otimes(\Tcut\geq k, \vA_{k-1}\in \Clive )^n.
\end{align*}
Since by \eqref{eq:AsympBootstrap} and \eqref{eq:RewriteA_TauSquare} we have $\proba_p^\otimes(\Tcut>N)=\Theta(\lambda_p^N)$, taking $n\to \infty$ concludes the proof.
\end{proof}

The main idea behind the proof of Theorem \ref{thm:MixingTime} is that, for typical $\vec A_1,\vec A_2,\vec A_3,\vec A_4$ such that $(\vec A_1,\vec A_2),(\vec A_3,\vec A_4)\notin \Ccut$,
we may modify $\vec A_2, \vec A_3$ locally so that we also have $(\vec A_2, \vec A_3)\notin \Ccut$.
This in turn implies that we can go from $\vec A_1$ to $\vec A_4$ with positive probability.
To this end, we use the functions $f^\uparrow_r$ and $f^\downarrow_r$, defined as $f_r$ was in the previous subsection, except that for $f^\uparrow_r$,
we do not open the edges near $0$,
	and to construct $f^\downarrow_r(A,\vec a)$,
	we do not perform the modifications in $\ball_r(\vec a)$ (see Figure \ref{fig:LocalModifB}).
 \begin{figure}[!h]
\centering
\includegraphics[scale=1.1]{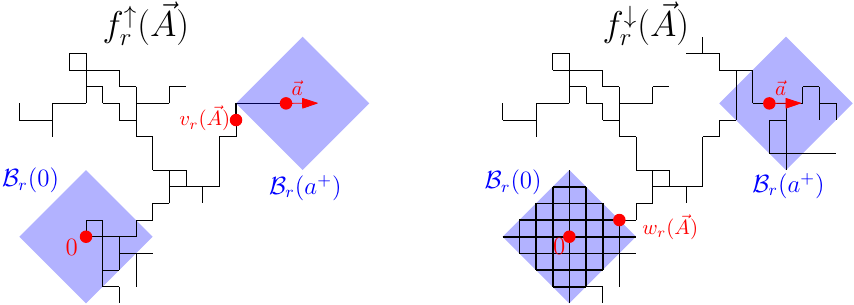}
\caption{$f^\uparrow_r(\vec A)$ and $f^\downarrow_r(\vec A)$ are represented,
where $\vec A$ is the same cluster as in Figure \ref{fig:LocalModif}.
The drawing conventions are also the same.
$f^\uparrow_r(\vec A)$ is obtained by modifying $\vec A$ around $\vec a$,
and $f^\downarrow_r(\vec A)$ is obtained by modifying $\vec A$ around $0$.}
 \label{fig:LocalModifB}
 \end{figure}
 To simplify the notation, we let $E_r(\vec A_1,\vec A_2,\vec A_3,\vec A_4)$ be the event that
 \[
 (\vec A_1,f_r^\uparrow(\vec A_2))\notin \Ccut \quad \text{and} \quad (f_r^\uparrow(\vec A_2),f_r^\downarrow(\vec A_3))\notin \Ccut \quad \text{and} \quad (f_r^\downarrow(\vec A_3),\vec A_4)\notin \Ccut .
\]
\begin{proposition} \label{pro:OrientedOpenTriangleSwitching}
Let $p<p_c$ and $C>0$ be such that \eqref{eq:initialization} holds. For every $\varrho>0$, if $r>0$ is large enough depending on $d,L,C,\varrho$, and if $\X(p)$ is large enough depending on $d,L,C,r,\varrho$, we have
\[
\proba^\otimes_p((\vec \A_1,\vec \A_2)\notin \Ccut ,(\vec \A_3,\vec \A_4)\notin\Ccut , E_r(\vec \A_1,\vec \A_2,\vec \A_3,\vec \A_4)) \geq (1-\varrho)\Zb_p^2.
\]
\end{proposition}
\begin{proof}
By directly adapting the proof of Proposition \ref{pro:OrientedOpenTriangle}, we obtain the following three statements:
\[
\proba^\otimes_p((\vA_1,\vA_2)\notin \Ccut, (\vec \A_1,f_r^\uparrow(\vec \A_2))\in \Ccut ) \leq \varrho,
\]
and
\[
\proba^\otimes_p((f_r^\uparrow(\vec \A_2),f_r^\downarrow(\vec \A_3))\in \Ccut) \leq \varrho,
\]
and
\[
\proba^\otimes_p((\vA_3,\vA_4)\notin \Ccut, (f_r^\downarrow(\vec \A_3),\vec \A_4)\in \Ccut ) \leq \varrho.
\]
We omit the details, as the proofs are essentially the same.

By a union bound, we deduce that the desired probability is at least
\[
\proba_p^\otimes((\vec \A_1,\vec \A_2)\notin \Ccut,(\vec \A_3,\vec \A_4)\notin \Ccut)-3\varrho =\Zb_p^2-3\varrho.
\]
Since $\varrho>0$ is arbitrary, and since by Lemma \ref{lem:LambdaLowMemory} and Theorem \ref{thm:LowerLambda}, $\Zb_p\geq \lambda_p^2$ is bounded below by a constant, this concludes the proof.
\end{proof}
We then write, dropping the dependence on $(p,r,\varrho')$, where $\varrho'$ will be chosen much later,
\[
\sigma(\vec A_1):=\proba^\otimes_p(\vA_1=\vec A_1,  (\vec \A_1,\vec \A_2)\notin \Ccut ,(\vec \A_3,\vec \A_4)\notin\Ccut , E_r(\vec \A_1,\vec \A_2,\vec \A_3,\vec \A_4)),
\]
and let $\Cgood$ be the set of $\vec A\in \Clive$ such that $\sigma(\vec A) \geq (1-\varrho') \nu^\uparrow_p(\vec A)\Zb_p$.
\begin{corollary} \label{cor:SizeCgood}
We have $\nu^\uparrow_p(\vec \Comp(0)\backslash \Cgood)\leq (\varrho/\varrho')\Zb_p$.
\end{corollary}
\begin{proof}
Note that for every $\vec A\in \vec \Comp(0)$, we have
\[
\sigma(\vec A)\leq \proba^\otimes_p(\vA_1=\vec A, (\vec \A_1,\vec \A_2)\notin \Ccut ,(\vec \A_3,\vec \A_4)\notin\Ccut)=\nu^\uparrow_p(\vec A)\Zb_p.
\]
Summing over all $\vec A$, since by Proposition \ref{pro:OrientedOpenTriangleSwitching}, $\sum_{\vec A} \sigma(\vec A)\geq (1-\varrho)\Zb_p^2$, we get
\[
(1-\varrho)\Zb_p^2\leq \nu^\uparrow_p(\Cgood)\Zb_p+(1-\varrho')\nu^\uparrow_p(\vec \Comp(0)\backslash \Cgood)\Zb_p=\nu^\uparrow_p(\vec \Comp(0))\Zb_p-\varrho'\nu^\uparrow_p(\vec \Comp(0)\backslash \Cgood)\Zb_p.
\]
The desired result then follows from $\nu^\uparrow_p(\vec \Comp(0))=\Zb_p$ and a basic calculation.
\end{proof}
We are particularly interested in $\Cgood$ because of the next property.
\begin{lemma} \label{lem:PrelimCoupling}
In the setting of Proposition \ref{pro:OrientedOpenTriangleSwitching}, there exists a constant $0<c<1$, depending on $d,L,C,r,\varrho$, such that for every $\vec A_1\in \Clive$,
\[
\sum_{\vec A_4 \in \Clive}\min \left (
\proba^\otimes_p\left (\vA_1=\vec A_1, \vA_4=\vec A_4, \substack{ (\vec A_1,\vec \A_2)\notin \Ccut \\ (\vec \A_2,\vec \A_3)\notin \Ccut \\ (\vec \A_3,\vec A_4)\notin \Ccut } \right ),
c\nu_p^\uparrow(\vec A_1 ) \nu_p^\downarrow (\vec A_4)
\right)\geq c\sigma(\vec A_1).
\]
\end{lemma}
\begin{proof}
It suffices to check that for every $\vec A_1,\vec A_4\in \Clive$ we have
\begin{equation}
\proba^\otimes_p\big ((\vec A_1,\vec \A_2)\notin \Ccut, (\vec \A_2,\vec \A_3)\notin \Ccut,(\vec \A_3,\vec A_4)\notin \Ccut\big ) \geq c\proba_p^\otimes(E_r(\vec A_1,\vec \A_2,\vec \A_3,\vec A_4)). \label{eq:BridgeCost}
\end{equation}
Indeed, if \eqref{eq:BridgeCost} holds, then the desired minimum is bounded below by
\[
c \proba^\otimes_p(\vec \A_1=\vec A_1, \vec \A_4=\vec A_4, (\vec \A_1,\vec \A_2)\notin \Ccut ,(\vec \A_3,\vec \A_4)\notin\Ccut , E_r(\vec \A_1,\vec \A_2,\vec \A_3,\vec \A_4)).
\]
If we omit the event $E_r(\cdot)$ in the probability, we get $ c\nu_p^\uparrow(\vec A_1 ) \nu_p^\downarrow (\vec A_4)$. Summing over $\vec A_4$, we deduce that the sum in the lemma is bounded below by $c\sigma(\vec A_1)$.

The proof of \eqref{eq:BridgeCost} is essentially the same as the proof of Proposition \ref{pro:RelateLocalLambda}. More precisely, we can rewrite the probability on the right-hand side of \eqref{eq:BridgeCost} as
\[
\sum_{\vec B_2, \vec B_3}
\proba_p^\otimes(f_r^\uparrow (\vec \A_2)=\vec B_2)
\proba_p^\otimes(f_r^\downarrow (\vec \A_3)=\vec B_3)
\1_{(\vec A_1, \vec B_2)\notin \Ccut}
\1_{(\vec B_2, \vec B_3)\notin \Ccut}
\1_{(\vec B_3, \vec A_4)\notin \Ccut},
\]
which we upper-bound by
\begin{equation}
M^\uparrow M^\downarrow  \sum_{\vec B_2, \vec B_3} \proba_p^\otimes(\vA_2=\vec B_2,\vA_3=\vec B_3) \1_{(\vec A_1, \vec B_2)\notin \Ccut}\1_{(\vec B_2, \vec B_3)\notin \Ccut}\1_{(\vec B_3, \vec A_4)\notin \Ccut}, \label{eq:RewritingBridge}
\end{equation}
where
\[
M^\uparrow : = \sup_{\vec B\in \vec \Comp(0)} \frac{\proba_p^\otimes(f^\uparrow_r(\vec \A_2)=\vec B)}{\proba_p^\otimes(\vA_2=\vec B)},
\]
and where $M^\downarrow$ is defined similarly.
By adapting the last part of the proof of Proposition \ref{pro:RelateLocalLambda}, the two suprema can be bounded above by a constant.
We recognise the sum in \eqref{eq:RewritingBridge} as the left-hand side of \eqref{eq:BridgeCost}.
This concludes the proof.
\end{proof}
We now wish to show an analog of Lemma \ref{lem:PrelimCoupling} using $\proba_p^h$, which introduces a bias through $h_p$. To this end, we first estimate a few related sums:
\begin{lemma} \label{eq:SumForCoupling}
We have
\[
\Zb_p^h:=\sum_{\vec A} \nu^{\downarrow}_p(\vec A) h_p(\vec A) = \lambda_p \E_p^\otimes[h_p(\vA_0)] \geq \lambda_p^2 \|h_p\|_\infty.
\]
Moreover, for every $v:\Clive \mapsto \R$, we have
\[
\|h_p\cdot v\|_1 \leq \lambda_p^{-2} \Zb^h_p \|v\|_1.
\]
\end{lemma}
\begin{proof}
First, we have
\[
\Zb_p^h=\sum_{\vec A_0,\vec A_1}
h_p(\vec A_1) \frac{\mu_p(\vec A_0) \mu_p(\vec A_1)}{\Omega^2\X(p)^2}
\1_{(\vec A_0,\vec A_1)\notin \Ccut}
= \sum_{\vec A_0} \frac{\mu_p(\vec A_0)}{\Omega\X(p)}
\E_p^\otimes\left [ h_p(\vA_1) \1_{(\vec A_0,\vA_1)\notin \Ccut} \right ].
\]
We recognize the expectation as $[\T_p(h_p)](\vec A_0)$, and the claimed equality follows from Lemma \ref{lem:vpd}.
For the first inequality, we have for every $\vec A_0$, using again $\T_p(h_p)=\lambda_p h_p$, and dropping the indicator,
\[
\lambda_p h_p(\vec A_0)\leq \E^\otimes_p[h_p(\vA_1)],
\]
and the inequality follows since $\vec A_0$ is arbitrary. Finally, for the second inequality,
\[
\|h_p\cdot v\|_1 \leq \|h_p\|_\infty \| v\|_1\leq \lambda_p^{-2} \Zb_p^h\|v\|_1. \qedhere
\]
\end{proof}
Next, note that by \eqref{eq:irreductible}, $\nu_p^\uparrow$ is positive on $\Clive$. We write
\[
\eta^\uparrow_p: \vec A \in \Clive \mapsto \frac{\Omega\X(p) \nu_p^\uparrow(\vec A)}{\mu_p(\vec A)}>0.
\]

\begin{lemma} \label{lem:PrelimCouplingBis}
In the setting of Lemma \ref{lem:PrelimCoupling}, if $\varrho'\leq \lambda_p^2/3$, we have for every $\vec A_1\in \Cgood$,
\[
\sum_{\vec A_4 \in \Clive}\min \left (\frac{h_p(\vec A_1)}{\eta^\uparrow_p(\vec A_1)}\P_p^3(\vec A_4|\vec A_1),c h_p(\vec A_4) \nu_p^\downarrow (\vec A_4) \right)\geq \frac{2}{3} c \Zb_p^h.
\]
\end{lemma}
\begin{proof}
Fix $\vec A_1\in \Cgood $. Let
\[
v: \vec A_4 \in \Clive \mapsto
\left (c\nu_p^\uparrow(\vec A_1 ) \nu_p^\downarrow (\vec A_4)
-\proba^\otimes_p\left (\vA_1=\vec A_1,\vA_4=\vec A_4,\substack{ (\vec A_1,\vec \A_2)\notin \Ccut \\ (\vec \A_2,\vec \A_3)\notin \Ccut \\ (\vec \A_3,\vec A_4)\notin \Ccut } \right ) \right )^+,
\]
where $(\cdot)^+=\max(0,\cdot)$ denotes the positive part of a real number. Using the generic identity $\min(x,y)=y-(y-x)^+$, valid for $x,y\in \R$, we may rewrite Lemma \ref{lem:PrelimCoupling} as
\[
\|v\|_1 \leq c(\nu_p^\uparrow(\vec A_1 )\Zb_p -\sigma(\vec A_1)) \leq c\varrho' \nu^{\uparrow}_p(\vec A_1) \Zb_p.
\]
It follows from Lemmas \ref{eq:SumForCoupling} and \ref{lem:LambdaLowMemory} that
\[
\|h_p\cdot v\|_1 \leq \lambda_p^{-2} \Zb_p^h \|v\|_1\leq (c/3)\nu^{\uparrow}_p(\vec A_1)\Zb_p^h.
\]
Furthermore, by \eqref{eq:DoobTransform} we may rewrite $h_p\cdot v$ as
\[
h_p\cdot v : \vec A_4 \in \Clive \mapsto \left (c\nu_p^\uparrow(\vec A_1 ) h_p(\vec A_4) \nu_p^\downarrow (\vec A_4)-\lambda_p^3 \frac{\mu_p(\vec A_1)}{\Omega\X(p)}h_p(\vec A_1)\P_p^3(\vec A_4|\vec A_1) \right )^+.
\]
Combining the bound on $\|h_p\cdot v\|_1$, the inequality $\lambda_p\leq 1$ (see Lemma \ref{lem:LambdaLowMemory}), and the generic identity $\min(x,y)=y-(y-x)^+$, it results that the sum in the lemma is at least
\[
\sum_{\vec A_4\in \Clive } ch_p(\vec A_4) \nu_p^\downarrow (\vec A_4) -\frac{\|h_p\cdot v\|_1}{\nu_p^\uparrow(\vec A_1)} \geq c\Zb_p^h - \frac{c}{3} \Zb_p^h.  \qedhere
\]
\end{proof}
We now introduce a few other notations.
Let $\Fc$ be the set of vectors $(v_{\vec A})_{\vec A\in \Clive}$ with $\|v\|_1<\infty$.
We let $\Fc^+$ be the set of such non-negative vectors.
Let $\Fc^+_1:=\{v\in \Fc^+,\|v\|_1=1\}$.
The Markov kernel $\P_p$ acts naturally on $\Fc,\Fc^+,\Fc^+_1$ by $\P_p:v\mapsto \left (\sum_{\vec A\in \Clive}v_{\vec A}\P_p(\vec A, \vec B) \right)_{\vec B\in \Clive}$.
For every pointed cluster $\vec A$, we write $\1_{\vec A}$ for the vector $(\1_{\vec A'=\vec A})_{\vec A'\in \Clive}\in \Fc^+_1$.
To show Theorem \ref{thm:MixingTime}, we need to upper-bound $\|\P^k_p(\1_{\vec A}-\vpg_p)\|_1$.
More generally, for every $v,w\in \Fc_1^+$, we bound the evolution of $\|\P^k_p(v-w)\|_1$ as $k$ increases.
Let $\Cgood'$ be the set of $\vec A\in \Clive$ with $\eta^\uparrow_p(\vec A)\geq \varrho''$, where $\varrho''$ will be chosen later.
The next result is the central argument for such an upper bound:
\begin{proposition} \label{pro:Coupling}
In the setting of Lemma \ref{lem:PrelimCouplingBis}, for every $v,w\in \Fc^+$ with support in $\Cgood \cap \Cgood'$, we have
\[
\|\min(\P^3_p(h_p\cdot v), \P^3_p(h_p\cdot w))\|_1\geq (c/3) \lambda_p^2\varrho'' \|h_p\|_\infty\min(\|v\|_1,\|w\|_1).
\]
\end{proposition}
\begin{proof}
Without loss of generality, after decreasing the larger vector, which does not change the right-hand side, we may focus on the case where $\|v\|_1=\|w\|_1$.
Also, if the statement of the proposition is satisfied for $(v,w)$ and $(v',w')$, then, by the triangle inequality, it is also satisfied for $(v+v',w+w')$.
Thus, it suffices to check the statement for $v,w$ of the form $\1_{\vec A_1}$ and $\1_{\vec A'_1}$.
In that case, we may rewrite the left-hand side of the proposition as
\[
\sum_{\vec A_4\in \Clive} \min \left (h_p(\vec A_1)\P_p^3(\vec A_4|\vec A_1),h_p(\vec A'_1)\P_p^3(\vec A_4|\vec A'_1)\right).
\]
Using $\eta^\uparrow_p(\vec A_1),\eta^\uparrow_p(\vec A'_1)\geq \varrho''$ by definition of $\Cgood'$, the last sum is at least
\[
\varrho''
\sum_{\vec A_4\in \Clive} \min \left (\frac{h_p(\vec A_1)}{\eta^\uparrow_p(\vec A_1)} \P_p^3(\vec A_4|\vec A_1),\frac{h_p(\vec A'_1)}{\eta^\uparrow_p(\vec A'_1)}  \P_p^3(\vec A_4|\vec A'_1)\right),
\]
Then, using the generic inequality $\min(x,z)\geq \min(x,y)+\min(y,z)-y$, valid on $\R_+$, and Lemma \ref{lem:PrelimCouplingBis}, this last sum is bounded below by
\[
(4/3)c\Zb_p^h-c\sum_{\vec A_4\in \Clive} h_p(\vec A_4)\nu_p^\downarrow(\vec A_4)=\frac{c}{3}\Zb_p^h \geq \frac{c\lambda_p^2}{3}\|h_p\|_\infty,
\]
using Lemma \ref{eq:SumForCoupling} for the last inequality.
\end{proof}
The previous results ensure that we may couple two Markov chains with kernel $\P_p$ with good probability, when started from $\vec A_1,\vec A'_1\in\Cgood\cap \Cgood'$.
To deal with the complementary set $\Clive \backslash (\Cgood\cap \Cgood')$,
we are interested in
\[
\|v\|_{h+}:=\sum_{\vec A\in \Clive} \frac{\max(v_{\vec A},0)}{h_p(\vec A)} \quad \text{and} \quad \|v\|_{h-}:=\sum_{\vec A\in \Clive} \frac{\max(-v_{\vec A},0)}{h_p(\vec A)},
\]
and, more precisely, in the evolution of $\|\P_p^k(v)\|_{h+}$ and $\|\P_p^k(v)\|_{h-}$ as $k$ increases.
We first show the following general bounds:
\begin{lemma} \label{lem:BasicEvoMass}
For every $v\in \Fc$, we have
\[
\|\P_p(v)\|_1\leq \|v\|_1 \quad \text{and} \quad \|\P_p(v)\|_{h+}\leq (1/\lambda_p)\|v\|_{h+} \quad \text{and} \quad \|\P_p(v)\|_{h-}\leq (1/\lambda_p)\|v\|_{h-}.
\]
\end{lemma}
\begin{proof}
The first inequality follows directly from the fact that $\P_p$ is a Markov kernel. For the second, by definition of $\P_p$, we have
\[
\|\P_p(v)\|_{h+}=\sum_{\vec B\in \Clive}\frac{1}{h_p(\vec B)}\max\left (0, \sum_{\vec A \in\Clive} \frac{h_p(\vec B)}{\lambda_p h_p(\vec A)}\frac{\mu_p(\vec B)}{\Omega\X(p)} v_{\vec A} \1_{(\vec A, \vec B)\notin \Ccut} \right ).
\]
The $h_p(\vec B)$ terms cancel each other. We can remove the indicator $\1_{(\vec A, \vec B)\notin \Ccut}$. We may then remove the contribution of the sum coming from the $\vec A$ with $v_{\vec A}<0$. Rearranging the sums, we get
\[
\|\P_p(v)\|_{h+} \leq \frac{1}{\lambda_p} \sum_{\vec A\in \Clive} \frac{\max(0,v_{\vec A})}{h_p(\vec A)}  \sum_{\vec B}\frac{\mu_p(\vec B)}{\Omega\X(p)} .
\]
The last sum over $\vec B$ equals $1$, and then we recognize $\|v\|_{h+}/\lambda_p$. The proof for $\|v\|_{h-}$ is the same.
\end{proof}
\begin{proposition} \label{pro:TrueCoupling}
Let $p<p_c$ and $C>0$ be such that \eqref{eq:initialization} holds. There exist $c',K>0$, depending on $d,L,C$, such that if $\X(p)\geq K$, then for every $v,w\in \Fc^+_1$, at least one of the following inequalities holds
\[
\frac{\|\P^4_p(v-w)\|_1}{\|v-w\|_1} \leq (1-c') \quad \text{or} \quad \frac{\|\P^4_p(v-w)\|_{h+}}{ \|v-w\|_{h+}} \leq \frac{\lambda_p^4}{2} \quad \text{or} \quad \frac{\|\P^4_p(v-w)\|_{h-}}{ \|v-w\|_{h-}}\leq \frac{\lambda_p^4}{2}.
\]
\end{proposition}
\begin{proof}
Let $\varrho'= \lambda_p^2/3$, $\varrho= \lambda_p^8\varrho'/6$ and $\varrho''= \lambda_p^8/6$. This choice is possible since, by Theorem \ref{thm:LowerLambda}, $\varrho, \varrho',\varrho''$ are bounded below by a constant depending only on $d,L,C$.
Note that it suffices to check the desired result for every $v,w$ with disjoint support.
Indeed, if the result holds for $v'=2\max(0,v-w)/(\|v-w\|_1)$ and $w'=2\max(0,w-v)/(\|v-w\|_1)$, which are both in $\Fc^+$, it also holds for $v,w$ since $v'-w'=2(v-w)/\|v-w\|_1$.
We will assume throughout the proof that $v,w$ indeed have disjoint support.
In that case,
\[
\|v-w\|_1=2\quad \text{and} \quad \|v-w\|_{h+}=\|v\|_{h+} \quad \text{and}  \quad \|v-w\|_{h-}=\|w\|_{h+}.
\]

We may decompose $ \P_p(v)$ as $v_1+v_2+v_3$,
where $v_1$ is supported on $\Cgood\cap \Cgood'$, $v_2$ is supported on $\Clive\backslash \Cgood$, and $v_3$ is supported on $\Cgood\backslash\Cgood'$. We decompose $\P_p(w)=w_1+w_2+w_3$ similarly. By the triangle inequality, we have
\[
\|\P^4_p(v)\|_{h+} \leq \|\P^3_p(v_1)\|_{h+}+ \|\P^3_p(v_2)\|_{h+}+ \|\P^3_p(v_3)\|_{h+}.
\]
Then, using Lemma \ref{lem:BasicEvoMass}, we get
\begin{equation}
\|\P^4_p(v)\|_{h+} \leq \lambda_p^{-3} \|v_1\|_{h+}+\lambda_p^{-2} \|\P_p(v_2)\|_{h+}+ \lambda_p^{-2}\|\P_p(v_3)\|_{h+}. \label{eq:SplitGoodBadCoupling}
\end{equation}

We have
\[
\|\P_p(v_2)\|_{h+}
= \sum_{\substack{\vec A_0,\vec A_2\in \Clive \\ \vec A_1\notin \Cgood}}
\frac{1}{h_p(\vec A_2)}\frac{h_p(\vec A_2)}{\lambda_p^2 h_p(\vec A_0)}
v_{\vec A_0} \proba_p^\otimes(\vA_1=\vec A_1,\vA_2=\vec A_2)
\1_{(\vec A_0,\vec A_1)\notin \Ccut}\1_{(\vec A_1,\vec A_2)\notin \Ccut}.
\]
The $h_p(\vec A_2)$ terms cancel. By increasing the sum, we drop the indicator $\1_{(\vec A_0,\vec A_1)\notin \Ccut}$, then sum over $\vec A_0$ to get a $\|v\|_{h+}$ term. Thus,
\[
\|\P_p(v_2)\|_{h+}\leq \lambda_p^{-2}\|v\|_{h+} \sum_{\vec A_1\notin \Cgood} \proba_p^\otimes(\vA_1=\vec A_1,(\vA_1,\vA_2)\notin\Ccut)=\lambda_p^{-2}\|v\|_{h+}\nu^\uparrow_p(\Clive \backslash\Cgood).
\]
By Corollary \ref{cor:SizeCgood}, Lemma \ref{lem:LambdaLowMemory}, and $\varrho= \lambda_p^8\varrho'/6$, we get
\[
\|\P_p(v_2)\|_{h+}\leq (\lambda_p^6/6)\|v\|_{h+} .
\]
Similarly, using \eqref{eq:OrientedX},
\[
\nu^\uparrow_p(\Clive\backslash \Cgood')=\sum_{\vec A_1\notin \Cgood'} \nu^\uparrow_p(\vec A_1)\leq \sum_{\vec A_1\notin \Cgood'} \frac{\mu_p(\vec A_1)}{\Omega\X(p)} \varrho'' \leq \frac{\lambda_p^8}{6},
\]
we get
\[
\|\P_p(v_3)\|_{h+}\leq \lambda_p^{-2}(\lambda_p^8/6)\|v\|_{h+}=(\lambda_p^6/6)\|v\|_{h+} .
\]
Therefore, by \eqref{eq:SplitGoodBadCoupling}, we either have $\|v_1\|_{h+}\geq (\lambda_p^7/6)\|v\|_{h+}$ or
\[
\|\P^4_p(v)\|_{h+}/\|v\|_{h+} \leq \lambda_p^{-3} (\lambda_p^7/6)+\lambda_p^{-2} (\lambda_p^6/6)+ \lambda_p^{-2}(\lambda_p^6/6)=\lambda_p^4/2.
\]
We proceed similarly for $\|\P^4_p(w)\|_{h+}$.

It remains to deal with the case $\|v_1\|_{h+}\geq (\lambda_p^7/6)\|v\|_{h+}$ and $\|w_1\|_{h+}\geq (\lambda_p^7/6)\|w\|_{h+}$. In that case, by applying Proposition \ref{pro:Coupling} to $v_1/h_p$ and $w_1/h_p$, we get
\[
\|\min(\P_p^3(v_1), \P_p^3(w_1))\|_1 \geq (c/3) \lambda_p^2\varrho''\|h_p\|_\infty \min(\|v_1\|_{h+},\|w_1\|_{h+}).
\]
Using $\P_p(v)\geq v_1$ and $\P_p(w)\geq w_1$ pointwise, we get
\[
\|\min(\P_p^4(v), \P^4_p(w))\|_1 \geq (c/3) \lambda_p^2\varrho''\|h_p\|_\infty (\lambda_p^7/6)\min(\|v\|_{h+},\|w\|_{h+}).
\]
Also note that $\|v\|_{h+}=\|v/h_p\|_1\geq \|v\|_1/\|h_p \|_\infty=1/\|h_p\|_\infty$. Similarly $\|w\|_{h+}\geq 1/\|h_p\|_\infty$. Therefore,
\[
\|\min(\P_p^4(v), \P^4_p(w))\|_1 \geq (c/3) \lambda_p^2\varrho''(\lambda_p^7/6).
\]
Hence, since $\P_p^4(v)$ and $\P^4_p(w)$ are in $\Fc_1^+$, we deduce
\[
\|\P_p^4(v)-\P^4_p(w)\|_1/2 \leq 1-(c/3) \lambda_p^2(\lambda_p^8/6)(\lambda_p^7/6).
\]
Since, by Theorem \ref{thm:LowerLambda}, $\lambda_p$ is bounded below by a positive constant, this concludes the last case.
\end{proof}
\begin{proof}[Proof of Theorem \ref{thm:MixingTime}]
Let $N>0$. Let
\[
f: v\in \Fc\mapsto \|v\|^N_1 \|v\|_{h+}\|v\|_{h-}.
\]
By Theorem \ref{thm:LowerLambda}, Lemma \ref{lem:BasicEvoMass}, and Proposition \ref{pro:TrueCoupling}, if $N>0$ is large enough (i.e. such that $(1-c')^N\leq \lambda_p^8/2$), we have for every $v,w\in \Fc_1^+$ that
\[
f(\P_p^4(v-w)) \leq (1/2)f(v-w).
\]
Stated differently, writing $\Fc_0$ for the subset of $v\in \Fc$ with $\sum_{\vec A} v_{\vec A}=0$, we have for every $v\in \Fc_0$,
\[
f(\P_p^4(v)) \leq (1/2)f(v).
\]
Since $\P_p^4$ also acts on $\Fc_0$, we can recursively apply this inequality to get, for every $k\in \N$, $v\in \Fc_0$,
\begin{equation}
f(\P_p^{4k}(v)) \leq (1/2)^kf(v). \label{eq:IteratedCoupling}
\end{equation}
Also note that for every $v\in \Fc_0$, we have
\[
\|v\|_{h+}= \|\max(0,v/h_p)\|_1\geq \|\max(0,v)\|_1/\|h_p\|_\infty =\|v\|_1/(2\|h_p\|_\infty).
\]
Similarly, we have $\|v\|_{h-}\geq \|v\|_1/(2\|h_p\|_\infty)$. Also,
\[
\max(\|v\|_{h-},\|v\|_{h+}) \leq \|v\|_1/\left (2\inf_{\vec A\in \Clive} h_p(\vec A) \right ).
\]
Hence, by \eqref{eq:IteratedCoupling} and Lemma \ref{lem:vpd}, we get for every $v\in \Fc_0$,
\[
\|\P_p^{4k}(v)\|^{N+2}_1\leq 2^{-k}\|v\|^{N+2}_1 \left (\frac{\|h_p\|_\infty}{\inf_{\vec A\in \Clive} h_p(\vec A)} \right )^2 \leq 2^{-k} \|v\|^{N+2}_1 \frac{\Omega^2 \X(p)^2}{(1-p)^{2\Omega}}.
\]
Applying the last result to $v=\1_{\vec A}-\vpg_p$, we get, for every $k\geq 0$,
\[
\|\P^{4k}_p(\1_{\vec A}-\vpg_p) \|_1 \leq \left (2^{-k}2^{N+2}  \frac{\Omega^2 \X(p)^2}{(1-p)^{2\Omega}}\right )^{1/(N+2)} \leq 2^{-(k+1)/(N+2)}\X(p),
\]
where the last inequality holds whenever $\X(p)$ is large enough. Finally, since, by Lemma \ref{lem:BasicEvoMass}, $k\mapsto  \|\P^{k}_p(\1_{\vec A}-\vpg_p) \|_1$ is decreasing, we have, for every $k\in \N$,
\[
\|\P^{k}_p(\1_{\vec A}-\vpg_p) \|_1 \leq 2^{-k/(4N+8)}\X(p).
\]
Since the left-hand side equals the sum in Theorem \ref{thm:MixingTime}, this concludes the proof.
\end{proof}
\subsection{High frequencies for $\Four[\tau_{p}^{\square,k}]$ for Section \ref{sec:TCL}} \label{Sec:HighFrequencies}
The aim of this subsection is to show the following result, which will only be used in Section \ref{sec:TCL}.
We do it here for convenience, because the proof shares many similarities with the proof of Theorem \ref{thm:MixingTime}.
First, recall, for $f:\Z^d\mapsto \R$, that $\Four[f]:\kappa\mapsto \sum_{x\in \Z^d} f(x)e^{\i\langle \kappa,x\rangle}$ denotes the Fourier transform of $f$.
\begin{theorem} \label{thm:FT_High}
Let $C>0$. There exists $F,c,C'>0$ depending only on $d,L,C$ such that for every $p<p_c$ satisfying \eqref{eq:initialization} and $\X(p)\geq F$, we have for every $k\geq 3$, for every $\kappa\in [-\pi,\pi]^{d}$,
\[
|\Four[\tau_p^{\square,k}](\kappa)| \leq C' (1-c\|\kappa\|^2_\infty)^k\|\tau_p^{\square,k}\|_1.
\]
\end{theorem}
For every $i\geq 0$, $\vec A_0, \vec A_i\in \vec \Comp(0)$, let
\[
P^\otimes_i(\vec A_0,\vec A_i):= \proba^\otimes_p(\Tcut>i|\vA_0=\vec A_0, \vA_i=\vec A_i).
\]
Then let $\Amp_{i,\kappa}(\vec A_0,\vec A_i)\in \R_+$ and $\Rot_{i,\kappa}(\vec A_0,\vec A_i)$ be such that $|\Rot_{i,\kappa}(\vec A_0,\vec A_i)|=1$ and such that
\[
\Amp_{i,\kappa}(\vec A_0,\vec A_i)\Rot_{i,\kappa}(\vec A_0,\vec A_i) = \E_p^\otimes \left [e^{\i\langle \kappa, \Delta(\vA_1)+\dots \Delta(\vA_{i-1})\rangle}\middle |\vA_0=\vec A_0, \vA_i=\vec A_i,\Tcut>i  \right ].
\]
By convention, $\Amp_{i,\kappa}(\vec A_0,\vec A_i)=1$ if $P^\otimes_i(\vec A_0,\vec A_i)=0$.
 Let us first assume the next technical result:
\begin{proposition} \label{pro:SmallRotation}
Let $C,\varrho,\varrho'>0$. There exists $F,c>0$ such that for every $p<p_c$ satisfying \eqref{eq:initialization} and $\X(p)\geq F$, for every $\kappa\in [-\pi,\pi]^d$, we have
\[
\proba_p^\otimes \left (\Upsilon_{3,\kappa}(\vA_0,\vA_3)>1-c\|\kappa\|^2_\infty \quad \text{ and } \quad P_3^\otimes(\vA_0,\vA_3)> \varrho \right ) \leq \varrho'.
\]
\end{proposition}
\begin{proof}[Proof of Theorem \ref{thm:FT_High} given Proposition \ref{pro:SmallRotation}]
First, by integrating \eqref{eq:RewriteA_TauSquare}, we have
\[
\Four[\tau_p^{\square,k}](\kappa)= \Omega^k \X(p)^{k+1} \E_p^\otimes[e^{\i\langle \SA_k-\delta(\vec \A_k),\kappa\rangle},\Tcut>k].
\]
We condition on $\vA_0,\vA_3,\vA_6,\dots, \vA_{3I}$, where $I$ is the largest integer such that $3I<k$, so that we may rewrite the last expectation as
\[
\E_p^\otimes  \left [
e^{-\i\langle \delta(\vA_k),\kappa\rangle }
\prod_{j=0}^{I} e^{\i \langle \Delta(\vA_{3j}),\kappa \rangle}
\prod_{j=0}^{I-1}[P_3^\otimes\cdot \Amp_{3,\kappa}\cdot  \Rot_{3,\kappa}](\vA_{3j},\vA_{3j+3})
\prod_{3I<j\leq k} e^{\i \langle \Delta(\vA_j),\kappa \rangle} \1_{(\vA_{j-1},\vA_j)\notin \Ccut}
\right ].
\]
Thus, by the triangle inequality,
\begin{equation}
|\Four[\tau_p^{\square,k}](\kappa)|
\leq \Omega^k \X(p)^{k+1}
\E_p^\otimes \left [
\prod_{j=0}^{I-1}[P_3^\otimes\cdot \Amp_{3,\kappa}](\vA_{3j},\vA_{3j+3})
\prod_{j=3I+1}^k \1_{(\vA_{j-1},\vA_j)\notin \Ccut}
\right ]. \label{eq:Fourrier_Partial_Cond}
\end{equation}

Fix $\varrho,\varrho'>0$. By Theorem \ref{thm:LowerLambda}, we may take $\varrho,\varrho'>0$ small enough such that $\varrho\leq \lambda_p^9/2$ and $\varrho'\leq \lambda_p^{54}/3^6$, whenever $\X(p)$ is large enough.
 Let $c$ be as in Proposition \ref{pro:SmallRotation}.
  Let
 \[
S_1:=\{(\vec A_0,\vec A_3)\in (\vec \Comp(0))^2,  \Amp_{3,\kappa}(\vec A_0,\vec A_3) \leq 1-c\|\kappa\|_\infty^2\},
\]
 and
  \[
  S_2:=\{(\vec A_0,\vec A_3)\in (\vec \Comp(0))^2,  P_3^\otimes(\vec A_0,\vec A_3)\leq \varrho\},
  \]
  and let
  \[
  S_3:=(\vec \Comp(0))^2\backslash (S_1\cup S_2).
  \]
	 We write $N_1,N_2,N_3$ for the number of pairs $(\vA_{3j},\vA_{3j+3})$ with $0\leq j <I$ that are in $S_1,S_2,S_3$, respectively.
We split the expectation in \eqref{eq:Fourrier_Partial_Cond} according to whether $N_1\geq I/3$, $N_2\geq I/3$, or $N_3\geq I/3$, and we denote the corresponding expectations by $E_1,E_2,E_3$, respectively.
	  First, for $E_1$, using that $\Amp_{3,\kappa}(\vA_{3j},\vA_{3j+3})\leq 1-c\|\kappa\|_\infty^2$ for at least $I/3$ pairs, and $0\leq \Amp_{3,\kappa}\leq 1$ pointwise,
\begin{align*}
E_1& \leq \E_p^\otimes \left [(1-c\|\kappa\|_\infty^2)^{I/3}\prod_{j=0}^{I-1}P_3^\otimes(\vA_{3j},\vA_{3j+3}) \prod_{j=3I+1}^k \1_{(\vA_{j-1},\vA_j)\notin \Ccut} \right ]
\\ & =(1-c\|\kappa\|^2_\infty)^{I/3}\proba_p^\otimes(\Tcut>k) .
\end{align*}

For $E_2$, using $0\leq P_3^\otimes\leq 1$ and $0\leq \Amp_{3,\kappa}\leq 1$ pointwise,
and $P_3^\otimes(\vA_{3j},\vA_{3j+3})\leq \varrho$ for at least $I/3$ values of $j$, we have $E_2\leq \varrho^{I/3}.$
Then, using Lemma \ref{lem:LambdaLowMemory} and $\varrho\leq \lambda_p^9/2$, we get
\[
E_2 \leq (1/2)^{I/3} \lambda_p^{3I} \leq (1/2)^{I/3}\proba_p^\otimes(\Tcut>k)/\lambda_p^4.
\]

For $E_3$, by splitting odd and even indices, and using basic inequalities for sums of independent Bernoulli random variables, we get
\[
E_3\leq \proba_p^\otimes(N_3\geq I/3) \leq 2^I \varrho'^{I/6}\leq (2/3)^I \lambda_p^{3I}\leq (2/3)^I\proba_p^\otimes(\Tcut>k)/\lambda_p^4,
\]
using again Lemma \ref{lem:LambdaLowMemory} for the last inequality.

Summing the bounds on $E_1,E_2,E_3$, and returning first to \eqref{eq:Fourrier_Partial_Cond} and then to \eqref{eq:RewriteA_TauSquare}, we deduce
\begin{align*}
|\Four[\tau_p^{\square,k}](\kappa)| & \leq ((1-c\|\kappa\|^2_\infty)^{I/3}+(1/2)^{I/3}/\lambda_p^4+ (2/3)^I/\lambda_p^4) \Omega^k \X(p)^{k+1}\proba_p^\otimes(\Tcut>k)
\\ & = ((1-c\|\kappa\|^2_\infty)^{I/3}+(1/2)^{I/3}/\lambda_p^4+ (2/3)^I/\lambda_p^4) \|\tau_p^{\square,k}\|_1.
\end{align*}
By Theorem \ref{thm:LowerLambda}, $\lambda_p$ is bounded below by a constant, and the desired result follows by fixing $c$ small enough and taking $\X(p)$ large enough.
 \end{proof}
 The main idea behind the proof of Proposition \ref{pro:SmallRotation} is that we can usually pay a multiplicative constant to locally modify $\vA_1,\vA_2$ to change $\Delta(\vA_1)+\Delta(\vA_2)$.
More precisely, let $(\delta_1,\dots,\delta_d)$ be the usual orthonormal basis of $(\R^d,\langle\cdot,\cdot \rangle)$.
Let $f_r^\uparrow (\vec A)$ and $f_r^\downarrow(\vec A)$ be defined as in the previous subsection (see Figure \ref{fig:LocalModifB}).
 For $1\leq i \leq d$, let $E_{r,i}(\vec A_0,\vec A_1,\vec A_2,\vec A_3)$ be the event that, 
 \[
(\vec A_0,f_r^\uparrow(\vec A_1))\notin \Ccut \quad \text{and} \quad (f_r^\uparrow(\vec A_1),f_r^\downarrow(\vec A_2)-\delta_i)\notin \Ccut \quad \text{and} \quad (f_r^\downarrow(\vec A_2)-\delta_i,\vec A_3)\notin \Ccut .
\]
 \begin{proposition} \label{pro:OrientedOpenTriangleSwitching2}
 Let $p<p_c$ and $C>0$ be such that \eqref{eq:initialization} holds. For every $\varrho>0$, if $r>0$ is large enough depending on $d,L,C,\varrho$, and if $\X(p)$ is large enough depending on $d,L,C,r,\varrho$, we have for every $1\leq i \leq d$,
\[
\proba^\otimes_p( (\vec \A_0,\vec \A_1)\notin \Ccut ,(\vec \A_2,\vec \A_3)\notin\Ccut , E_{r,i}(\vec \A_0,\vec \A_1,\vec \A_2,\vec \A_3)) \geq (1-\varrho)\Zb_p^2.
\]
\end{proposition}
\begin{proof}
The proof is exactly the same as that of Proposition \ref{pro:OrientedOpenTriangleSwitching} and is thus omitted.
\end{proof}
Recall the definitions of $\nu^\downarrow_p$, $\nu^\uparrow_p$, and $\Zb_p$ from \eqref{eq:def_nu_up+down} in the previous subsection.
Let
\[
\sigma_{i}(\vec A_0,\vec A_3):=\proba^\otimes_p(\vA_0=\vec A_0, \vA_3=\vec A_3, (\vec \A_0,\vec \A_1)\notin \Ccut ,(\vec \A_2,\vec \A_3)\notin\Ccut , E_{r,i}(\vec \A_0,\vec \A_1,\vec \A_2,\vec \A_3)),
\]
and let $\vec \Comp_{i,\varrho}$ be the set of $\vec A_0,\vec A_3\in \Clive$ such that $\sigma_i(\vec A_0,\vec A_3) \geq (1-\sqrt{\varrho}) \nu^\uparrow_p(\vec A_0)\nu_p^\downarrow(\vec A_3)$.
\begin{corollary} \label{cor:SizeCgood2}
We have $\nu^\uparrow_p\times \nu^\downarrow_p((\vec \Comp(0))^2\backslash \vec \Comp_{i,\varrho})\leq \sqrt{\varrho} \Zb_p^2$.
\end{corollary}
\begin{proof}
The proof is essentially the same as that of Corollary \ref{cor:SizeCgood}, with $\varrho'=\sqrt{\varrho}$, and is thus omitted.
\end{proof}
\begin{corollary} \label{cor:Bad_Case_SmallRotation}
We have
\[
\proba_p^\otimes ((\vA_0,\vA_3)\notin \vec \Comp_{i,\varrho},\quad P_3^\otimes(\vA_0,\vA_3)> \varrho^{1/4}  ) \leq \varrho^{1/4}.
\]
\end{corollary}
\begin{proof}
By definition of $\nu^\uparrow_p,\nu^\downarrow_p$ (see \eqref{eq:def_nu_up+down}), we have
\[
\nu^\uparrow_p\times \nu^\downarrow_p((\vec \Comp(0))^2\backslash \vec \Comp_{i,\varrho})=\E_p^\otimes[\nu^\uparrow_p(\vA_0)\nu_p^\downarrow(\vA_3)\1_{(\vA_0,\vA_3)\notin \vec \Comp_{i,\varrho}}].
\]
Then it follows directly from the definition of $P_3^\otimes$ and \eqref{eq:def_nu_up+down} that $ P_3^\otimes(\vA_0,\vA_3)\leq \nu_p^\uparrow(\vA_0)\nu_p^\downarrow(\vA_3)$.
Thus, the desired result follows from Corollary \ref{cor:SizeCgood2} and Markov's inequality.
\end{proof}
Since $\varrho>0$ can be taken arbitrarily small, it suffices to deal with the case $(\vA_0,\vA_3)\in \vec \Comp_{i,\varrho}$.
\begin{proposition}\label{pro:RotateVsRotate}
In the setting of Proposition \ref{pro:OrientedOpenTriangleSwitching2},
there exists $c>0$, depending only on $d,L,C,r,\varrho$, such that
if $\X(p)$ is large enough, depending only on $d,L,C,r,\varrho$,
we have, for every $1\leq i \leq d$,
every $(\vec A_0,\vec A_3)\in \vec \Comp_{i,\varrho}$ with $P_3^\otimes (\vec A_0,\vec A_3)\geq 2\sqrt{\varrho}$,
every $\kappa=(\kappa_i)_{1\leq i \leq d}\in [-\pi,\pi]^d$,
\[
\Upsilon_{3,\kappa}(\vec A_0,\vec A_3) \leq 1-c\kappa^2_i.
\]
\end{proposition}
\begin{proof}
Fix $(\vec A_0, \vec A_3)\in \vec \Comp_{i,\varrho}$. Let $S$ be the set of all pairs $(\vec A_1,\vec A_2)\in \vec \Comp(0)$ such that $(\vec A_0,\vec A_1)$, $(\vec A_1,\vec A_2)$, $(\vec A_2,\vec A_3)\notin \Ccut$. Let
$g:\vec A_1,\vec A_2\mapsto e^{\i\langle \Delta(\vec A_1)+\Delta(\vec A_2),\kappa\rangle}$.
We have
\[
\Amp_{3,\kappa} (\vec A_0, \vec A_3) = \left | \E_p^\otimes\left [g(\vA_1,\vA_2) \middle | (\vA_1,\vA_2)\in S\right ] \right |.
\]
Thus, for some $\zeta=\zeta(\vec A_0,\vec A_3) \in \R$, we have
 \[
 \Amp_{3,\kappa} (\vec A_0, \vec A_3) = \E_p^\otimes\left [\langle g(\vA_1,\vA_2),e^{\i\zeta} \rangle \middle | (\vA_1,\vA_2)\in S\right ],
 \]
 where  $\langle \cdot, \cdot \rangle$ denotes the usual scalar product on the complex plane $\R+\i\R$.
Let $S'$ be the set of all pairs $(\vec A_1,\vec A_2)$ such that $g(\vec A_1,\vec A_2) \in D_{\zeta,\kappa}:=\{e^{\i x}, x\in (\zeta-|\kappa_i|/2,\zeta+|\kappa_i|/2)\}$.
 By Markov's inequality, we have
  \begin{equation}
  \Amp_{3,\kappa}\leq 1-(1-\cos(\kappa_i/2)) \proba_p^\otimes \left [(\vA_1,\vA_2)\notin S'\middle | (\vA_1,\vA_2)\in S\right ]
  . \label{eq:Concentred_Periodicity}
  \end{equation}
Since $1-\cos(\kappa_i/2)=\Theta(\kappa_i^2)$, it suffices to lower-bound the above probability by a constant depending on $d,L,C,r,\varrho$ to deduce the inequality of the proposition.

To this end, let $S_E$ be the set of all pairs $(\vec A_1, \vec A_2)$ such that the event $E_{r,i} (\vec A_0,\vec A_1,\vec A_2,\vec A_3)$ holds. By a union bound,
\[
\proba_p^\otimes((\vA_1,\vA_2)\in S'\cap S_E)\geq \proba_p^\otimes((\vA_1,\vA_2)\in S)-\proba_p^\otimes((\vA_1,\vA_2)\in S\backslash S_E)-\proba_p^\otimes((\vA_1,\vA_2)\in S\backslash S').
\]
The first probability is $P_3^\otimes(\vec A_0,\vec A_3)$.
Since $(\vec A_0,\vec A_3)\in \vec \Comp_{i,\varrho}$, the second probability is at most $\sqrt{\varrho} \nu^\uparrow_p(\vec A_0)\nu_p^\downarrow(\vec A_3)\leq \sqrt{\varrho}$.
We may assume that the last probability is at most $P_3^\otimes(\vec A_0,\vec A_3)/4$, otherwise the conditional probability in \eqref{eq:Concentred_Periodicity} is at least $1/4$, which would conclude the proof.
Hence,
\begin{equation}
\proba_p^\otimes((\vA_1,\vA_2)\in S'\cap S_E)\geq (3/4)P_3^\otimes(\vec A_0,\vec A_3)-\sqrt{\varrho}\geq P_3^\otimes(\vec A_0,\vec A_3)/4, \label{eq:Prob_S'+S_E}
\end{equation}
using $P_3^\otimes(\vec A_0,\vec A_3)\geq 2\sqrt{\varrho}$ for the last inequality.

We now claim that for some constant $c>0$, depending only on $d,L,C,r,\varrho$, we have
\begin{equation}
\proba_p^\otimes((\vA_1,\vA_2)\in S\backslash S') \geq c \proba_p^\otimes((\vA_1,\vA_2)\in S'\cap S_E). \label{eq:Rotating_Inequality}
\end{equation}
Provided this inequality holds, the desired result follows directly from \eqref{eq:Concentred_Periodicity} and \eqref{eq:Prob_S'+S_E}. It remains to prove \eqref{eq:Rotating_Inequality}. We may rewrite the probability on its right-hand side as
\[
\proba_p^\otimes((f^\uparrow_r(\vA_1),f_r^\downarrow(\vA_2)-\delta_i)\in S, \quad
 g(\vA_1,\vA_2)\in D_{\zeta,\kappa}).
 \]
 Since $\Delta(f^\uparrow_r(\vA_1))=\Delta(\vA_1)$ and $\Delta(f^\downarrow_r(\vA_2))=\Delta(\vA_2)$, we have
 \[
 g(f^\uparrow_r(\vA_1),f^\downarrow_r(\vA_2)-\delta_i)=g(\vA_1,\vA_2)e^{-\i\kappa_i}\notin D_{\zeta,\kappa},
 \]
	  provided $g(\vA_1,\vA_2)\in D_{\zeta,\kappa}$. Thus, the probability on the right-hand side of \eqref{eq:Rotating_Inequality} is bounded by
  \[
  \sum_{\vec B_1,\vec B_2} \proba_p^\otimes(f^\uparrow_r(\vA_1)=\vec B_1) \proba_p^\otimes(f^\downarrow_r(\vA_2)-\delta_i=\vec B_2) \1((\vec B_1,\vec B_2)\in S\backslash S'),
  \]
	  which is bounded by
  \[
  \left(\sup_{\vec B_1\in \vec \Comp(0)} \frac{ \proba_p^\otimes(f^\uparrow_r(\vA_1)=\vec B_1)}{\proba_p^\otimes(\vA_1=\vec B_1)}\right)
  \left(\sup_{\vec B_2\in \vec \Comp(0)} \frac{\proba_p^\otimes(f^\downarrow_r(\vA_2)-\delta_i=\vec B_2)}{\proba_p^\otimes(\vA_2=\vec B_2)}\right)
  \proba_p^\otimes ((\vA_1,\vA_2)\in S\backslash S').
  \]
Exactly as in the proof of Proposition \ref{pro:RelateLocalLambda} and Lemma \ref{lem:PrelimCoupling},
we can bound the two suprema by a constant $c>0$ depending only on $d,L,r$,
which concludes the proof of \eqref{eq:Rotating_Inequality}.
\end{proof}
\begin{proof}[Proof of Proposition \ref{pro:SmallRotation}]
Fix $\varrho''>0$ small enough so that $(\varrho'')^{1/4}\leq \varrho,\varrho'$ and $2\sqrt{\varrho''}\leq\varrho$. Take $i$ such that $|\kappa_i|=\|\kappa\|_\infty$. By a union bound, writing $\vA_{0,3}=(\vA_0,\vA_3)$, the probability of the proposition is bounded by
\[
\proba_p^\otimes (\vA_{0,3}\notin \vec \Comp_{i,\varrho''}, P_3^\otimes(\vA_{0,3})> (\varrho'')^{1/4})+\proba_p^\otimes(\vA_{0,3}\in \vec \Comp_{i,\varrho''}, P_3^\otimes(\vA_{0,3})> 2\sqrt{\varrho''}, \Upsilon_{3,\kappa}(\vA_{0,3}) > 1-c\kappa^2_i).
\]
By Corollary \ref{cor:Bad_Case_SmallRotation}, the first probability is at most $(\varrho'')^{1/4}\leq \varrho'$.
Provided $c$ is small enough and $\X(p)$ is large enough, Proposition \ref{pro:RotateVsRotate} shows that the second probability is null.
\end{proof}

\section{Bounding the errors, and exponential tail for $\tau_p$} \label{sec:FirstApplications}
The goal of this section is twofold: bound $\tau_{p'}-\tau_{p,p'}^\square$ using the diagrammatic errors from Section~\ref{sec:Diagrams}, and prove a proper exponential tail for $\tau_p$.
We first prove in Subsection~\ref{sec:OpenTriangle} some generic convolution bounds that are used throughout the section. Subsection~\ref{sec:L1AtP} then shows a rudimentary bound on $\|\tau_{p'}-\tau_{p,p'}^\square\|_1$, which is needed to establish the exponential decay of $\tau_p$ in
Subsection~\ref{sec:ExpTail}. The main result of this section, Theorem~\ref{thm:PointwiseError}, is then proved in Subsection~\ref{Sec:1+1=2}. It provides both pointwise and $L_1$ bounds on $\tau_{p'}-\tau_{p,p'}^\square$.
Finally, Subsection~\ref{sec:L2AtPPrime} derives an $L_2$ bound on $\tau_{p'}-\tau_{p,p'}^\square$.

In this section, we will only assume \eqref{eq:initialization}. This means that we will not have access to the consequences of \eqref{eq:eta}, notably \eqref{eq:Plateau}, which would otherwise make many computations much easier.
\label{sec:FirstCalculus}
\subsection{Open triangle condition, and uniform square condition under \eqref{eq:initialization}.} \label{sec:OpenTriangle}
We start with some generic inequalities that we will repeatedly use in this section.
\begin{lemma} \label{lem:TruncatedPoly}
Let
$a,b >0$, and let $f,g:x\in \Z^d \mapsto \R$ be such that pointwise $|f(x)|\leq \langle x\rangle ^{-a}$ and $|g(x)|\leq \langle x\rangle ^{-b}$. There exists $C=C(d,L,a,b)\in (0,\infty)$ such that for every $x\in \Z^d$,
\[
|(f\ast g)(x)| \leq \begin{cases} C\langle x\rangle^{-b} & \text{ if } a>d \text{ and } a\geq b
\\ C\langle x\rangle^{d-a-b} & \text{ if } a,b<d \text{ and } a+b>d
\end{cases}
\]
Assume furthermore that $1\leq \|f\|_1<\infty $. Then
\[
|(f\ast g)(x)| \leq \begin{cases} C\log \left (2+\|f\|^{1/(d-a)}_1/\langle x\rangle \right ) & \text{ if } a+b=d
\\ C  \|f\|_1^{1-b/(d-a)} & \text{ if }  a+b<d
\end{cases}
\]
\end{lemma}
\begin{proof}
In this proof, $C$ denotes a constant depending on $d,L,a,b$ that may change from line to line. The first two cases are from Proposition 1.7 of \cite{MR1959796}.
For the other two cases, we write
\[
f\ast g(x) =\sum_{y\in \Z^d} f(y) g(x-y).
\]
We divide the sum into several cases. We bound the sum corresponding to the case $\|x-y\|^{d-a}_2\geq \|f\|_1$ by
\[
C\sum_{y\in \Z^d} f(y)\|f\|_1^{-b/(d-a)}=C\|f\|_1^{1-b/(d-a)}.
\]

For the case $\|x-y\|_2\geq \|x\|_2/2$ and $\|y\|_2\geq \|x\|_2/2$ and $\|x-y\|^{d-a}_2< \|f\|_1$, by the triangle inequality $\|x-y\|_2/\|y\|_2\in [1/3,3]$, so we can upper-bound the corresponding sum by
\[
C\sum_{\substack{y\in \Z^d,\\ \|x\|_2/2 \leq \|y\|_2\leq 3\|f\|_1^{1/(d-a)}}} \langle y\rangle^{-a-b}\leq \begin{cases} C \log(1+6\|f\|^{1/(d-a)}_1/\langle x\rangle) & \text{ if } a+b=d
\\ C  \|f\|_1^{1-b/(d-a)} & \text{ if } a+b<d
\end{cases} .
\]

We upper-bound the sum corresponding to the case $\|x-y\|_2 \geq \|x\|_2/2$ and $\|y\|_2< \|x\|_2/2$ by
\[
C\sum_{\substack{ y\in \Z^d, \\ \|y\|_2\leq \|x\|_2/2 }} f(y)\langle x\rangle^{-b} \leq \min(\|f\|_1,C\|x\|_2^{d-a})\langle x\rangle^{-b} \leq C\|f\|_1^{1-b/(d-a)},
\]
using $0<a, b<d$ and $a+b\leq d$ for the last inequality.

We upper-bound the contribution of the case $\|x-y\|_2< \|x\|_2/2$ and $\|y\|_2\geq \|x\|_2/2$ and $\|x-y\|^{d-a}_2< \|f\|_1$ by
\[
C \hspace{-1em}\sum_{\substack{ y\in \Z^d, \\ \|x-y\|_2 < \min(\|f\|^{1/(d-a)}_1,\|x\|_2/2)}} \hspace{-2em}\langle x\rangle^{-a} \langle x-y\rangle^{-b} \leq C \langle x\rangle^{-a} \min(\|f\|^{1/(d-a)}_1,\|x\|_2)^{d-b} \leq C\|f\|_1^{1-b/(d-a)},
\]
using $0<a, b<d$ and $a+b\leq d$ for the last inequality. This concludes this last case.
\end{proof}
\begin{lemma} \label{lem:TwoTriangleSquare}
For every $C>0$, there exist $C_2, C_3,C_4>0$ such that the following statements hold for every $p$ satisfying \eqref{eq:initialization}:
\begin{itemize}
\item[(a)] For every $x\in \Z^d$, we have $\tau_p\ast \tau_p (x) \leq C_2 \langle x\rangle^{4-d}.$
\item[(b)] For every $x\in \Z^d$, we have  $\tau_p\ast \tau_p \ast \tau_p(x) \leq C_3  \langle x\rangle^{6-d}.$
\item[(c)] For every $x\in \Z^d$,  if $\X(p)\geq 2$, we have $\tau_p\ast \tau_p\ast \tau_p \ast \tau_p(x)\leq C_4\psi_d(\X(p)).$
\end{itemize}
\end{lemma}
\begin{proof}
(a), (b), (c) are proved inductively using \eqref{eq:initialization} and Lemma \ref{lem:TruncatedPoly}, with $a=b=d-2$, then $a=d-2, b=d-4$, then $a=d-2, b=d-6$.
\end{proof}
\subsection{$L_1$ bound assuming only \eqref{eq:initialization} at $p$.} \label{sec:L1AtP}
The goal of this section is to prove a bound on $\|\tau_{p,p'}^\square-\tau_{p'}\|_1$ assuming only \eqref{eq:initialization} at $p$ and that $\X(p')$ is not too large.
In Section \ref{Sec:1+1=2}, we will show better estimates when we assume either that \eqref{eq:initialization} is satisfied at $p'$, or an analogous bound for $\tau_{p,p'}^\square$.
In this subsection we write $\mathfrak{s}(p):=(\X(p)/\psi_d(\X(p)))^{1/3}$.
\begin{proposition} \label{pro:BeforeBootstrap}
Let $C>0$. There exist $c,F,K>0$, depending only on $d,L,C$, such that for every $p<p'$, if the following properties hold:
\begin{itemize}
\item[(a)] $\X(p)\geq F$ and $(p,C)$ satisfies \eqref{eq:initialization}
\item[(b)] $\X(p')\leq c\X(p) \mathfrak{s}(p)$
\item[(c)] $\|\tau_{p,p'}^\square\|_1 \leq c\X(p) \mathfrak{s}(p)$
\end{itemize}
then we have
\[
\|\tau_{p,p'}^\square-\tau_{p'}\|_1\leq K \left (\min(\X(p'),\|\tau_{p,p'}^\square\|_1)/\X(p) \right)^4 \psi_d(\X(p)).
\]
\end{proposition}

Before proving Proposition \ref{pro:BeforeBootstrap}, let us show why it implies the following more important result.
\begin{proposition} \label{pro:AfterBootstrap}
If one removes one of the conditions (b) or (c) in Proposition \ref{pro:BeforeBootstrap}, and $c>0$ is small enough, then the conclusion still holds.
\end{proposition}
\begin{proof}
Assume that Proposition \ref{pro:BeforeBootstrap} holds for $c,F,K>0$, and take $0<\eps<1$ small enough.
Let $M=\eps c\X(p)\mathfrak{s}(p)$.
Let $p'':=\sup\{\tilde p: \max(\X(\tilde p),\|\tau_{p,\tilde p}^\square\|_1)\leq M\}$.
Since $\tilde p\mapsto \X(\tilde p)$ is continuous below $p_c$ and diverges as $\tilde p\to p_c$, we have $\X(p'')\leq M$.
Similarly, by Proposition  \ref{pro:PrelimBootstrap}, we have $\|\tau_{p,p''}^\square\|_1\leq M$.
Also, by continuity, we must have $\max(\X(p''),\|\tau_{p,p''}^\square\|_1)=M$.
Then by Proposition \ref{pro:BeforeBootstrap} applied to $p,p''$, we have
\[
|\|\tau_{p,p''}^\square\|_1-\X(p'')| \leq K(M/\X(p))^4\psi_d(\X(p))=M K(\eps c)^3\mathfrak{s}(p)^3\psi_d(\X(p))/\X(p)=K(\eps c)^3 M.
\]
Thus if $K(\eps c)^3< 1/2$, we get
\[
\min(\|\tau_{p,p''}^\square\|_1,\X(p''))> \max(\|\tau_{p,p''}^\square\|_1,\X(p'')) -M/2=M/2.
\]
Therefore, if either $\X(p')\leq M/2$ or $\|\tau_{p,p'}^\square\|_1\leq M/2$, then we have $p'\leq p''$, and $\max(\X(p'),\|\tau_{p,p'}^\square\|_1)\leq M$.
Hence, $p'$ satisfies conditions (b) and (c) in Proposition \ref{pro:BeforeBootstrap}, whose conclusion thus applies.
\end{proof}
To prove Proposition \ref{pro:BeforeBootstrap}, we need an upper bound on $p'-p$.
\begin{lemma} \label{lem:p'-p}
For every $C>0$, there exist $C',K>0$ depending only on $d,L,C$ such that for every $p$ satisfying \eqref{eq:initialization} and $\X(p)>K$, and every $p'$ with $\|\tau_{p,p'}^\square \|_1<\infty$, we have $(p'-p)/(1-p)\leq C'/\X(p).$
\end{lemma}
\begin{proof}
By Proposition \ref{pro:PrelimBootstrap}, $p'<\tilde p_c(p)=p+(1-p)/(\Omega \X(p)\lambda_p)$. The conclusion follows from Theorem \ref{thm:LowerLambda}.
\end{proof}

We first show a basic inequality:
\begin{lemma} \label{lem:ExpandSquare}
Let $0<p<p_c$. For every $k\geq 1$, we have $\tau^{\square,k}_p\leq \tau^{\square,k-1}_p\ast D\ast \tau_p$. For every $p'>p$, we have $\tau^\square_{p,p'}\leq \tau_p\ast \bar \tau_{p,p'}^\square$.
\end{lemma}
\begin{proof}
By dropping the last "off" in \eqref{eq:def_TauSquare}, we get that $\tau_p^{\square,k}(x)$ is bounded by
\[
\sum_{\substack{ \vec e_1,\dots \vec e_k\in \vedges \\  e_0^+=0, \, \, e_{k+1}^-=x}}
\hspace{-0.5em} \sum_{\substack{ A_0,A_1,\dots, A_k\in \Comp \\ A_{-1}=\emptyset}}
\prod_{i=0}^k \proba_p(\C(e_i^+)=A_i)
\1(x\in A_k)
\prod_{i=0}^{k-1} \1(e_{i+1}^-\in (A_i\offx{e_i^+} A_{i-1})).
\]
First, summing over $A_k$ yields a term $\tau_p(x-e_k^+)$. Then we recognize $\tau_p^{\square,k-1}(e_k^-)$ when we sum over everything except $\vec e_k$. In other words,
\[
\tau_p^{\square,k}(x) \leq \sum_{\vec e_k\in \vedges} \tau_p^{\square,k-1}(e_k^-)\tau_p(x-e_k^+)=\tau_p^{\square,k-1}\ast D\ast \tau_p(x),
\]
which yields the first inequality of the lemma since $x\in \Z^d$ is arbitrary. The second inequality follows directly from the first since $\tau_{p,p'}^\square=\sum_{k\geq 0} ((p'-p)/(1-p))^k\tau^{\square,k}_p$ and $\tau_p^{\square,0}=\tau_p$. Thus,
\[
\tau_{p,p'}^\square\leq \tau_{p}+(p'-p)/(1-p)\tau_p\ast D \ast \tau_{p,p'}^\square=\tau_p\ast \bar \tau_{p,p'}^\square. \qedhere
\]
\end{proof}
\begin{lemma} \label{lem:NaiveSquare}
In the setting of Proposition \ref{pro:BeforeBootstrap}, if $\X(p)$ is large enough, then for some constant $K_\square=K_\square(d,L,C,c)$,
\[
\|\square_{p,p'}\|_1\leq K_\square \psi_d(\X(p))\|\tau_{p,p'}^\square \|_1/\X(p) .
\]
\end{lemma}
\begin{proof}
By Lemma \ref{lem:ExpandSquare},
\[
\square_{p,p'}=\tau_p[\tau_p\ast D\ast \tau_{p,p'}^\square \ast D\ast \tau_p] \leq \tau_p[\tau_p\ast D\ast \tau_p\ast \bar \tau^\square_{p,p'} \ast D\ast \tau_p].
\]
So
\[
\|\square_{p,p'}\|_1\leq [\tau_p\ast (\tau_p\ast D\ast \tau_p\ast \bar \tau^\square_{p,p'} \ast D\ast \tau_p)](0).
\]
Then using the generic bound $\|f\ast g\|_\infty\leq  \|f\|_1\|g\|_\infty$, we get
\[
\|\square_{p,p'}\|_1 \leq \| \tau_p\ast \tau_p\ast \tau_p\ast \tau_p\|_\infty \|D\|^2_1 \|\bar \tau^\square_{p,p'}\|_1
.
\]
Observe that $\|D\|_1=\Omega$, where we recall that $\Omega$ is the degree of a vertex in $(\Z^d,\edges)$.
Since $(p'-p)/(1-p)\leq C'/\X(p)$ from Lemma \ref{lem:p'-p}, and since $\X(p)\leq \|\tau_{p,p'}^\square\|_1$ from $\tau_p\leq \tau_{p,p'}^\square$, we have
\begin{equation}
\|\bar \tau^\square_{p,p'}\|_1=1+((p'-p)/(1-p))\|D\ast \tau_{p,p'}^\square\|_1\leq (1+C'\Omega) \|\tau_{p,p'}^\square\|_1/\X(p). \label{eq:BarNorme1a}
\end{equation}
Then, by Lemma \ref{lem:TwoTriangleSquare},
\[
\| \tau_p\ast \tau_p\ast \tau_p\ast \tau_p\|_\infty \leq C_4(C) \psi_d(\X(p)).
\]
Thus,
\[
\|\square_{p,p'}\|_1\leq C_4(C) \Omega^2 \psi_d(\X(p)) (1+C'\Omega )\|\tau_{p,p'}^\square\|_1/\X(p). \qedhere
\]
\end{proof}
\begin{lemma} \label{lem:NaiveDiamond}
In the setting of Proposition \ref{pro:BeforeBootstrap}, if $\X(p)$ is large enough, then for some constant $K_\Diamond=K_\Diamond(d,L,C,c)$,
\[
\|\Diamond_{p,p'}\|_1\leq K_\Diamond \psi_d(\X(p))\|\tau_{p,p'}^\square \|_1\X(p')/\X(p)^2.
\]
\end{lemma}
\begin{proof}
Similarly to the last proof, we have
\[
\|\Diamond_{p,p'}\|_1=(\tau_p\ast D\ast \bar \tau^\square_{p,p'} \ast \tau_p) \ast (\tau_p\ast D\ast \bar \tau_{p,p'} \ast \tau_p)(0),
\]
so
\[
\|\Diamond_{p,p'}\|_1 \leq \|\tau_p\ast\tau_p\ast \tau_p\ast \tau_p\|_\infty \|D\|_1^2   \|\bar \tau_{p,p'}\|_1 \|\bar \tau^\square_{p,p'}\|_1.
\]
Substituting the bounds from the previous proof for $\|\tau_p\ast\tau_p\ast \tau_p\ast \tau_p\|_\infty$, $\|D\|_1$, and $\|\bar \tau^\square_{p,p'}\|_1$, and using from Lemma \ref{lem:p'-p} that
\begin{equation}
\|\bar \tau_{p,p'}\|_1 = 1+((p'-p)/(1-p))\|D\ast \tau_{p'}\|_1 \leq (1+C'\Omega)\X(p')/\X(p) \label{eq:BarNorme1b},
\end{equation}
we get
\[
\|\Diamond_{p,p'}\|_1 \leq C_4(C) \Omega^2(1+C'\Omega)^2 \psi_d(\X(p)) \|\tau^\square_{p,p'}\|_1\X(p')/\X(p)^2. \qedhere
\]
\end{proof}
\begin{lemma} \label{lem:NaiveTwosquare}
In the setting of Proposition \ref{pro:BeforeBootstrap}, if $\X(p)$ is large enough, then for some constant $K_\twosquare=K_\twosquare(d,L,C,c)$,
\[
\|\twosquare_{p,p'}\|_1\leq K_\twosquare \psi_d(\X(p))\|\tau_{p,p'}^\square \|_1/\X(p).
\]
\end{lemma}
\begin{remark}
$\|\twosquare_{p,p'}\|_1$ is actually much smaller than $\|\square_{p,p'}\|_1$ and $\|\Diamond_{p,p'}\|_1$ (see Lemmas \ref{lem:TruncatedNormSquare}, \ref{lem:TruncatedNormDiamond}, and \ref{lem:TruncatedNormTwoSquare} for the true orders).
\end{remark}
\begin{proof}
Summing \eqref{eq:deftwosquare} over $z$, we obtain
\[
\|\twosquare_{p,p'}\|_1 = \sum_{v,y\in \Z^d}\tau_p(v)\tau_p(v-y)  [\tau_p\ast D\ast \bar \tau_{p,p'}^\square\ast \tau_p(y)] [\tau_p\ast D\ast \tau_p\ast \tau_p(v-y)],
\]
which we upper-bound by
\[
\|\twosquare_{p,p'}\|_1\leq \Omega \|\tau_p\ast \tau_p\ast \tau_p\|_\infty\sum_{v,y\in \Z^d}\tau_p(v)\tau_p(v-y)  [\tau_p\ast D\ast \bar \tau_{p,p'}^\square\ast \tau_p(y)].
\]
By (a) and Lemma \ref{lem:TwoTriangleSquare}, we have $\|\tau_p\ast \tau_p\ast \tau_p\|_\infty\leq C_3(C)$, and the above sum can be rewritten as
\[
[\tau_p\ast D\ast \bar \tau_{p,p'}^\square\ast \tau_p\ast \tau_p\ast \tau_p](0),
\]
and we can then finish exactly as in the proof of Lemma \ref{lem:NaiveSquare}.
\end{proof}

\begin{lemma} \label{lem:BeforeBootstrap2}
In the setting of Proposition \ref{pro:BeforeBootstrap}, we have for some constant $K'=K'(d,L,C,c)$,
\[
\|\tau_{p,p'}^\square-\tau_{p'}\|_1\leq K' \left (\max(\X(p'),\|\tau_{p,p'}^\square\|_1)/\X(p) \right)^4 \psi_d(\X(p)).
\]
\end{lemma}
\begin{proof}
By Theorems \ref{thm:LowerBoundTau} and \ref{thm:UpperBoundTau}, we have
\[
\| \tau_{p'}-\tau_{p,p'}^\square \|_1\leq ((p'-p)/(1-p))^2
(\|\bar \tau^\square_{p,p'} \ast \tau_p\ast (\square_{p,p'}+\Diamond_{p,p'})
\ast \tau_p \ast \bar \tau_{p,p'}^\square\|_1
+ \| \bar \tau_{p,p'}^\square \ast \tau_p\ast \twosquare_{p,p'}
\ast \tau_p \ast \bar \tau_{p,p'}\|_1).
\]
Using the general inequality $\|f\ast g\|_1\leq \|f\|_1\|g\|_1$,
and using Lemma \ref{lem:p'-p}, we obtain
\[
\| \tau_{p'}-\tau_{p,p'}^\square \|_1\leq C'^2
(\|\bar \tau^\square_{p,p'}\|_1
(\|\square_{p,p'}\|_1+\|\Diamond_{p,p'}\|_1)
\|\bar \tau_{p,p'}^\square\|_1
+ \|\bar \tau^\square_{p,p'}\|_1
\|\twosquare_{p,p'} \|_1
\|\bar \tau_{p,p'}\|_1).
\]
The result follows by substituting the bounds from \eqref{eq:BarNorme1a}, \eqref{eq:BarNorme1b}, and Lemmas \ref{lem:NaiveSquare}, \ref{lem:NaiveDiamond}, and \ref{lem:NaiveTwosquare}.
\end{proof}
\begin{proof}[Proof of Proposition \ref{pro:BeforeBootstrap}]
It suffices to replace the maximum in Lemma \ref{lem:BeforeBootstrap2} by a minimum. We have by (b), (c), and $\mathfrak{s}(p)^3=\X(p)/\psi_d(\X(p))$ that
\[
|\|\tau_{p,p'}^\square\|_1-\X(p')| \leq \| \tau_{p,p'}^\square-\tau_{p'}\|_1 \leq c^3 K' \max(\X(p'),\|\tau_{p,p'}^\square\|_1).
\]
Note that if we decrease $c$, then Lemma \ref{lem:BeforeBootstrap2} still holds with the same constant $K'$. So we may assume $c^3 K'\leq 1/2$, and the previous inequality implies that
\[
\max(\|\tau_{p,p'}^\square\|_1,\X(p')) \leq 2 \min(\|\tau_{p,p'}^\square\|_1,\X(p')),
\]
and the desired result then follows from Lemma \ref{lem:BeforeBootstrap2}.
\end{proof}
\subsection{Exponential decay for $\tau_p(x)$ under \eqref{eq:initialization}} \label{sec:ExpTail}
The aim of this subsection is to show the following result.
\begin{theorem}\label{thm:ExponentialTail}
For every $\eps,C>0$ there exists $f_\eps(C)>0$ such that for every $p<p_c$ satisfying $\X(p)\geq f_\eps(C)$ and \eqref{eq:initialization}, we have
\begin{equation}
\forall x\in \Z^d, \quad \tau_p(x) \leq \exp \left (1-\frac{\|x\|_2}{\X(p)^{1/2+\eps}}\right ). \label{eq:ExpDecay}
\end{equation}
\end{theorem}
We start by proving a classical result for $\eps=1/2$, using Duminil-Copin--Tassion's approach \cite{duminil2015new}.
\begin{lemma} \label{lem:NaiveExpTail}
Let $p<p_c$. We have for every $x\in \Z^d$,
\[
\tau_p(x)\leq \exp \left (1-\left \lfloor \frac{\|x\|_1}{3L\X(p)} \right \rfloor \right ),
\]
where for every $y>0$, $\lfloor y\rfloor$ denotes the integer part of $y$.
\end{lemma}
\begin{proof}
First, for every $i\in \N$, let $S_i:=\{y\in \Z^d: L(i-1)<\|y\|_1\leq Li \}$. Note that the $S_i$ are disjoint and that $\bigcup_{i\in \N} S_i=\Z^d\backslash \{0\}$, so
\[
\sum_{i\in \N} \sum_{y\in S_i} \tau_p(y)= \X(p)-1.
\]
Thus, by the pigeonhole principle, there exists $k\leq e\X(p)$ such that $\sum_{y\in S_k} \tau_p(y)\leq 1/e$.

We show by induction on $n$ that
\begin{equation}
\forall x\in \Z^d,  \quad \|x\|_1\geq 3nL \X(p)\, \Rightarrow\, \tau_p(x)\leq e^{1-n}. \tag{$E_n$}
\end{equation}
The result is trivial for $n=1$. Let $n\geq 2$ be such that $E_{n-1}$ holds. Let $x$ be such that $\|x\|_1\geq 3 nL \X(p)$. If $0\slr x$, there exists $y\in S_k$ such that $(0\slr y) \circ (y\slr x)$. Thus, by the BK inequality,
\[
\tau_p(x)\leq \sum_{y\in S_k} \tau_p(y)\tau_p(x-y).
\]
Note that $\|x-y\|_1\geq \|x\|_1-\|y\|_1 \geq 3nL\X(p)-e\X(p)L>3(n-1)L\X(p)$, so by ($E_{n-1}$),
\[
\tau_p(x)\leq \sum_{y\in S_k} \tau_p(y)e^{1-(n-1)}\leq e^{1-n}.
\]
This concludes the induction, and hence the proof.
\end{proof}
We show Theorem \ref{thm:ExponentialTail} by induction.
More precisely, we assume that \eqref{eq:initialization} and \eqref{eq:ExpDecay} hold for some $\eps,C>0$, and prove better exponential tails for $\tau_{p'}$ where $p'$ satisfies $\X(p')=\X(p)^{1+\eps'}$ with $0<\eps'<1/8$.
To do so, by Theorem \ref{thm:UpperBoundTau}, it suffices to prove suitable exponential tails for $\tau_{p,p'}^\square$, or equivalently for $\tau^{\square,k}_p$.
Recalling Section \ref{sec:Rewriting}, we do it by proving exponential tails for $\SA_k$ under $\proba_p^h$.

First, let us introduce some notation. For $\kappa\in \R^d$ and $f\in \Z^d\mapsto \R_+$, we write
\begin{equation}
\Lap_{\kappa}(f):=\sum_{x\in \Z^d} e^{\langle \kappa,x\rangle} f(x)
\label{eq:def-Laplace-transform}
\end{equation}
for its Laplace transform. By Theorem \ref{thm:MixingTime}, there exists $C_l=C_l(d,L,C)>0$ such that, writing $l(p)=\lfloor C_l \log(\X(p))\rfloor$, if $\X(p)$ is large enough, then for every $\vec A\in \Clive$,
\begin{equation}
\sum_{\vec B\in \Clive} \left |\proba_p^h(\vA_{l(p)-1}=\vec B|\vA_0=\vec A)-\vpg_p(\vec B) \right | \leq 1/\X(p)^8.\label{eq:def_l(p)}
\end{equation}
We take $C_l$ large enough below, depending only on $d,L,C$.

We start with some rough bounds that we use when we are below the mixing time:
\begin{lemma} \label{lem:OneStepBias}
Let $C>0$. Let $p<p_c$. There exist $F,K,K'$ depending on $d,L,C$ such that if $\X(p)\geq F$ and \eqref{eq:initialization} is satisfied, then for every $x\in \Z^d$ we have
\[
\proba_p^h(\Delta(\vA_0)=x)\leq \frac{K}{\X(p)} D\ast \tau_p(x),
\]
and for every $\vec A_0\in \Clive$, we have
\[
\proba_p^h (\Delta(\vA_1)=x |\vA_0=\vec A_0)\leq K' D\ast \tau_p(x).
\]
\end{lemma}
\begin{proof}
The first probability is equal to
\begin{equation}
\proba_p^\otimes(\Delta(\vA_0)=x|\vA_0\in \Clive)\leq \proba_p^\otimes(\Delta(\vA_0)=x)/\Plive_p= \frac{1}{\Omega\X(p)\Plive_p} \sum_{A\in \Comp(0)}\sum_{a^-\in A,(a^-,x)\in \vedges} \proba_p(\C(0)=A). \label{eq:naive}
\end{equation}
We recognize the last sum as $\tau_p\ast D(x)$. Also by Lemma \ref{lem:LambdaLowMemory},
\begin{equation}
\Plive_p\geq \lambda_p. \label{eq:Plive>lambda}
\end{equation}
Thus, the first inequality follows from \eqref{eq:naive}, \eqref{eq:Plive>lambda}, and Theorem \ref{thm:LowerLambda}.

For the second inequality, by \eqref{eq:MarkovTransition} and Lemma \ref{lem:vpd}, for every $\vec A_1\in \Clive$, we have
\[
\P_p^1(\vec A_1|\vec A_0) \leq \frac{h_p(\vec A_1)}{\lambda_p h_p(\vec A_0)} \proba_p^\otimes(\vA_1=\vec A_1) \leq \frac{\Omega\X(p)}{\lambda_p(1-p)^\Omega}\proba_p^\otimes(\vA_1=\vec A_1).
\]
Summing over $\vec A_1\in \Clive$ such that $\Delta(\vec A_1)=x$, and using
$\proba_p^\otimes(\Delta(\vA_1)=x)=\tau_p\ast D(x)/(\Omega\X(p))$, we obtain
\[
\proba_p^h (\Delta(\vA_1)=x |\vA_0=\vec A_0)
\leq \frac{1}{\lambda_p(1-p)^\Omega}D\ast \tau_p(x).
\]
Theorem \ref{thm:LowerLambda} concludes the proof.
\end{proof}
\begin{lemma} \label{lem:Tail_vpg} In the same setting, we have 
  \[ \sum_{\vec A\in \Clive} \vpg_p(\vec A)\1(\Delta(\vec A)=x)\leq K' D\ast \tau_p(x).\]
\end{lemma}
\begin{proof} For every $\vec A_1\in \Clive$, we have $\vpg_p(\vec A_1)=\sum_{\vec A_0\in \Clive}\vpg_p(\vec A_0)\P_p^1(\vec A_1|\vec A_0)$. The proof is then essentially the same as before.
\end{proof}
\begin{lemma} \label{lem:TrashTailBelowMixing}
Let $C,\eps>0$. Let $p<p_c$. There exists $F>0$, depending only on $d,L,C,\eps$, such that if $\X(p)\geq F$ and \eqref{eq:initialization} and \eqref{eq:ExpDecay} are satisfied, then for every $\kappa\in \R^d$ with
\begin{equation}
 4d\log(\X(p))\X(p)^{1/2+\eps} \leq \|\kappa\|_2^{-1}\leq \X(p)^3, \label{eq:Def_Kappa}
\end{equation}
 we have
\[
\E_p^h[e^{\langle \kappa,\Delta(\vA_0)\rangle} ]\leq \X(p)^2.
\]
Moreover, for every $\vec A_0\in \Clive$,
\[
\E_p^h\left [e^{\langle \kappa,\Delta(\vA_1)\rangle} \middle  |\vA_0=\vec A_0\right  ]\leq \X(p)^3 .
\]
\end{lemma}
\begin{proof}
By Lemma \ref{lem:OneStepBias}, it suffices to bound $\Lap_\kappa(\tau_p\ast D)$. We have
\[
\Lap_\kappa(\tau_p\ast D) \leq \Lap_\kappa(D) \sum_{x\in \Z^d, \|x\|_2\leq \|\kappa\|_2^{-1}} \tau_p(x) e^{\langle x,\kappa \rangle}+\Lap_\kappa(D)\sum_{x\in \Z^d, \|x\|_2>\|\kappa\|_2^{-1}} e^{1+\langle \kappa,x\rangle -\|x\|_2/\X(p)^{1/2+\eps}}.
\]
By the Cauchy--Schwarz inequality, we have $\langle x,\kappa\rangle \leq 1$ in the first sum, so we may upper-bound the first sum by $e\X(p)$.
For the second sum, the Cauchy--Schwarz inequality and $\|\kappa\|_2 \leq 1/(4d\log(\X(p))\X(p)^{1/2+\eps} )$ give $\langle x,\kappa\rangle \leq \|x\|_2/\X(p)^{1/2+\eps}/(4d)$.
Then, by a basic calculation, as long as $\X(p)$ is large enough, the second sum is bounded by 1.
\end{proof}
We use the following bound after the mixing time:
\begin{lemma} \label{lem:NextStepExpTail}
Let $C,\eps>0$.
Let $p<p_c$.
There exists $F>0$, depending only on $d,L,C,\eps$,
such that if $\X(p)\geq F$,
and \eqref{eq:initialization} and \eqref{eq:ExpDecay} are satisfied,
then for every $\vec A\in \Clive$,
and every $\kappa\in \R^d$ satisfying \eqref{eq:Def_Kappa},
we have
\[
\E^h_p\left [e^{\langle \kappa,\Delta(\vA_{l(p)}) \rangle} \middle |\vA_0=\vec A \right ] \leq 1+\|\kappa\|_2^2l(p)^2\X(p)^{1+2\eps}.
\]
\end{lemma}
\begin{proof}
Fix $\vec A$. By \eqref{eq:def_l(p)}, there exists $w:\Clive \mapsto \R_+$ such that $\|w\|_1\leq 1/\X(p)^8$ and for every $\vec B\in \Clive$,
\[
\proba_p^h(\vA_{l(p)-1}=\vec B|\vA_0=\vec A) \leq \vpg_p(\vec B)+w(\vec B).
\]
Applying $\P_p$, it follows that for every $\vec B\in \Clive$,
\[
\proba_p^h(\vA_{l(p)}=\vec B|\vA_0=\vec A) \leq \vpg_p(\vec B)+[\P_p(w)](\vec B).
\]
Hence, the expectation in the lemma is bounded by $S_\vpg+S_w$, where
\[
S_\vpg:= \sum_{\vec B\in \Clive} \vpg_p(\vec B) e^{\langle \kappa, \Delta(\vec B)\rangle} \quad \text{and} \quad S_w:= \sum_{\vec B\in \Clive} [\P_p(w)](\vec B) e^{\langle \kappa, \Delta(\vec B)\rangle}.
\]
By Lemma \ref{lem:TrashTailBelowMixing}, we have $S_w\leq \|w\|_1 \X(p)^3\leq 1/\X(p)^5 \leq \|\kappa\|_2^2 \X(p)$.

So it suffices to bound $S_\vpg$. By Lemma \ref{lem:Symmetric_hp_vpg}, for every $\vec B\in \Clive$ we have $\vpg_p(\vec B)=\vpg_p(-\vec B)$.
Also, $\Delta(\vec B)=-\Delta(-\vec B)$. Furthermore, $\vpg_p$ is a symmetric probability measure. Hence,
\[
S_\vpg=1+\sum_{\vec B\in \Clive} \vpg_p(\vec B) \left (e^{\langle \kappa, \Delta(\vec B)\rangle} -1-\langle \kappa, \Delta(\vec B)\rangle \right ).
\]
Let $r(p)=2d\log(\X(p))\X(p)^{1/2+\eps}$. We split the last sum into two parts according to whether $\|\Delta(\vec B)\|_2\leq r(p)$ or $\|\Delta(\vec B)\|_2> r(p)$. Let $S_\vpg^-$ and $S_\vpg^+$ denote the respective sums.
For the first sum, note that when $-1\leq \langle \kappa,x\rangle\leq 1$, we have  
\[
|e^{\langle \kappa, x\rangle} -1-\langle \kappa,x \rangle|\leq 2\langle \kappa,x\rangle^2 \leq 2\|\kappa\|_2^2 \|x\|_2^2.
\]
Thus, since $\vpg_p$ is a probability measure, and $\|\kappa\|_2r(p)\leq 1$, we have
\[ S_\vpg^- \leq 2\|\kappa\|_2^2r(p)^2\leq \frac{32d^2}{C_l^2}\|\kappa\|_2^2l(p)^2\X(p)^{1+2\eps}.\]

For the second sum, Lemma \ref{lem:Tail_vpg} yields, for some constant $K>0$ depending only on $d,L,C$,
\[
S_\vpg^+\leq K\sum_{x\in \Z^d,\|x\|_2>r(p)}\tau_p\ast D(x) |e^{\langle \kappa, x\rangle} -1-\langle \kappa,x \rangle|.
\]
Using \eqref{eq:ExpDecay}, if $\X(p)$ is large enough, the last sum is then bounded by 
\[\Omega \sum_{x\in \Z^d: \|x\|_2> r(p)} \exp \left (2-\frac{\|x\|_2}{\X(p)^{1/2+\eps}}\right )e^{1+|\langle \kappa,x\rangle|}.
\]
By the Cauchy--Schwarz inequality, the last sum is bounded by
\[
\sum_{x\in \Z^d: \|x\|_2> r(p)} \exp \left (3-\frac{\|x\|_2}{\X(p)^{1/2+\eps}}+\|x\|_2\|\kappa\|_2\right ).
\]
Using $\|\kappa\|_2\leq 1/(4d\log(\X(p))\X(p)^{1/2+\eps})$, we deduce by a basic calculation that, if $\X(p)$ is large enough, then the last sum is bounded by
$\|\kappa\|_2^2 \X(p)^{1+2\eps}$. Taking $C_l$ large enough and then $\X(p)$ large enough, the desired inequality follows by summing the bounds on $S_w$, $S_\vpg^-$, and $S_\vpg^+$.
\end{proof}
\begin{lemma} \label{lem:kStepExpTail}
In the setting of Lemma \ref{lem:NextStepExpTail}, we have for every $k\geq 0$,
\[
\E_p^h[e^{\langle \kappa /l(p),\SA_k\rangle } ] \leq (1+\|\kappa\|_2^2l(p)\X(p)^{1+2\eps})^k \X(p)^3.
\]
\end{lemma}
\begin{proof}
We decompose $\SA_k$ as $\SA_k^0+\dots+\SA_k^{l(p)-1}$, where for every $0\leq j <l(p)$ and $k\geq 0$,
\[
\SA_{k}^j := \sum_{0\leq i \leq k} \1_{(i-j)/l(p)\in \Z} \Delta(\vA_i).
\]
By H\"{o}lder's inequality, for every $k\geq 0$, we have
\begin{equation}
\E^h_p \left[ e^{ \langle \kappa/l(p), \SA_k\rangle} \right ]\leq \left (\prod_{0\leq j<l(p)} \E^h_p\left [ e^{ \langle \kappa , \SA_k^j\rangle } \right ] \right )^{1/l(p)}.  \label{eq:HolderTail}
\end{equation}

By Lemma \ref{lem:NextStepExpTail}, for every $k'\geq l(p)$ and $0\leq j < l(p)$ with $(k'-j)/l(p) \in \Z$, we have
\[
\E^h_p\left [ e^{ \langle \kappa, \SA_{k'}^j \rangle} \right ]=
\sum_{\vec A\in \Clive}
 \E^h_p\left [ e^{ \langle \kappa, \SA_{k'-l(p)}^j \rangle}\1_{\vA_{k'-l(p)}=\vec A} \E_p^h \left [e^{ \langle \kappa,\Delta(\vA_{k'})\rangle}  \middle | \vA_{k'-l(p)}=\vec A \right ] \right ].
 \]
By Lemma \ref{lem:NextStepExpTail}, the conditional expectation is uniformly bounded by $1+\|\kappa\|_2^2l(p)^2\X(p)^{1+2\eps}$. Thus,
\[
\E^h_p\left [ e^{ \langle \kappa, \SA_{k'}^j \rangle} \right ]\leq  (1+\|\kappa\|_2^2l(p)^2\X(p)^{1+2\eps})\E^h_p\left [ e^{ \langle \kappa, \SA_{k'-l(p)}^j \rangle} \right ] .
\]
Since for every $0\leq j<l(p)$, the map $k'\mapsto \SA_{k'}^j$ is constant on every interval of the form $[j+il(p),j+(i+1)l(p)-1]$, the previous statement still holds when we only assume $k'\geq j+l(p)$.
Hence, by induction, for every $0\leq j <l(p)$ with $k\geq j$ we have
\[
\E^h_p\left [ e^{ \langle \kappa , \SA_k^j\rangle } \right ] \leq (1+\|\kappa\|_2^2l(p)^2\X(p)^{1+2\eps})^{k/l(p)} \E_p^h[e^{\langle \kappa, \Delta(\vA_j)\rangle }].
\]
Note that by symmetry, the right-hand side is larger than 1, so the bound also holds when $k<j$. Going back to \eqref{eq:HolderTail}, whenever $k\geq 0$, we get
\[
\E_p^h[e^{\langle \kappa/l(p),\SA_k\rangle } ] \leq (1+\|\kappa\|_2^2l(p)^2\X(p)^{1+2\eps})^{k/l(p)} \left (\prod_{0\leq j<l(p)} \E_p^h \left [e^{\langle \kappa, \Delta(\vA_j)\rangle } \right ]\right )^{1/l(p)}.
\]
Also, using that for $n\in \N$ and $x\geq 0$ we have $(1+x)^{1/n}\leq 1+x/n$, we get
\[
(1+\|\kappa\|_2^2l(p)^2\X(p)^{1+2\eps})^{k/l(p)}\leq (1+\|\kappa\|_2^2l(p)\X(p)^{1+2\eps})^k
\]

It remains to show that, for $j\geq 0$,
\begin{equation}
\E_p^h \left [e^{\langle \kappa, \Delta(\vA_j)\rangle } \right ] \leq \X(p)^3 .\label{eq:IndividualTail}
\end{equation}
The bound for $j=0$ follows directly from Lemma \ref{lem:TrashTailBelowMixing}. Also, by Lemma \ref{lem:TrashTailBelowMixing}, we have for $j\geq 1$,
\[
\E_p^h \left [e^{\langle \kappa, \Delta(\vA_j)\rangle } \right ]=\sum_{\vec A \in \Clive} \proba_p^h(\vA_{j-1}=\vec A) \E_p^h \left [e^{\langle \kappa, \Delta(\vA_j)\rangle } \middle |\vA_{j-1}=\vec A\right ]\leq \X(p)^3. \qedhere
\]
\end{proof}

\begin{lemma} \label{lem:AllStepExpTail}
In the setting of Lemma \ref{lem:NextStepExpTail}, we have for every $k\geq 1$,
\[
\Lap_{\kappa/l(p)}(\tau_{p}^{\square,k}) \leq \X(p)^7 (\Omega\X(p)\lambda_p)^{k-1} (1+\|\kappa\|_2^2l(p)\X(p)^{1+2\eps})^{k-1}.
\]
\end{lemma}
\begin{proof}
By \eqref{eq:RewriteB_TauSquare}, $\Lap_{\kappa/l(p)}(\tau_p^{\square,k})$ is equal to
\[
\X(p)(\Omega\X(p)\lambda_p)^{k-1} \Plive_p
\E^h_p \left[\frac{h_p(\vec \A_0)}{h_p(\vec \A_{k-1})}
\sum_{\vec A_k\in \vec \Comp(0)}   \mu_p(\vec A_k)
\1_{(\vec \A_{k-1},\vec A_k)\notin\Ccut}
e^{ \langle \kappa/l(p), \SA_{k-1}+\Delta(\vec A_k)-\delta(\vec A_k) \rangle} \right ].
\]
Dropping the indicator, and using Lemma \ref{lem:vpd}, $\Lap_{\kappa/l(p)}(\tau_p^{\square,k})$ is bounded by
\[
\X(p)(\Omega\X(p)\lambda_p)^{k-1} \frac{\Omega\X(p)}{(1-p)^\Omega}
\E^h_p \left[ e^{ \langle \kappa/l(p), \SA_{k-1}\rangle} \right ]
\sum_{\vec A_k\in \vec \Comp(0)}  \mu_p(\vec A_k)
e^{ \langle \kappa/l(p), \Delta(\vec A_k)-\delta(\vec A_k) \rangle}.
\]
We first deal with the sum above. We can rewrite it as in the proof of Lemma \ref{lem:NextStepExpTail} as
\[
\Omega \sum_{x\in \Z^d} \tau_p(x)e^{\langle \kappa/l(p),x\rangle } \leq C_1(d,L)\X(p),
\]
using that $\X(p)$ is large enough and \eqref{eq:ExpDecay} for the inequality. Thus, if $\X(p)$ is large enough, we have
\[
\Lap_{\kappa/l(p)}(\tau_p^{\square,k}) \leq \X(p)^4(\Omega\X(p)\lambda_p)^{k-1} \E^h_p \left[ e^{ \langle \kappa/l(p), \SA_{k-1}\rangle} \right ] .
\]
Lemma \ref{lem:kStepExpTail} then concludes the proof.
\end{proof}
\begin{lemma} \label{eq:FirstExpTail_TauSquare}
Let $C,\eps>0$.
There exist $F,K,K'>0$ depending on $d,L,C,\eps$ such that for every $p<p'<p_c$ with $\X(p)\geq F$ and $K'\X(p)\leq \X(p')\leq \X(p)^{10/9}$,
if \eqref{eq:initialization} and \eqref{eq:ExpDecay} are satisfied at $p$,
then for every $\kappa\in \R^d$ with $\|\kappa \|_2=1/(K l(p)^2 \X(p')^{1/2}\X(p)^{\eps})$,
we have
$ \Lap_\kappa(\tau^\square_{p,p'})\leq \X(p)^{10}.$
\end{lemma}
\begin{proof}
Let
\[
A= \|\kappa\|_2^2 l(p)^3\X(p)^{1+2\eps} \leq \frac{1}{K^2}\frac{\X(p)}{\X(p')}.
\]
Note that if $K$ is large enough, then $l(p)\kappa$ satisfies \eqref{eq:Def_Kappa}. So, by definition of $\tau_{p,p'}^\square$ and Lemma \ref{lem:AllStepExpTail},
\begin{align}
\Lap_\kappa(\tau_{p,p'}^\square) & =\Lap_\kappa(\tau_p)+\sum_{k=1}^\infty \left (\frac{p'-p}{1-p} \right )^k \Lap_\kappa(\tau_p^{\square,k}) \notag
\\ & \leq \Lap_\kappa(\tau_p)+\X(p)^7  \sum_{k=0}^\infty  \left (\frac{p'-p}{1-p} \right )^k (\Omega \X(p)\lambda_p)^k (1+A)^k. \label{eq:Sum_Laplace}
\end{align}
Then, by definition of $\tilde p_c(p)$ (see \eqref{eq:def_pcp}), we have
\[
\frac{p'-p}{1-p}\Omega \X(p)\lambda_p (1+A)=\frac{p'-p}{1-p} \frac{1-p}{\tilde p_c(p)-p} (1+A)=(1+A)\left (1- \frac{\tilde p_c(p)-p'}{\tilde p_c(p)-p} \right).
\]
We thus want to lower-bound the last fraction. By \eqref{eq:def_pcp} and Theorem \ref{thm:LowerLambda}, if $\X(p)$ is large enough,
\[
\tilde p_c(p)-p \leq C_\lambda /\X(p),
\]
 for some $C_\lambda$ depending only on $d,L,C$.
Then let $p_2'$ be such that $\X(p_2')=2\X(p')$.
We have $\X(p_2')\leq 2\X(p)^{10/9}$, so by Proposition \ref{pro:AfterBootstrap},
if $\X(p)$ is large enough, we have $\|\tau_{p,p'_2}^\square\|_1<\infty$.
So $p'_2\leq \tilde p_c(p)$.
Moreover, we have the well-known inequality $\frac{\partial \X(p)}{\partial p}\leq \Omega \X(p)^2$.
Thus, by integration, for some constant $C_\Omega$ depending only on $d,L$, we have
\[
\tilde p_c(p)-p'\geq p_2'-p'\geq C_\Omega/\X(p').
\]
Hence, if $K$ is large enough, we have for some constant $c>0$ depending only on $d,L,C$,
\[
\frac{p'-p}{1-p}\Omega \X(p)\lambda_p (1+A) \leq \left (1+\frac{1}{K^2}\frac{\X(p)}{\X(p')} \right )\left (1- \frac{C_\Omega}{C_\lambda}\frac{\X(p)}{\X(p')} \right ) \leq 1-c \frac{\X(p)}{\X(p')}.
\]
Using $\sum_{k=0}^\infty (1-x)^k=1/x$ for $x>0$, \eqref{eq:Sum_Laplace} gives
\[
\Lap_\kappa (\tau_{p,p'}^\square)\leq \Lap_\kappa(\tau_p)+\X(p') \X(p)^6/c,
\]
and we can use \eqref{eq:ExpDecay} to bound $\Lap_\kappa(\tau_p)$.
\end{proof}
\begin{lemma} \label{lem:LaplaceLogConvexe}
In the setting of Lemma \ref{eq:FirstExpTail_TauSquare}, if $K$ is large enough then
\[
\Lap_\kappa(\tau^\square_{p,p'})\leq \X(p)^{1/9}\|\tau^\square_{p,p'}\|_1.
\]
\end{lemma}
\begin{proof}
Since the Laplace transform of a positive function is log-convex, it suffices to increase $K$ by a factor $90$.
\end{proof}
To bound $\Lap_\kappa(\tau_{p'})$, we want to use Theorem \ref{thm:UpperBoundTau}. To this end, we first bound $\Lap_\kappa(\twosquare_{p,p'})$.
\begin{lemma} \label{lem:ExpTail_TwoSquare}
In the setting of Lemma \ref{eq:FirstExpTail_TauSquare}, if $K$ is large enough, and $\X(p)$ is large enough, then
\[
\Lap_\kappa(\twosquare_{p,p'})\leq e\|\twosquare_{p,p'}\|_1+1.
\]
\end{lemma}
\begin{proof}
We have
\begin{equation}
\Lap_{\kappa}(\twosquare_{p,p'})=\sum_{x\in \Z^d,\|x\|_2\leq 1/\|\kappa\|_2} \twosquare_{p,p'}(x)e^{\langle \kappa,x\rangle} + \sum_{x\in \Z^d,\|x\|_2>1/\|\kappa\|_2} \twosquare_{p,p'}(x)e^{\langle \kappa,x\rangle}. \label{eq:Split_Laplace_TwoSquare}
\end{equation}
By the Cauchy--Schwarz inequality, the first sum is bounded by $e\|\twosquare_{p,p'}\|_1$. For the second, by dropping the term $\tau_p(y-v)$ in \eqref{eq:deftwosquare}, we get for every $x\in \Z^d$,
\[
\twosquare_{p,p'}(x)\leq [\tau_p\ast \tau_p](x)[ \tau_p\ast D\ast \bar \tau_{p,p'}^\square \ast \tau_p\ast \tau_p\ast D\ast \tau_p](x) \leq [\tau_p\ast \tau_p](x) \Omega^2 \X(p)^4\|\bar \tau_{p,p'}^\square\|_1,
\]
using $\|\cdot\|_\infty \leq \|\cdot \|_1$ for the last inequality.
By Proposition \ref{pro:AfterBootstrap},
\[
\|\bar \tau_{p,p'}^\square\|_1\leq 1+\Omega \|\tau_{p,p'}^\square\|_1 \leq \X(p)^2.
\]
So the second sum in \eqref{eq:Split_Laplace_TwoSquare} is bounded by
\[
\Omega^3 \X(p)^6 \sum_{x\in \Z^d, \|x\|_2>\|\kappa\|^{-1}_2 }\tau_p\ast \tau_p(x) e^{\langle \kappa,x\rangle}\leq 2 \Omega^3 \X(p)^6 \Lap_\kappa(\tau_p)\sum_{x\in \Z^d, \|x\|_2> \|\kappa\|^{-1}_2/2 }\tau_p(x)e^{\langle\kappa,x\rangle},
\]
using for the last inequality that if $x_1+x_2=x$ and $\|x\|_2>\|\kappa\|^{-1}_2$ then either $\|x_1\|_2> \|\kappa\|^{-1}_2/2$ or $\|x_2\|_2> \|\kappa\|^{-1}_2/2$. 
Since $\|\kappa\|_2\leq 1/(K\log(\X(p))\X(p)^{1/2+\eps})$, we have by the Cauchy--Schwarz inequality $\langle x,\kappa \rangle <\|x\|_2/(2 \X(p)^{1/2+\eps})$. 
So by \eqref{eq:ExpDecay}, $\Lap_\kappa(\tau_p)\leq 2\X(p)^{1+d\eps}$, provided $\X(p)$ is large enough. Then, by \eqref{eq:ExpDecay}, the right-hand side is bounded by
\[
4\Omega^3 \X(p)^{6+1+d\eps} \sum_{x\in \Z^d} \1\left (\|x\|_2 > K\log(\X(p)) \X(p)^{1/2+\eps} \right ) \exp \left (1-\frac{\|x\|_2}{2 \X(p)^{1/2+\eps}}\right ).
\]
If $K$ is large enough, the last display goes to $0$ as $\X(p)\to \infty$.
\end{proof}
\begin{lemma} \label{lem:NaiveTailTau}
In the setting of Lemma \ref{eq:FirstExpTail_TauSquare}, if $K$ is large enough, and $\X(p)$ is large enough, then
\[
\Lap_\kappa(\tau_{p})\leq 3\X(p).
\]
\end{lemma}
\begin{proof}
As above, this can be derived from \eqref{eq:ExpDecay}. We omit the details.
\end{proof}
\begin{lemma} \label{lem:Tail_Tau_p'}
In the setting of Lemma \ref{eq:FirstExpTail_TauSquare}, if $K$ is large enough, and $\X(p)$ is large enough, then
\[
\Lap_\kappa(\tau_{p'})\leq 2\Lap_\kappa(\tau_{p,p'}^\square)+\X(p).
\]
\end{lemma}
\begin{proof}
By Theorem \ref{thm:UpperBoundTau}, we have
\[
\Lap_\kappa(\tau_{p'}) \leq \Lap_\kappa(\tau_{p,p'}^\square)+\left(\frac{p'-p}{1-p}\right)^2\Lap_\kappa(\bar \tau^\square_{p,p'}) \Lap_\kappa^2 (\tau_p)\Lap_\kappa(\twosquare_{p,p'})\Lap_\kappa(\bar \tau_{p,p'}).
\]
If $\X(p)$ is large enough, Proposition \ref{pro:AfterBootstrap} and $\X(p')\leq \X(p)^{10/9}$ give $\|\bar \tau^\square_{p,p'}\|_1\leq 2 \X(p')$.
Lemma \ref{lem:NaiveTwosquare} then gives $\|\twosquare_{p,p'}\|_1 \leq 2 K_\twosquare \psi_d(\X(p))\X(p')/\X(p)$.
Thus, by Lemma \ref{lem:ExpTail_TwoSquare}, if $K$ is large enough,
\[
\Lap_\kappa(\twosquare_{p,p'}) \leq C \psi_d(\X(p))\frac{\X(p')}{\X(p)}.
\]
Moreover, by Lemmas \ref{lem:p'-p}, \ref{lem:LaplaceLogConvexe}, and \ref{lem:NaiveTailTau},
\[
\Lap_\kappa(\bar \tau^\square_{p,p'})\leq C \X(p)^{1/9}\frac{\X(p')}{\X(p)}, \qquad
\Lap_\kappa(\tau_p)\leq 3\X(p), \qquad
\Lap_\kappa(\bar \tau_{p,p'})\leq C\left(1+\frac{\Lap_\kappa(\tau_{p'})}{\X(p)}\right).
\]
We get for some constant $C_\Lap$ depending on $d,L,C$,
\begin{align*}
\Lap_\kappa(\tau_{p'})&  \leq \Lap_\kappa(\tau_{p,p'}^\square)+C_\Lap\frac{1}{\X(p)^2} \left(\X(p)^{1/9}\frac{\X(p')}{\X(p)}\right)\X(p)^2
\left(\psi_d(\X(p))\frac{\X(p')}{\X(p)}\right)
\left(1+\frac{\Lap_\kappa(\tau_{p'})}{\X(p)}\right)
\\ & = \Lap_\kappa(\tau_{p,p'}^\square)+
C_\Lap  \psi_d(\X(p))\X(p)^{1/9}\frac{\X(p')^2}{\X(p)^2}
+\left(C_\Lap \psi_d(\X(p))\X(p)^{1/9} \frac{\X(p')^2}{\X(p)^3}\right) \Lap_\kappa(\tau_{p'})  .
\end{align*}
Since $ \X(p')\leq \X(p)^{10/9}$ and $\psi_d(\X(p))\leq \X(p)^{1/2}$, we deduce that if $\X(p)$ is large enough,
\[
\Lap_\kappa(\tau_{p'})  \leq \Lap_\kappa(\tau_{p,p'}^\square)+ \frac{\X(p)}{2}+ \frac{1}{2}\Lap_\kappa(\tau_{p'}). \qedhere
\]
\end{proof}
\begin{proof}[Proof of Theorem \ref{thm:ExponentialTail}]
We prove the result by induction on $\eps>0$.
By Lemma \ref{lem:NaiveExpTail}, the result is true for every $\eps>1/2$.
Let us assume that it is true at $\eps>0$, and let us show that it is true for every $\eps'>(9/10)\eps$.
By Lemmas \ref{lem:LaplaceLogConvexe}, \ref{lem:Tail_Tau_p'},
and Proposition \ref{pro:AfterBootstrap},
there exists $K>0$ depending on $d,L,C,\eps$
such that for every $p<p'$ with $\X(p)$ large enough depending on $d,L,C,\eps$,
and $\X(p')=\X(p)^{10/9}$,
for every $\kappa\in \R^d$
with $\|\kappa\|_2= 1/(K l(p)^2 \X(p')^{1/2}\X(p)^{\eps})$
we have
\[
\Lap_\kappa(\tau_{p'})\leq 3 \X(p')^{10/9}.
\]
Now let $x\in \Z^d$
and take $\kappa$ collinear to $x$.
We deduce from the last inequality and Markov's inequality that
\[
\tau_{p'}(x)\leq 3\X(p')^{10/9} e^{-\|x\|_2  \|\kappa\|_2}. 
 \]
Then, by definition of $\|\kappa\|_2$ and since $\X(p')=\X(p)^{10/9}$, as $\X(p)\to \infty$ with $\eps'>(9/10)\eps$ fixed, we have $\X(p')^{1/2+\eps'}\|\kappa\|_2/2\gg \log(\X(p'))$. Using this inequality twice, if $\X(p)$ is large enough, and $\|x\|_2\geq \X(p')^{1/2+\eps'}$, we get
\[
\tau_{p'}(x) \leq \exp(1-\|x\|_2\|\kappa\|_2/2)\leq \exp\left (1-\frac{\|x\|_2}{\X(p')^{1/2+\eps'}} \right ).
\]
This final bound is also trivially true for every $x\in \Z^d$ with $\|x\|_2 \leq \X(p')^{1/2+\eps'}$.
Hence, \eqref{eq:ExpDecay} with $p$ replaced by $p'$ and $\eps$ replaced by $\eps'$ holds.
Since $\eps'>(9/10)\eps$ is arbitrary, this concludes the induction and hence the proof.
\end{proof}

\subsection{Pointwise bound and $L_1$ bound assuming \eqref{eq:initialization} at $p'$} \label{Sec:1+1=2}
First, let us introduce some notation.
For every $r\geq 0$ and $f:\Z^d\mapsto \R$, we write $\|f\|_{\infty,r}:=\max_{x:\|x\|_2\geq r} |f(x)|$, and $\|f\|_{1,r}:=\sum_{x:\|x\|_2<r} |f(x)|$.
The following error term will appear in our computations.
For every $r\geq 0$, and $p\leq p'$ with $\X(p)<\infty$ let
\begin{equation}
\Xi(r,p,p'):=r^{d-2}\|\tau_p\|_{\infty,r} \frac{\min(r^2,\X(p'))^2}{\X(p)\X(p')}.
\label{eq:def-Xi-error}
\end{equation}
Under \eqref{eq:initialization} and suitable exponential tails for $\tau_p$,
we will usually think of $\Xi(x,p,p')$ as bounded by a constant
(see e.g. Lemma \ref{lem:AnnoyingXi} or \eqref{eq:Plateau}).
The aim of this subsection is to prove the following bounds:
\begin{theorem} \label{thm:PointwiseError}
Let $C>0$. There exist $c,F,K,K'>0$ depending only on $d,L,C$ such that for every $p\leq p'$, if $F\leq \X(p)<\infty$,
and at least one of the following lines holds:
\begin{itemize}
\item[(a)] $\X(p')\leq c\X(p)^2/\psi_d(\X(p))$ and $(p',C)$ satisfies \eqref{eq:initialization}
\item[(b)] $\|\tau_{p,p'}^\square\|_1\leq  c\X(p)^2/\psi_d(\X(p))$ and for every $x\in \Z^d$, we have $\tau_{p,p'}^\square(x)\leq C\langle x\rangle^{2-d}$,
\end{itemize}
then we have
\[
\forall x\in \Z^d, \quad |\tau_{p'}(x)-\tau_{p,p'}^\square(x)| \leq K\langle x\rangle^{2-d}\frac{\X(p')\psi_d(\X(p))}{\X(p)^2}(1+\Xi(\|x\|_2/15,p,p')),
\]
and
\[
\|\tau_{p'}-\tau_{p,p'}^\square\|_1 \leq K' \frac{\min(\X(p'),\|\tau_{p,p'}^\square\|_1)^2\psi_d(\X(p))}{\X(p)^2}.
\]
\end{theorem}
As in the previous subsection, we first study the case where (a) and (b) are both satisfied, and then use a bootstrap argument to conclude.
Note that (a) and (b) both imply that \eqref{eq:initialization} is satisfied at $p$ since $\tau_p\leq \tau_{p'}$ and $\tau_{p}\leq \tau_{p,p'}^\square$.
We frequently use the following general inequality:
\begin{lemma} \label{lem:SplitConvolTail}
For every $k\in \N$, functions $f_1,f_2,\dots, f_k:\Z^d \mapsto \R_+$, and $r_1,r_2,\dots,r_k\geq 0$ with $\sum_{i=1}^k r_i=r$, we have
\[
\|f_1\ast f_2\ast \dots \ast f_k \|_{\infty, r} \leq \sum_{i=1}^k \|f_i\|_{\infty,r_i} \prod_{1\leq j<i} \|f_j\|_{1,r_j} \prod_{i<j\leq k} \|f_j\|_1.
\]
\end{lemma}
\begin{proof}
For every $x$ with $\|x\|_2 \geq r$ we have
\[
f_1\ast f_2\ast \dots \ast f_k(x)=\sum_{\substack{ y_1,y_2,\dots, y_k\in \Z^d, \\ \sum_{i=1}^k y_i=x}} f_1(y_1)f_2(y_2)\dots f_k(y_k).
\]
We must have $\|y_i\|_2 \geq r_i$ for some $1\leq i \leq k$. In that case, we have $f_i(y_i)\leq \|f_i\|_{\infty,r_i}$, so
\[
f_1\ast f_2\ast \dots f_k(x)\leq \sum_{i=1}^k \|f_i\|_{\infty,r_i}  \sum_{\substack{\{y_j\}_{j\neq i} \\ \forall j<i, \|y_j\|_2<r_j}} \prod_{j\neq i} f_j(y_j),
\]
which equals the right-hand side of the lemma.
\end{proof}

We start by bounding pointwise the different errors in Theorems \ref{thm:LowerBoundTau} and \ref{thm:UpperBoundTau}.
Since the errors can be expressed as a convolution of $\tau_p$,
$\bar \tau^\square_{p,p'}$, $\bar \tau_{p,p'}$
and either $\square_{p,p'}$, $\Diamond_{p,p'}$, $\twosquare_{p,p'}$,
it suffices to estimate their truncated norms separately.
For $r\geq 0$, we write $\langle r\rangle =\max(1,r)$.
We first have:
\begin{lemma} \label{lem:TruncatedNormTau}
For every $(p,C)$ satisfying \eqref{eq:initialization}, $r\geq 0$, we have for some $C'$ depending only on $d,C$,
\[
\|\tau_p\|_{\infty,r} \leq C \langle r\rangle^{2-d} \quad \text{and} \quad  \|\tau_p\|_{1,r} \leq C' \langle r\rangle^{2} .
\]
\end{lemma}
\begin{proof}
The first inequality is simply \eqref{eq:initialization} rewritten differently. For the second one, we have
\[
\|\tau_p\|_{1,r}=\sum_{x,\|x\|_2<r} \tau_p(x)\leq C\sum_{x,\|x\|_2<r} \langle x\rangle^{2-d}\leq C'\langle r\rangle^2. \qedhere
\]
\end{proof}
\begin{lemma} \label{lem:TruncatedNormTau2}
In the setting of Theorem \ref{thm:PointwiseError}, there exists $C''>0$ depending only on $d,L,C$ such that if (a) holds, and $\X(p)$ is large enough, then for every $r\geq 0$ we have
\[
\|\bar \tau_{p,p'}-\1_{\{0\}} \|_{\infty,r} \leq C'' \langle r\rangle^{2-d}/\X(p) \quad ; \quad  \|\bar \tau_{p,p'}-\1_{\{0\}}\|_{1, r} \leq C''\min(\langle r\rangle^2,\X(p'))/\X(p).
\]
\end{lemma}
\begin{proof}
Recall that $\X(p')=\|\tau_{p'}\|_1$ and $\bar \tau_{p,p'}=\1_{\{0\}}+(p'-p)/(1-p)\tau_{p'}\ast D$. Thus, the result follows directly from Lemmas \ref{lem:p'-p} and \ref{lem:TruncatedNormTau}.
\end{proof}
\begin{lemma} \label{lem:TruncatedNormTauSquare}
In the setting of Theorem \ref{thm:PointwiseError}, there exists $C'''>0$ depending only on $d,L,C$ such that if (b) holds, and $\X(p)$ is large enough, then for every $r\geq 0$ we have
\[
\|\bar \tau_{p,p'}^\square-\1_{\{0\}} \|_{\infty,r} \leq C''' \langle r\rangle^{2-d}/\X(p) \quad ; \quad  \|\bar \tau^\square_{p,p'}-\1_{\{0\}}\|_{1, r} \leq C'''\min(\langle r\rangle^2,\|\tau^\square_{p,p'}\|_1)/\X(p).
\]
\end{lemma}
\begin{proof}
The proof is exactly the same as for Lemma \ref{lem:TruncatedNormTau2}, using (b) instead of (a) to bound $\|\tau^\square_{p,p'}\|_{\infty,r}$ and $\|\tau^\square_{p,p'}\|_{1,r}$. We omit the details.
\end{proof}
We now deal with the main terms $\square_{p,p'}$, $\Diamond_{p,p'}$ and $\twosquare_{p,p'}$.
\begin{lemma} \label{lem:TruncatedNormSquare}
In the setting of Theorem \ref{thm:PointwiseError}, there exist $K^\square_1,K^\square_\infty>0$ depending only on $d,L,C$ such that if (a) and (b) hold, and $\X(p)$ is large enough, then for every $r\geq 0$ we have
\[
\|\square_{p,p'}\|_{\infty,r} \leq K^\square_\infty \langle r\rangle^{6-d} \|\tau_p\|_{\infty,r} \quad \text{and}  \quad \|\square_{p,p'}\|_1 \leq K^\square_1 \psi_d(\X(p)).
\]
\end{lemma}
\begin{proof}
By definition, for every $x\in \Z^d$, we have
\[
\square_{p,p'}(x)=\tau_p(x)\cdot[\tau_p\ast D\ast \tau_{p,p'}^\square \ast D\ast \tau_p](x).
\]
By assumptions (a) and (b) from Theorem \ref{thm:PointwiseError}, we can
inductively use Lemma \ref{lem:TruncatedPoly} to get, for some constant $K=K(d,L,C)$ and every $x\in \Z^d$,
\begin{equation}
[\tau_p\ast D\ast \tau_{p,p'}^\square \ast D\ast \tau_p](x)\leq K\langle x\rangle^{6-d}. \label{eq:ChapeauForSquare}
\end{equation}
The first inequality of the lemma then follows by product. For the second inequality, note that
\[
\|\square_{p,p'}\|_1=[\tau_p\ast (\tau_p\ast D\ast \tau_{p,p'}^\square \ast D\ast \tau_p)](0),
\]
so the second inequality follows from Lemma \ref{lem:TruncatedPoly}, (a), and \eqref{eq:ChapeauForSquare}.
\end{proof}
\begin{lemma} \label{lem:TruncatedNormDiamond}
In the setting of Theorem \ref{thm:PointwiseError},
there exist $K^\Diamond_1,K^\Diamond_\infty>0$ depending only on $d,L,C$
such that if (a) and (b) hold, and $\X(p)$ is large enough,
then the following inequalities hold:
For every $r\geq 0$,
\[
\|\Diamond_{p,p'}\|_{\infty,r} \leq K^\Diamond_\infty  \langle r\rangle^{8-2d}.
\]
For every $r\geq (\X(p))^{1/2}$ we have
\[
\|\Diamond_{p,p'}\|_{\infty,r} \leq K^\Diamond_\infty(r^{4-2d}\X(p)^2+\|\tau_{p}\|_{\infty,r/3} r^{6-d}).
\]
And,
\[
\|\Diamond_{p,p'}\|_1 \leq K^\Diamond_1 \psi_d(\X(p)).
\]
\end{lemma}
\begin{proof}
In this proof, $K$ denotes a constant that depends only on $d,L,C$ and may change from line to line. By definition of $\Diamond_{p,p'}$, for every $x\in \Z^d$, we have
\[
\Diamond_{p,p'}(x) = [\tau_p\ast D \ast \bar \tau_{p,p'}^\square\ast \tau_p](x)\cdot [\tau_p\ast D \ast \bar \tau_{p,p'}\ast \tau_p](x).
\]
By Lemma \ref{lem:SplitConvolTail}, for every $r\geq 0$, $\|\tau_p\ast D \ast \bar \tau_{p,p'}^\square\ast \tau_p\|_{\infty,r}$ is bounded by
\[
\|\tau_p\ast D\ast \tau_p\|_{\infty,r}+\X(p)^2\|D\ast (\bar \tau^\square_{p,p'}-\1_{\{0\}})\|_{\infty,r/3}+2\X(p)\|D\ast (\bar \tau_{p,p'}^\square-\1_{\{0\}})\|_{1,r/3} \|\tau_p\|_{\infty,r/3}.
\]
Also, by Lemma \ref{lem:SplitConvolTail} and \eqref{eq:initialization}, we have
\[
\|\tau_p\ast D\ast \tau_p\|_{\infty,r}\leq \X(p)(\|D\ast \tau_p\|_{\infty,r/2}+\Omega \|\tau_p\|_{\infty,r/2})\leq K\X(p) \langle r\rangle^{2-d}.
\]
Thus, by Lemma \ref{lem:TruncatedNormTauSquare},
\begin{equation}
\|\tau_p\ast D \ast \bar \tau_{p,p'}^\square\ast \tau_p\|_{\infty,r} \leq K(\X(p) \langle r\rangle ^{2-d}+\langle r\rangle ^2\|\tau_p\|_{\infty,r/3} ).\label{eq:TruncatedCross}
\end{equation}
Also, by inductively using Lemma \ref{lem:TruncatedPoly}, (a), and Lemma \ref{lem:TruncatedNormTauSquare}, for every $r\geq 0$, we have
\[
\|\tau_p\ast D \ast \bar \tau_{p,p'}^\square\ast \tau_p\|_{\infty,r} \leq K(\langle r\rangle^{4-d}+\langle r\rangle^{6-d}/\X(p)).
\]
Taking the infimum of these two bounds, we get
 \begin{align}
 \|\tau_p\ast D \ast \bar \tau_{p,p'}^\square\ast \tau_p\|_{\infty,r} & \leq K \inf\left (\X(p)\langle r\rangle ^{2-d}+\langle r\rangle^2 \|\tau_p\|_{\infty,r/3},\langle r\rangle^{4-d}+ \langle r\rangle^{6-d}/\X(p) \right ). \notag
\\ & \leq K\langle r\rangle ^{4-d} \label{eq:TruncatedCross'}
  \end{align}

By a similar argument, using Lemma \ref{lem:TruncatedNormTau2} instead of Lemma \ref{lem:TruncatedNormTauSquare}, we have
  \begin{equation}
  \|\tau_p\ast D \ast \bar \tau_{p,p'} \ast \tau_p\|_{\infty,r}\leq K(\X(p)\langle r\rangle ^{2-d}+\langle r\rangle^2 \|\tau_p\|_{\infty,r/3}), \label{eq:TruncatedCross2'}
  \end{equation}
and
 \begin{equation}
 \|\tau_p\ast D \ast \bar \tau_{p,p'}\ast \tau_p\|_{\infty,r}\leq  K\langle r\rangle^{4-d}.\label{eq:TruncatedCross2}
 \end{equation}
The first bound of the lemma is obtained by taking the product of \eqref{eq:TruncatedCross'} and \eqref{eq:TruncatedCross2}.

By taking the product of \eqref{eq:TruncatedCross} with \eqref{eq:TruncatedCross2'}, and using the general inequality $(a+b)^2\leq 2(a^2+b^2)$, we obtain
\[
\|\Diamond_{p,p'}\|_{\infty,r}\leq K(\X(p)^2 \langle r\rangle^{4-2d}+\langle r\rangle^4 \|\tau_p\|_{\infty,r/3}^2).
\]
The second inequality is obtained using (a).

For the last inequality, we rewrite
\[
\|\Diamond_{p,p'}\|_1=[\tau_p\ast \tau_p\ast \tau_p]\ast [D\ast D\ast \tau_p\ast \bar \tau_{p,p'} \ast \bar \tau_{p,p'}^\square](0).
\]
By inductively using Lemma \ref{lem:TruncatedPoly}, with (a) and Lemmas \ref{lem:TruncatedNormTau2} and \ref{lem:TruncatedNormTauSquare}, we get for every $r\geq 0$,
\[
\|D\ast D\ast\tau_p\ast \bar \tau_{p,p'} \ast \bar \tau_{p,p'}^\square\|_{\infty,r} \leq K(\langle r\rangle ^{2-d}+\langle r\rangle ^{4-d}/\X(p)+\langle r\rangle^{6-d}/\X(p)^2).
\]
Also recall that by Lemma \ref{lem:TwoTriangleSquare} (b) we have
$\|\tau_p\ast \tau_p\ast \tau_p\|_{\infty,r}\leq \langle r\rangle^{6-d}.$ It follows by Lemma \ref{lem:TruncatedPoly} applied with $a=d-6$, $b\in \{d-2,d-4,d-6\}$ that we have when $d=7$,
\[
\|\Diamond_{p,p'}\|_1 \leq K ((\X(p)^3)^{1-5/6} +(\X(p)^3)^{1-3/6}/\X(p)+(\X(p)^3)^{1-1/6}/\X(p)^2)=3K\X(p)^{1/2}.
\]
Similarly, when $d=8$, by Lemma \ref{lem:TruncatedPoly} we have
\[
\|\Diamond_{p,p'}\|_1 \leq K (3\log(\X(p)) +(\X(p)^3)^{1-4/6}/\X(p)+(\X(p)^3)^{1-2/6}/\X(p)^2)=3K\log(\X(p))+2K.
\]
When $d=9$, still by Lemma \ref{lem:TruncatedPoly}, we have
\[
\|\Diamond_{p,p'}\|_1 \leq K (1+(\X(p)^3)^{1-5/6}/\X(p)+(\X(p)^3)^{1-3/6}/\X(p)^2)=K+2K/\sqrt{\X(p)}.
\]
We can do similar calculations for $d\in \{10,11,12\}$ and for $d\geq 13$. We omit the details.
\end{proof}
We now deal with $\twosquare_{p,p'}$, which corresponds to more complex correlations, and so is smaller.
\begin{lemma} \label{lem:TruncatedNormTwoSquare}
In the setting of Theorem \ref{thm:PointwiseError},
there exist $K^\twosquare_1,K^\twosquare_\infty>0$ depending only on $d,L,C$
such that if (a) and (b) hold, and $\X(p)$ is large enough,
then the following inequalities hold:
For every $0\leq r \leq \sqrt{\X(p)}$,
\[
\|\twosquare_{p,p'}\|_{\infty,r} \leq K^\twosquare_\infty  \langle r\rangle ^{7-2d}.
\]
For every $r\geq \sqrt{\X(p)}$,
\[
\|\twosquare_{p,p'}\|_{\infty,r} \leq K^\twosquare_\infty \|\tau_p\|_{\infty,r/2} r^{4-d} \psi_d(\X(p)).
\]
And,
\[
\|\twosquare_{p,p'}\|_1 \leq K^\twosquare_1 \psi_{d+1}(\X(p)).
\]
\end{lemma}
\begin{proof}
As in the previous proof, $K$ will denote a constant depending only on $d,L,C$ that may change from line to line. By definition, for every $z\in \Z^d$,
\begin{equation}
\twosquare_{p,p'}(z) =\sum_{v,y\in \Z^d} \tau_p(v) \tau_p(z-v) \tau_p(y-v)[\tau_p\ast D\ast \bar \tau_{p,p'}^\square\ast \tau_p](y)[\tau_p\ast D\ast \tau_p](z-y). \label{eq:recall_twosquare}
\end{equation}
We distinguish two cases, which we treat similarly. Either $\|v\|_2\geq \|z\|_2/2$ or $\|z-v\|_2\geq \|z\|_2/2$. We bound the contribution of the first case by
\begin{align}
\twosquare^1_{p,p'}(z)  := &\sum_{v,y\in \Z^d} \| \tau_p\|_{\infty,\|z\|_2/2}  \tau_p(z-v) \tau_p(y-v)[\tau_p\ast D\ast \bar \tau_{p,p'}^\square\ast \tau_p](y)[\tau_p\ast D\ast \tau_p](z-y) \notag
\\  = & \| \tau_p\|_{\infty,\| z\|_2/2} [\tau_p\ast D\ast \bar \tau_{p,p'}^\square\ast \tau_p \ast \Diamond^1_{p,p'}](z), \label{eq:def:twosquare^1}
\end{align}
where
\[
\Diamond^1_{p,p'}:x\mapsto [\tau_p\ast \tau_p](x)[\tau_p\ast D\ast \tau_p](x).
\]
By Lemma \ref{lem:TwoTriangleSquare} (a), for every $r\geq 0$,
\[
\|\tau_p\ast \tau_p\|_{\infty,r}\leq K \langle r\rangle^{4-d}\quad \text{and} \quad \|\tau_p\ast D\ast \tau_p\|_{\infty,r}\leq K \langle r\rangle^{4-d}.
\]
Thus, by taking products,
\begin{equation}
\|\Diamond^1_{p,p'}\|_{\infty,r}\leq K \langle r\rangle^{8-2d} \leq K  \langle r\rangle^{1-d} . \label{eq:Diamond_1}
\end{equation}

Applying Lemma \ref{lem:TruncatedPoly} twice, together with \eqref{eq:initialization} and Lemma \ref{lem:TruncatedNormTauSquare}, we deduce, for every $r\geq 0$,
\[
\|D\ast \bar \tau_{p,p'}^\square\ast \tau_p \ast \Diamond^1_{p,p'}\|_{\infty,r} \leq K(\langle r \rangle^{3-d}+ \langle r \rangle^{5-d}/\X(p)).
\]
Another application of Lemma \ref{lem:TruncatedPoly}, together with \eqref{eq:initialization}, gives
\[
\|\tau_p\ast D\ast \bar \tau_{p,p'}^\square\ast \tau_p \ast \Diamond^1_{p,p'}\|_{\infty,r} \leq K \left (\langle r \rangle^{5-d}+\log \left (2+ \frac{\sqrt{\X(p)}}{\langle r\rangle} \right ) \frac{\langle r \rangle^{7-d}}{\X(p)}\right ).
\]
Since $x\in \Z^d \mapsto \log(2+1/x)x^{2}$ is bounded on $(0,1]$, by \eqref{eq:def:twosquare^1} and \eqref{eq:initialization}, for every $r\leq \sqrt{\X(p)}$ we get
\begin{equation}
\|\twosquare^1_{p,p'}\|_{\infty,r}\leq K \langle r\rangle^{2-d}\langle r\rangle^{5-d}=K\langle r\rangle^{7-2d}.
\label{eq:Intermediary_L1_TwoSquare_A}
\end{equation}

On the other hand, Lemma \ref{lem:TwoTriangleSquare} gives
\begin{equation}
\|\Diamond^1_{p,p'}\|_1=[\tau_p\ast \tau_p\ast \tau_p\ast \tau_p\ast D](0)\leq K\psi_d(\X(p)). \label{eq:Diamond_1'}
\end{equation}
By Lemma \ref{lem:SplitConvolTail}, for every $r\geq 0$, we have
\begin{align*}
\|\tau_p\ast D\ast \bar \tau_{p,p'}^\square\ast \tau_p \ast \Diamond^1_{p,p'}\|_{\infty,r}
\leq \X(p)^2\|\Diamond^1_{p,p'}\|_1 & \|D\ast \bar\tau_{p,p'}^\square\|_{\infty,r/4}
+\X(p)^2 \|D\ast\bar\tau_{p,p'}^\square\|_{1,r/4}  \|\Diamond^1_{p,p'}\|_{\infty,r/4}\\
&+2 \X(p) \|\tau_p\|_{\infty,r/4}\|D\ast \bar \tau_{p,p'}^\square\|_{1,r/4} \|\Diamond^1_{p,p'}\|_1 .
\end{align*}
Substituting the bounds from Lemma \ref{lem:TruncatedNormTauSquare}, \eqref{eq:Diamond_1}, and \eqref{eq:Diamond_1'}, we get, for every $r\geq \sqrt{\X(p)}$,
\begin{align*}
\|\tau_p\ast D\ast \bar \tau_{p,p'}^\square\ast \tau_p \ast \Diamond^1_{p,p'}\|_{\infty,r} & \leq K\big ( \X(p) \psi_d(\X(p)) r^{2-d} + \X(p) r^2 r^{8-2d}+\|\tau_p\|_{\infty,r/4} r^2\psi_d(\X(p)) \big )
\\ & \leq 3Kr^{4-d}\psi_d(\X(p)),
\end{align*}
using $r\geq \sqrt{\X(p)}\geq \psi_d(\X(p))$ and \eqref{eq:initialization} for the last inequality. Thus, by \eqref{eq:def:twosquare^1}, for $r\geq \sqrt{\X(p)}$, we get
\begin{equation}
\|\twosquare^1_{p,p'}\|_{\infty,r} \leq 3K\|\tau_p\|_{\infty,r/2}r^{4-d}\psi_d(\X(p))
\label{eq:Intermediary_L1_TwoSquare_B}
\end{equation}

Going back to \eqref{eq:recall_twosquare}, we now have to deal with the case $\|z-v\|_2\geq \|z\|_2/2$. We bound the contribution  of this case to $\twosquare_{p,p'}$ by
\begin{align*}
\twosquare^2_{p,p'}(z)  := &\sum_{v,y\in \Z^d} \|\tau_p\|_{\infty,\|z\|_2/2} \tau_p(v) \tau_p(y-v)[\tau_p\ast D\ast \bar \tau_{p,p'}^\square\ast \tau_p](y)[ \tau_p\ast D\ast \tau_p](z-y)
\\  = & \|\tau_p\|_{\infty,\|z\|_2/2} [\Diamond^2_{p,p'} \ast \tau_p\ast D\ast \tau_p](z),
\end{align*}
where
\[
\Diamond^2_{p,p'}:y\mapsto [\tau_p\ast \tau_p](y)[\tau_p\ast D\ast \bar \tau^\square_{p,p'}\ast \tau_p](y).
\]
With the same proof as for Lemma \ref{lem:TruncatedNormDiamond}, we get that $ \Diamond^2_{p,p'}$ also satisfies \eqref{eq:Diamond_1} and \eqref{eq:Diamond_1'}.
Then, with the same proof as for $\twosquare^1_{p,p'}$, we get that $\twosquare^2_{p,p'}$ also satisfies the analogues of \eqref{eq:Intermediary_L1_TwoSquare_A} and \eqref{eq:Intermediary_L1_TwoSquare_B}.
We omit the details. The bounds on $\| \twosquare_{p,p'}\|_{\infty,r}$ are obtained by sum, since $\twosquare_{p,p'}\leq \twosquare^1_{p,p'}+\twosquare^2_{p,p'}$.

Finally, we have
\[
\|\twosquare_{p,p'}\|_1 \leq  \sum_{\substack{ x\in \Z^d, \\ \langle x\rangle\leq \sqrt{\X(p)}}} \|\twosquare_{p,p'}\|_{\infty,\langle x\rangle}+ \sum_{\substack{ x\in \Z^d, \\ \langle x\rangle> \sqrt{\X(p)}}} \|\twosquare_{p,p'}\|_{\infty,\langle x\rangle}.
\]
Substituting the bounds on $\|\twosquare_{p,p'}\|_{\infty,\langle x\rangle}$, we find that the first sum is bounded by
\[
K_\infty^\twosquare \sum_{\substack{ x\in \Z^d, \\ \langle x\rangle\leq \sqrt{\X(p)}}}\langle x\rangle ^{7-2d} \leq K\psi_{d+1}(\X(p)).
\]
Using \eqref{eq:initialization}, the second sum is bounded by
\[
K_\infty^\twosquare \hspace{-1em} \sum_{\substack{ x\in \Z^d, \\ \langle x\rangle> \sqrt{\X(p)}}}
 \|\tau_p\|_{\infty,\langle x\rangle /2}\langle x\rangle^{4-d}\psi_d(\X(p))
\leq K\X(p)^{(6-d)/2}\psi_d(\X(p)).
\]
Since $d\geq 7$, the last sum is bounded by a constant.
\end{proof}
We can now prove the main result for the bootstrap argument:
\begin{proposition} \label{pro:PointwiseError}
Let $C>0$. There exist $c,F,K,K'>0$ depending only on $d,L,C$ such that for every $p\leq p'$ with $\X(p)\geq F$, if (a) and (b) of Theorem \ref{thm:PointwiseError} hold, then its conclusion also holds.
\end{proposition}
\begin{proof}
As usual, $K=K(d,L,C)$ and $K'=K'(d,L,C)$ denote constants that may change from line to line. We start by bounding $\|\tau_{p'}-\tau_{p,p'}^\square\|_1$. By Theorems \ref{thm:LowerBoundTau} and \ref{thm:UpperBoundTau}, we have
\[
\|\tau_{p'}-\tau_{p,p'}^\square\|_1 \leq ((p'-p)/(1-p))^2 \max( \|\bar \tau_{p,p'}\|_1,\|\bar \tau^\square_{p,p'}\|_1)^2 \|\tau_p\|_1^2 (\|\square_{p,p'}\|_1+\|\Diamond_{p,p'}\|_1+\|\twosquare_{p,p'}\|_1).
\]
Substituting the bounds from Lemmas \ref{lem:p'-p}, \ref{lem:TruncatedNormTau2}, \ref{lem:TruncatedNormTauSquare}, \ref{lem:TruncatedNormSquare}, \ref{lem:TruncatedNormDiamond}, and \ref{lem:TruncatedNormTwoSquare}, we get
\begin{equation}
\|\tau_{p'}-\tau_{p,p'}^\square\|_1 \leq K\max(\X(p'),\|\tau_{p,p'}^\square\|_1)^2 \psi_d(\X(p))/\X(p)^2. \label{eq:DifL1withMax}
\end{equation}
Furthermore, using (a) and (b), we get
\[
\|\tau_{p'}-\tau_{p,p'}^\square\|_1 \leq cK\max(\X(p'),\|\tau_{p,p'}^\square\|_1).
\]
So if $c$ is small enough depending on the previous $K$, which we can now assume since we never used $c$ before,
\[
|\X(p')-\|\tau_{p,p'}^\square\|_1| \leq \frac{1}{2}\max(\X(p'),\|\tau_{p,p'}^\square\|_1),
\]
and so
\begin{equation}
\max(\X(p'),\|\tau_{p,p'}^\square\|_1)\leq 2\min(\X(p'),\|\tau_{p,p'}^\square\|_1). \label{eq:Max<Min}
\end{equation}
The bound on $\|\tau_{p'}-\tau_{p,p'}^\square\|_1$ follows from \eqref{eq:DifL1withMax}.

We now bound $\| \tau_{p'}-\tau^\square_{p,p'}\|_{\infty,5r}$ for every $r\geq 0$. By Lemma \ref{lem:SplitConvolTail} and Theorems \ref{thm:LowerBoundTau} and \ref{thm:UpperBoundTau}, it suffices to bound separately
\begin{align*}
& \text{(1)} \quad \|\bar \tau^\square_{p,p'}\|_{\infty,r}\|\bar \tau^\square_{p,p'}\|_1 \|\tau_p\|_1^2(\|\square_{p,p'}\|_1+\|\Diamond_{p,p'}\|_1)
 & &\text{(2)} \quad \|\bar \tau^\square_{p,p'}\|_{\infty,r} \|\bar \tau_{p,p'}\|_1 \|\tau_p\|_1^2 \|\twosquare_{p,p'}\|_1
\\ &\text{(3)} \quad  \|\bar \tau^\square_{p,p'}\|_1 \|\bar \tau_{p,p'}\|_{\infty,r} \|\tau_p\|_1^2 \|\twosquare_{p,p'}\|_1
 & &\text{(4)} \quad \|\bar \tau^\square_{p,p'}\|_{1,r}^2 \|\tau_p\|_{\infty,r}\|\tau_p\|_1(\|\square_{p,p'}+\Diamond_{p,p'}\|_1)
\\ & \text{(5)} \quad \|\bar \tau^\square_{p,p'}\|_{1,r} \|\bar \tau_{p,p'} \|_{1,r} \|\tau_p\|_{\infty,r}  \|\tau_p\|_1 \|\twosquare_{p,p'}\|_1
& &\text{(6)} \quad \|\bar \tau^\square_{p,p'}\|_{1,r}^2 \|\tau_p\|_{1,r}^2 (\|\square_{p,p'}+\Diamond_{p,p'}\|_{\infty,r})
\\
& \text{(7)} \quad  \|\bar \tau^\square_{p,p'}\|_{1,r}\|\bar \tau_{p,p'}\|_{1,r} \|\tau_p\|_{1,r}^2 \|\twosquare_{p,p'}\|_{\infty,r}.
\end{align*}
For (1), by Lemmas \ref{lem:TruncatedNormTauSquare}, \ref{lem:TruncatedNormSquare}, and \ref{lem:TruncatedNormDiamond}, and by $\|\tau_{p,p'}^\square\|_1\leq 2 \X(p')$, we have
\[
\|\bar \tau^\square_{p,p'}\|_{\infty,r}\|\bar \tau^\square_{p,p'}\|_1 \|\tau_p\|_1^2(\|\square_{p,p'}\|_1+\|\Diamond_{p,p'}\|_1) \leq K \langle r\rangle^{2-d} \X(p')\psi_d(\X(p)).
\]
We deal with (2) and (3) similarly by using Lemmas \ref{lem:TruncatedNormTau2}, \ref{lem:TruncatedNormTauSquare}, and \ref{lem:TruncatedNormTwoSquare}.
(Note that the bounds on $\twosquare_{p,p'}$ in Lemma \ref{lem:TruncatedNormTwoSquare} are strictly better than the ones on $\square_{p,p'}$ and $\Diamond_{p,p'}$.)

For (4), by Lemmas \ref{lem:TruncatedNormTauSquare}, \ref{lem:TruncatedNormSquare}, and \ref{lem:TruncatedNormDiamond}, and by $\|\tau_{p,p'}^\square\|_1\leq 2 \X(p')$, we have
\begin{align*}
\|\bar \tau^\square_{p,p'}\|_{1,r}^2 \|\tau_p\|_{\infty,r}\|\tau_p\|_1(\|\square_{p,p'}\|_1+\|\Diamond_{p,p'}\|_1) & \leq K\|\tau_p\|_{\infty,r}(1+ \min(r^2, \X(p'))/\X(p))^2\X(p) \psi_d(\X(p)),
\\ & \leq K'(1+\Xi(r/3,p,p'))r^{2-d}\X(p')\psi_d(\X(p)),
\end{align*}
by definition of $\Xi$. We deal with (5) similarly.

For (6), if $r^2\leq \X(p)$, by Lemmas \ref{lem:TruncatedNormTau}, \ref{lem:TruncatedNormTauSquare}, \ref{lem:TruncatedNormSquare}, \ref{lem:TruncatedNormDiamond}, we have
\[
\|\bar \tau^\square_{p,p'}\|_{1,r}^2 \|\tau_p\|_{1,r}^2 (\|\square_{p,p'}\|_{\infty,r}+\|\Diamond_{p,p'}\|_{\infty,r}) \leq K r^4 r^{8-2d} \leq K'r^{2-d} \X(p')\psi_d(\X(p)),
\]
using $r^3\leq \X(p)^{3/2}$ when $d=7$ and $r^{10-d}\leq \X(p)$ when $d\geq 8$. If $r\geq \sqrt{\X(p)}$, by the same lemmas and by $\|\tau_{p,p'}^\square\|_1\leq 2 \X(p')$, we have
\[
\|\bar \tau^\square_{p,p'}\|_{1,r}^2 \|\tau_p\|_{1,r}^2 (\|\square_{p,p'}\|_{\infty,r}+\|\Diamond_{p,p'}\|_{\infty,r}) \leq K \min(r^2,\X(p'))^2 (r^{4-2d}\X(p)^2 +\|\tau_p\|_{\infty,r/3}r^{6-d}).
\]
Since $r^2\geq \X(p)$ we have $\X(p)^2 \leq r^{d-4}\psi_d(\X(p))$, and so
\[
\min(r^2,\X(p'))^2 (r^{4-2d}\X(p)^2)\leq \X(p')r^{6-2d} \X(p)^2 \leq  \X(p') r^{2-d} \psi_d(\X(p)).
\]
For the other term, we have
\[
\min(r^2,\X(p'))^2\|\tau_p\|_{\infty,r/3}r^{6-d} \leq \Xi(r/4,p,p') \X(p') \X(p) r^{2-d}r^{6-d} \leq \Xi(r/4,p,p') \X(p') r^{2-d} \psi_d(\X(p))  .
\]
This concludes the treatment of (6). We deal with (7) similarly.

Summing our bounds on (1)--(7), and using Theorems \ref{thm:LowerBoundTau} and \ref{thm:UpperBoundTau} with Lemmas \ref{lem:p'-p} and \ref{lem:SplitConvolTail}, we get
\[
\| \tau_{p'}-\tau^\square_{p,p'}\|_{\infty,5r}\leq K\langle r\rangle^{2-d} \frac{\X(p')\psi_d(\X(p))}{\X(p)^2}(1+\Xi(r/3,p,p')).
\]
Taking $r=\| x\|_2/5$ concludes the proof.
\end{proof}
Our goal is now to show Theorem \ref{thm:PointwiseError} using Proposition \ref{pro:PointwiseError} and a bootstrap argument.
To this end, we first have to deal with a small issue.
Typically, bootstrap arguments are performed with a finite number of functions, but here we have two for every $x\in \Z^d$: $p'\mapsto \tau_{p'}(x)$ and $p'\mapsto \tau^\square_{p,p'}(x)$.
To avoid that issue, we need to prove separately some rough estimates on the tails of $\tau_{p'}(x)$ and $\tau^\square_{p,p'}(x)$.
For $\tau_{p'}(x)$, Lemma \ref{lem:NaiveExpTail} does the job.
For $\tau^\square_{p,p'}(x)$, we use the following two lemmas:
\begin{lemma} \label{lem:CheapTailTauSquare}
Let $0<p<p'<\inf(p_c,\tilde p_c(p))$. Assume that
\[
\varpi(p,p'):=((p'-p)/(1-p))^3\X(p)^2 \|\square_{p,p'}+\Diamond_{p,p'}\|_1\|\bar \tau_{p,p'}^\square\|_1<2^{-d}/\Omega,
\]
then, as $\|x\|_2\to \infty$, we have
\[
\tau^\square_{p,p'}(x)=o(\|x\|_2^{2-d}).
\]
\end{lemma}
\begin{proof}
By Lemma \ref{lem:NaiveExpTail}, there exists $l>0$ such that as $r\to \infty$ we have
\[
\|\tau_{p'}\|_{\infty,l\log(r)}=o(r^{-d}).
\]
Since $\varpi(p,p')<1/2$, we have $\|\tau_{p,p'}^\square\|_1<\infty$, and so, as $r\to \infty$,
\[
\|\square_{p,p'}\|_{\infty,4l\log(r)} \leq \|\tau_{p}\|_{\infty,4l\log(r)} \|\tau_p\ast D\ast \tau_{p,p'}^\square\ast D\ast \tau_p\|_\infty =o(r^{-d}).
\]
Similarly, Lemma \ref{lem:SplitConvolTail} gives
\[
\|\Diamond_{p,p'}\|_{\infty,4l\log(r)} =O(\|\tau_p\ast D\ast \bar \tau_{p,p'}\ast \tau_p\|_{\infty,4l\log(r)} )=O(\|\tau_{p'}\|_{\infty,4l\log(r)})=o(r^{-d}).
\]
Recall that by Theorem \ref{thm:LowerBoundTau} we have
\[
\tau_{p,p'}^\square\leq \tau_{p'}+((p'-p)/(1-p))^2\bar \tau_{p,p'}^\square \ast \tau_p\ast (\square_{p,p'}+\Diamond_{p,p'})\ast \tau_p\ast \bar \tau^\square_{p,p'}.
\]
Also, for $s\geq L$,
\[
\|\bar \tau_{p,p'}^\square\|_{\infty,s}\leq ((p'-p)/(1-p))\Omega \|\tau_{p,p'}^\square\|_{\infty,s-L}.
\]
Thus, by Lemma \ref{lem:SplitConvolTail}, with the decomposition $r=(r/2-3l\log(r))+l\log(r)+4l\log(r)+l\log(r)+(r/2-3l\log(r))$, we have as $r\to \infty$,
\[
\|\tau_{p,p'}^\square\|_{\infty,r} \leq o(r^{-d})+ 2\Omega\varpi(p,p') \|\tau_{p,p'}^\square\|_{\infty,r/2-4l\log(r)}.
\]
The Master Theorem then yields the desired result.
\end{proof}
\begin{lemma} \label{lem:KappaContinuous}
For every $p<p_c$, the map $p'\mapsto \varpi(p,p')$ is continuous on $[p,\min(p_c,\tilde p_c(p)))$.
\end{lemma}
\begin{proof}
By Proposition \ref{pro:PrelimBootstrap}, $p'\mapsto \|\bar \tau_{p,p'}^\square\|_1$ is continuous below $\tilde p_c(p)$. Next, note that $p'\mapsto \|\square_{p,p'}\|_1$ is of the form
\[
A^\square_p+\sum_{x\in \Z^d} a^\square_p(x)\tau^\square_{p,p'}(x),
\]
where $A^\square_p\geq 0$, and for every $x\in \Z^d$, we have $a^\square_p(x)\geq 0$, and by Proposition \ref{pro:PrelimBootstrap}, $p'\mapsto \tau^\square_{p,p'}(x)$ is continuous below $\tilde p_c(p)$.
Moreover, for every $p'<\tilde p_c(p)$, we have
\[
\|\square_{p,p'}\|_1\leq \|D\|_1^2 \|\tau_p\|_1^3 \|\tau_{p,p'}^\square\|_1<\infty.
\]
Thus, by dominated convergence, $p'\mapsto \|\square_{p,p'}\|_1$ is continuous below $\tilde p_c(p)$.

Similarly, we have
\[
\|\Diamond_{p,p'}\|_1\leq \|D\|_1^2 \|\tau_p\|_1^4 \|\bar \tau_{p,p'}^\square\|_1\|\bar \tau_{p,p'}\|_1 <\infty,
\]
whenever $p'<\min(p_c,\tilde p_c(p))$. Thus, again by dominated convergence, $p'\mapsto \|\Diamond_{p,p'}\|_1$ is continuous below $\min(p_c,\tilde p_c(p))$. The map $p'\mapsto \varpi(p,p')$ is then continuous by sums and products.
\end{proof}

We then deal with $\Xi(r,p,p')$.
\begin{lemma} \label{lem:AnnoyingXi}
For every $C,\eps>0$, if $(p,C)$ satisfies \eqref{eq:initialization}, $\X(p)$ is large enough, and $\X(p')\geq \X(p)^{1+\eps}$, then for every $r\geq 0$, we have
\[
\Xi(r,p,p')\leq 1.
\]
\end{lemma}
\begin{proof}
By Theorem \ref{thm:ExponentialTail}, we may assume that $\X(p)$ is large enough that for every $r\geq \X(p)^{1/2+\eps/5}$ we have
\begin{equation}
\|\tau_p\|_{\infty,r}\leq r^{-d}. \label{eq:tailRewritten}
\end{equation}
Thus, for every $r\geq \X(p)^{1/2+\eps/5}$,
\[
\Xi(r,p,p') = r^{d-2}\|\tau_p\|_{\infty,r} \frac{\min(r^2,\X(p'))^2}{\X(p)\X(p')} \leq 1.
\]
Also, for every $r\leq \X(p)^{1/2+\eps/3}$, \eqref{eq:initialization} gives
\[
\Xi(r,p,p') \leq C \frac{r^4}{\X(p)\X(p')} \leq C\X(p)^{-\eps/5}\leq 1,
\]
provided that $\X(p)$ is large enough.
\end{proof}

\begin{proof}[Proof of Theorem \ref{thm:PointwiseError}]
Dropping the dependence on $d,L$ in the notation, let $c_C,F_C,K_C,K'_C$ denote the constants of Proposition \ref{pro:PointwiseError}.
After reducing $c_C$, we may assume that $c_C$ is small enough depending on $C,K_C,K'_C$.
Also, we may assume without loss of generality that $c$ is decreasing in $C$, and that $F,K,K'$ are increasing.

Fix $C\geq 1$.
Let $p<p_c$ be such that $\X(p)\geq F_{2C}$, and assume that $(p,C)$ satisfies \eqref{eq:initialization}.
We first assume that $\X(p')\leq (c_{2C}/3)\X(p)^2/\psi_d(\X(p))$ and that $(p',C)$ satisfies \eqref{eq:initialization}.
Let $S$ denote the set of all $p''\in [p,p']$ such that $\|\tau^\square_{p,p''}\|_1\leq c_{2C}\X(p)^2/\psi_d(\X(p))$ and such that for every $x\in \Z^d$, we have $\tau^\square_{p,p''}(x)\leq 2C\langle x\rangle^{2-d}$.
Note that, by \eqref{eq:initialization} at $p$ and by $\tau^\square_{p,p}=\tau_p$, if $\X(p)$ is large enough, then $p\in S$.
Let $\tilde p=\sup S$. We first check that $\tilde p\in S$. By Proposition \ref{pro:PointwiseError}, for every $p''\in S$, we have
\begin{equation}
\forall x\in \Z^d, \quad |\tau_{p''}(x)-\tau_{p,p''}^\square(x)| \leq K_{2C}\langle x\rangle^{2-d}\frac{\X(p'')\psi_d(\X(p))}{\X(p)^2}(1+\Xi(\|x\|_2/20,p,p'')), \label{eq:ResultInBootstrapA}
\end{equation}
and
\begin{equation}
\|\tau_{p''}-\tau_{p,p''}^\square\|_1 \leq K'_{2C} \frac{\min(\X(p''),\|\tau_{p,p''}^\square\|_1)^2\psi_d(\X(p))}{\X(p)^2}. \label{eq:ResultInBootstrapB}
\end{equation}
Then, as in the proof of \eqref{eq:Max<Min}, we can assume that $c_{2C}$ is small enough that
\[
\max(\X(p''),\|\tau_{p,p''}^\square\|_1)\leq 2\min(\X(p''),\|\tau_{p,p''}^\square\|_1)\leq (2/3) c_{2C}\X(p)^2/\psi_d(\X(p)).
\]
By taking $p''\in S\to \tilde p$, it follows from Proposition \ref{pro:PrelimBootstrap}, and the continuity of $p''\mapsto \X(p'')$ below $p_c$, that we also have
\begin{equation}
\max(\X(\tilde p),\|\tau_{p,\tilde p}^\square\|_1)\leq (2/3) c_{2C}\X(p)^2/\psi_d(\X(p)). \label{eq:ResultInBootstrapB'}
\end{equation}
Then, by Proposition \ref{pro:PrelimBootstrap}, for every $x\in \Z^d$, the map $\tau^\square_{p,p''}(x)$ is continuous below $\tilde p_c(p)>\tilde p$. Also, the map
$p''\mapsto \tau_{p''}(x)$ is continuous below $p_c>\tilde p$. Thus \eqref{eq:ResultInBootstrapA} also holds when $p''=\tilde p$. We thus get for every $x\in \Z^d$,
\begin{equation}
|\tau_{\tilde p}(x)-\tau_{p,\tilde p}^\square(x)|\leq 2c_{2C}K_{2C}\langle x\rangle^{2-d}\leq (C/2)\langle x\rangle^{2-d}. \label{eq:ResultInBootstrapA'}
\end{equation}
Indeed, when $\X(\tilde p)\geq \X(p)^{6/5}$, we can use Lemma \ref{lem:AnnoyingXi} to upper-bound $\Xi(\|x\|_2/20,p,\tilde p)\leq 1$, and the inequality then follows from \eqref{eq:ResultInBootstrapB'}.
When $\X(\tilde p)\leq \X(p)^{6/5}$, we can use \eqref{eq:initialization} and the definition of $\Xi$ to obtain
\[
\Xi(\|x\|_2/20,p,\tilde p)\leq C \frac{\X(\tilde p)}{\X(p)}\leq C\X(p)^{1/4}\leq c_{2C} \frac{\X(p)^2}{ \X(\tilde p)\psi_d(\X(p))},
\]
provided that $\X(p)$ is large enough. Again \eqref{eq:ResultInBootstrapA'} follows. Together, \eqref{eq:ResultInBootstrapB'} and \eqref{eq:ResultInBootstrapA'} yield $\tilde p\in S$.

We now suppose for contradiction that $\tilde p<p'$, and show that there exists $\tilde p'\in S$ with $\tilde p'>\tilde p$, which contradicts $\tilde p=\sup S$.
First, by Lemmas \ref{lem:TruncatedNormSquare} and \ref{lem:TruncatedNormDiamond} and $\tilde p\in S$, there exist $K^\square_1(2C)>0$ and $K^\Diamond_1(2C)>0$ such that
\[
\|\square_{p,\tilde p}\|_1\leq K^\square_1(2C)\psi_d(\X(p)) \quad \text{and} \quad  \|\Diamond_{p,\tilde p}\|_1\leq K^\Diamond_1(2C)\psi_d(\X(p)).
\]
Thus, by \eqref{eq:ResultInBootstrapB'} and Lemma \ref{lem:p'-p}, if $c_{2C}$ is small enough, we have
\[
\varpi(p,\tilde p) =((\tilde p-p)/(1-p))^3\X(p)^2 \|\square_{p,\tilde p}+\Diamond_{p,\tilde p}\|_1\|\bar \tau_{p,\tilde p}^\square\|_1<2^{-d}/\Omega.
\]
So, by \eqref{eq:ResultInBootstrapB'} and Lemma \ref{lem:KappaContinuous}, there exists $p_1>\tilde p$ such that $\varpi(p,p_1)<2^{-d}/\Omega$.
By Lemma \ref{lem:CheapTailTauSquare}, there exists $R\in \R$ such that, for every $x$ with $\|x\|_2 \geq R$, we have
\begin{equation}
\tau_{p,p_1}^\square(x)\leq C\langle x\rangle^{2-d}. \label{eq:KappaTail}
\end{equation}
Similarly, by Lemma \ref{lem:NaiveExpTail}, if $R\in \R$ is large enough, then for every $\|x\|_2\geq R$, we have
\begin{equation}
\tau_{p_1}(x)\leq C\langle x\rangle^{2-d}. \label{eq:KappaTail2}
\end{equation}
Also, by \eqref{eq:ResultInBootstrapB'}
and Proposition \ref{pro:PrelimBootstrap},
for every $x\in \Z^d$ the maps $p''\mapsto \tau_{p''}(x)$,
$p''\mapsto \tau_{p,p''}^\square(x)$,
$p''\mapsto \X(p'')$
and $p''\mapsto \|\tau_{p,p''}^\square\|_1$
are continuous in a neighborhood of $\tilde p$.
So, by \eqref{eq:ResultInBootstrapB'} and \eqref{eq:ResultInBootstrapA'}, for every $p''>\tilde p$ close enough to $\tilde p$, we have
\begin{equation}
\max(\X(p''),\|\tau_{p,p''}^\square\|_1)\leq c_{2C}\X(p)^2/\psi_d(\X(p)), \label{eq:ResultInBootstrapB''}
\end{equation}
and for every $x\in \Z^d$, with $\|x\|_2<R$,
\begin{equation}
|\tau_{p''}(x)-\tau_{p,p''}^\square(x)|\leq C \langle x\rangle^{2-d}. \label{eq:ResultInBootstrapA''}
\end{equation}
By \eqref{eq:KappaTail} and \eqref{eq:KappaTail2}, the previous equation still holds whenever $\|x\|_2\geq R$ and $p<p''\leq p_1$.
Hence, there exists $p''>\tilde p$ such that \eqref{eq:ResultInBootstrapB''} holds and \eqref{eq:ResultInBootstrapA''} is satisfied for every $x\in \Z^d$, so $p''\in S$.
Therefore, we have the desired contradiction.
Thus, $p'\in S$, and we can use Proposition \ref{pro:PointwiseError} at $p'$.

The proof for the case where we only have (b) is essentially the same.
We take $p'$ such that $\|\tau_{p,p'}^\square\|_1\leq (c_{2C}/3)\X(p)^2/\psi_d(\X(p))$ and such that for every $x\in \Z^d$ we have $\tau_{p,p'}^\square(x)\leq C\langle x\rangle^{2-d}$.
We fix $S$ to be the set of all $p''\in [p,p']$ such that $\X(p'')\leq c_{2C} \X(p)^2/\psi_d(\X(p))$ and such that $(p'',2C)$ satisfies \eqref{eq:initialization}.
Again $p\in S$, and we let $\tilde p:=\sup S$.
The proof that $\tilde p\in S$ is essentially the same.
The proof that we cannot have $\tilde p<p'$ is also almost the same, so we omit the details.
\end{proof}
\subsection{$L_2$ bound assuming \eqref{eq:initialization} at $p'$} \label{sec:L2AtPPrime}
The goal of this subsection is to prove the following bound: (Recall that for a function $f:\Z^d\mapsto \R$, we have $\|f\|_2^2=\sum_{x\in \Z^d} |f(x)|\|x\|_2^2$.)
\begin{proposition} \label{pro:L2Error}
Let $\eps>0$.
In the setting of Theorem \ref{thm:PointwiseError},
there exists $K''$ depending on $d,L,C$,
such that if (a) or (b) holds,
$\X(p)$ is large enough depending on $d,L,C,\eps$,
and $c$ is small enough,
then
\[
\|\tau_{p'}-\tau_{p,p'}^\square\|_2^2 \leq K'' \frac{\X(p')^{3+\eps} \psi_d(\X(p))}{\X(p)^2}
\]
\end{proposition}
\begin{remark}
In many inequalities of this section, we have an extra $\X(p)^\eps$ term.
This term comes from the fact that we use Theorem \ref{thm:ExponentialTail} to estimate the tail of $\tau_p$.
Under \eqref{eq:eta}, the computations can easily be adapted to remove this extra term, by working instead with \eqref{eq:Plateau}.
\end{remark}
By Theorem \ref{thm:PointwiseError}, if (a) or (b) is satisfied then the other assumption is also satisfied up to slightly increase $C$, and to assume that $c$ is small enough (see the proof of Theorem \ref{thm:PointwiseError}).
So we may assume, without loss of generality, that both (a) and (b) are satisfied.
\begin{lemma} \label{lem:tail2ndMoment}
Let $C,\eps>0$. For every $p$ satisfying \eqref{eq:initialization}, if $\X(p)$ is large enough, then $\|\tau_p\|_2^2 \leq \X(p)^{2+\eps}$.
\end{lemma}
\begin{proof}
We have
\[
\|\tau_p\|_2^2=\sum_{x\in \Z^d} \tau_p(x)\|x\|_2^2
\leq \sum_{x\in \Z^d, \|x\|_2^2 \leq \X(p)^{1+\eps/2}} \tau_p(x) \X(p)^{1+\eps/2}
+\sum_{x\in \Z^d, \|x\|_2^2> \X(p)^{1+\eps/2}} \tau_p(x)\|x\|_2^2.
\]
The first sum is bounded by $\X(p) \X(p)^{1+\eps/2}$, and by Theorem \ref{thm:ExponentialTail} the second sum is bounded by $1$ as long as $\X(p)$ is large enough.
\end{proof}
\begin{lemma} \label{lem:L2NormBasicDiagrams}
Let $\eps>0$.
In the setting of Theorem \ref{thm:PointwiseError}, there exist $K^\square_2,K^\Diamond_2,K^\twosquare_2 >0$ depending only on $d,L,C$ such that if (a) and (b) hold and $\X(p)$ is large enough depending on $d,L,C,\eps$, then
\[
\|\square_{p,p'}\|_2^2 \leq K^\square_2 \psi_{d-2}(\X(p))  \X(p)^{\eps} \quad ; \quad \|\Diamond_{p,p'}\|_2^2 \leq K^\Diamond_2\psi_{d-2}(\X(p)) \X(p)^{\eps},
\]
and
\[
\|\twosquare_{p,p'}\|_2^2\leq K^\twosquare_2\psi_{d-1}(\X(p)) \X(p)^{\eps}.
\]
\end{lemma}
\begin{proof}
By Lemma \ref{lem:TruncatedNormDiamond},
\begin{align*}
\|\Diamond_{p,p'}\|_2^2 & \leq \sum_{x\in \Z^d} \|x\|_2^2 \| \Diamond_{p,p'}\|_{\infty,\langle x\rangle}
\\ & \leq K_\infty^\Diamond \sum_{\substack{ x\in \Z^d,\\  \langle x\rangle^2\leq \X(p)}} \langle x\rangle^{10-2d}
+ K_\infty^\Diamond \X(p)^2
\hspace{-1em} \sum_{\substack{ x\in \Z^d,\\  \langle x\rangle^2> \X(p)}} \hspace{-1em}
\langle x\rangle^{6-2d}
+ K_\infty^\Diamond
\hspace{-1em}\sum_{\substack{ x\in \Z^d,\\  \langle x\rangle^2 >\X(p)}} \hspace{-1em}
\|\tau_p\|_{\infty,\langle x\rangle/3} \langle x\rangle^{8-d}  .
\end{align*}
By basic analysis, the first sum has order $\psi_{d-2}(\X(p))$, and the second sum has order $\X(p)^{(6-d)/2}\leq \psi_{d-2}(\X(p))/\X(p)$.
For the last sum, we further split according to $\|x\|_2^2\geq \X(p)^{1+\eps/3}$ or not.
By Theorem \ref{thm:ExponentialTail}, the contribution of the case $\|x\|_2^2\geq \X(p)^{1+\eps/3}$ is at most $1$ if $\X(p)$ is large enough.
For the case $\X(p)<\|x\|_2^2< \X(p)^{1+\eps/3}$,
we first use \eqref{eq:initialization}
to upper-bound $\|\tau_p\|_{\infty,\langle x\rangle/3}$ by $C' \langle x\rangle^{2-d}$,
and then, by basic analysis, the contribution is bounded by
$\psi_{d-2}(\X(p)^{1+\eps/3})\leq \psi_{d-2}(\X(p))\X(p)^{\eps}$.
The bound for $\square_{p,p'}$ is obtained in exactly the same way, using Lemma \ref{lem:TruncatedNormSquare} instead.

The computations for $\twosquare_{p,p'}$ are also essentially the same using Lemma \ref{lem:TruncatedNormTwoSquare}, with the exponents for $\langle x\rangle$ reduced by $1$. We omit the details.
\end{proof}

\begin{proof}[Proof of Proposition \ref{pro:L2Error}]
In this proof, $K,K'$ denote constants depending on $d,L,C$ that may change from line to line. It will be crucial at a late step of the proof that we may decrease $c$ without increasing those constants.
By Theorems \ref{thm:LowerBoundTau} and \ref{thm:UpperBoundTau}, we have
\[
\|\tau_{p'}-\tau_{p,p'}^\square\|_2^2
\leq \frac{(p'-p)^2}{(1-p)^2}
\left ( \| \bar \tau_{p,p'}^\square\ast  \tau_p\ast (\square_{p,p'}+\Diamond_{p,p'})
\ast \tau_p \ast \bar \tau_{p,p'}^\square\|_2^2
+\|\bar \tau^\square_{p,p'}\ast  \tau_p \ast  \twosquare_{p,p'}
\ast \tau_p \ast \bar \tau_{p,p'}\|_2^2 \right ).
\]
Since, for every pair of functions $f,g$ satisfying $f(x)=f(-x)$ and $g(x)=g(-x)$ on $\Z^d$, we have $\|f\ast g\|_2^2=\|f\|_2^2 \|g\|_1+ \|f\|_1 \|g\|_2^2$, it suffices to upper-bound each of the following terms separately:
\begin{align*}
& \text{(1)} \quad \|\bar \tau^\square_{p,p'}\|_2^2 \|\bar \tau^\square_{p,p'}\|_1 \|\tau_p\|_1^2(\|\square_{p,p'}\|_1+\|\Diamond_{p,p'}\|_1)
 & &\text{(2)} \quad \|\bar \tau^\square_{p,p'}\|_2^2 \|\bar \tau_{p,p'}\|_1 \|\tau_p\|_1^2 \|\twosquare_{p,p'}\|_1
\\ &\text{(3)} \quad  \|\bar \tau^\square_{p,p'}\|_1 \|\bar \tau_{p,p'}\|_2^2 \|\tau_p\|_1^2 \|\twosquare_{p,p'}\|_1
 & &\text{(4)} \quad \|\bar \tau^\square_{p,p'}\|^2_1 \|\tau_p\|_2^2\|\tau_p\|_1(\|\square_{p,p'}\|_1+\|\Diamond_{p,p'}\|_1)
\\ & \text{(5)} \quad \|\bar \tau^\square_{p,p'}\|_1\|\bar \tau_{p,p'} \|_1\|\tau_p\|^2_2 \|\tau_p\|_1 \|\twosquare_{p,p'}\|_1
& &\text{(6)} \quad \|\bar \tau^\square_{p,p'}\|_1^2 \|\tau_p\|_1^2 (\|\square_{p,p'}\|_2^2+\|\Diamond_{p,p'}\|_2^2)
\\
& \text{(7)} \quad  \|\bar \tau^\square_{p,p'}\|_1 \|\bar \tau_{p,p'}\|_1 \|\tau_p\|_1^2 \|\twosquare_{p,p'}\|_2^2.
\end{align*}
Recall that $\bar \tau_{p,p'}=\1_{\{0\}}+(p'-p)/(1-p)D\ast \tau_{p'}$, and  $\bar \tau^\square_{p,p'}=\1_{\{0\}}+(p'-p)/(1-p)D\ast \tau^\square_{p,p'}$, so by Lemma \ref{lem:p'-p},
\begin{equation}
\|\bar \tau_{p,p'}\|_2^2 \leq 1+K\frac{\|\tau_{p'}\|_2^2}{\X(p)} \quad \text{ and }  \quad   \|\bar \tau^\square_{p,p'}\|_2^2 \leq 1+K\frac{\|\tau^\square_{p,p'}\|_2^2}{\X(p)} \label{eq:2normBar}
\end{equation}
Also note that by Theorem \ref{thm:PointwiseError}, if $c$ is small enough then $\X(p')/\|\tau^\square_{p,p'}\|_1\in (1/2,2)$.

For (1) and (2), by Lemma \ref{lem:p'-p} and assumptions (a) and (b), we have
\[
\max(\|\bar \tau^\square_{p,p'}\|_1,  \|\bar \tau_{p,p'}\|_1)\leq 1+K \X(p')/\X(p)\leq  1+K'\X(p)/\psi_d(\X(p)),
\]
and then, plugging in the bounds of \eqref{eq:2normBar}, Lemmas \ref{lem:TruncatedNormSquare}, \ref{lem:TruncatedNormDiamond} and \ref{lem:TruncatedNormTwoSquare}, and $\|\tau_p\|_1=\X(p)$, we get
 \[
 \max((1),(2)) \leq K \|\tau^\square_{p,p'}\|_2^2 \X(p')\psi_d(\X(p))+K \X(p')^{3+\eps}\psi_d(\X(p)).
 \]

For (3), substituting the bounds from Lemmas \ref{lem:TruncatedNormTwoSquare} and \ref{lem:tail2ndMoment}, \eqref{eq:2normBar}, and (b) directly gives
\[
(3) \leq K \X(p')^{3+\eps} \psi_d(\X(p)).
\]
We can deal with (4) and (5) similarly. For (6) and (7), it suffices to substitute the bounds from Lemma \ref{lem:L2NormBasicDiagrams} and (a) and (b). We omit the details.

Combining these bounds and using Lemma \ref{lem:p'-p}, we get
\begin{equation}
\|\tau_{p'}-\tau_{p,p'}^\square\|_2^2 \leq K \|\tau^\square_{p,p'}\|_2^2\frac{\X(p')\psi_d(\X(p))}{\X(p)^2} + K' \frac{\X(p')^{3+\eps}\psi_d(\X(p))}{\X(p)^2}. \label{eq:L_2_Incomplete}
\end{equation}
Taking $c$ small enough that $cK\leq 1/2$, we have by assumption (a) that the first term on the right-hand side of \eqref{eq:L_2_Incomplete} is at most $\|\tau_{p,p'}^\square\|_2^2/2$. So, by the triangle inequality,
\[
\|\tau_{p,p'}^\square\|_2^2/2 \leq \|\tau_{p'}\|_2^2 + K' \frac{\X(p')^{3+\eps}\psi_d(\X(p))}{\X(p)^2}.
\]
By Lemma \ref{lem:tail2ndMoment}, we have $\|\tau_{p'}\|_2^2\leq \X(p')^{2+\eps}$. So by (a),
\[
\|\tau_{p,p'}^\square\|_2^2\leq K \X(p')^{2+\eps}.
\]
The desired result follows by substituting the corresponding bound on $\|\tau_{p,p'}^\square\|_2^2$ into \eqref{eq:L_2_Incomplete}.
\end{proof}

\section{Sharp asymptotics under \eqref{eq:etaUpper}} \label{Sec:SharpAsymptotics}
The goal of this section is to prove Theorems \ref{thm:susceptibility} and \ref{thm:gyration} under the following weaker assumption. There exists $B>0$ such that
\begin{equation}
\forall x\in \Z^d, \quad \tau_{p_c}(x) \leq B \langle x\rangle^{2-d}. \label{eq:etaUpper}
\tag{B}
\end{equation}
Those two theorems will be used to prove Theorems \ref{thm:Main} and \ref{thm:tau_local_CLT}. We also show technical results under condition \eqref{eq:initialization}, to reuse them in Section \ref{sec:TCL}. 
Subsection~\ref{sec:p_c-p_c} bounds $|\tilde p_c(p)-p_c|$. Subsection~\ref{sec:Chi} then proves Theorem~\ref{thm:susceptibility}.
In Subsection~\ref{sec:SecondMoment}, we prove several second-moment estimates, which will then be used in Subsection~\ref{sec:Gyr} to prove Theorem~\ref{thm:gyration}.
\subsection{Asymptotic for $\tilde p_c(p)-p_c$} \label{sec:p_c-p_c}
The goal of this section is to prove the following result:
\begin{theorem} \label{thm:p_c-p_c}
If \eqref{eq:etaUpper} holds, then we have as $p\uparrow p_c$,
\[
|\tilde p_c(p)-p_c|=O(\psi_d(\X(p))/\X(p)^2).
\]
\end{theorem}
The main idea of the proof is to use Theorem \ref{thm:PointwiseError}, together with the facts that $\X(p)\approx (p_c-p)^{-1}$ as $p\uparrow p_c$, and that for $p'$ large enough we have $\|\tau_{p,p'}^\square\|_1=\Theta(1/(\tilde p_c(p)-p))$.
For the first approximation, we recall the following well-known result:
\begin{lemma} \label{lem:X'}
Let $C>0$. There exist $F,C_1,C_2$ depending on $d,L,C$ such that for every $p<p'\leq p_c$ with $\X(p)\geq F$, if \eqref{eq:initialization} holds at $p'$ then
\[
C_1(p'-p) \leq \frac{1}{\X(p)}- \frac{1}{\X(p')}\leq C_2(p'-p).
\]
The upper-bound still hold without \eqref{eq:initialization}.
\end{lemma}
\begin{proof}
By Lemma \ref{lem:TwoTriangleSquare} (b), the triangle condition is bounded in $[p,p']$, so the result follows from (1.4) and Theorem 1.1 in \cite{HutchcroftTriangle}.
\end{proof}
\begin{remark}
It is possible to recover $C_1(p'-p)\leq 1/\X(p)$ by using Lemma \ref{lem:p'-p} inductively.
\end{remark}
The proof of the second approximation relies on Section \ref{sec:Rewriting}. Beforehand, let us deal with a technical issue arising when using
\eqref{eq:RewriteB_TauSquare}. Because the following sum appears in \eqref{eq:RewriteB_TauSquare}, it will frequently appear in our computations. For $\vec A\in \Clive$, let
\begin{equation}
\Gamma_p(\vec A):=\frac{1}{h_p(\vec A)}
\sum_{\vec B\in \vec \Comp(0)}\mu_p(\vec B)\1_{(\vec A,\vec B)\notin \Ccut} .
\label{eq:def_Gamma_p}
\end{equation}
We will use $\Gamma_p$ to deal with the last step, which is not handled by the Doob $h_p$-transform.
\begin{lemma} \label{lem:Gamma}
For every $p<p_c$ we have $\|\Gamma_p\|_\infty \leq \Omega^2\X(p)^2(1-p)^{-\Omega}/\|h_p\|_\infty$.
\end{lemma}
\begin{proof}
By dropping the indicator, and using Lemma \ref{lem:vpd} and \eqref{eq:OrientedX}, we have
\[
\|\Gamma_p\|_\infty
\leq \frac{1}{\inf_{\vec A\in \Clive}h_p(\vec A)}
\sum_{\vec B\in \vec \Comp(0)} \mu_p(\vec B)
\leq  \Omega^2\X(p)^2(1-p)^{-\Omega}/\|h_p\|_\infty . \qedhere
\]
\end{proof}
The following probability measure will frequently appear in this section, since we will make many approximations using the mixing time of the Markov chain $(\vA_i)_{i\geq 0}$ under $\proba_p^h$. Let 
\begin{equation}
\proba_p^\vpg(\cdot)=\sum_{\vec A_0\in \Clive}\vpg_p(\vec A_0)\proba_p^h(\cdot |\vA_0=\vec A_0),
\label{eq:Proba^vpg}
\end{equation}
and let $\E_p^\vpg$ denote the associated expectation.

We can now focus on estimating $\|\tau_{p,p'}^\square\|_1$. Let $E_{p,0}:=\X(p)$, and for $k\geq 1$, let
\[
E_{p,k}:=\X(p)\lambda_p^{-k}\proba_p^\otimes(\Tcut>k)=\frac{ \Plive_p}{\Omega \lambda_p}  \E^h_p \left[h_p(\vec \A_0) \Gamma_p(\vA_{k-1}) \right ],
\]
so that, by \eqref{eq:RewriteB_TauSquare} and the definition of $\lambda_p$, we have for $k\geq 1$,
\begin{equation}
\|\tau_p^{\square,k}\|_1 =(\Omega \X(p)\lambda_p)^k E_{p,k}=\left (\frac{1-p}{\tilde p_c(p)-p}\right)^k E_{p,k} .
\label{eq:tau_E_p,infty}
\end{equation}
Also, by Theorem \ref{thm:MixingTime} and Lemma \ref{lem:vpd}, $E_{p,k}$ converges exponentially fast toward
\[
E_{p,\infty} =(\Plive_p/\Omega \lambda_p)\E_p^h[h_p(\vA_0)]\E_p^\vpg[\Gamma_p(\vA_0)] .
\]
Thus, it is natural for every $p\leq p'<\tilde p_c(p)$ to rewrite $\|\tau_{p,p'}^\square\|_1$ as
\begin{align}
\|\tau_{p,p'}^\square\|_1 & =\sum_{k=0}^\infty \left (\frac{p'-p}{1-p} \right )^k \|\tau_p^{\square,k}\|_1=\sum_{k=0}^\infty \left (\frac{p'-p}{\tilde p_c(p)-p } \right )^kE_{p,k} \label{eq:Tau_X_Diff_prelim}
\\ & =\sum_{k=0}^\infty \left (\frac{p'-p}{\tilde p_c(p)-p } \right )^k(E_{p,k}-E_{p,\infty}) +\frac{\tilde p_c(p)-p}{\tilde p_c(p)-p'}E_{p,\infty}. \label{eq:Tau_X_Diff}
\end{align}
We first deal with the indices $k$ beyond the mixing time for $(\vA_i)_{i\geq 0}$ under $\proba_p^h$.
\begin{lemma}\label{eq:L1Diff_After_Mixing}
Let $C>0$. There exist $F>0$ and $0<c<1$ depending on $d,L,C$ such that for every $p<p_c$ satisfying \eqref{eq:initialization} and $\X(p)\geq F$, every $N\geq 0$, we have
\[
\sum_{k\geq N} |E_{p,k}-E_{p,\infty}|  \leq \X(p)^4c^N.
\]
In particular, if $K_1$ is large enough, depending on $d,L,C$, we have 
\[
\sum_{k\geq K_1\log(\X(p))} |E_{p,k}-E_{p,\infty}|  \leq 1/\X(p)^4.
\]
\end{lemma}
\begin{proof}
We have for $k\geq 1$,
\[
E_{p,k}-E_{p,\infty}=\frac{\Plive_p}{\Omega \lambda_p} \sum_{\vec A_0, \vec A\in \Clive} h_p(\vec A_0)\proba_p^h(\vA_0=\vec A_0)\Gamma_p(\vec A) \left (\P^{k-1}_p(\vec A|\vec A_0)-\vpg_p(\vec A)\right )  .
\]
Thus, by Lemma \ref{lem:Gamma} and Theorem \ref{thm:LowerLambda}, if $\X(p)$ and $K$ are large enough depending on $d,L,C$, then
\begin{align*}
|E_{p,k}-E_{p,\infty}|
& \leq K\|h_p\|_\infty\|\Gamma_p\|_\infty
\max_{\vec A_0\in \Clive }
\left (
\sum_{\vec A\in \Clive}
\left |\P^{k-1}_p(\vec A|\vec A_0)-\vpg_p(\vec A)  \right |
\right )
\\ & \leq K\X(p)^2
\max_{\vec A_0\in \Clive }
\left (
\sum_{\vec A\in \Clive}
\left |\P^{k-1}_p(\vec A|\vec A_0)-\vpg_p(\vec A)  \right |
\right ).
\end{align*}
The desired result then follows by summing the bound in Theorem \ref{thm:MixingTime} over $k\geq N$.
\end{proof}
We then have the following analogue of Lemma \ref{lem:X'}:
\begin{proposition} \label{pro:p_c(p)-p_2}
Let $C>0$. There exist $c,F,F',K_2$ depending on $d,L,C$ such that for every $p<p_2<p_c$ such that $\X(p)\geq F$, \eqref{eq:initialization} is satisfied at $p_2$, and
\[
F' \log(\X(p)) \X(p) \leq \X(p_2) \leq c\X(p)^2 /\psi_d(\X(p)),
\]
we have
\[
0\leq \tilde p_c(p)-p_2\leq K_2/\X(p_2).
\]
\end{proposition}
\begin{proof}
Let $\ell(p)=\lceil K_1\log(\X(p)) \rceil$.
First, by Theorem \ref{thm:PointwiseError}, if $F$ is large enough, we have $\|\tau^\square_{p,p_2}\|_1<\infty$, so $p_2<\tilde p_c(p)$.
Let $p_3$ be such that $\X(p_3)=\X(p_2)/2$.
By \eqref{eq:Tau_X_Diff_prelim}, $\|\tau^\square_{p,p_2}\|_1$ equals
\[
\sum_{k <\ell(p) } \left (\frac{p_2-p}{\tilde p_c(p)-p } \right )^k E_{p,k}
+ \sum_{k\geq \ell(p)} \left (\frac{p_2-p}{\tilde p_c(p)-p}\right )^k (E_{p,k}-E_{p,\infty})
+ \left (\frac{p_2-p}{\tilde p_c(p)-p } \right )^{\ell(p)}
\frac{\tilde p_c(p)-p}{\tilde p_c(p)-p_2}E_{p,\infty},
\]
and the same equality holds at $p_3$. By Lemma \ref{eq:L1Diff_After_Mixing}, the second sum is bounded above and below by $\pm 1/\X(p)^4$. Then, by comparing this identity with the one obtained at $p_3<p_2$, we get
\begin{equation}
\|\tau^\square_{p,p_2}\|_1 \leq \frac{1}{\X(p)^4}+ \left (\frac{p_2-p}{p_3-p} \right )^{\ell(p)}  \frac{\tilde p_c(p)-p_3}{\tilde p_c(p)-p_2} \left (\|\tau_{p,p_3}^\square \|_1+\frac{1}{\X(p)^4} \right ) . \label{eq:low_k_do_not_increase}
\end{equation}
Then, if $\X(p)$ is large enough and $c$ is small enough, Theorem \ref{thm:PointwiseError} gives, for some constant $K'$,
\[
\|\tau_{p_2}-\tau_{p,p_2}^\square\|_1 \leq K' \frac{\X(p_2)^2\psi_{d}(\X(p))}{\X(p)^2} \leq \X(p_2)/9,
\]
and similarly at $p_3$. Hence, by \eqref{eq:low_k_do_not_increase}, we have
\begin{equation}
\X(p_2) \leq \frac{4}{3} \left (\frac{p_2-p}{p_3-p} \right )^{\ell(p)}  \frac{\tilde p_c(p)-p_3}{\tilde p_c(p)-p_2} \X(p_3). \label{eq:low_k_do_not_increase_bis}
\end{equation}

Next, using the general inequality $1+x\leq e^x$, then Lemma \ref{lem:X'} with $\X(p_3)=\X(p_2)/2$ and $\X(p_2)\geq F'\log(\X(p))\X(p)$, we get
\[
\left (\frac{p_2-p}{p_3-p} \right )^{\ell(p)} \leq e^{\ell(p) (p_2-p_3)/(p_3-p)} \leq (5/4),
\]
provided $F'$ is large enough. We thus deduce from \eqref{eq:low_k_do_not_increase_bis} and $\X(p_2)=2\X(p_3)$ that we must have
\[
\frac{\tilde p_c(p)-p_3}{\tilde p_c(p)-p_2} \geq 6/5,
\]
or equivalently,
\[
p_2-p_3 \geq (\tilde p_c(p)-p_2)/5,
\]
and the desired result follows from Lemma \ref{lem:X'}.
\end{proof}
\begin{proof}[Proof of Theorem \ref{thm:p_c-p_c}]
We keep the notation of the previous proposition.
Take $p_2=p_2(p)$ such that $\X(p_2)=c\X(p)^2/\psi_d(\X(p))$.
Proposition \ref{pro:p_c(p)-p_2} gives $|\tilde p_c(p)-p_2|\leq K_2/\X(p_2)$.
By Lemma \ref{lem:X'}, we have $|p_c-p_2|\leq 1/(C_1\X(p_2))$.
Thus,
\[
\tilde p_c(p)-p_c=O(1/\X(p_2))=O(\psi_d(\X(p))/\X(p)^2).\qedhere
\]
\end{proof}
We finish with some other useful consequences of Proposition \ref{pro:p_c(p)-p_2}.
\begin{lemma} \label{lem:Small_k_L1}
Let $C>0$.
Let $K_1>0$ as in Lemma \ref{eq:L1Diff_After_Mixing}.
There exist $F,F',K_3$ depending on $d,L,C$ such that for every $p<p_c$ such that $\X(p)\geq F$, if \eqref{eq:initialization} is satisfied at $p_2$ with $\X(p_2)= F' \log(\X(p)) \X(p)$, then
\[
\sum_{k\leq K_1 \log(\X(p))} E_{p,k} \leq K_3 \log(\X(p)) \X(p).
\]
\end{lemma}
\begin{proof}
Let $\ell(p):=\lfloor K_1\log(\X(p))\rfloor+1$.
By \eqref{eq:Tau_X_Diff_prelim} and $p\leq p_2\leq \tilde p_c(p)$, we have
\[
\|\tau_{p,p_2}^\square\|_1\geq \left (\frac{p_2-p}{\tilde p_c(p)-p} \right)^{\ell(p)}  \sum_{k<\ell(p)} E_{p,k}.
\]
Using that $y/(x+y)\geq e^{-x/y}$ for $x,y\geq 0$, we deduce
\[
\|\tau_{p,p_2}^\square\|_1\geq e^{-\ell(p)(\tilde p_c(p)-p_2) /(p_2-p)} \sum_{k<\ell(p)} E_{p,k}.
\]
By Proposition \ref{pro:p_c(p)-p_2}, Lemma \ref{lem:X'}, and $\X(p_2)=F'\log(\X(p))\X(p)$, the exponential is lower-bounded by a constant $c>0$ depending on $d,L,C$.
By Theorem \ref{thm:PointwiseError}, we have $\|\tau_{p,p_2}^\square\|_1\leq 2 \X(p_2)$.
Thus,
\[
\sum_{k<\ell(p)} E_{p,k} \leq (2/c)\X(p_2),
\]
and the desired result follows.
\end{proof}
\begin{lemma}\label{lem:Large_k_L1}
In the setting of Lemma \ref{lem:Small_k_L1}, there exists $K_4$ such that if $F'$ is large enough, under the same assumptions, we have
\[
E_{p,\infty} \leq K_4 \X(p).
\]
\end{lemma}
\begin{proof}
By \eqref{eq:Tau_X_Diff_prelim} and Lemma \ref{eq:L1Diff_After_Mixing}, we have
\[
\|\tau_{p,p_2}^\square\|_1 \geq \left (\frac{p_2-p}{\tilde p_c(p)-p} \right )^{\ell(p)} \frac{\tilde p_c(p)-p}{\tilde p_c(p)-p_2}E_{p,\infty}-1/\X(p)^4.
\]
As in the previous proof, we have $\|\tau_{p,p_2}^\square\|_1\leq 2 \X(p_2)$, and $((p_2-p)/(\tilde p_c(p)-p))^{\ell(p)}\geq c$ for some constant $c>0$ depending on $d,L,C$. By Lemma \ref{lem:X'} and Proposition \ref{pro:p_c(p)-p_2},
\[
\frac{\tilde p_c(p)-p}{\tilde p_c(p)-p_2} \geq \frac{p_2-p}{\tilde p_c(p)-p_2}\geq \frac{1/(2C_2\X(p))}{K_2/\X(p_2)}=\frac{1}{2K_2C_2}\frac{\X(p_2)}{\X(p)}.
\]
Hence,
\[
E_{p,\infty} \leq \frac{2\X(p_2)+1/\X(p)^4}{c/2} 2K_2 C_2 \frac{\X(p)}{\X(p_2)} \leq 9K_2 C_2 \X(p). \qedhere
\]
\end{proof}

\subsection{Asymptotic for $\X(p)$: proof of Theorem \ref{thm:susceptibility}} \label{sec:Chi}
We assume throughout the section that \eqref{eq:etaUpper} holds.
The main idea is to approximate $\X(p')$ by $\|\tau_{p,p'}^\square\|_1$ using Theorem \ref{thm:PointwiseError} and \eqref{eq:Tau_X_Diff}.
The main issue is that Theorem \ref{thm:PointwiseError} is more accurate when $\X(p')$ is close to $\X(p)$, while the smaller $\X(p')$ is, the less precisely we can estimate $\|\tau_{p,p'}^\square\|_1$.
This trade-off would usually call for taking an optimal $p'$, but we can actually be even more precise by performing some well-chosen operations on \eqref{eq:Tau_X_Diff}.
Our first step is to replace $\tilde p_c(p)$ by $p_c$ in \eqref{eq:Tau_X_Diff}.
Throughout the section we let $K_1$ be as in Lemma \ref{eq:L1Diff_After_Mixing}.
We also let $\ell(p):= K_1 \log(\X(p))$.
\begin{lemma} \label{lem:replace_pc}
As $p,p'\uparrow p_c$ with $\X(p)\log(\X(p))\ll \X(p') \ll \X(p)^2/\psi_d(\X(p))$, we have
\[
\|\tau_{p,p'}^\square \|_1 =\sum_{0\leq k<\ell(p)} \left (\frac{p'-p}{p_c-p } \right )^k (E_{p,k}-E_{p,\infty}) + \frac{p_c-p}{p_c-p'} E_{p,\infty} +O(\psi_d(\X(p))\X(p')^2/\X(p)^2).
\]
\end{lemma}
\begin{proof}
Let $S$ denote the above sum.
First, by Theorem \ref{thm:PointwiseError}, if $\X(p)$ is large enough, we have $\|\tau_{p,p'}^\square\|_1<\infty$, so $p'<\tilde p_c(p)$. Recall that, by \eqref{eq:Tau_X_Diff}, we have
\[ \|\tau_{p,p'}^\square\|_1= \sum_{k=0}^\infty \left (\frac{p'-p}{\tilde p_c(p)-p } \right )^k(E_{p,k}-E_{p,\infty}) +\frac{\tilde p_c(p)-p}{\tilde p_c(p)-p'}E_{p,\infty}. \]
So $|S-\|\tau_{p,p'}^\square\|_1|\leq S_1+S_2+S_3$, where $S_1, S_2, S_3$ are defined below. First, using Lemma \ref{eq:L1Diff_After_Mixing} and $p<p'<\tilde p_c(p)$, we obtain
\[ S_1 := \sum_{k\geq \ell(p)} \left (\frac{p'-p}{p_c-p } \right )^k |E_{p,k}-E_{p,\infty}|\leq 1/\X(p)^4. \]
Then, using the fact that $|(a/b)^k-(a/c)^k|\leq k|b-c|/b$ for every $0<a<b,c$ and $k\geq 0$, we get
\[ S_2:= \sum_{0\leq k<\ell(p)} \left |\left (\frac{p'-p}{p_c-p } \right )^k-\left (\frac{p'-p}{\tilde p_c(p)-p}  \right )^k \right | |E_{p,k}-E_{p,\infty}|\leq \ell(p)\frac{|p_c-\tilde p_c(p)|}{p_c-p} \sum_{0\leq k <\ell(p)}|E_{p,k}-E_{p,\infty}|.\]
Thus, using Theorem \ref{thm:p_c-p_c}, Lemmas \ref{lem:X'}, \ref{lem:Small_k_L1}, \ref{lem:Large_k_L1}, and $\X(p)\log(\X(p))\ll \X(p')$, we get
\[ S_2=O(\ell(p)^2\psi_d(\X(p)))=O(\psi_d(\X(p))\X(p')^2/\X(p)^2).\]
Finally, by Theorem \ref{thm:p_c-p_c}, Lemma \ref{lem:X'}, Lemma \ref{lem:Large_k_L1}, and Proposition \ref{pro:p_c(p)-p_2}, we get
\[ S_3:= E_{p,\infty}\left |\frac{p_c-p}{p_c-p'}-\frac{\tilde p_c(p)-p}{\tilde p_c(p)-p'} \right |=O(\X(p))\frac{(p'-p)|p_c-\tilde p_c(p)|}{(\tilde p_c(p)-p')(p_c-p')} =O(\psi_d(\X(p))\X(p')^2/\X(p)^2).\]
The desired result follows by summing the bounds on $S_1,S_2,S_3$.
\end{proof}
From now on, we fix $0<\eps <1/4$ and take $\X(p')$ such that $\X(p)^{1+\eps/2} \leq \X(p')\leq \X(p)^{1+\eps}$. Fix $N$ large enough depending on $\eps$. By Taylor's theorem, as $p\uparrow p_c$, uniformly over $k\leq \ell(p)$, we have
\[
\left (\frac{p'-p}{p_c-p } \right )^k= \left (1- \frac{p_c-p'}{p_c-p } \right )^k=\sum_{i=0}^N \binom{k}{i} \left ( -\frac{p_c-p'}{p_c-p} \right )^i +O(\X(p)^{-2}).
\]
Thus, summing and using Lemmas \ref{lem:Small_k_L1}, \ref{lem:Large_k_L1}, and \ref{lem:replace_pc}, we have as $p\uparrow p_c$,
\begin{equation}
\|\tau_{p,p'}^\square \|_1= \sum_{0\leq k <\ell(p)}\sum_{i=0}^N \binom{k}{i} \left (-\frac{p_c-p'}{p_c-p} \right )^i (E_{p,k}-E_{p,\infty}) + \frac{p_c-p}{p_c-p'} E_{p,\infty} +O(\psi_d(\X(p))\X(p)^{2\eps}). \label{eq:Approx_Poly_X}
\end{equation}
By Theorem \ref{thm:PointwiseError}, we also have $\X(p')= \|\tau_{p,p'}^\square \|_1+O(\psi_d(\X(p))\X(p)^{2\eps})$.

The main idea is then to use the following operators to cancel every term in \eqref{eq:Approx_Poly_X}.
First, let $\Phi :x\in [0,p_c]\mapsto \X(p_c-x)$.
Fix an arbitrary $\alpha>1$.
For every $\lambda \in \R$, let $P_\lambda$ be the operator such that for every $f:[0,p_c]\mapsto \R$, we have
\[
P_\lambda(f):x\mapsto f(x)-\alpha^\lambda f(x/\alpha).
\]
Note that if $f$ is of the form $f:p\mapsto a x^{\lambda'}$, then $P_\lambda(f)$ is also of the form $P_\lambda (f)\mapsto a' x^{\lambda'}$; moreover, if $\lambda=\lambda'$, then $P_\lambda(f)=0$.
Thus, \eqref{eq:Approx_Poly_X} yields the following lemma.
\begin{lemma} \label{lem:Filled_Operator_X}
For every $\eps>0$, if $N$ is large enough, we have as $x\to 0$,
\[
[P_{-1}P_0P_1\dots P_N]\Phi(x)=O(\psi_d(1/x)x^{-2\eps}).
\]
\end{lemma}
\begin{proof}
For $p<p_c$, let
\[
\Phi_p: x\in (0,p_c]\mapsto \sum_{i=0}^N (-x)^i \sum_{0\leq k <\ell(p)} \binom{k}{i}\frac{E_{p,k}-E_{p,\infty}}{(p_c-p)^i} +\frac{p_c-p}{x}E_{p,\infty}.
\]
Recall that $\X(p)\leq \X(p')$, and by Lemma \ref{lem:X'}, $\X(p')=\Theta(1/(p_c-p'))$. Thus, by \eqref{eq:Approx_Poly_X} and Theorem \ref{thm:PointwiseError}, writing $x=p_c-p'$, we have
\begin{equation}
\Phi(x)= \Phi_p(x) + O(\psi_d(1/x)x^{-2\eps}). \label{eq:Gamma-->Gamma_p}
\end{equation}
The formula holds as long as $\X(p)^{1+\eps/2}\leq \X(p')\leq \X(p)^{1+\eps}$.
In other words, by Lemma \ref{lem:X'}, the formula holds as long as $1/(\X(p)^{1+\eps}C_1) \leq x\leq 1/(\X(p)^{1+\eps/2}C_2)$.
Thus, for every $x$ small enough, there exists $p=p(x)<p_c$ such that \eqref{eq:Gamma-->Gamma_p} holds for every $y\in [x/\alpha^N,\alpha x]$.
Hence,
\[
[P_{-1}P_0P_1\dots P_N]\Phi(x)= [P_{-1}P_0P_1\dots P_N]\Phi_{p(x)}(x)+O(\psi_d(1/x)x^{-2\eps})=O(\psi_d(1/x)x^{-2\eps}),
\]
using for the last equality the fact that $\Phi_{p(x)}$ is of the form $y\mapsto \sum_{i=-1}^N a_i y^i$.
\end{proof}
The rest of the proof is basic analysis to invert the operator $P_{-1}P_0\dots P_N$.
\begin{lemma}\label{lem:downward_PolyPerateur1}
Let $\theta<\lambda$. If $P_\lambda f(x)=O(x^{\theta})$ as $x\to 0$, and $f$ is bounded on $(p_c/\alpha,p_c]$, then $f=O(x^{\theta})$ as $x\to 0$.
\end{lemma}
\begin{proof}
By taking $n$ such that $\alpha^n x\in (p_c/\alpha,p_c]$, we have
\begin{align*}
f(x) & =\sum_{i=0}^{n-1} \left ( \alpha^{-i\lambda}f(x\alpha^i)- \alpha^{-(i+1)\lambda} f(x\alpha^{i+1}) \right ) + \alpha^{-n\lambda}f(x\alpha^n)
\\ & =\sum_{i=0}^{n-1} \alpha^{-(i+1)\lambda}P_\lambda f(x\alpha^{i+1})+\alpha^{-n\lambda}f(x\alpha^n).
\end{align*}
Since $f$ is bounded on $(p_c/\alpha,p_c]$, we have $ \alpha^{-n\lambda}f(x\alpha^n)=O(x^{\lambda})$. So
\[
f(x) =O\left ( \sum_{i=0}^\infty  \alpha^{-i\lambda} (x\alpha^i)^{\theta} \right ) +O(x^{\lambda})=O(x^{\theta}). \qedhere
\]
\end{proof}
\begin{lemma} \label{lem:downward_PolyPerateur2}
Let $\lambda<\theta$. Assume that $f$ is monotone, and that for every $\alpha>1$, we have $P_\lambda f(x)=O(x^{\theta})$ as $x\to 0$. Then there exists $l\in \R$ such that $f(x)=l x^\lambda +O(x^{\theta})$ as $x\to 0$.
\end{lemma}
\begin{proof}
For every $n\in \N$, we have
\begin{equation}
f(x)=\sum_{i=0}^{n-1}
\left ( \alpha^{i\lambda}f\left (\frac{x}{\alpha^i} \right)
- \alpha^{(i+1)\lambda} f\left (\frac{x}{\alpha^{i+1}}\right ) \right )
+ \alpha^{n\lambda}f\left (\frac{x}{\alpha^n}\right)
=\sum_{i=0}^{n-1} \alpha^{i\lambda}P_\lambda f\left (\frac{x}{\alpha^i}\right )
+\alpha^{n\lambda}f\left (\frac{x}{\alpha^n}\right). \label{eq:Telescopage}
\end{equation}
Using $\lambda<\theta$ and $P_\lambda f(x)=O(x^{\theta})$, we get that the last sum converges as $n\to \infty$.
Thus, as $n\to \infty$, $\alpha^{n\lambda}f(x/\alpha^n)$ converges.
Since, by assumption, $\alpha>1$ is arbitrary, and $f$ is monotone, we deduce that $f(x)/x^\lambda$ converges toward some limit $l\in \R$.
By taking $n\to \infty$ in \eqref{eq:Telescopage}, we deduce
\[
f(x)=O\left ( \sum_{i=0}^\infty  \alpha^{i\lambda} \left (\frac{x}{\alpha^i}\right )^{\theta} \right ) +lx^\lambda =lx^\lambda+O(x^{\theta}). \qedhere
\]
\end{proof}
\begin{proof}[Proof of Theorem \ref{thm:susceptibility}]
Let $-1<\theta<-\1_{d=7}/2-2\eps$.
By Lemma \ref{lem:Filled_Operator_X} we have $P_{-1}\dots P_N \Phi(x)=O(x^{\theta})$.
A downward induction on $i\in [-1,N]$ shows that $P_{-1}\dots P_i \Phi(x)=O(x^{\theta})$ as $x\to 0$.
Indeed, if the result holds at $i\geq 0$, then by Lemma \ref{lem:downward_PolyPerateur1}, it holds at $i-1$.
We deduce that $P_{-1}\Phi(x)=O(x^{\theta})$.
Thus, by Lemma \ref{lem:downward_PolyPerateur2} we get that there exists $l\in \R$ such that $\Phi(x)=l/x+O(x^{\theta})$.
By Lemma \ref{lem:X'}, we must have $l>0$.
Since $0<\eps<1/3$ and $-1<\theta<-\1_{d=7}/2-2\eps$ are arbitrary, this concludes the proof.
\end{proof}
\subsection{Second moment computations} \label{sec:SecondMoment}
The goal of this subsection and the next one is to prove the following theorem, from which we will deduce Theorem \ref{thm:gyration}. Along the way, we will derive several second moment results that will be useful for the proof of the local central limit theorem in Section \ref{sec:LowFrequencies}.
\begin{theorem} \label{thm:Gyration_NotRenormalized}
Let $d>6$ and $L\geq 1$ be such that \eqref{eq:etaUpper} holds. There exists $C'_{\xi}$ such that for every $\eps>0$, as $p\uparrow p_c$, we have
\[
\|\tau_p\|_2^2 = C'_\xi(p_c-p)^{-2}+O((p_c-p)^{-1-\eps}\psi_d(\X(p))).
\]
\end{theorem}
We proceed as for Theorem \ref{thm:susceptibility}. 
By Proposition \ref{pro:L2Error}, for $p\leq p'<\tilde p_c(p)$, $\|\tau_{p'}\|_2^2$ is close to
\[
\|\tau_{p,p'}^\square\|_2^2=\sum_{k=0}^\infty \left (\frac{p'-p}{1-p} \right )^k \|\tau_p^{\square,k} \|_2^2.
\]
By \eqref{eq:RewriteA_TauSquare}, we have 
\[ \|\tau_p^{\square,k}\|_2^2=\Omega^k \X(p)^{k+1} \E_p^\otimes \left [ \|\SA_k-\delta(\vec \A_k)\|_2^2 \1_{\Tcut>k} \right ]
 \]
Since $\SA_{k}=\sum_{i=0}^{k} \Delta(\vA_i)$, and since the $\Delta(\vA_i)$ have sub-exponential tails (see Lemma \ref{lem:TrashTailBelowMixing}), we may expand the sum $\SA_k$. For $0\leq i,j\leq k$ let 
\[ V_{i,j,k}(p) := \E_p^\otimes[\langle \Delta(\vA_i),\Delta(\vA_j)\rangle |\Tcut>k].\]
Note also that, under the conditioning on $\Tcut>k$, $\delta(\vec \A_k)$ is independent of $\SA_{k}-\delta(\vec \A_k)$ since it corresponds to an independent edge at the end of the chain. Thus, by expanding $\SA_k$ and using $\|\tau_p^{\square,k}\|_1=\Omega^k\X(p)^{k+1} \proba_p^\otimes(\Tcut>k)$ from \eqref{eq:RewriteA_TauSquare}, we get
\begin{equation}
\|\tau_p^{\square,k}\|_2^2 = \|\tau_p^{\square,k}\|_1\left (\sum_{0\leq i,j\leq k}V_{i,j,k}(p)-\|D\|_2^2/\Omega \right ).
\label{eq:DecomposeVariance}
\end{equation}
We split the sum according to whether $|i-j|$ is large or $\min(i,j)$ is large or $k-\max(i,j)$ is large.

Before doing so, let us first show some rough bounds on the $\Delta(\vA_i)$ and then on the $V_{i,j,k}$.
\begin{lemma}\label{lem:2stepPerturbation} Let $p<p_c$. Let $\vec A_0\in \Clive$ and $\vec A_2\in \vec \Comp(0)$ be such that there exists $\vec A\in \vec \Comp(0)$ with $(\vec A,\vec A_2)\notin \Ccut$. For every $x\in \Z^d$ we have 
\[ \proba^\otimes_p(\Delta(\vA_1)=x|\vA_0=\vec A_0, \vA_2=\vec A_2, \Tcut>2) \leq \frac{D\ast \tau_p(x)}{(1-p)^\Omega}.\]
\end{lemma}
\begin{proof} Let $\vec A=(A,\vec a)$ such that $(\vec A,\vec A_2)\notin \Ccut$. 
Let $\vec A_1=((\{0\},\emptyset),(0,a^+-a^-))$ be the pointed cluster of $0$ consisting of a single vertex with edge parallel to $\vec a$. Note that $(\vec A_1,\vec A_2)\notin \Ccut$. 
Note also that $(\vec A_0,\vec A_1)\notin \Ccut$ since $\vec A_0\in \Clive$. Thus, 
\[ \proba^\otimes_p(\Tcut>2\mid\vA_0=\vec A_0, \vA_2=\vec A_2)\geq \proba_p^\otimes (\vA_1=\vec A_1)=(1-p)^\Omega/(\Omega \X(p)). \]
Moreover, for every $x\in \Z^d$ we have 
\[ \proba^\otimes_p(\Delta(\vA_1)=x)=D\ast \tau_p(x)/(\Omega \X(p)).\]
The desired result is obtained by division.
\end{proof}
\begin{lemma} \label{lem:GlobalPerturbation} Let $p<p_c$. For every $0\leq i \leq k$ and every $x\in \Z^d$, we have
  \[ \proba_p^\otimes(\Delta(\vA_i)=x|\Tcut>k)\leq \frac{D\ast \tau_p(x)}{(1-p)^\Omega}.\]
\end{lemma}
\begin{proof} If $0<i<k$, then, by splitting according to the value of $\vA_{i-1}$ and $\vA_{i+1}$, the probability of the lemma is equal to
  \[ \sum_{\vec A_{i-1},\vec A_{i+1}} \proba \left (\substack{\vA_{i-1}=\vec A_{i-1}, \\ \vA_{i+1}=\vec A_{i+1}} \middle |\Tcut>k \right ) \proba \left (\Delta(\vA_i)=x, (\vA_{i-1},\vA_i)\notin \Ccut,(\vA_{i},\vA_{i+1})\notin \Ccut \middle |\substack{ \vA_{i-1}=\vec A_{i-1}, \\ \vA_{i+1}=\vec A_{i+1}} \right ). \]
Lemma \ref{lem:2stepPerturbation} bounds the second probability, and the desired result is obtained by summing. The cases $i=0$ and $i=k$ can be treated similarly. We omit the details.
\end{proof}
\begin{lemma} \label{lem:BoundedConditionalMoments} Let $C>0,\eps>0,n\in \N$. There exists $F$ depending only on $d,L,C,n,\eps$ such that for every $p<p_c$ satisfying \eqref{eq:initialization} and $\X(p)\geq F$, and every $0\leq i\leq k$, we have
  \[ \E_p^\otimes[\|\Delta(\vA_i)\|_2^n|\Tcut>k]\leq \X(p)^{n/2+\eps}.\]
\end{lemma}
\begin{proof} By Lemma \ref{lem:GlobalPerturbation}, the expectation in the lemma equals
 \begin{align*} & \E_p^\otimes\left [\|\Delta(\vA_i)\|_2^n \1_{\|\Delta(\vA_i)\|_2\leq \X(p)^{1/2+\eps/n}/2} \middle |\Tcut>k \right ]+ \E_p^\otimes \left[\|\Delta(\vA_i)\|_2^n \1_{\|\Delta(\vA_i)\|_2>\X(p)^{1/2+\eps/n}/2} \middle |\Tcut>k \right]
 \\ & \leq \frac{\X(p)^{n/2+\eps}}{2} + \sum_{x\in \Z^d, \|x\|_2>\X(p)^{1/2+\eps/n}/2} \|x\|_2^n \frac{D\ast \tau_p(x)}{(1-p)^\Omega}.
\end{align*}
 The desired upper bound is then a direct consequence of Theorem \ref{thm:ExponentialTail}.
\end{proof}
\begin{corollary} \label{lem:BoundVariance} Let $C>0,\eps>0$. There exists $F$ depending only on $d,L,C,\eps$ such that for every $p<p_c$ satisfying \eqref{eq:initialization} and $\X(p)\geq F$, and every $0\leq i,j\leq k$, we have $|V_{i,j,k}(p)|\leq \X(p)^{1+\eps}$.
\end{corollary}
\begin{proof} The case $i=j$ is already treated by the previous lemma applied to $n=2$. For the general case, simply note that for every $x,y\in \Z^d$ we have $2|\langle x,y\rangle| \leq \|x\|_2^2+\|y\|_2^2$, so $|V_{i,j,k}(p)|\leq (V_{i,i,k}(p)+V_{j,j,k}(p))/2$.
\end{proof}
By a direct adaptation of the proof of Lemma \ref{lem:BoundedConditionalMoments}, we also have the following variation:
\begin{lemma} \label{lem:BoundedConditionalMoments2} Let $C>0,\eps>0,n\in \N$. There exists $F$ depending only on $d,L,C,n,\eps$ such that for every $p<p_c$ satisfying \eqref{eq:initialization} and $\X(p)\geq F$, every $0<i\leq k$, and every $\vec A_0\in \Clive$, we have
  \[ \E_p^\otimes[\|\Delta(\vA_i)\|_2^n|\vA_0=\vec A_0, \Tcut>k]\leq \X(p)^{n/2+\eps}.\]
\end{lemma}
We now show a technical corollary of Theorem \ref{thm:MixingTime} that we will repeatedly use to estimate $V_{i,j,k}$.
For every $f$ from $\vec \Comp(0)$ to $\R$ or $\R^d$, we write $\|f\|_{2,\infty}=\sup_{\vec A}\|f(\vec A)\|_2$ and let
\[ \mathfrak{S}_p(f):=\sum_{\vec B\in \Clive} \frac{\vpg_p(\vec B)f(\vec B)}{h_p(\vec B)}. \]
\begin{proposition} \label{pro:MixingTimeFunction} Let $p<p_c$ and $C>0$ be such that \eqref{eq:initialization} holds.
There exist $0<c<1$ and $K>0$, depending on $d,L,C$, such that if $\X(p)\geq K$, then for every $\vec A_0 \in \Clive$,
every $k\geq 0$, and every $f$ from $\vec \Comp(0)$ to $\R$ or $\R^d$ with $f=0$ on $\Ckill$, we have 
\[
\left \|  \E_p^\otimes[f(\vA_k)\1_{\Tcut>k}|\vA_0=\vec A_0]-\lambda_p^k h_p(\vec A_0)\mathfrak{S}_p(f)\right \|_2 \leq \lambda_p^k c^k \X(p)^3 \|f\|_{2,\infty}.
\]
\end{proposition}
\begin{remark} By Lemma \ref{lem:Symmetric_hp_vpg}, $\vpg_p$ and $h_p$ are symmetric, so
if $f$ is antisymmetric, $\mathfrak{S}_p(f)=0.$
\end{remark}
\begin{proof} By the Doob transform \eqref{eq:DoobTransform}, since $f$ is supported on $\Clive$, we have
\begin{align*}  \E_p^\otimes[f(\vA_k)\1_{\Tcut>k}|\vA_0=\vec A_0] 
  & =\lambda_p^k \E^h_p \left [\frac{h_p(\vec \A_0)}{h_p(\vec \A_k)} f(\vec \A_k)\middle | \vA_0=\vec A_0 \right ] 
\\ & = \lambda_p^k \sum_{\vec B\in \Clive}\P^k_p(\vec B|\vec A_0)\frac{h_p(\vec A_0)}{h_p(\vec B)}f(\vec B), \end{align*}
using Fubini's theorem for the last equality, which is justified by Lemma \ref{lem:vpd} and the boundedness of $f$. It follows from the triangle inequality that the quantity in the lemma is bounded by
\[ \lambda_p^k \frac{h_p(\vec A_0)}{\inf_{\vec B'\in \Clive} h_p(\vec B')} \|f\|_{2,\infty}\sum_{\vec B\in \Clive} \left |  \P_p^k(\vec B|\vec A_0)-\vpg_p(\vec B)\right |.\]
The desired inequality follows from Lemma \ref{lem:vpd} and Theorem \ref{thm:MixingTime} when $\X(p)$ is large enough.
\end{proof}
We then deduce from Proposition \ref{pro:MixingTimeFunction} a general result that we will apply many times in this subsection and in Section \ref{sec:LowFrequencies}. For $F$ from $(\vec\Comp(0))^{n+1}$ to either $\R$ or $\R^d$, we write 
\[ \mathfrak{S}^-_p(F):=\E_p^\otimes[F(\vA_0,\dots,\vA_n)h_p(\vA_n)\1_{\Tcut>n}], \]
and 
\[ \mathfrak{S}^+_p(F):=\sum_{\vec A\in \Clive} \frac{\vpg_p(\vec A)}{h_p(\vec A)} \E_p^\otimes[F(\vA_0,\dots,\vA_n)\1_{\Tcut>n}|\vA_0=\vec A].\] 
We also let 
\[ \|F\|^-:=\E_p^\otimes[\|F(\vA_0,\dots,\vA_n)\|_2\1_{\Tcut>n}], \]
and 
\[ \|F\|^+:=\sup_{\vec A\in \Clive}\E_p^\otimes[\|F(\vA_0,\dots,\vA_n)\|_2\1_{\Tcut>n}|\vA_0=\vec A].\]
We also naturally extend the notation $\langle x,y\rangle$ to the case where $x$ or $y$ are in $\R$.
\begin{proposition} \label{pro:MixingTimeFunc} Let $p<p_c$ and $C>0$ be such that \eqref{eq:initialization} holds. 
There exist $0<c<1$ and $K>0$, depending on $d,L,C$, such that if $\X(p)\geq K$, then for every $0\leq i \leq j\leq k$ and every $F_1$ and $F_2$ from $(\vec \Comp(0))^{i+1}$ and $(\vec \Comp(0))^{k-j+1}$, to either $\R$ or $\R^d$, if either $j<k$ or $F_2=0$ on $\Ckill$, we have 
\[
\left \|  \E_p^\otimes[\langle F_1(\vA_0,\dots,\vA_i),F_2(\vA_j,\dots, \vA_k)\rangle\1_{\Tcut>k}]-\lambda_p^{j-i} \langle\mathfrak{S}^-_p(F_1),\mathfrak{S}^+_p(F_2)\rangle\right \|_2 \leq (\lambda_pc)^{j-i}\X(p)^3 \|F_1\|^-\|F_2\|^+.
\]
\end{proposition}
\begin{remark} By Lemma \ref{lem:Symmetric_hp_vpg}, if $F_2$ is antisymmetric, i.e., $F_2(-\vec A_j,\dots, -\vec A_k)=-F_2(\vec A_j,\dots, \vec A_k)$, then $\mathfrak{S}^+_p(F_2)=0$. Similarly, if $F_1$ is antisymmetric, then $\mathfrak{S}^-_p(F_1)=0$.
\end{remark}
\begin{proof}
For $\vec A\in \vec\Comp(0)$, let
\[
f(\vec A):=\E_p^\otimes[F_2(\vA_0,\dots,\vA_{k-j})\1_{\Tcut>k-j}|\vA_0=\vec A].
\]
Then $f=0$ on $\Ckill$, and $\|f\|_{2,\infty}\leq \|F_2\|^+$.
By conditioning on $\vA_i$, we have
\begin{align*}
&\E_p^\otimes[\langle F_1(\vA_0,\dots,\vA_i),F_2(\vA_j,\dots, \vA_k)\rangle\1_{\Tcut>k}]\\
&\qquad =
\E_p^\otimes\left[
\left\langle F_1(\vA_0,\dots,\vA_i),
\E_p^\otimes[f(\vA_{j-i})\1_{\Tcut>j-i}|\vA_0=\vA_i]
\right\rangle
\1_{\Tcut>i}
\right],
\end{align*}
where the use of Fubini's theorem is justified whenever $\|F_1\|^-<\infty$ and $\|F_2\|^+<\infty$.
By applying Proposition \ref{pro:MixingTimeFunction} to $f$ between times $i$ and $j$ and using the triangle inequality, we may approximate the last term by
\[
\lambda_p^{j-i}
\left\langle
\E_p^\otimes[F_1(\vA_0,\dots,\vA_i)h_p(\vA_i)\1_{\Tcut>i}],
\mathfrak{S}_p(f)
\right\rangle
=\lambda_p^{j-i}\langle\mathfrak{S}^-_p(F_1),\mathfrak{S}^+_p(F_2)\rangle,
\]
by the definitions of $\mathfrak{S}^-_p$ and $\mathfrak{S}^+_p$. The error term is bounded by
\[
(\lambda_p c)^{j-i}\X(p)^3 \|f\|_{2,\infty}
\E_p^\otimes[\|F_1(\vA_0,\dots,\vA_i)\|_2\1_{\Tcut>i}]
\leq
(\lambda_p c)^{j-i}\X(p)^3\|F_1\|^-\|F_2\|^+. \qedhere
\]
\end{proof}
We then prove a rough bound on $\proba_p^\otimes(\Tcut>k)$.
\begin{lemma} \label{lem:BadLowerTcut} Let $C>0$. There exists $F>0$ depending only on $d,L,C$ such that for every $p<p_c$ satisfying \eqref{eq:initialization} and $\X(p)\geq F$, for every $\vec A_0\in \vec \Comp(0)$, we have 
$ \proba_p^\otimes(\Tcut>k|\vA_0=\vec A_0) \leq \lambda_p^{k+1} \X(p)^2.$
\end{lemma}
\begin{proof} Recall the definition of $\Gamma_p$ from Section \ref{sec:p_c-p_c}. If $k=0$, the result follows from Theorem \ref{thm:LowerLambda}. If $\vec A_0\notin \Clive$ and $k>0$ then the probability is $0$. If $\vec A_0\in \Clive$ and $k>0$, by \eqref{eq:DoobTransform}, we have
  \[ \proba_p^\otimes(\Tcut>k|\vA_0=\vec A_0)= \frac{1}{\Omega\X(p)}\lambda_p^{k-1} \E_p^h[h_p(\vec A_0)\Gamma_p(\vA_{k-1})|\vA_0=\vec A_0]. \]
The desired result then follows directly from Lemma \ref{lem:Gamma} and Theorem \ref{thm:LowerLambda}.
\end{proof}
\begin{lemma} \label{lem:CovarianceDecay} Let $C>0$. There exist $F>0,c<1$ depending only on $d,L,C$ such that for every $p<p_c$ satisfying \eqref{eq:initialization} and $\X(p)\geq F$, for every $0\leq i<j\leq k$, we have
  \[ |V_{i,j,k}(p)|\leq \X(p)^{10} c^{j-i-1}\]
\end{lemma}
\begin{proof} Let $F_1:(\vec A_0,\dots \vec A_i)\mapsto \Delta(\vec A_i)$ and $F_2:(\vec A_{j-1},\vec A_j, \dots \vec A_k)\mapsto \Delta(\vec A_j)$. Note that $F_1$ and $F_2$ are antisymmetric. 
  Also, by Lemma \ref{lem:BoundedConditionalMoments2}, if $\X(p)$ is large enough, we have
  \[ \|F_1\|^-\leq \X(p)\proba_p^\otimes(\Tcut>i) \quad \text{and} \quad \|F_2\|^+\leq \X(p)\max_{\vec A_0}\proba_p^\otimes(\Tcut>k-j|\vA_0=\vec A_0).\]
  Thus, by Proposition \ref{pro:MixingTimeFunc}, we have
\[ |V_{i,j,k}(p)|\proba_p^\otimes(\Tcut>k)\leq (\lambda_p c)^{j-i-1}\X(p)^5\proba_p^\otimes(\Tcut>i)\max_{\vec A_0}\proba_p^\otimes(\Tcut>k-j|\vA_0=\vec A_0).\]
The final bound is obtained by applying Lemmas \ref{lem:LambdaLowMemory} and \ref{lem:BadLowerTcut}.
\end{proof}
We now use the same method to estimate $V_{i,j,k}$ in the cases where either $\min(i,j)$ or $k-\max(i,j)$ is large. We focus on the case where both are large, which is the most technical, and omit the proof for the others.
For $i\leq j$, let
\[
V_{j-i}(p)=V_{i-j}(p):=\E_p^\vpg[\langle \Delta(\vA_i),\Delta(\vA_j)\rangle].
\]
Note that $V_{j-i}(p)$ depends only on $j-i$ because $(\vA_i)_{i\geq 0}$ is a stationary Markov chain under $\proba_p^\vpg$.
We start by proving the following intermediate result:
\begin{lemma} \label{lem:ForgetPosForTcut} Let $C>0$. There exist $F>0,c<1$ depending only on $d,L,C$ such that for every $p<p_c$ satisfying \eqref{eq:initialization} and $\X(p)\geq F$, for every $\vec A_0\in \Clive$, we have 
\[  \left \| \proba_p^\otimes(\Tcut>k|\vA_0=\vec A_0)-\frac{\lambda_p^{k-1}}{\Omega \X(p)} h_p(\vec A_0)\E_p^\vpg[\Gamma_p(\vA_0)] \right \|_2  \leq \lambda_p^{k-1} c^{k-1}\X(p)^3 . \]
\end{lemma}
\begin{proof} For $\vec A\in \vec \Comp(0)$, let $f(\vec A)=\proba_p^\otimes((\vec A,\vA_1)\notin \Ccut)$. Clearly, $\|f\|_\infty \leq 1$. Moreover, we have
  \[ \proba_p^\otimes(\Tcut>k|\vA_0=\vec A_0)=\E_p^\otimes[f(\vA_{k-1})\1_{\Tcut>k-1}|\vA_0=\vec A_0].\]
Also $\mathfrak{S}_p(f)=\E_p^\vpg[\Gamma_p(\vA_0)]/(\Omega\X(p))$.
The result then directly follows from Proposition \ref{pro:MixingTimeFunction}.
\end{proof}
\begin{lemma} \label{lem:Bulk_L2}
Let $C>0$. There exist $F>0,c<1$ depending only on $d,L,C$ such that for every $p<p_c$ satisfying \eqref{eq:initialization} and $\X(p)\geq F$, 
for every $1\leq i\leq j< k$, for every $\vec A_0\in \Clive$, we have 
\[
|V_{i,j,k}(p)-V_{j-i}(p)|\leq \X(p)^{11}(c^i+c^{k-j-1})
\]
\end{lemma}
\begin{proof} We first apply Proposition \ref{pro:MixingTimeFunc} between the block $(\vA_i,\dots,\vA_j)$ and the terminal time $k$, with $F_1:(\vec A_0,\dots,\vec A_j)\mapsto \langle\Delta(\vec A_i),\Delta(\vec A_j) \rangle $ and $F_2:\vec A\mapsto \proba_p^\otimes((\vec A,\vA_1)\notin \Ccut)$.
By Lemma \ref{lem:BoundedConditionalMoments2} and Lemma \ref{lem:BadLowerTcut}, we have
\[ \|F_1\|^-\leq \X(p)^2\proba_p^\otimes(\Tcut>j)\leq \lambda_p^{j+1}\X(p)^5,\]
 whenever $\X(p)$ is large enough. Clearly, $\|F_2\|^+\leq 1$.
Also note that $\mathfrak{S}_p^+(F_2)=\E_p^\vpg[\Gamma_p(\vA_0)]/(\Omega\X(p))$. 
Thus, 
\begin{equation}
\left | \E_p^\otimes[\langle \Delta(\vA_i),\Delta(\vA_j)\rangle\1_{\Tcut>k}]-
\lambda_p^{k-j-1}
\mathfrak{S}^-_p(F_1)\mathfrak{S}_p^+(F_2) \right | \leq c^{k-j-1}\lambda_p^k\X(p)^8.
\label{eq:LongRangeVariancePart1}
\end{equation}
Moreover, we have 
\[ \mathfrak{S}^-_p(F_1)= \E_p^\otimes[\langle \Delta(\vA_i),\Delta(\vA_j)\rangle h_p(\vA_j)\1_{\Tcut>j}],\]
so that we may reuse Proposition \ref{pro:MixingTimeFunc} with $F'_1(\vec A_0)=1$ and $F'_2(\vec A_i,\dots, \vec A_j)=\langle \Delta(\vec A_i),\Delta(\vec A_j)\rangle h_p(\vec A_j)$.
Note that $\mathfrak{S}_p^{-}(F'_1)=\E_p^\otimes[h_p(\vA_0)]$. By a Doob transform, we have 
\[ \mathfrak{S}_p^+(F'_2)=\lambda_p^{j-i}V_{j-i}(p).  \]
Also, $\|F'_1\|^-=1$. By Lemma \ref{lem:BoundedConditionalMoments2} and Lemma \ref{lem:BadLowerTcut}, we have
\[ \|F'_2\|^+\leq \lambda_p^{j-i}\X(p)^2\|h_p\|_\infty. \]
Thus, by Proposition \ref{pro:MixingTimeFunc}, we get
\[\left |\mathfrak{S}^-_p(F_1) -\lambda_p^j\E_p^\otimes[h_p(\vA_0)]V_{j-i}(p) \right |\leq c^i\lambda_p^j\X(p)^5\|h_p\|_\infty.\]
Furthermore, by Lemma \ref{lem:Gamma}, we have $\mathfrak{S}_p^+(F_2)\leq \X(p)^3/\|h_p\|_\infty$ whenever $\X(p)$ is large enough; thus,
\[\left |\lambda_p^{k-j-1}
\mathfrak{S}^-_p(F_1)\mathfrak{S}_p^+(F_2) -\lambda_p^{k-1}\E_p^\otimes[h_p(\vA_0)]\mathfrak{S}_p^+(F_2)V_{j-i}(p) \right |\leq c^i\lambda_p^{k-1}\X(p)^8 .\]
Combining this with \eqref{eq:LongRangeVariancePart1}, we get
\[ \left |\E_p^\otimes[\langle \Delta(\vA_i),\Delta(\vA_j)\rangle\1_{\Tcut>k}]-\lambda_p^{k-1}\E_p^\otimes[h_p(\vA_0)]\mathfrak{S}_p^+(F_2)V_{j-i}(p) \right |\leq \lambda_p^{k-1}\X(p)^8(c^i+c^{k-j-1}).\]
Using Lemma \ref{lem:ForgetPosForTcut} and Corollary \ref{lem:BoundVariance}, we deduce
\[ \lambda_p^{k-1}\E_p^\otimes[h_p(\vA_0)]\mathfrak{S}_p^+(F_2)\left |V_{i,j,k}- V_{j-i}(p) \right |\leq \lambda_p^{k-1}\X(p)^8(c^i+c^{k-j-1}).\]
The desired result follows by division. Indeed, since $\mathfrak{S}_p^+(F_2)=\E_p^\vpg[\Gamma_p(\vA_0)]/(\Omega\X(p))$, we have
\[ \E_p^\otimes[h_p(\vA_0)]\mathfrak{S}_p^+(F_2)\geq \frac{\Plive_p \min_{\vec A\in \Clive} (h_p(\vec A))}{\|h_p\|_\infty}\sum_{\vec A,\vec B\in \Clive }\vpg_p(\vec A)\lambda_p \P_p(\vec B|\vec A)\geq \frac{(1-p)^\Omega}{\Omega\X(p)}\lambda_p^2,  \]
using Lemmas \ref{lem:vpd}, \ref{lem:vpg} and \ref{lem:LambdaLowMemory} for the last inequality. Moreover, $\lambda_p$ is bounded below by a constant by Theorem \ref{thm:LowerLambda}.
\end{proof}
We deal with the other two cases in exactly the same way:
\begin{lemma} \label{lem:Start_L2}
Let $C>0$. There exist $F>0,c<1$ depending only on $d,L,C$ such that for every $p<p_c$ satisfying \eqref{eq:initialization} and $\X(p)\geq F$, 
for every $0\leq i\leq j$, there exists $V_{i,j,\infty}(p)=V_{j,i,\infty}(p)$ such that for every $k\geq j+1$, we have 
\[
|V_{i,j,k}(p)-V_{i,j,\infty}(p)|\leq \X(p)^{11}c^{k-j-1}
\]
\end{lemma}
\begin{proof}
The proof is essentially the same as for Lemma \ref{lem:Bulk_L2}, with the difference that we skip the second application of Proposition \ref{pro:MixingTimeFunc}. We omit the details.
\end{proof}
\begin{lemma} \label{lem:End_L2}
Let $C>0$. There exist $F>0,c<1$ depending only on $d,L,C$ such that for every $p<p_c$ satisfying \eqref{eq:initialization} and $\X(p)\geq F$, 
for every $0\leq i\leq j$, there exists $V_{\infty,i,j}(p)=V_{\infty,j,i}(p)$ with $|V_{\infty,i,j}(p)|\leq \X(p)^3$ such that for every $k>j$, we have
\[
|V_{k-j,k-i,k}(p)-V_{\infty,i,j}(p)|\leq \X(p)^{11}c^{i}
\]
\end{lemma}
\begin{proof}
The proof is essentially the same as for Lemma \ref{lem:Bulk_L2}, with the difference that we skip the first application of Proposition \ref{pro:MixingTimeFunc}. We omit the details.
\end{proof}
When $k$ is small, we also use the following bound:
\begin{lemma} \label{lem:First_Term_L2_BadBound}
Let $C>0$. There exists $F>0$ depending only on $d,L,C$ such that for every $p<p_c$ with $\X(p)\geq F$, and such that \eqref{eq:initialization} is satisfied at $p'$ with $\X(p')=\log(\X(p))^3\X(p)$, we have
\[
\sum_{k=0}^{3\log(\X(p))^2} \left (\frac{\tilde p_c(p)-p}{1-p} \right )^k  \|\tau_p^{\square,k}\|_2^2 \leq \X(p)^3.
\]
\end{lemma}
\begin{proof}
By Proposition \ref{pro:L2Error}, we have $\|\tau^\square_{p,p'}\|_2^2\leq \|\tau_{p'}\|_2^2+o(\X(p)^3)$. It follows directly from, e.g., Theorem \ref{thm:ExponentialTail} that $\|\tau_{p'}\|_2^2=o(\X(p)^3)$. Thus
\begin{equation}
\sum_{k=0}^{3\log(\X(p))^2} \left (\frac{p'-p}{1-p} \right )^k \|\tau_p^{\square,k} \|_2^2 \leq  \|\tau_{p,p'}^\square\|_2^2=o(\X(p)^3). \label{eq:Naive_Truncated_L2_Bound}
\end{equation}
Proposition \ref{pro:p_c(p)-p_2} then gives $0\leq \tilde p_c(p)-p'\leq K/\X(p')$ for some constant $K>0$ depending on $d,L,C$.
By Lemma \ref{lem:X'}, we have $p'-p\geq K'/\X(p)$ for some constant $K'>0$.
It follows, using $x/(x+y)\geq e^{-y/x}$ for $x>0,y\geq 0$,
\[
\left (\frac{p'-p}{\tilde p_c(p)-p} \right )^{3\log(\X(p))^2}\geq \exp\left (-\frac{3K}{K'}\frac{\log(\X(p))^2\X(p)}{\X(p')}\right )=1-o(1).
\]
The result thus follows from \eqref{eq:Naive_Truncated_L2_Bound} and $p\leq p'\leq \tilde p_c(p)$.
\end{proof}
\subsection{Asymptotic for $\xi_2(p)$: proof of Theorem \ref{thm:gyration}} \label{sec:Gyr}
We assume \eqref{eq:etaUpper} throughout the section. As in Section \ref{sec:Chi}, we first approximate
$\|\tau_{p,p'}^\square\|_2^2$ by a polynomial, then use a deconvolution estimate.
\begin{proposition} \label{pro:PolynomialApproximation_For_Xi2}
For every $p<p_c$ with $\X(p)$ large enough,
and $\eps>0$ small enough,
there exist $N(\eps)\in\N$
and coefficients $A_{-2}(p,\eps),A_{-1}(p,\eps),\dots, A_{N}(p,\eps)$
such that uniformly over all $p'$ satisfying
$\X(p)^{1+\eps/2} \leq\X(p')\leq \X(p)^{1+2\eps}$,
we have
\[
\|\tau_{p,p'}^\square\|_2^2 = \sum_{k=-2}^{N(\eps)} A_{k}(p,\eps)(\tilde p_c(p)-p')^k+O(1/\X(p)^2).
\]
\end{proposition}
\begin{proof} Let $\ell(p):=\lfloor\log(\X(p))^2\rfloor$.
First, by \eqref{eq:tau_E_p,infty} and \eqref{eq:DecomposeVariance}, we have
\begin{equation}
\|\tau_p^{\square,k}\|_2^2
=\left(\frac{1-p}{\tilde p_c(p)-p}\right)^k E_{p,k}
\left(\sum_{0\leq i,j\leq k}V_{i,j,k}(p)-\frac{\|D\|_2^2}{\Omega}\right).
\label{eq:def_SE}
\end{equation}
Thus, by summing,
\[
\|\tau_{p,p'}^\square\|_2^2
=\sum_{k\leq 3 \ell(p)}
\left ( \frac{p'-p}{1-p} \right )^k\|\tau^{\square,k}_{p}\|_2^2
+\sum_{k>3\ell(p)}
\left ( \frac{p'-p}{\tilde p_c(p)-p} \right )^k E_{p,k}
\left(\sum_{0\leq i,j\leq k}V_{i,j,k}(p)-\frac{\|D\|_2^2}{\Omega}\right).
\]
Since $p'<\tilde p_c(p)$ by Proposition \ref{pro:p_c(p)-p_2}, the sum under consideration converges absolutely by Proposition \ref{pro:PrelimBootstrap}, Lemma \ref{lem:BoundedConditionalMoments}, and the Cauchy--Schwarz inequality.
For $k>3\ell(p)$ and $0\leq i,j \leq k$, we now define
\[
E_{i,j,k}(p):=\begin{cases} 0 & \text{ if } |i-j|>\ell(p)
\\ E_{p,\infty} V_{j-i}(p) & \text{ if } |i-j|\leq \ell(p) \text{ and } i,j\in (\ell(p),k-\ell(p))
\\ E_{p,\infty} V_{i,j,\infty}(p) & \text{ if } |i-j|\leq \ell(p) \text{ and } \min(i,j)\leq \ell(p)
\\ E_{p,\infty} (V_{\infty,k-i,k-j}(p)-\1_{i=j=k}\frac{\|D\|_2^2}{\Omega}) & \text{ if } |i-j|\leq \ell(p) \text{ and } \max(i,j)\geq k- \ell(p)
\end{cases}
\]
We approximate $\|\tau_{p,p'}^\square\|_2^2$ by 
\[ S_E:=  \sum_{k\leq 3 \ell(p)}
\left ( \frac{p'-p}{1-p} \right )^k\|\tau^{\square,k}_{p}\|_2^2
+\sum_{k>3\ell(p)}
\left ( \frac{p'-p}{\tilde p_c(p)-p} \right )^k \sum_{0\leq i ,j \leq k}E_{i,j,k}(p). \]
Lemmas \ref{eq:L1Diff_After_Mixing} and \ref{lem:Large_k_L1} give, uniformly for $k>3\ell(p)$,
$|E_{p,k}-E_{p,\infty}|\leq \X(p)^{-200}$ and $E_{p,\infty}\leq \X(p)^2$. Also, by 
Lemma \ref{lem:CovarianceDecay}, we have $|V_{i,j,k}|\leq \X(p)^{-200}$ whenever $|i-j|>\ell(p)$. Similarly, the second, third, and fourth cases in the definition of $E_{i,j,k}(p)$ are dealt with using, respectively,
 Lemmas \ref{lem:Bulk_L2}, \ref{lem:Start_L2}, and \ref{lem:End_L2}. Hence,
\[
\left |\|\tau_{p,p'}^\square\|_2^2-S_E\right | =O(\X(p)^{-100}) \sum_{k> 3\ell(p)} \left (\frac{p'-p}{\tilde p_c(p)-p} \right)^k (k+1)^2.
\]
Since we have $\sum_{k\geq 0} x^k (k+1)^2=O(1+(1-x)^{-3})$ uniformly over $0\leq x\leq 1$, we deduce
\[
\left |\|\tau_{p,p'}^\square\|_2^2-S_E\right | =O(\X(p)^{-100})\left (1+\left (\frac{\tilde p_c(p)-p}{\tilde p_c(p)-p'}\right)^3 \right ).
\]
By Lemma \ref{lem:X'} and Theorem \ref{thm:p_c-p_c}, we have
\begin{equation}
\tilde p_c(p)-p'\geq p_c-p'-O(\psi_d(\X(p))/\X(p)^2)\geq K /\X(p') \label{eq:p_c(p)-p'}
\end{equation}
for some constant $K>0$ depending only on $d,L,C$. Thus,
\begin{equation}
\left |\|\tau_{p,p'}^\square\|_2^2-S_E\right |=o(1/\X(p)^{2}). \label{eq:tau-SE}
\end{equation}

It is thus enough to estimate $S_E$.
We start with the case $k\leq 3\ell(p)$.
Lemma \ref{lem:X'} and Theorem \ref{thm:p_c-p_c} give $\tilde p_c(p)-p\geq K/\X(p)$ for some $K>0$.
Proposition \ref{pro:p_c(p)-p_2} gives $\tilde p_c(p)-p'\leq K_2/\X(p')$ for some constant $K_2>0$.
By a Taylor expansion, we thus get that there exists $N(\eps)\in \N$ such that uniformly over all $k\leq 3\ell(p)$, and $p'$ satisfying $\X(p)^{1+\eps/2} \leq\X(p')\leq \X(p)^{1+2\eps}$, we have
\begin{equation}
\left (\frac{p'-p}{\tilde p_c(p)-p} \right )^k=\left (1-\frac{\tilde p_c(p)-p'}{\tilde p_c(p)-p} \right )^k =\sum_{i=0}^{N(\eps)} \binom{k}{i} \left (-\frac{\tilde p_c(p)-p'}{\tilde p_c(p)-p} \right )^i +O(\X(p)^{-100}). \label{eq:Taylor2}
\end{equation}
Write $P_{\eps,k}$ for the above sum. By summing over all $k\leq 3\ell(p)$, we have
\begin{align*}
\sum_{k\leq 3 \ell(p)} \left ( \frac{p'-p}{1-p} \right )^k\|\tau^{\square,k}_{p}\|_2^2
= & \sum_{k\leq 3 \ell(p)} \left ( \frac{\tilde p_c(p)-p}{1-p} \right )^k
\|\tau^{\square,k}_{p}\|_2^2 P_{\eps,k} \\
& +O(\X(p)^{-100}) \sum_{k\leq 3\ell(p)}
\left (\frac{\tilde p_c(p)-p}{1-p} \right )^k
\|\tau_p^{\square,k} \|_2^2.
\end{align*}
The first sum has the polynomial form desired in the proposition. By Lemma \ref{lem:First_Term_L2_BadBound}, the second sum is $O(\X(p)^3)$. This proves the case $k\leq 3\ell(p)$.

We now deal with the case $k>3\ell(p)$. Recall that $\Plive_p\leq 1$, and by Theorem \ref{thm:LowerLambda}, we have $\lambda_p\geq c>0$ for some constant depending on $d,L,C$, so it suffices to estimate
\[
S_E^+ := \sum_{k>3\ell(p)}  \left ( \frac{p'-p}{\tilde p_c(p)-p} \right )^k \sum_{0\leq i,j\leq k} E_{i,j,k}(p).
\]
By \eqref{eq:Taylor2}, and since $\sup_{i,j,k} |E_{i,j,k}|\leq \X(p)^6$, we have
\[
\left | S_E^+ -P_{\eps,3\ell(p)}\sum_{k>3\ell(p)}  \left ( \frac{p'-p}{\tilde p_c(p)-p} \right )^{k-3\ell(p)} \sum_{0\leq i,j\leq k} E_{i,j,k}(p)\right | =O(\X(p)^{-94}) \sum_{k>3\ell(p)} (k+1)\left (\frac{p'-p}{\tilde p_c(p)-p} \right)^k.
\]
Proceeding exactly as in the argument preceding \eqref{eq:tau-SE}, we get that the last sum is $O(\X(p)^{90})$.
It remains to compute the sum over $k$ on the left-hand side, which we denote by $S_E'$.
To this end, counting the terms in $\{E_{i,j,k}(p)\}_{0\leq i,j\leq k}$ shows that there exist $A(p), B(p)$ such that for every $k>3\ell(p)$ we have
\[
\sum_{0\leq i,j\leq k} E_{i,j,k}(p)=A(p)(k-3\ell(p))+B(p).
\]
Thus, with the change of variable $k\gets k-3\ell(p)$, and using for $x>0$ the identities $\sum_{k\geq 1} x^k=x/(1-x)$ and $\sum_{k\geq 1} kx^k=x/(1-x)^2$, we get that $S_E'$ equals
\[
\sum_{k>0} \left ( \frac{p'-p}{\tilde p_c(p)-p} \right )^k (A(p)k+B(p))=A_{-2}(p)\left (\frac{\tilde p_c(p)-p'}{\tilde p_c(p)-p} \right)^{-2}+A_{-1}(p)\left (\frac{\tilde p_c(p)-p'}{\tilde p_c(p)-p} \right)^{-1}+A_0(p),
\]
for some functions $A_{-2},A_{-1},A_0$ depending only on $p$. In other words, $S_E^+$ is well approximated by $P_{\eps,3\ell(p)} S_E'$, which has the desired polynomial form. This proves the case $k>3\ell(p)$.
\end{proof}
Next, to use the method in Section \ref{sec:Chi}, we want a polynomial approximation in $p_c-p'$.
But we have a polynomial approximation in $\tilde p_c(p)-p'$.
To avoid this issue, we evaluate Proposition \ref{pro:PolynomialApproximation_For_Xi2} not at $p'$ but at $p'+\tilde p_c(p)-p_c$.
This requires us to bound $\|\tau_{p'}\|_2^2-\|\tau_{p'+\tilde p_c(p)-p_c}\|_2^2$, for which we use the following lemma.
\begin{lemma} \label{lem:Derive_L2}
For every $\eps>0$, as $p\uparrow p_c$, we have
\[
0\leq  \frac{\partial \|\tau_{p}\|_2^2}{\partial p} =O(\X(p)^{3+\eps}).
\]
\end{lemma}
\begin{proof}
For the first inequality, simply note that $\tau_p$ is positive and increasing in $p$. For the second, we use the following well-known consequence of the Margulis--Russo formula and the BK inequality: for every $x\in \Z^d$,
\[
\frac{\partial \tau_p}{\partial p}(x) \leq \tau_p\ast D\ast \tau_p(x).
\]
Thus,
\[
\frac{\partial \|\tau_{p}\|_2^2}{\partial p} =\sum_{x\in \Z^d} \|x\|_2^2  \frac{\partial \tau_p}{\partial p}(x)\leq \|\tau_p\ast D\ast \tau_p\|_2^2 = 2\|\tau_p\|^2_2 \|D\|_1\|\tau_p\|_1+ \|\tau_p\|_1^2 \|D\|_2^2=O(\X(p)^{3+\eps}),
\]
using, e.g., Theorem \ref{thm:ExponentialTail} for the last inequality.
\end{proof}
We finish as in Section \ref{sec:Chi}.
First, let $\Phi_2:p\in [0,p_c]\mapsto \|\tau_{p_c-p}\|_2^2$.
Fix $\alpha>1$.
For every $\lambda\in \R$, let $P_\lambda$ be the operator such that for every $f:[0,p_c]\mapsto \R$ we have $P_\lambda(f):x\mapsto f(x)-\alpha^\lambda f(x/\alpha)$.
\begin{lemma} \label{lem:OpOpOp2}
For every $\eps>0$, if $N$ is large enough, we have as $x\to 0$,
\[
[P_{-2}P_{-1}\dots P_N]\Phi_2(x)=O(\psi_d(1/x)x^{-1-9\eps}).
\]
\end{lemma}
\begin{proof}
Let $\eps>0$ be small enough. For $p<p_c$, let
\[
\Phi_p: x\in (0,p_c]\mapsto \sum_{k=-2}^{N(\eps)} x^k A_k(p,\eps),
\]
where the $N(\eps), A_k$ are those of Proposition \ref{pro:PolynomialApproximation_For_Xi2}.
Let $p<p_c$ with $\X(p)$ large enough.
By Proposition \ref{pro:PolynomialApproximation_For_Xi2}, uniformly for every $p'$ with $\X(p)^{1+\eps/2}\leq \X(p')\leq \X(p)^{1+2\eps}$, we have
\[
\|\tau_{p,p'}^\square\|_2^2= \Phi_p(\tilde p_c(p)-p')+O(1/\X(p)^2).
\]
Proposition \ref{pro:L2Error} then gives
\[
\|\tau_{p'}\|_2^2= \Phi_p(\tilde p_c(p)-p')+O(\X(p')^{3+\eps}\psi_d(\X(p))/\X(p)^2).
\]
Theorem \ref{thm:p_c-p_c} also gives $\tilde p_c(p)-p_c=O(\psi_d(\X(p))/\X(p)^2)$.
Moreover, Lemma \ref{lem:X'} yield $\X(\tilde p)=\Theta((p_c-\tilde p)^{-1})$ as $\tilde p\uparrow p_c$.
Thus, for every $\tilde p$ between $p'$ and $p'+\tilde p_c(p)-p_c$, we have $\X(p')/2\leq \X(\tilde p)\leq 2\X(p')$.
Hence, by Lemma \ref{lem:Derive_L2}, 
\[
\|\tau_{p'+\tilde p_c(p)-p_c}\|_2^2= \Phi_p(p_c-p')+O \left(\X(p'+\tilde p_c(p)-p_c)^{3+\eps}\psi_d(\X(p))/\X(p)^2 \right).
\]
In other words, for every $p'$ satisfying $2\X(p)^{1+\eps/2}\leq \X(p')\leq \X(p)^{1+2\eps}/2$,
\[
\|\tau_{p'}\|_2^2= \Phi_p(p_c-p')+O(\X(p')^{3+\eps}\psi_d(\X(p))/\X(p)^2).
\]
Together with Lemma \ref{lem:X'}, this implies that for every $x$ satisfying $K\X(p)^{-1-2\eps}\leq x\leq K'\X(p)^{-1-\eps/2}$ for some constant $K,K'>0$ depending on $d,L,C$, we have
\[
\Phi_2(x)=\Phi_p(x)+O(x^{-1-9\eps}\psi_d(1/x)).
\]
Finally, for every $x$ small enough, we may choose $p$ such that the preceding estimate is valid for every $x'\in \{\alpha^{-i} x\}_{-2\leq i\leq N(\eps)}$. Thus, by applying $P_{-2}P_{-1}\dots P_N$, we get
\[
[P_{-2}P_{-1}\dots P_N]\Phi_2(x)= [P_{-2}P_{-1}\dots P_N]\Phi_p(x)+O(x^{-1-9\eps}\psi_d(1/x)).
\]
Moreover, $[P_{-2}P_{-1}\dots P_N]\Phi_p=0$ since each of the terms $x\mapsto x^i$ is canceled by $P_i$.
\end{proof}
\begin{lemma} \label{lem:Gyr}
There exists a constant $C'_\xi\geq 0$ such that for every $\theta>1+\1_{d=7}/2$, as $p\uparrow p_c$,
 \[
 \|\tau_p\|_2^2=C'_\xi (p_c-p)^{-2}+O((p_c-p)^{-\theta})
 \]
\end{lemma}
\begin{proof}
Let $\theta\in(1+\1_{d=7}/2,2)$.
By Lemma \ref{lem:OpOpOp2}, we have $[P_{-2}P_{-1}\dots P_N]\Phi_2(x)=O(x^{-\theta})$.
Using Lemma \ref{lem:downward_PolyPerateur1} inductively, we deduce that for every $-2\leq i \leq N$, $[P_{-2}P_{-1}\dots P_i]\Phi_2(x)=O(x^{-\theta})$.
So by Lemma \ref{lem:downward_PolyPerateur2}, for some $l\in \R$, we have $\Phi_2(x)=l x^{-2}+O(x^{-\theta})$ as $x\to 0$.
The desired asymptotic follows by taking $p=p_c-x$.
Moreover, $C'_\xi\geq 0$ since $\|\tau_p\|_2^2$ is positive.
\end{proof}
\begin{lemma}
We have $C'_\xi>0$.
\end{lemma}
\begin{proof}
As $p\uparrow p_c$, for every $r>0$, we have
\[
\X(p)\leq \sum_{x,\|x\|_2\leq r} \tau_p(x)+\sum_{x,\|x\|_2\geq r} \tau_p(x)\|x\|_2^2/r^2=O(r^2)+\|\tau_p\|_2^2/r^2,
\]
using \eqref{eq:initialization} for the last inequality. Thus, by taking $r=\|\tau_p\|_2^{1/2}$, we get
\[
\X(p)=O(\|\tau_p\|_2),
\]
and the desired result follows from Lemma \ref{lem:X'}.
\end{proof}
\begin{proof}[Proof of Theorem \ref{thm:gyration}]
By Theorem \ref{thm:susceptibility} and Lemma \ref{lem:Gyr}, for every $\theta>\frac{\1_{d=7}}{2}$, we have, as $p\uparrow p_c$,
\[
\xi_2(p) := \sqrt{\frac{1}{\X(p)} \sum_{x\in \Z^d} \|x\|_2^2 \tau_p(x)} =\sqrt{\frac{C'_\xi(p_c-p)^{-2}+O((p_c-p)^{-1-\theta})}{C_\X(p_c-p)^{-1}+O((p_c-p)^{-\theta})}}
\]
Since $C'_\xi>0$ and $C_\X>0$, we deduce
\[
\xi_2(p)=(C'_\xi/C_\X)^{1/2}(p_c-p)^{-1/2}+O((p_c-p)^{1/2-\theta}). \qedhere
\]
\end{proof}

\section{Local central limit estimate for the two-point function} \label{sec:TCL}
In this section, we prove Theorems~\ref{thm:Main} and~\ref{thm:tau_local_CLT}. We first prove an LTCL for $\tau_p^{\square,k}$ in Subsection~\ref{seq:TCL_square}, using Subsections~\ref{Sec:HighFrequencies}, \ref{sec:MiddleFrequencies}, and~\ref{sec:LowFrequencies}, that estimate its Fourier transform in different frequency regimes. In Subsection~\ref{seq:TCL_square_2}, we deduce an LTCL for $\tau_{p,p'}^\square$ by summing over $k$.
In Subsection~\ref{sec:RenormalizationInduction}, we perform the renormalization induction to derive \eqref{eq:etaUpper} from this LTCL and \eqref{eq:initialization}. Finally, in Subsection~\ref{sec:FinalMainProof}, we prove the LTCL of Theorem~\ref{thm:tau_local_CLT} under condition \eqref{eq:etaUpper}, thus completing the proof of Theorem~\ref{thm:Main}.
\subsection{Intermediate frequencies for $\Four[\tau_{p}^{\square,k}]$} \label{sec:MiddleFrequencies}
The main goal of this subsection is to show the following rough upper bound on the Fourier transform of $\tau_{p}^{\square,k}$ at intermediate frequencies.
\begin{theorem} \label{thm:FT_Middle}
Let $C>0,\eps,\eps'\in(0,1/2)$. There exists $F,K>0$ depending only on $d,L,C,\eps,\eps'$ such that for every $p<p_c$ satisfying \eqref{eq:initialization} and $\X(p)\geq F$, for every $k\geq 3$, and for every $\kappa\in [-\pi,\pi]^{d}$ with $\|\kappa\|_\infty\leq \X(p)^{-\eps}$, we have
\[
|\Four[\tau_p^{\square,k}](\kappa)| \leq K \left (1-\X(p)^{-\eps'}\min(1,\X(p)\|\kappa\|_\infty^2) \right )^k\|\tau_p^{\square,k}\|_1.
\]
\end{theorem}
The proof follows essentially the same approach as that of Theorem \ref{thm:FT_High}. We first show an analogue of Proposition \ref{pro:SmallRotation}.
For $\vec A_0,\vec A_i\in \vec \Comp(0)$, recall the definitions of $P^\otimes_i(\vec A_0,\vec A_i)$, $\Amp_{i,\kappa}(\vec A_0,\vec A_i)$, and $\Rot_{i,\kappa}(\vec A_0,\vec A_i)$ from Section \ref{Sec:HighFrequencies}.
\begin{proposition} \label{pro:MiddleRotation}
Let $C>0$, $\eps,\eps'\in(0,1/2)$, and $\delta>0$.
There exists $F>0$, depending only on $d,L,C,\eps,\eps',\delta$, such that
for every $p<p_c$ satisfying \eqref{eq:initialization} and $\X(p)\geq F$,
every $\kappa\in[-\pi,\pi]^d$ with $\|\kappa\|_\infty\leq \X(p)^{-\eps}$ 
and every $\vec A_0,\vec A_2\in\vec\Comp(0)$ with
$
P_2^\otimes(\vec A_0,\vec A_2)\geq\delta,
$
we have
\[
\Amp_{2,\kappa}(\vec A_0,\vec A_2)
\leq 1-\X(p)^{-\eps'}\min(1,\X(p)\|\kappa\|_\infty^2).
\]
\end{proposition}
The proof of Proposition \ref{pro:MiddleRotation} uses the following technical bound:
\begin{lemma} \label{lem:NotAligned} There exists $K>0$,
depending only on $d$, such that for every $a>0$, every
$\kappa\in[-\pi,\pi]^d$, every $\theta\in\R$, and
every $R\geq 1$, we have 
\[
\sum_{\substack{x\in\Z^d:\ \|x\|_\infty\leq R\\
\operatorname{dist}(\langle\kappa,x\rangle-\theta,2\pi\Z)\leq a}}
\langle x\rangle^{2-d}
\leq KR(1+\|\kappa\|_\infty R)\left(1+a\|\kappa\|_\infty^{-1}\right).
\]
\end{lemma}
\begin{proof} By a dyadic decomposition of the sum, it suffices to show that, for some constant $K'$ and every $r\geq 1$,
\[\#\left \{x: r\leq\|x\|_\infty\leq 2r,  \operatorname{dist}(\langle\kappa,x\rangle-\theta,2\pi\Z)\leq a \right \} \leq K'r^{d-1}(1+\|\kappa\|_\infty r)\left(1+a\|\kappa\|_\infty^{-1}\right).\]
To this end, choose $j$ such that $|\kappa_j|=\|\kappa\|_\infty$. Every interval in the $j$-th direction of length $2\pi/\|\kappa\|_\infty$ contains at most $O(a/\|\kappa\|_\infty+1)$ vertices counted in the above set. The ball $\{x:\|x\|_\infty\leq 2r\}$ can be covered by $O(r^{d-1}+r^d\|\kappa\|_\infty)$ such intervals. The desired result follows by multiplication.
\end{proof}
\begin{proof}[Proof of Proposition \ref{pro:MiddleRotation}]
Fix $\vec A_0,\vec A_2\in\vec\Comp(0)$ as in the statement and define the
finite measure $\nu$ on $\Z^d$ by
\[
\nu(x):=
\proba_p^\otimes\left(
\Delta(\vA_1)=x,\ \Tcut>2
\,\middle|\,
\vA_0=\vec A_0,\ \vA_2=\vec A_2
\right).
\]
Its total mass is
\[
\nu(\Z^d)=P_2^\otimes(\vec A_0,\vec A_2)\geq\delta.
\]
Moreover, by \eqref{eq:def-pointed-cluster-law}, for every $x\in\Z^d$,
\[
\nu(x)
\leq \proba_p^\otimes(\Delta(\vA_1)=x) =\frac{1}{\Omega \X(p)}
\sum_{\substack{\vec A=(A,\vec a)\in\vec\Comp(0)\\a^+=x}}
\mu_p(\vec A)
=\frac{1}{\Omega \X(p)}
\sum_{\substack{y\in\Z^d\\(y,x)\in\vedges}}
\tau_p(y)
=\frac{(\tau_p\ast D)(x)}{\Omega \X(p)}.
\]
Next, set $a:=3\X(p)^{-\eps'/2}\min(1,\X(p)^{1/2}\|\kappa\|_\infty)$ and let $\theta$ be such that $\Rot_{2,\kappa}(\vec A_0,\vec A_2)=e^{\i\theta}$. Then let
\[
\Hc_{\theta,a}:=
\left\{x\in\Z^d:
\operatorname{dist}(\langle\kappa,x\rangle-\theta,2\pi\Z)\leq a
\right\}.
\]

We first bound $\nu(\Hc_{\theta,a})$. 
To this end, let $R:=\X(p)^{1/2+\min(\eps,\eps')/8}$ and $B_R:=\{x\in\Z^d:\|x\|_\infty\leq R\}$.
Using \eqref{eq:initialization} and Lemma
\ref{lem:NotAligned}, we have
\begin{align*}
\nu(\Hc_{\theta,a}\cap B_R)
&\leq \frac{1}{\Omega\X(p)}
\sum_{\substack{x\in\Hc_{\theta,a}\\ \|x\|_\infty\leq R}}
(\tau_p\ast D)(x) \\
&\leq \frac{K}{\Omega\X(p)}
R(1+\|\kappa\|_\infty R)\left(1+a\|\kappa\|_\infty^{-1}\right) \\
&\leq \frac{K}{\Omega\X(p)}
\X(p)^{1/2+\min(\eps,\eps')/8}(1+\X(p)^{1/2+\eps/8-\eps}+3\X(p)^{1/2-\eps'/2}+3\X(p)^{1/2+\eps'/8-\eps'/2})
\\ & \leq \delta/4,
\end{align*}
as long as $\X(p)$ is large enough. Then, applying Theorem
\ref{thm:ExponentialTail} with parameter $\min(\eps,\eps')/9$, we get, as long as $\X(p)$ is large enough,
\[\nu(\Hc_{\theta,a}\cap B_R^c)\leq \delta/4. \]
Therefore, as long as $\X(p)$ is large enough,
\[ \nu(\Hc_{\theta,a})\leq \delta/2\leq P_2^\otimes(\vec A_0,\vec A_2)/2. \]

Next, note that 
\[
P_2^\otimes(\vec A_0,\vec A_2)\Amp_{2,\kappa}(\vec A_0,\vec A_2)e^{\i\theta} 
= \sum_{x\in\Z^d}e^{\i\langle\kappa,x\rangle}\nu(x).
\]
Thus, by multiplying by $e^{-\i\theta}$ and taking the real part, we obtain
\begin{align*}
P_2^\otimes(\vec A_0,\vec A_2)\Amp_{2,\kappa}(\vec A_0,\vec A_2)
&=\sum_{x\in\Z^d}
\cos(\langle\kappa,x\rangle-\theta)\nu(x)\\
&\leq\nu(\Hc_{\theta,a})
+\cos(a)\big(P_2^\otimes(\vec A_0,\vec A_2)-\nu(\Hc_{\theta,a})\big)\\
&\leq P_2^\otimes(\vec A_0,\vec A_2)-\frac{P_2^\otimes(\vec A_0,\vec A_2)}{2}(1-\cos(a)).
\end{align*}
Using $a\leq 1$, whenever $\X(p)$ is large enough, we have $1-\cos(a)\geq a^2/4$. Hence,
\[ P_2^\otimes(\vec A_0,\vec A_2)\Amp_{2,\kappa}(\vec A_0,\vec A_2) \leq P_2^\otimes(\vec A_0,\vec A_2)(1-\X(p)^{-\eps'}\min(1,\X(p)\|\kappa\|_\infty^2)).\]
The desired result follows by dividing by $P_2^\otimes(\vec A_0,\vec A_2)\neq 0$.
\end{proof}

\begin{proof}[Proof of Theorem \ref{thm:FT_Middle}]
First, by integrating \eqref{eq:RewriteA_TauSquare}, we have
\[
\Four[\tau_p^{\square,k}](\kappa)= \Omega^k \X(p)^{k+1} \E_p^\otimes[e^{\i\langle \SA_k-\delta(\vec \A_k),\kappa\rangle},\Tcut>k].
\]
We condition on $\vA_0,\vA_2,\vA_4,\dots, \vA_{2I}$ where $I$ is the largest integer such that $2I<k$, so that we may rewrite the last expectation as
\[
\E_p^\otimes  \left [
e^{-\i\langle \delta(\vA_k),\kappa\rangle }
\prod_{j=0}^{I} e^{\i \langle \Delta(\vA_{2j}),\kappa \rangle}
\prod_{j=0}^{I-1}[P_2^\otimes\cdot \Amp_{2,\kappa}\cdot  \Rot_{2,\kappa}](\vA_{2j},\vA_{2j+2})
\prod_{2I<j\leq k} e^{\i \langle \Delta(\vA_j),\kappa \rangle} \1_{(\vA_{j-1},\vA_j)\notin \Ccut}
\right ].
\]
Thus, by the triangle inequality,
\begin{equation}
|\Four[\tau_p^{\square,k}](\kappa)|
\leq \Omega^k \X(p)^{k+1}
\E_p^\otimes \left [
\prod_{j=0}^{I-1}[P_2^\otimes\cdot \Amp_{2,\kappa}](\vA_{2j},\vA_{2j+2})
\prod_{j=2I+1}^k \1_{(\vA_{j-1},\vA_j)\notin \Ccut}
\right ]. \label{eq:Fourrier_Partial_Cond_Middle}
\end{equation}

By Theorem \ref{thm:LowerLambda}, $\lambda_p$ is bounded below by a constant depending only on $d,L,C$. Thus, we may choose $\delta>0$ small enough that $\delta\leq \lambda_p^4/2$ whenever $\X(p)$ is large enough.
Let 
 \[ S:=\{(\vec A_0,\vec A_2)\in (\vec \Comp(0))^2,  P_2^\otimes(\vec A_0,\vec A_2) \geq \delta\}.\]
We write $N$ for the number of pairs $(\vA_{2j},\vA_{2j+2})$ with $0\leq j <I$ in $S$.
We split the expectation in \eqref{eq:Fourrier_Partial_Cond_Middle} according to the two cases $N\geq I/2$ and $N<I/2$, and we denote the corresponding expectations by $E_1$ and $E_2$.
For $E_1$, applying Proposition \ref{pro:MiddleRotation} at least $I/2$ times, we get
\begin{align*}
E_1 & \leq \E_p^\otimes \left [(1-\X(p)^{-\eps'}\min(1,\X(p)\|\kappa\|_\infty^{2}))^{I/2}\prod_{j=0}^{I-1}P_2^\otimes(\vA_{2j},\vA_{2j+2}) \prod_{j=2I+1}^k \1_{(\vA_{j-1},\vA_j)\notin \Ccut} \right ]
\\ & =(1-\X(p)^{-\eps'}\min(1,\X(p)\|\kappa\|_\infty^{2}))^{I/2}\proba_p^\otimes(\Tcut>k) .
\end{align*}

For $E_2$, using $0\leq P_2^\otimes\leq 1$ and $0\leq \Amp_{2,\kappa}\leq 1$ pointwise
and $P_2^\otimes(\vA_{2j},\vA_{2j+2})< \delta$ for at least $I/2$ values of $j$, we have $E_2\leq \delta^{I/2}.$
Lemma \ref{lem:LambdaLowMemory} and $\delta\leq \lambda_p^4/2$ then give
\[
E_2 \leq (1/2)^{I/2} \lambda_p^{2I} \leq (1/2)^{I/2}\proba_p^\otimes(\Tcut>k)/\lambda_p^4.
\]

By \eqref{eq:RewriteA_TauSquare}, $\|\tau_p^{\square,k}\|_1=\Omega^k\X(p)^{k+1}\proba_p^\otimes(\Tcut>k)$. Summing the bounds on $E_1$ and $E_2$ thus gives
\[
|\Four[\tau_p^{\square,k}](\kappa)|\leq  \left ((1-\X(p)^{-\eps'}\min(1,\X(p)\|\kappa\|_\infty^2))^{I/2}+(1/2)^{I/2}/\lambda_p^4 \right ) \|\tau_p^{\square,k}\|_1.
\]
By Theorem \ref{thm:LowerLambda}, $\lambda_p$ is bounded below by a constant, and the desired result follows by taking $\X(p)$ large enough.
 \end{proof}
\subsection{Low frequencies for $\Four[\tau_{p}^{\square,k}]$}\label{sec:LowFrequencies}
The objective of this subsection is to show the following low-frequency estimate for $\Four[\tau_{p}^{\square,k}]$.
For every $i\geq 0$, let $V_i(p)=\E_p^\vpg[\langle \Delta(\vA_0),\Delta(\vA_i)\rangle]$ and let $V(p)=V_0(p)+2\sum_{i\geq 1} V_i(p)$.
\begin{theorem} \label{thm:FT_Low}
Let $C>0,\eps,\eps'>0$. There exists $F>0$ depending only on $d,L,C,\eps,\eps'$ such that for every $p<p_c$ satisfying \eqref{eq:initialization} and $\X(p)\geq F$, for every $1\leq k \leq \X(p)^{1/\eps'}$,
 for every $\kappa\in [-\pi,\pi]^{d}$, 
\[
\left |\frac{\Four[\tau_p^{\square,k}](\kappa)}{\|\tau_p^{\square,k}\|_1}-e^{-k\|\kappa\|_2^2V(p)/(2d)} \right | \leq \X(p)^{-1/\eps'}+\X(p)^{1/2+\eps}k^{1/4}\|\kappa\|_2+\|\kappa\|_2^4k^{3/2}\X(p)^{2+\eps}.
\]
\end{theorem}
The proof of Theorem \ref{thm:FT_Low} is organized as follows.
We use a standard block-decomposition method for central limit estimates.
The main idea is to split the interval $[0,k]$ into well-spaced intervals, called blocks, and a set of bad indices.
This splits $\SA_k$ into a sum of essentially independent subsums, one for each block, plus the contribution of the bad indices.
We first upper-bound the contribution of the bad indices in Lemma \ref{lem:SumBadIndicesBlock}.
We then justify the approximation by independent blocks in Lemmas \ref{lem:EndingFourier} and \ref{lem:BlockDecomposition}.
After that, we prove a one-block central limit in Proposition \ref{pro:TCL_One_Block}. We do this by computing the second and fourth moments of each block in Lemmas \ref{lem:SecondMomentBlock} and \ref{lem:FourthMomentBlock}, respectively.
Finally, we combine these estimates to prove Theorem \ref{thm:FT_Low}.

We first remove the indices not belonging to the blocks.
For $S\subset [0,k]\cap \Z$, set
\[ \Four_{S,k}(\kappa):=\lambda_p^{-k}\E_p^\otimes[e^{\i\langle \SA_S,\kappa\rangle }\1_{\Tcut>k}], \]
where $\SA_S:=\sum_{s\in S}\Delta(\vec \A_s)$. Note that we omit the dependence on $p$ from the notation.
We also let $\N_0=\{0\}\cup \N$, and for $a<b\in \Z$, let  $\llbracket a, b\rrbracket:=[a,b]\cap \Z$. 
\begin{lemma} \label{lem:SumBadIndicesBlock} Let $C>0,\eps>0$. There exists $F$ depending on $d,L,C,\eps$ such that for every $p<p_c$ satisfying \eqref{eq:initialization} and $\X(p)\geq F$, for every $k\geq 0$, $S\subset \llbracket 0,k \rrbracket$, $\kappa\in [-\pi,\pi]^d$, we have
  \[ |\Four[\tau_{p}^{\square,k}](\kappa) -\Omega^k\lambda_p^k \X(p)^{k+1}\Four_{S,k}(\kappa)|\leq (L+\sqrt{k+1-|S|}\X(p)^{1/2+\eps})\|\kappa\|_2 \|\tau_{p}^{\square,k}\|_1. \]
\end{lemma}
\begin{proof} By \eqref{eq:RewriteA_TauSquare}, 
\[ \Four[\tau_{p}^{\square,k}](\kappa) -\Omega^k\lambda_p^k \X(p)^{k+1} \Four_{S,k}(\kappa)= \Omega^k \X(p)^{k+1}\E_p^\otimes[(e^{\i\langle \SA_S,\kappa\rangle }-e^{\i\langle \SA_k-\delta(\vA_k),\kappa\rangle })\1_{\Tcut>k}].\]
Thus, since $y\mapsto e^{\i\langle \kappa,y\rangle}$ is $\|\kappa\|_2$-Lipschitz, we have 
\begin{equation} \left |\Four[\tau_{p}^{\square,k}](\kappa) -\Omega^k\lambda_p^k \X(p)^{k+1}\Four_{S,k}(\kappa) \right | \leq  \Omega^k \X(p)^{k+1} \|\kappa\|_2 \E_p^\otimes[(L+\|\SA_S-\SA_k\|_2)\1_{\Tcut>k}].\label{eq:RemovePartFourier}\end{equation}
The term $L$ simply corresponds, by \eqref{eq:RewriteA_TauSquare}, to the term $L\|\kappa\|_2\|\tau_p^{\square,k}\|_1$ in the final bound. The Cauchy--Schwarz inequality yields a bound for the main term by 
\[ \Omega^k \X(p)^{k+1} \sqrt{\E_p^\otimes[\|\SA_S-\SA_k\|^2_2 \1_{\Tcut>k}]\proba_p^\otimes(\Tcut>k)}.\]
Fubini's theorem, justified by Lemma \ref{lem:BoundedConditionalMoments}, then gives
\[ \E_p^\otimes[\|\SA_S-\SA_k\|^2_2\1_{\Tcut>k}]= \sum_{s\in \llbracket 0,k\rrbracket\backslash S}\sum_{s'\in \llbracket 0,k\rrbracket\backslash S}\E_p^\otimes[\langle \Delta(\vA_s),\Delta(\vA_{s'})\rangle\1_{\Tcut>k}].\]
For each fixed $s$, we upper-bound the sum over $s'$ as follows. When $|s-s'|\leq \X(p)^{\eps/3}$, Corollary \ref{lem:BoundVariance} gives
\[ \E_p^\otimes[\langle \Delta(\vA_s),\Delta(\vA_{s'})\rangle\1_{\Tcut>k}]\leq \X(p)^{1+\eps/3}\proba_p^\otimes(\Tcut>k). \] 
When $|s-s'|\geq \X(p)^{\eps/3}$, Lemma \ref{lem:CovarianceDecay} gives
\[\E_p^\otimes[\langle \Delta(\vA_s),\Delta(\vA_{s'})\rangle\1_{\Tcut>k}]\leq c^{|s'-s|-1}\X(p)^{10}\proba_p^\otimes(\Tcut>k)\leq \X(p)^{1+\eps/3}\proba_p^\otimes(\Tcut>k),\]
 for some constant $c<1$ depending only on $d,L,C$. Therefore, by summing, whenever $\X(p)$ is large enough, 
\[ \E_p^\otimes[\|\SA_S-\SA_k\|^2_2\1_{\Tcut>k}]\leq \sum_{s\in \llbracket 0,k\rrbracket\backslash S} \X(p)^{1+\eps}\proba_p^\otimes(\Tcut>k).\]
Hence, the main term in \eqref{eq:RemovePartFourier} is bounded by
\[ \Omega^k \X(p)^{k+1} \|\kappa\|_2\sqrt{(k+1-|S|)\X(p)^{1+\eps}}\proba_p^\otimes(\Tcut>k)=\sqrt{(k+1-|S|)}\|\kappa\|_2 \X(p)^{1/2+\eps/2}\|\tau_{p}^{\square,k}\|_1, \]
using again \eqref{eq:RewriteA_TauSquare} for the last equality. This concludes the proof.
\end{proof} 
We next replace the terminal cut conditioning by the endpoint weight that is adapted to the Doob transform.
For every set $S\subset \{0\}\cup\N$, let $S^+:=\max S$, $S^-:=\min S$, and 
\[ \Four^h_{S}(\kappa):=\lambda_p^{-S^+}\E_p^\otimes[e^{\i\langle \SA_S,\kappa\rangle }h_p(\vA_{S^+})\1_{\Tcut>S^+}]. \]
By convention, when $S=\emptyset$, we let $S^-=S^+=0$ and $\Four^h_\emptyset(\kappa):=\E_p^\otimes[h_p(\vA_{0})]$.
\begin{lemma} \label{lem:EndingFourier} Let $C>0$. There exists $F>0,c<1$ depending only on $d,L,C$ such that for every $p<p_c$ satisfying \eqref{eq:initialization} and $\X(p)\geq F$, for every $k\in \N$, $S\subset \llbracket 0,k-1 \rrbracket$, $\kappa\in [-\pi,\pi]^d$, we have 
\[ \left | \Four_{S,k}(\kappa)-\frac{\proba_p^\otimes(\Tcut>k)}{\lambda_p^{k}\E_p^\otimes[h_p(\vA_0)]}\Four^h_{S}(\kappa) \right | \leq c^{k-1-S^+}\X(p)^{8}. \]
\end{lemma}
\begin{proof} We first apply Proposition \ref{pro:MixingTimeFunc} with $F_1:(\vec A_0,\dots, \vec A_{S^+})\mapsto e^{\i\langle \sum_{s\in S}\Delta(\vec A_s),\kappa\rangle }$, and $F_2:\vec A_{k-1}\mapsto \proba_p^\otimes((\vec A_{k-1},\vA_k)\notin \Ccut)$.
These functions satisfy $\mathfrak{S}_p^-(F_1)=\lambda_p^{S^+}\Four_S^h(\kappa)$, and $\mathfrak{S}_p^+(F_2)=\E_p^\vpg[\Gamma_p(\vA_0)]/(\Omega\X(p))$, and $\|F_1\|^-\leq \proba_p^\otimes(\Tcut>S^+)$, and $\|F_2\|^+\leq 1$. Thus, Proposition \ref{pro:MixingTimeFunc} yields
\[ \left | \Four_{S,k}(\kappa)-\frac{\E_p^\vpg[\Gamma_p(\vA_0)]}{\Omega \lambda_p\X(p)}\Four^h_{S}(\kappa) \right | \leq c^{k-1-S^+}\X(p)^5. \]
Furthermore, Lemma \ref{lem:BadLowerTcut} gives
$|\Four^h_S(\kappa)|\leq \lambda_p^{-S^+} \|h_p\|_\infty\proba_p^\otimes(\Tcut>S^+)\leq \|h_p\|_\infty \X(p)^2.$
So, by Lemma \ref{lem:ForgetPosForTcut}, we have
\[ \left | \frac{\proba_p^\otimes(\Tcut>k)}{\lambda_p^{k}\E_p^\otimes[h_p(\vA_0)]}\Four^h_{S}(\kappa)-\frac{\E_p^\vpg[\Gamma_p(\vA_0)]}{\Omega \lambda_p\X(p)}\Four^h_{S}(\kappa) \right | \leq c^{k-1}\X(p)^5\frac{\|h_p\|_\infty}{\E_p^\otimes[h_p(\vA_0)]} . \]
Lemma \ref{lem:vpd} bounds the last fraction by $\X(p)^2$ whenever $\X(p)$ is large enough. The desired inequality is obtained by combining the last two displays using the triangle inequality.
\end{proof}
We then separate the contribution of two well-spaced sets of indices.
For finite sets $S,I\subset \N_0$, we write $S\preceq I$ whenever $S^+< I^-$. Then let 
\[ \Four^v_{S}(\kappa):=\E_p^\vpg[e^{\i\langle \SA_S,\kappa\rangle }]. \]
\begin{lemma} \label{lem:BlockDecomposition} Let $C>0$. There exists $F>0,c<1$ depending only on $d,L,C$ such that for every $p<p_c$ satisfying \eqref{eq:initialization} and $\X(p)\geq F$, all finite subsets $S,I$ of $\N_0$ with $S\preceq I$, and every $\kappa\in [-\pi,\pi]^d$, we have
\[ \left |\Four^h_{S\cup I}(\kappa)-\Four^h_{S}(\kappa)\Four^v_{I}(\kappa) \right | \leq c^{I^- - S^+} \|h_p\|_\infty \X(p)^{10}.\]
\end{lemma}
\begin{proof}
We apply Proposition \ref{pro:MixingTimeFunc} with
\[
\begin{aligned}
F_1:(\vec A_0,\dots, \vec A_{S^+})
&\mapsto e^{\i\langle \sum_{s\in S}\Delta(\vec A_s),\kappa\rangle },\\
F_2:(\vec A_{I^-},\dots, \vec A_{I^+})
&\mapsto e^{\i\langle \sum_{s\in I}\Delta(\vec A_s),\kappa\rangle }h_p(\vec A_{I^+}).
\end{aligned}
\]
Note that
\[
\mathfrak{S}_p^-(F_1)=\lambda_p^{S^+}\Four_S^h(\kappa),
\qquad
\|F_1\|^-\leq \proba_p^\otimes(\Tcut>S^+).
\]
A bit more care is required for $F_2$.
We have
\[
\mathfrak{S}_p^+(F_2)
=
\sum_{\vec A\in \Clive}
\frac{\vpg_p(\vec A)}{h_p(\vec A)}
\E_p^\otimes[
e^{\i\langle \SA_{I-I^-},\kappa\rangle }
h_p(\vec A_{I^+-I^-})
\1_{\Tcut>I^+-I^-}
|\vA_0=\vec A],
\]
where $I-I^-$ denotes the natural translate of $I$ by $I^-$.
Using the Doob transform, this becomes
\[
\mathfrak{S}_p^+(F_2)
=
\lambda_p^{I^+-I^-}
\E_p^\vpg[
e^{\i\langle \SA_{I-I^-},\kappa\rangle }
].
\]
Since $(\vA_i)_{i\geq 0}$ is a Markov chain under $\proba_p^h$, we recognize the last expectation as $\Four_{I}^v(\kappa)$.
Also,
\[
\|F_2\|^+
=
\sup_{\vec A\in \Clive}
\E_p^\otimes[
| e^{\i\langle \SA_{I-I^-},\kappa\rangle }
h_p(\vec A_{I^+-I^-})|
\1_{\Tcut>I^+-I^-}
|\vA_0=\vec A]
\leq
\proba_p^\otimes(\Tcut>I^+-I^-)\|h_p\|_\infty .
\]
Therefore, by Proposition \ref{pro:MixingTimeFunc},
\[
\left |\Four^h_{S\cup I}(\kappa)
-\Four^h_{S}(\kappa)\Four^v_{I}(\kappa) \right |\\
\leq
\lambda_p^{-I^+}
(c\lambda_p)^{I^- - S^+}
\proba_p^\otimes(\Tcut>S^+)
\proba_p^\otimes(\Tcut>I^+-I^-)
\|h_p\|_\infty\X(p)^3.
\]
Lemma \ref{lem:BadLowerTcut} bounds the last two probabilities by $\lambda_p^{S^++1}\X(p)^2$ and $\lambda_p^{I^+-I^-+1}\X(p)^2$.
Computing the last product concludes the proof.
\end{proof}
It remains to estimate the contribution of one large block.
\begin{proposition} \label{pro:TCL_One_Block}Let $C>0,\eps,\eps'>0$. There exists $F>0$ depending only on $d,L,C,\eps,\eps'$ such that for every $p<p_c$ satisfying \eqref{eq:initialization} and $\X(p)\geq F$, for every interval $I\subset \N_0$ of length $|I|\leq \X(p)^{1/\eps'}$,
 for every $\kappa\in [-\pi,\pi]^{d}$, we have 
\[
\left |\Four^v_{I}(\kappa)-e^{-\|\kappa\|_2^2|I|V(p)/(2d)} \right | \leq \|\kappa\|_2^2\X(p)^{1+\eps}+\|\kappa\|_2^4|I|^2\X(p)^{2+\eps}.
\]
\end{proposition}
To show Proposition \ref{pro:TCL_One_Block}, we approximate $\Four^v_{I}(\kappa)$ by a Taylor expansion.
By symmetry, we have 
\[ \Four^v_{I}(\kappa)=\E_p^\vpg[\cos(\langle \SA_I,\kappa\rangle)].\] 
Taylor's formula applied to $\cos$ gives
\[ \left | \Four^v_{I}(\kappa) -\left (1- \frac{\E_p^\vpg[\langle \kappa,\SA_I\rangle^2]}{2} \right ) \right |\leq \frac{\E_p^\vpg[\langle \kappa,\SA_I\rangle^4]}{24}.\]
Since for every $x\in \R_+$ we have $|e^{-x}-(1-x)|\leq x^2/2$, we get, by the Cauchy--Schwarz inequality,
\begin{equation} \left | \Four^v_{I}(\kappa) -e^{- \E_p^\vpg[\langle \kappa,\SA_I\rangle^2]/2} \right |\leq \frac{\E_p^\vpg[\langle \kappa,\SA_I\rangle^4]}{24}+\frac{\E_p^\vpg[\langle \kappa,\SA_I\rangle^2]^2}{8}\leq \frac{\E_p^\vpg[\langle \kappa,\SA_I\rangle^4]}{6}. \label{eq:Taylor1block}\end{equation}
We first identify the second moment.
By the $\Z^d$ symmetry, we have $\E_p^\vpg[\langle \kappa,\SA_I\rangle^2]=\|\kappa\|_2^2 \E_p^\vpg[\|\SA_I\|^2_2]/d$. Hence, to show Proposition \ref{pro:TCL_One_Block}, 
it is enough to prove $|\E_p^\vpg[\|\SA_I\|^2_2]-|I| V(p)| \leq \X(p)^{1+\eps}$ and $\E_p^\vpg[\langle \kappa,\SA_I\rangle^4] \leq \|\kappa\|_2^4|I|^2\X(p)^{2+\eps}$, which are the contents of Lemmas \ref{lem:SecondMomentBlock} and \ref{lem:FourthMomentBlock}, respectively.
\begin{lemma} \label{lem:V(p)}\label{lem:SecondMomentBlock} Let $C>0,\eps>0$. There exists $F>0$ depending on $d,L,C,\eps$ such that for every $p$ satisfying \eqref{eq:initialization} and $\X(p)\geq F$, we have $V(p)\leq \X(p)^{1+\eps}$. Moreover, for every interval
   $I$, we have
   \[ \left |\E_p^\vpg[\|\SA_I\|^2_2]-|I| V(p) \right | \leq \X(p)^{1+\eps}.\]
\end{lemma}
\begin{proof} Since $\vA_i$ is a stationary Markov chain under $\proba_p^\vpg$, we have
  \[ \left |\E_p^\vpg[\|\SA_I\|^2_2]-|I|V(p) \right |= \left | \sum_{i,j\in I} V_{|j-i|}(p)-|I| \left ( V_0+2\sum_{i\geq 1} V_i(p) \right ) \right | \leq \sum_{i\geq 0} 9(i+1) |V_i(p)|. \]
  Lemma \ref{lem:Bulk_L2} and Corollary \ref{lem:BoundVariance} give, when $\X(p)$ is large enough, that $|V_i(p)|\leq \X(p)^{1+\eps/2}$ for every $i\geq0$.
  So the contribution of $i\leq \X(p)^{\eps/3}$ to the last sum is $o(\X(p)^{1+\eps})$.
  Also, by Lemmas \ref{lem:BadLowerTcut} and \ref{lem:Bulk_L2}, there exists $c<1$ depending only on $d,L,C$ such that, if $\X(p)$ is large enough, we have $|V_i(p)|\leq c^i \X(p)^{10}$.
   Thus, the contribution of $i> \X(p)^{\eps/3}$ to the last sum is also $o(\X(p)^{1+\eps})$. The same argument also upper-bounds $V(p)$.
\end{proof}

We then bound the fourth moment.
To this end, we use a stationary analogue of Proposition \ref{pro:MixingTimeFunc}.
For $F$ from $(\vec\Comp(0))^{n+1}$ to either $\R$ or $\R^d$, we write 
\[ \mathfrak{S}^\vpg_p(F):=\E_p^\vpg[F(\vA_0,\dots,\vA_n)]. \]
We also let 
\[ \|F\|_\vpg^-:=\E_p^\vpg[\|F(\vA_0,\dots,\vA_n)\|_2], \]
and 
\[ \|F\|^+_\vpg:=\sup_{\vec A\in \Clive}\E_p^\vpg[\|F(\vA_0,\dots,\vA_n)\|_2|\vA_0=\vec A].\]
\begin{proposition} \label{pro:MixingTimeFuncVpg} Let $p<p_c$ and $C>0$ be such that \eqref{eq:initialization} holds. 
There exist $0<c<1$ and $K>0$, depending on $d,L,C$, such that if $\X(p)\geq K$, then for every $0\leq i \leq j\leq k$, every $F_1$ from $(\vec \Comp(0))^{i+1}$, and every $F_2$ from $(\vec \Comp(0))^{k-j+1}$ to either $\R$ or $\R^d$, we have
\[
\left \|  \E_p^\vpg[\langle F_1(\vA_0,\dots,\vA_i),F_2(\vA_j,\dots, \vA_k)\rangle]-\langle\mathfrak{S}^\vpg_p(F_1),\mathfrak{S}^\vpg_p(F_2)\rangle\right \|_2 \leq c^{j-i}\X(p) \|F_1\|^-_\vpg\|F_2\|^+_\vpg.
\]
\end{proposition}
\begin{remark} By Lemma \ref{lem:Symmetric_hp_vpg}, if $F$ is antisymmetric, then $\mathfrak{S}^\vpg_p(F)=0$.
\end{remark}
\begin{proof} 
For $\vec A\in \vec\Comp(0)$, let $f(\vec A):=\E_p^\vpg[F_2(\vA_0,\dots,\vA_{k-j})|\vA_0=\vec A].$
Then $\|f\|_{2,\infty}\leq \|F_2\|^+_\vpg$.
By conditioning on $\vA_i$, we have 
\[
\E_p^\vpg[\langle F_1(\vA_0,\dots,\vA_i),F_2(\vA_j,\dots, \vA_k)\rangle]\\
=
\E_p^\vpg\left[
\left\langle F_1(\vA_0,\dots,\vA_i),
\E_p^\vpg[f(\vA_j)|\vA_i]
\right\rangle
\right].
\]
The use of Fubini's theorem is justified whenever $\|F_1\|_\vpg^-<\infty$ and $\|F_2\|_\vpg^+<\infty$.
We also have 
\[ \langle\mathfrak{S}^\vpg_p(F_1),\mathfrak{S}^\vpg_p(F_2)\rangle= \E_p^\vpg\left[
\left\langle F_1(\vA_0,\dots,\vA_i),\mathfrak{S}^\vpg_p(F_2)
\right\rangle
\right].\]
Thus, taking the difference between the two previous right-hand sides and applying first the triangle inequality and then the Cauchy--Schwarz inequality, the difference is bounded by
\[ \E_p^\vpg[\|F_1(\vA_0,\dots,\vA_i)\|_2]\max_{\vec A_i\in \Clive} \left \|\E_p^\vpg[f(\vA_j)|\vA_i=\vec A_i ]-\mathfrak{S}^\vpg_p(F_2) \right \|_2.  \]
We recognize the first expectation as $\|F_1\|_\vpg^-$. As for the second, by conditioning on $\vA_j$, we have for every $\vec A_i$, 
\[ \E_p^\vpg[f(\vA_j)|\vA_i=\vec A_i ]-\mathfrak{S}^\vpg_p(F_2)=\sum_{\vec A_j\in \Clive}(\P_p^{j-i}(\vec A_j|\vec A_i)-\vpg_p(\vec A_j))\E_p^\vpg[F_2(\vA_j,\dots,\vA_k)|\vA_j=\vec A_j].  \]
The use of Fubini's theorem to get the previous equality is justified as long as $\|F_2\|^+_\vpg<\infty$. Moreover, using the triangle inequality, we deduce 
\[ \| \E_p^\vpg[f(\vA_j)|\vA_i=\vec A_i ]-\mathfrak{S}^\vpg_p(F_2)\|_2\leq \sum_{\vec A_j\in \Clive}|\P_p^{j-i}(\vec A_j|\vec A_i)-\vpg_p(\vec A_j)| \|F_2\|^+_\vpg\leq c^{j-i}\X(p)\|F_2\|^+_\vpg,  \]
using Theorem \ref{thm:MixingTime} for the last inequality. Taking the maximum over $\vec A_i$ concludes the proof.
\end{proof}
We apply this mixing estimate to the following products.
For $i_1,i_2,\dots,i_n\in \N_0$, let us write $$W_{\kappa,p}(i_1,i_2,\dots,i_n):=\E_p^\vpg\left[\prod_{j=1}^n\langle \kappa,\Delta(\vA_{i_j})\rangle \right ].$$
\begin{lemma} \label{lem:ScalarMoment} Let $C>0,\eps>0,n\in \N$. There exists $F>0$ depending on $d,L,C,n,\eps$ such that for every $p$ satisfying \eqref{eq:initialization} and $\X(p)\geq F$, for every $i_1,i_2,\dots,i_n\in \N_0$, for every $\kappa\in [-\pi,\pi]^d$ we have 
  \[ |W_{\kappa,p}(i_1,i_2,\dots,i_n)|\leq \E_p^\vpg\left [\prod_{j=1}^n|\langle \kappa,\Delta(\vA_{i_j}) \rangle| \right ] \leq \|\kappa\|_2^n\X(p)^{n/2+\eps}, \]
  and, if $i_1,\dots,i_n\geq 1$, for every $\vec A_0\in \Clive$, we have
\[ \E_p^\vpg\left [\prod_{j=1}^n|\langle \kappa,\Delta(\vA_{i_j})\rangle| \mid \vA_0=\vec A_0 \right ]\leq \|\kappa\|_2^n\X(p)^{n/2+\eps}. \]
\end{lemma}
\begin{proof}  
 Since $(\vA_i)_{i\geq 0}$ is a stationary Markov chain under $\proba_p^\vpg$, we have $W_{\kappa,p}(i_1,i_2,\dots,i_n)=W_{\kappa,p}(i_1+1,i_2+1,\dots,i_n+1)$, so the first inequality follows by integrating the second over $\vec A_0$. 
  For the second, by convexity of $x\mapsto x^n$ and the Cauchy--Schwarz inequality, we have
\begin{align*} \E_p^\vpg [\prod_{j=1}^n|\langle \kappa,\Delta(\vA_{i_j})\rangle | |\vA_0=\vec A_0 ] 
  & \leq \max_{i\geq 1}\E_p^\vpg[|\langle \kappa, \Delta(\vA_i)\rangle|^n | \vA_0=\vec A_0 ]
  \\ & \leq \|\kappa\|_2^n \max_{i\geq 1}\E_p^\vpg[\|\Delta(\vA_i)\|_2^n | \vA_0=\vec A_0 ]. 
\end{align*}
Then, by Lemma \ref{lem:OneStepBias} and Theorem \ref{thm:ExponentialTail}, the contribution of the case $\|\Delta(\vA_i)\|_2\geq \X(p)^{1/2+\eps/(2n)}$
to the last expectation is at most
\[ K'\sum_{x\in \Z^d}\tau_p(x)\langle x\rangle^n\1 \left (\langle x\rangle\geq \X(p)^{1/2+\eps/(2n)}-L\right )\leq 1, \]
for some constant $K'>0$ depending only on $d,L,C$. The rest of the expectation is bounded by $\X(p)^{n/2+\eps/2}$. This concludes the proof.
\end{proof}
The next estimate gives the gap decay needed for the fourth moment.
\begin{lemma} \label{lem:FastDecayFourthMoment} Let $C>0$. There exists $F>0,c<1$ depending only on $d,L,C$ such that for every $p<p_c$ satisfying \eqref{eq:initialization} and $\X(p)\geq F$, for every $i_1\leq i_2\leq i_3 \leq i_4$, for every $\kappa\in [-\pi,\pi]^d$, we have 
  \[ |W_{\kappa,p}(i_1,i_2,i_3,i_4)|\leq \|\kappa\|_2^4 \X(p)^5 \min(c^{i_2-i_1},c^{i_4-i_3}).\]
\end{lemma}
\begin{proof} When $i_1<i_2-2$, we first apply Proposition \ref{pro:MixingTimeFuncVpg} with
  \[ F_1:(\vec A_0,\dots, \vec A_{i_1})\mapsto \langle \kappa, \Delta(\vec A_{i_1})\rangle  \quad \text{and} \quad F_2:(\vec A_{i_2-1},\dots, \vec A_{i_4})\mapsto \prod_{2\leq j \leq 4}\langle \kappa, \Delta(\vec A_{i_j})\rangle. \]
By antisymmetry, we have $\mathfrak{S}^\vpg_p(F_1)=\mathfrak{S}^\vpg_p(F_2)=0$. By Lemma \ref{lem:ScalarMoment} we have $\|F_1\|_\vpg^-\leq \|\kappa\|_2\X(p)$ and $\|F_2\|_\vpg^+\leq \|\kappa\|_2^3\X(p)^2$ as long as $\X(p)$ is large enough. Thus,
\[ |W_{\kappa,p}(i_1,i_2,i_3,i_4)|\leq \|\kappa\|_2^4\X(p)^5 c^{i_2-i_1}, \]
whenever $i_1<i_2-2$. The inequality still holds when $i_2\in \{i_1,i_1+1,i_1+2\}$ by Lemma \ref{lem:ScalarMoment}. Similarly, by applying Proposition \ref{pro:MixingTimeFuncVpg} with
  \[ F'_1:(\vec A_0,\dots, \vec A_{i_3})\mapsto \prod_{1\leq j \leq 3}\langle \kappa, \Delta(\vec A_{i_j})\rangle  \quad \text{and} \quad F'_2:(\vec A_{i_4-1},\dots, \vec A_{i_4})\mapsto \langle \kappa, \Delta(\vec A_{i_4})\rangle, \]
  we get 
\[ |W_{\kappa,p}(i_1,i_2,i_3,i_4)|\leq \|\kappa\|_2^4 \X(p)^5 c^{i_4-i_3}. \qedhere\]
\end{proof}
\begin{lemma} \label{lem:FourthMomentBlock} Let $C>0,\eps,\eps'>0$. There exists $F>0$ depending on $d,L,C,\eps,\eps'$ such that for every $p$ satisfying \eqref{eq:initialization} and $\X(p)\geq F$, for every interval
   $I$ satisfying $|I|\leq \X(p)^{1/\eps'}$, for every $\kappa\in [-\pi,\pi]^d$, we have 
   \[ \E_p^\vpg[\langle \kappa,\SA_I\rangle^4] \leq \|\kappa\|_2^4|I|^2\X(p)^{2+\eps}.\]
\end{lemma}
\begin{proof} We have 
  \[ \E_p^\vpg[\langle \kappa,\SA_I\rangle^4]=\sum_{i_1,i_2,i_3,i_4\in I} W_{\kappa,p}(i_1,i_2,i_3,i_4)\leq 24 \sum_{\substack{i_1\leq i_2\leq i_3\leq i_4\\i_1,i_2,i_3,i_4\in I}} |W_{\kappa,p}(i_1,i_2,i_3,i_4)|.\]
When $\max(i_2-i_1,i_4-i_3)\leq \X(p)^{\eps/10}$, we use $|W_{\kappa,p}(i_1,i_2,i_3,i_4)|\leq \|\kappa\|_2^4 \X(p)^{2+\eps/2}$ from Lemma \ref{lem:ScalarMoment}. There are at most $|I|^2(1+\X(p)^{\eps/10})^2$ such tuples. When $\max(i_2-i_1,i_4-i_3)> \X(p)^{\eps/10}$, we instead use Lemma \ref{lem:FastDecayFourthMoment} to get
$|W_{\kappa,p}(i_1,i_2,i_3,i_4)|\leq \|\kappa\|_2^4 \X(p)^{-100/\eps'}$, whenever $\X(p)$ is large enough. There are at most $|I|^4$ such tuples. In total, whenever $\X(p)$ is large enough, we get
  \[ \E_p^\vpg[\langle \kappa,\SA_I\rangle^4]\leq 24\|\kappa\|_2^4(2\X(p)^{2+\eps/2+2\eps/10}|I|^2 +\X(p)^{-100/\eps'}|I|^4)\leq\|\kappa\|_2^4|I|^2\X(p)^{2+\eps}. \qedhere \]
\end{proof}
We now have all the elements to conclude the subsection.
\begin{proof}[Proof of Theorem \ref{thm:FT_Low}]
Without loss of generality, we may prove the main result with $\eps$ replaced by $10\eps$.
We may also assume that
\begin{equation}
\X(p)^{1/2+10\eps}k^{1/4}\|\kappa\|_2\leq 2
\quad \text{and} \quad
\|\kappa\|_2^4k^{3/2}\X(p)^{2+10\eps}\leq 2, \label{eq:HypoSimplifyThmLowFrequencies}
\end{equation}
since otherwise we can trivially bound the left-hand side of Theorem \ref{thm:FT_Low} by $2$.
Let $\lmix:=\lfloor\X(p)^\eps\rfloor$ and $\lblock:=\lfloor \sqrt{k}\rfloor -1$.
For $i\geq 1$, let
\[
a_i=i\lmix+(i-1)\lblock,
\qquad
b_i=i\lmix+i\lblock,
\qquad
I_i=[a_i,b_i].
\]
Let $q$ be the largest integer such that $a_{q+1}\leq k$.
We may assume $q\geq 1$, since otherwise $k\leq 4\X(p)^{2\eps}$ and the conclusion follows directly from Lemma \ref{lem:SumBadIndicesBlock} with $S=\emptyset$.
Let $S=\bigcup_{i=1}^q I_i$.
Then
\begin{equation}
k+1-|S|\leq 3\sqrt{k}\X(p)^\eps,
\qquad
q\leq 2\sqrt{k}. \label{eq:NumberBadIndices}
\end{equation}

The first step is to remove the indices outside the blocks.
By Lemmas \ref{lem:SumBadIndicesBlock} and \ref{lem:EndingFourier}, and using
$\|\tau_p^{\square,k}\|_1=\Omega^k\X(p)^{k+1}\proba_p^\otimes(\Tcut>k)$, we have
\[
\left |
\frac{\Four[\tau_p^{\square,k}](\kappa)}{\|\tau_p^{\square,k}\|_1}
-\frac{\Four^h_S(\kappa)}{\E_p^\otimes[h_p(\vA_0)]}
\right |
\leq
(L+\sqrt{k+1-|S|}\X(p)^{1/2+\eps})\|\kappa\|_2
+\frac{\lambda_p^k}{\proba_p^\otimes(\Tcut>k)}c^{\lmix-1}\X(p)^8 .
\]
Since $\lmix=\X(p)^\eps$, the last term is at most $\X(p)^{-1/\eps'-10}$ whenever $\X(p)$ is large enough.
Using the bound on $k+1-|S|$, we get
\begin{equation}
\left |
\frac{\Four[\tau_p^{\square,k}](\kappa)}{\|\tau_p^{\square,k}\|_1}
-\frac{\Four^h_S(\kappa)}{\E_p^\otimes[h_p(\vA_0)]}
\right |
\leq
\X(p)^{-1/\eps'-10}
+\X(p)^{1/2+2\eps}k^{1/4}\|\kappa\|_2 .
\label{eq:LowFrequencyRemoveBadIndices}
\end{equation}

The second step is to replace $\Four_S^h$ by a product of independent block contributions.
For $0\leq r\leq q$, set $S_{-r}:=\bigcup_{1\leq i\leq q-r} I_i$, with the convention $S_{-q}=\emptyset$.
By telescoping,
\[
\Four^h_{S}(\kappa)-\E_p^\otimes[h_p(\vA_{0})]\prod_{i=1}^q \Four^v_{I_i}(\kappa)
=
\sum_{j=0}^{q-1}
\prod_{0\leq i<j}\Four^v_{I_{q-i}}(\kappa)
\left (
\Four^h_{S_{-j}}(\kappa)
-\Four^h_{S_{-j-1}}(\kappa)\Four^v_{I_{q-j}}(\kappa)
\right ).
\]
Thus, by Lemma \ref{lem:BlockDecomposition} and since $|\Four^v_{I_i}(\kappa)|\leq 1$, we get
\begin{equation}
\left |
\frac{\Four^h_{S}(\kappa)}{\E_p^\otimes[h_p(\vA_{0})]}
-\prod_{i=1}^q \Four^v_{I_i}(\kappa)
\right |
\leq
q c^{\lmix}\frac{\|h_p\|_\infty}{\E_p^\otimes[h_p(\vA_{0})]}\X(p)^{10}
\leq
\X(p)^{-1/\eps'-10},
\label{eq:LowFrequencyIndependentBlocks}
\end{equation}
where the last inequality follows from Lemma \ref{lem:vpd} and $\lmix=\X(p)^\eps$, provided $\X(p)$ is large enough.

The third step is to approximate each block by its Gaussian factor.
Another telescoping gives
\[
\prod_{i=1}^q \Four^v_{I_i}(\kappa)
-e^{-\|\kappa\|_2^2|S| \frac{V(p)}{2d}}
=
\sum_{j=1}^q
e^{-(q-j)\|\kappa\|_2^2(\lblock+1)\frac{V(p)}{2d}}
\prod_{i=1}^{j-1} \Four^v_{I_i}(\kappa)
\left(
\Four^v_{I_j}(\kappa)
-e^{-\|\kappa\|_2^2(\lblock+1)\frac{V(p)}{2d}}
\right).
\]
Thus, using $|\Four^v_{I_i}(\kappa)|\leq 1$ and Proposition \ref{pro:TCL_One_Block}, we get
\begin{align}
\left |
\prod_{i=1}^q \Four^v_{I_i}(\kappa)
-e^{-\|\kappa\|_2^2|S| V(p)/(2d)}
\right |
&\leq
q\left(
\|\kappa\|_2^2\X(p)^{1+\eps}
+\|\kappa\|_2^4(\lblock+1)^2\X(p)^{2+\eps}
\right) \notag\\
&\leq
\X(p)^{1/2+\eps}k^{1/4}\|\kappa\|_2
+\|\kappa\|_2^4k^{3/2}\X(p)^{2+\eps},
\label{eq:LowFrequencyBlockGaussianApproximation}
\end{align}
using $\X(p)^{1/2}k^{1/4} \|\kappa\|_2\leq 1$, which follows from \eqref{eq:HypoSimplifyThmLowFrequencies}, for the last inequality.
It remains only to replace $|S|$ by $k$ in the Gaussian factor.
By Lemma \ref{lem:SecondMomentBlock} applied to an interval of length $1$, we have $V(p)\leq \X(p)^{1+\eps}$ whenever $\X(p)$ is large enough.
Since $x\mapsto e^{-x}$ is $1$-Lipschitz on $\R_+$, \eqref{eq:NumberBadIndices} gives
\begin{equation}
\left |
e^{-\|\kappa\|_2^2|S| V(p)/(2d)}
-e^{-k\|\kappa\|_2^2 V(p)/(2d)}
\right |
\leq
3\|\kappa\|_2^2 k^{1/2}\X(p)^{1+2\eps}
\leq
3\X(p)^{1/2+\eps}k^{1/4}\|\kappa\|_2,
\label{eq:LowFrequencyRestoreFullTime}
\end{equation}
again using $\X(p)^{1/2}k^{1/4} \|\kappa\|_2\leq 1$, which follows from \eqref{eq:HypoSimplifyThmLowFrequencies}, for the last inequality.
Combining \eqref{eq:LowFrequencyRemoveBadIndices}, \eqref{eq:LowFrequencyIndependentBlocks}, \eqref{eq:LowFrequencyBlockGaussianApproximation}, and \eqref{eq:LowFrequencyRestoreFullTime} with the triangle inequality yields the desired estimate.
\end{proof}
\subsection{Local central limit estimates for $\tau_p^{\square,k}$} \label{seq:TCL_square}
We first show the following result.
\begin{lemma} \label{lem:LowerBound_V(p)} Let $C>0,\eps,\eps'>0$. There exists $F>0$ depending on $d,L,C,\eps,\eps'$ such that for every $p$ satisfying \eqref{eq:initialization} and $\X(p)\geq F$, we have $V(p)\geq \X(p)^{1-\eps}$. 
\end{lemma}
\begin{proof} Let $k=\lfloor \X(p)^{5\eps}\rfloor$. Let $\kappa\in [-\pi,\pi]^d$ such that $\|\kappa\|_2=\X(p)^{-1/2-2\eps}$. By Theorem \ref{thm:FT_Middle} with $\eps'=\eps/10$, we obtain
\[
|\Four[\tau_p^{\square,k}](\kappa)| 
\leq K \left(1-\X(p)^{-\eps'}\min(1,\X(p)^{1-1-4\eps}) \right )^{k}\|\tau_p^{\square,k}\|_1 
\leq \X(p)^{-100}\|\tau_p^{\square,k}\|_1,
\]
as long as $\X(p)$ is large enough. On the other hand, by applying Theorem \ref{thm:FT_Low} with $\eps/9$ taking the role of $\eps$, and $\eps'>0$ small enough, we have 
\[
\left |\frac{\Four[\tau_p^{\square,k}](\kappa)}{\|\tau_p^{\square,k}\|_1}-e^{-k\|\kappa\|_2^2V(p)/(2d)} \right | \leq \X(p)^{-100}+\X(p)^{1/2+\eps/9+(5/4)\eps-1/2-2\eps}+\X(p)^{-2-8\eps+15\eps/2+2+\eps/9},
\]
which is smaller than $\X(p)^{-\eps/3}$ whenever $\X(p)$ is large enough. Combining the last two inequalities, we get 
\[ e^{-k\|\kappa\|_2^2V(p)/(2d)} \leq 2\X(p)^{-\eps/3}. \] 
Therefore, $k\|\kappa\|_2^2V(p)/(2d)\geq 1$, which implies $ V(p)\geq 2d\X(p)^{1-\eps}$. Since $\eps>0$ is arbitrary, this concludes the proof.
\end{proof}
To simplify notation, let $\Gauss_0: x \mapsto \1_{x=0}$, and for $x\in \Z^d$ and $v>0$, write
\[ \Gauss_v(x):= (2\pi v/d)^{-d/2}\exp(-(\|x\|^2_2 d)/(2v)).\]
\begin{theorem} \label{thm:LTCL_tau_p^square,k}
Let $C>0,\eps,\eps'>0$. There exists $F>0$ depending only on $d,L,C,\eps,\eps'$ such that for every $p<p_c$ satisfying \eqref{eq:initialization} and $\X(p)\geq F$, for every $\X(p)^{\eps'}\leq k\leq \X(p)^{1/\eps'}$, we have
\[
\left \|\tau_p^{\square,k}/\|\tau_p^{\square,k}\|_1-\Gauss_{kV(p)} \right \|_\infty \leq \X(p)^{\eps}k^{-1/4}(\X(p)k)^{-d/2}.
\]
\end{theorem}
\begin{proof}
By Fourier inversion, the difference in Theorem \ref{thm:LTCL_tau_p^square,k} is bounded by
\[ \frac{1}{(2\pi)^d} \int_{\kappa\in \R^d} \left |\Four[\tau_p^{\square,k}](\kappa)/\|\tau_p^{\square,k}\|_1\1_{\|\kappa\|_\infty\leq \pi }-e^{-k\|\kappa\|_2^2V(p)/(2d)} \right |\,d\kappa. \]
We split the integral above into the four regions where $\|\kappa\|_\infty$ belongs to $[0, \X(p)^{-1/2+\eps}k^{-1/2})$, $[\X(p)^{-1/2+\eps}k^{-1/2},k^{-1/3})$, $[k^{-1/3},\pi)$, and $[\pi,+\infty)$. We write $\gamma_1,\gamma_2,\gamma_3,\gamma_4$ for the respective integrals.
We further divide $\gamma_2$ into $\gamma^\Four_2:= \int |\Four[\tau_p^{\square,k}](\kappa)/\|\tau_p^{\square,k}\|_1|$ and $\gamma^{\Gauss}_2:=\int e^{-k\|\kappa\|_2^2V(p)/(2d)}$. We treat $\gamma_3\leq \gamma_3^\Four+\gamma_3^{\Gauss}$ similarly.
First, by Theorem \ref{thm:FT_Low}, we have 
\begin{align*} \gamma_1 & \leq \int_{\kappa:\|\kappa\|_\infty\leq \X(p)^{-1/2+\eps}k^{-1/2}} \left(\X(p)^{-2}+ \X(p)^{1/2+\eps}k^{1/4}\|\kappa\|_2+\|\kappa\|_2^4k^{3/2}\X(p)^{2+\eps}\right)\,d\kappa 
  \\ & \leq K(\X(p)k)^{-d/2}\X(p)^{d\eps}\left (\X(p)^{-2}+ \frac{\X(p)^{1/2+\eps}k^{1/4}}{\X(p)^{1/2-\eps}k^{1/2}}+\frac{\X(p)^{2+\eps}k^{3/2}}{ (\X(p)^{1/2-\eps}k^{1/2})^4} \right ),
\end{align*}
for some constant $K$ depending only on $d,L$, using $\int_{\kappa:\|\kappa\|_\infty\leq r} \|\kappa\|_2^n=O(r^{d+n})$. Thus, if $3\eps\leq \eps'$, which we may assume without loss of generality,
\[ \gamma_1\leq 3K (\X(p)k)^{-d/2}\X(p)^{\eps(d+5)} k^{-1/4}.\]
Next, for the part where $\X(p)^{-1/2+\eps}k^{-1/2}\leq \|\kappa\|_\infty\leq k^{-1/3}$, Theorem \ref{thm:FT_Middle} gives, for some constant $K$ depending only on $d,L,C,\eps$ and whenever $\X(p)$ is large enough,
\begin{align*} \gamma_2^\Four & \leq K(1-\X(p)^{-\eps}\min(1,\X(p)\X(p)^{-1+2\eps}k^{-1}) )^k
  \\ & \leq K(1-\X(p)^{\eps}k^{-1})^k
  \\ & \leq K e^{-\X(p)^{\eps}}\leq \X(p)^{-100d/\eps}.
\end{align*}
Next, for the part where $\|\kappa\|_\infty\in[k^{-1/3},\pi]$, Theorem \ref{thm:FT_High} gives, for some constants $K,c>0$ depending only on $d,L,C,\eps$ and whenever $\X(p)$ is large enough,
\[
\gamma_3^\Four \leq K (1-c/k^{2/3})^k\leq \X(p)^{-100d/\eps}. \]
Finally, by elementary integration, using $V(p)\geq \X(p)^{1-\eps/10}$ from Lemma \ref{lem:LowerBound_V(p)}, we have
\[ \gamma^{\Gauss}_2+\gamma^{\Gauss}_3+\gamma_4\leq \X(p)^{-100d/\eps}. \] 
The desired result follows by combining the previous inequalities, since $\eps>0$ is arbitrary.
\end{proof}
We conclude this subsection by showing a variant of Theorem \ref{thm:LTCL_tau_p^square,k}. Recall $E_{p,k}$ and $E_{p,\infty}$ from Section \ref{sec:p_c-p_c}. (The only important fact for later is $E_{p,\infty}=O(\X(p))$ from Lemma \ref{lem:Large_k_L1}.)
\begin{theorem} \label{thm:LTCL_tau_p^square,k_non_renormalized}
Let $C>0,\eps>0$. There exist $F,F'>0$ depending only on $d,L,C,\eps$ such that for every $p<p_c$ with $\X(p)\geq F$, for every $\X(p)^{2\eps}\leq k \leq \X(p)^{2/\eps}$, if \eqref{eq:initialization} is satisfied at $p_2$ with $\X(p_2)= F' \log(\X(p)) \X(p)$, then for every $x\in \Z^d$, we have 
\[
\left \|\tau_p^{\square,k}\left ((\tilde p_c(p)-p)/(1-p)\right)^k-E_{p,\infty}\Gauss_{kV(p)} \right \|_\infty \leq \X(p)^{1+\eps}k^{-1/4}(\X(p)k)^{-d/2}.
\]
\end{theorem}
\begin{proof} Recall from \eqref{eq:tau_E_p,infty} that $\|\tau_p^{\square,k}\|_1=((1-p)/(\tilde p_c(p)-p))^kE_{p,k}$. Thus, since $|E_{p,k}-E_{p,\infty}|\leq \X(p)^4c^k$ by Lemma \ref{eq:L1Diff_After_Mixing}, since $E_{p,\infty}\leq K_4\X(p)$ by Lemma \ref{lem:Large_k_L1}, and since $V(p)\geq \X(p)^{1/2-\eps/d}$ by Lemma \ref{lem:LowerBound_V(p)}, it suffices to multiply the inequality in Theorem \ref{thm:LTCL_tau_p^square,k} by $E_{p,k}$.
\end{proof}
We end the section by proving a convenient sub-gamma tail for $\tau_{p}^{\square,k}$. 
\begin{lemma} \label{lem:gamma_tail_tau_pk} Let $C>0,\eps>0$. There exists $F>0$ depending on $d,L,C,\eps$ such that for every $p$ satisfying \eqref{eq:initialization} and $\X(p)\geq F$, for every $k\geq 0$, for every $x\in \Z^d$, we have
\[ \left (\frac{\tilde p_c(p)-p}{1-p} \right )^k \tau_p^{\square,k}(x) \leq \X(p)^9 \exp\left ( -\frac{\| x\|_2}{\X(p)^{1/2+\eps}} \min \left (1, \frac{\| x\|_2}{k\X(p)^{1/2}}\right )+k\X(p)^{-4}\right ), \]
with the convention $\langle x\rangle/(k\X(p)^{1/2})=\infty$ when $k=0$.
\end{lemma}
\begin{proof} The case $k=0$ holds by Theorem \ref{thm:ExponentialTail}. 
  Let $k\geq 1$. Let $\kappa\in \R^d$ be collinear to $x$ such that
  \begin{equation}
  \|\kappa\|_2:=\max(2\X(p)^{-1/2-\eps}\min(1,
  \| x\|_2/(k\X(p)^{1/2})),\X(p)^{-3}). \label{eq:kappa_for_gamma}
  \end{equation}
By using Lemma \ref{lem:AllStepExpTail} with $\eps'=\eps/100$, we have
\[
\sum_{y\in \Z^d} \tau_{p}^{\square,k}(y)e^{\langle \kappa, y\rangle} \leq \X(p)^7 (\Omega\X(p)\lambda_p)^{k-1} (1+\|\kappa\|_2^2\X(p)^{1+\eps/10})^{k-1}.
\]
By keeping only the term $y=x$ in the above sum, and using $\Omega \X(p)\lambda_p=((1-p)/(\tilde p_c(p)-p))$, and $1+z\leq e^z$ for $z\in \R_+$, we get 
\[ ((\tilde p_c(p)-p)/(1-p))^k \tau_p^{\square,k}(x) e^{\langle\kappa,x\rangle} \leq \Omega \lambda_p \X(p)^8 e^{k\|\kappa\|_2^2 \X(p)^{1+\eps/10}}.\]  
Since $\lambda_p\leq 1$, it remains to upper bound $e^{k\|\kappa\|^2_2\X(p)^{1+\eps/10}-\langle \kappa,x\rangle}$. The factor $e^{k\X(p)^{-4}}$ handles the case where $\|\kappa\|_2=\X(p)^{-3}$. For the other case, we have, if $\X(p)$ is large enough,
\[ 2k\|\kappa\|^2_2\X(p)^{1+\eps/10} \leq \|\kappa\|_2\|x\|_2=\langle \kappa,x\rangle. \]
Thus $e^{k\|\kappa\|^2_2\X(p)^{1+\eps/10}-\langle \kappa,x\rangle}\leq e^{-\|\kappa\|_2\|x\|_2/2}$, and the desired bound follows from \eqref{eq:kappa_for_gamma}.
\end{proof}
\subsection{Local central limit estimates for $\tau_{p,p'}^\square$} \label{seq:TCL_square_2}

We first show a local central limit estimate for $\tau_{p,p'}^\square=\sum_{k\geq 0}((p'-p)/(1-p))^k\tau_p^{\square,k}$. Naturally, we cannot have a full LTCL, as we cannot have a TCL for $\tau_p^{\square,k}$ when $k$ is small. So we focus on the case $\|x\|_2\gg \X(p)^{1/2}$. 
We approximate $\tau_{p,p'}^\square$ by the Green function of a Brownian motion killed at an exponential time, given for $v>0, \beta\in [0,1]$ by
\[ \Gauss_{v,\beta}(x):= \int_0^\infty \beta^t \Gauss_{vt}(x) dt. \]
We also write $C^\Gauss_v>0$ for the constant such that, whenever $x\neq 0$, we have $\Gauss_{v,1}(x)=C^\Gauss_v\langle x\rangle^{2-d}$.
\begin{lemma} \label{lem:LTCL_partial_tau_pp'^k} Let $C>0,\eps>0$. There exist $F,F'>0$ depending only on $d,L,C,\eps$ such that for every $p<p_c$ with $\X(p)\geq F$, if \eqref{eq:initialization} is satisfied at $p_2$ with $\X(p_2)= F' \log(\X(p)) \X(p)$, then for every $p< p'<\tilde p_c(p)$ such that $(p'-p)/(\tilde p_c(p)-p)<1-1/\X(p)^3$, and every  $x\in \Z^d$ with $\langle x\rangle\geq \X(p)^{1/2+\eps}$, we have
\[
\left |\tau_{p,p'}^\square(x) - E_{p,\infty}\Gauss_{V(p),(p'-p)/(\tilde p_c(p)-p)}(x) \right | \leq \X(p)^\eps \left (\langle x\rangle^2/\X(p)\right )^{-1/4}\langle x\rangle^{2-d} .
\]
\end{lemma}
\begin{proof} For $v>0,\beta>0$, let $S_{v,\beta}(x)=\sum_{k=0}^\infty \beta^k \Gauss_{kv}(x)$. By a sum-integral comparison, for every $v,\beta$, we have
  \[ |\Gauss_{v,\beta}(x)-S_{v,\beta}(x)|\leq 2\max_{t\in \R_+} \beta^t \Gauss_{vt}(x)=O(\langle x\rangle^{-d}). \]
  Thus, since $E_{p,\infty}\leq K_4\X(p)$ for some $K_4>0$ depending only on $d,L,C$ by Lemma \ref{lem:Large_k_L1}, it suffices to show the lemma with $\Gauss$ replaced by $S$.
  
  Fix $\beta=(p'-p)/(\tilde p_c(p)-p)$. By definition of $\tau_{p,p'}^\square$, we have 
  \[ |\tau_{p,p'}^\square(x)- E_{p,\infty}S_{V(p),\beta}(x)|\leq \sum_{k=0}^\infty \beta^k \left |\tau_p^{\square,k}(x)\left ((\tilde p_c(p)-p)/(1-p)\right)^k-E_{p,\infty}\Gauss_{kV(p)}(x) \right |. \]
We split the above sum into three parts according to whether $k\leq \langle x\rangle^2 \X(p)^{-1-\eps}$, or $k> \langle x\rangle^2 \X(p)^{-1-\eps}$ and $k\leq \X(p)^{100}$, or $k> \langle x\rangle^2 \X(p)^{-1-\eps}$ and $k> \X(p)^{100}$.
We write $S_1, S_2, S_3$ for the corresponding sums. We start with $S_2$, which gives the main error term. By Theorem \ref{thm:LTCL_tau_p^square,k_non_renormalized},
\begin{align*} S_2 & \leq \sum_{k\geq \langle x\rangle^2 \X(p)^{-1-\eps}} \X(p)^{1+\eps} k^{-1/4} (\X(p)k )^{-d/2}
\\  & \leq K \X(p)^{1+\eps-d/2}(\langle x\rangle^2 \X(p)^{-1-\eps})^{3/4-d/2}
\\  & = K \X(p)^{\eps(1/4+d/2)}(\langle x\rangle^2/\X(p))^{-1/4}\langle x\rangle^{2-d}
\end{align*}
  where the second line uses $\sum_{k\geq r} k^{-1/4-d/2} \leq Kr^{1-1/4-d/2}$ for some constant $K$ depending on $d$.

By Lemma \ref{lem:Large_k_L1}, we have $E_{p,\infty}\leq K_4\X(p)$ for some $K_4>0$ depending only on $d,L,C$. Then, Lemma \ref{lem:V(p)} gives $V(p)\leq \X(p)^{1+\eps/10}$. Thus, using a standard Gaussian tail bound,
    \[
  E_{p,\infty}\sum_{0\leq k\leq \langle x\rangle^2\X(p)^{-1-\eps}}
  \beta^k\Gauss_{kV(p)}(x)
  \leq \X(p)^{-100}\langle x\rangle^{3/2-d}.
  \]
  Moreover, by Lemma \ref{lem:gamma_tail_tau_pk} and $\beta=(p'-p)/(\tilde p_c(p)-p)\leq 1-1/\X(p)^3$, we have for every $0\leq k\leq \langle x\rangle^2\X(p)^{-1-\eps}$, whenever $\X(p)$ is large enough,
  \[   \beta^k\left(\frac{\tilde p_c(p)-p}{1-p}\right)^k\tau_p^{\square,k}(x) \leq e^{-\X(p)^{\eps/10}-\langle x\rangle/((k+1)\X(p)^2)} e^{-k\X(p)^{-4}}.\]
Summing this inequality over $k$, we again obtain a global bound smaller than $\X(p)^{-100}\langle x\rangle^{3/2-d}$.
Thus,
  \[
  S_1 \leq 2 \X(p)^{-100}\langle x\rangle^{3/2-d} .
  \]
Similarly, with the only difference that now $\beta$ gives the main exponential decay, we have
  \[
  S_3 \leq 2 \X(p)^{-100}\langle x\rangle^{3/2-d}.
  \]
Summing the bounds on $S_1,S_2,S_3$ concludes the proof, since $\eps>0$ is arbitrary.
\end{proof}
The following lemma will allow us to upper-bound $\tau_{p,p'}^\square(x)$ when $\langle x\rangle$ is small. This will be relevant in the proof of Theorem \ref{thm:Main} when applying the pointwise estimates in Theorem \ref{thm:PointwiseError}.
\begin{lemma} \label{lem:Future_diff_trick}Let $C>0,\eps,\eps'>0$. Let $F,F'>0$ be large enough depending only on $d,L,C,\eps,\eps'$. Let $p<p_c$ with $\X(p)\geq F$. Assume that \eqref{eq:initialization} is satisfied at $p_2$ with $\X(p_2)= F' \log(\X(p)) \X(p)$. Let $p< p'<p''<\tilde p_c(p)$ with $(p''-p)/(\tilde p_c(p)-p)<1-1/\X(p)^3$. Let 
  $\gamma=(p'-p)/(p''-p')$. If $\gamma\geq \X(p)^{\eps'}$, then for every $x\in \Z^d$ with $\langle x\rangle^2\leq \X(p) \gamma^{(d-4)/(2d-4)}$
  we have 
\[ \tau_{p,p''}^\square(x)-\tau_{p,p'}^\square(x) \leq 2\gamma^{-1/2}\tau_{p,p'}^\square(x)+(2\gamma^{-1/2}\X(p)^\eps+\X(p)^{-100})\langle x\rangle^{2-d}. \]
\end{lemma}
\begin{proof} Recall that $\tau_{p,p''}^\square(x)=\sum_{k=0}^\infty ((p''-p)/(1-p))^k \tau_{p}^{\square,k}$. Using $p'<p''$, we bound the contribution of the case $k\leq \sqrt{\gamma}$ by
\begin{equation} \left (\frac{p''-p}{p'-p}\right )^{\sqrt{\gamma}}\sum_{0\leq k \leq \sqrt{\gamma}} \left (\frac{p'-p}{1-p}\right )^k \tau_{p}^{\square,k}(x) \leq (1+1/\gamma)^{\sqrt{\gamma}} \tau_{p,p'}^\square(x) \leq (1+2\gamma^{-1/2}) \tau_{p,p'}^\square(x),
\label{eq:compare_tau_ppp_low}\end{equation}
using $\gamma\geq 1$ for the last inequality. The case $k>\X(p)^{100}$ is treated exactly as $S_2$ and $S_3$ in the proof of Lemma \ref{lem:LTCL_partial_tau_pp'^k}, giving the bound
\begin{equation} \sum_{k\geq \X(p)^{100}}((p''-p)/(1-p))^k \tau_{p}^{\square,k}(x) \leq \X(p)^{-100} \langle x\rangle^{2-d}.
\label{eq:compare_tau_ppp_high}\end{equation} 
Lastly, for the case $\sqrt{\gamma}\leq k \leq \X(p)^{100}$, we use Theorem \ref{thm:LTCL_tau_p^square,k_non_renormalized} to get
\[ \sum_{\sqrt{\gamma}\leq k \leq \X(p)^{100}} \left (\frac{p''-p}{1-p}\right )^k \tau_{p}^{\square,k}(x)\leq \sum_{\sqrt{\gamma}\leq k} \left (E_{p,\infty}\| \Gauss_{k V(p)}\|_\infty+\X(p)^{1+\eps}k^{-1/4}(\X(p)k)^{-d/2} \right )  .\]
Then, by Lemmas \ref{lem:Large_k_L1} and \ref{lem:LowerBound_V(p)}, whenever $\X(p)$ is large enough, the last sum is bounded by
\[ \sum_{\sqrt{\gamma}\leq k} \X(p)^\eps \left (\X(p) (k\X(p))^{-d/2}+\X(p) k^{-1/4}(\X(p)k)^{-d/2} \right )\leq \X(p)^{2\eps} \X(p)^{1-d/2}\gamma^{1/2-d/4}, \]
using $\sum_{k\geq r} k^{-d/2}+k^{-d/2-1/4} \leq K r^{1-d/2}$ for every $r\geq 1$ for some $K>0$ depending only on $d$. Using $\langle x\rangle^2\leq \X(p) \gamma^{(d-4)/(2d-4)}$ and $\gamma\geq 1$, we deduce
\[ \sum_{\sqrt{\gamma}\leq k \leq \X(p)^{100}}((p''-p)/(1-p))^k \tau_p^{\square,k}(x)\leq \X(p)^{2\eps}\gamma^{-1/2}\langle x\rangle^{2-d} . \]
Summing the last inequality with \eqref{eq:compare_tau_ppp_low} and \eqref{eq:compare_tau_ppp_high} yields the desired result.
\end{proof}
\subsection{Renormalization induction: Proof of \eqref{eq:etaUpper} given \eqref{eq:initialization}} \label{sec:RenormalizationInduction}
The goal of this subsection is to show the following weaker version of Theorem \ref{thm:Main}.
\begin{theorem} \label{thm:MainWeaker} Let $C>0$. There exists $F>0$ depending only on $d,L,C$ such that if there exists $p_0<p_c$ satisfying \eqref{eq:initialization} 
and $\X(p_0)\geq F$, then for every $x\in \Z^d$, we have
\begin{equation}
\tau_{p_c}(x)\leq 2C  \langle x\rangle^{2-d}.\label{eq:pc_initialization}
\end{equation}
\end{theorem}
We prove Theorem \ref{thm:MainWeaker} first to have access to all our results under \eqref{eq:initialization} or \eqref{eq:etaUpper} sufficiently close to $p_c$, which will greatly simplify the analysis.
The main tool to show Theorem \ref{thm:MainWeaker} is the following induction property. 
For every $f:\Z^d\mapsto \R$, let $\Mc(f):=\max_{x\in \Z^d}f(x)\langle x\rangle ^{d-2}$.
\begin{proposition} \label{pro:InductionScheme} There exists $\eta>0$ such that the following result holds. For every $C>0$, there exists $F(C)>0$ such that for every $p_1<p_c$ satisfying \eqref{eq:initialization} and $\X(p_1)\geq F(C)$, if $p_2$ is such that $\X(p_2)=\X(p_1)^{1.1}$, then
\[ \Mc(\tau_{p_2})\leq \Mc(\tau_{p_1})+\X(p_1)^{-\eta}.\]
\end{proposition}
\begin{proof}[Proof of Theorem \ref{thm:MainWeaker} given Proposition \ref{pro:InductionScheme}]
Let $C>0$. Let $p_0$ be such that $\X(p_0)\geq F(2C)$, where $F$ is given in Proposition \ref{pro:InductionScheme},
 and such that $\sum_{i\geq 0} \X(p_0)^{-\eta 1.1^i}\leq C$.
 For every $i\geq 1$, define $p_i<p_c$ by $\X(p_i)=\X(p_{i-1})^{1.1}$. Assume that \eqref{eq:initialization} holds at $p_0$ with $C$.
 Let $i\geq 0$. If $\Mc(\tau_{p_i})\leq C+\sum_{j=0}^{i-1} \X(p_0)^{-\eta 1.1^j}$
 then by Proposition \ref{pro:InductionScheme} we also have $\Mc(\tau_{p_{i+1}})\leq C+\sum_{j=0}^{i} \X(p_0)^{-\eta 1.1^j}$.
 Thus, by induction, the property holds for every $i\geq 0$. Taking $i\to \infty$ gives \eqref{eq:pc_initialization}.
\end{proof}
\begin{proof}[Proof of Proposition \ref{pro:InductionScheme}] Let $p<p_1$ be such that $\X(p_1)=\X(p)^{1.1}$. First, by Theorem \ref{thm:PointwiseError} and Lemma \ref{lem:AnnoyingXi}, we have, for some constant $K>0$ depending only on $d,L,C$,
\begin{equation} \left |\Mc(\tau_{p_1})-\Mc(\tau_{p,p_1}^\square) \right |\leq K \frac{\X(p_1)\psi_d(\X(p))}{\X(p)^2}\leq K \X(p)^{-0.4}.\label{eq:MC_MC_a}\end{equation}
Also, by Theorem \ref{thm:PointwiseError}, we have
\[ |\|\tau_{p,p_1}^\square\|_1-\X(p_1)|\leq K \X(p)^{0.7}.\]
In particular, $p_1<\tilde p_c(p)$. Equation \eqref{eq:Tau_X_Diff} then gives
\[ \|\tau_{p,p_1}^\square\|_1= \sum_{k=0}^\infty \left (\frac{p_1-p}{\tilde p_c(p)-p } \right )^k(E_{p,k}-E_{p,\infty}) +\frac{\tilde p_c(p)-p}{\tilde p_c(p)-p_1}E_{p,\infty}.\]
By Lemmas \ref{eq:L1Diff_After_Mixing}, \ref{lem:Small_k_L1}, and \ref{lem:Large_k_L1}, the first sum is bounded by $K \log(\X(p))\X(p)$ for some $K>0$ depending only on $d,L,C$, as long as $\X(p)$ is large enough. Thus,
\begin{equation} \left | \frac{\tilde p_c(p)-p}{\tilde p_c(p)-p_1}E_{p,\infty}-\X(p_1)\right |\leq 2K\log(\X(p))\X(p).\label{eq:ratio1renormalisation}\end{equation}

Then let $p'_2$ be the parameter such that
\begin{equation} \frac{\tilde p_c(p)-p}{\tilde p_c(p)-p'_2}E_{p,\infty}=2\X(p_2). \label{eq:ratio2renormalisation}\end{equation}
Similarly, reusing \eqref{eq:Tau_X_Diff} yields
\begin{equation} \left |2\X(p_2)- \|\tau_{p,p'_2}^\square\|_1 \right | \leq K\log(\X(p))\X(p) . \label{eq:X_of_p'_2}\end{equation}
Furthermore, taking the ratio between \eqref{eq:ratio1renormalisation} and \eqref{eq:ratio2renormalisation} gives, whenever $\X(p)$ is large enough, 
\begin{equation} \left |\frac{\tilde p_c(p)-p_1}{\tilde p_c(p)-p'_2}- \frac{2\X(p_2)}{\X(p_1)} \right | \leq 1.\label{eq:ratio_diff_pc}\end{equation}

We then estimate $\Mc(\tau_{p,p'_2}^\square)$ by splitting it into two parts. For every $f:\Z^d\mapsto \R$, and $r>0$, let
\[ \Mc_{\leq r}(f):=\max_{x\in \Z^d:\|x\|_2\leq r}f(x)\langle x\rangle ^{d-2},\] 
and define $\Mc_{>r}(f)$ similarly so that $\Mc(f)=\max(\Mc_{\leq r}(f),\Mc_{>r}(f))$. Let $r=\X(p)^{0.51}$. We first deal with the case $\|x\|_2\leq r$. 
By Lemma \ref{lem:X'} we have for some constants $K_1,K_2>0$, depending only on $d,L,C$, that $p_1-p\geq K_1/\X(p)$. 
Proposition \ref{pro:p_c(p)-p_2} also gives, for some constant $K_2>0$, that $\tilde p_c(p)-p_1\leq K_2/\X(p_1)$. Thus,
\[ \gamma:=\frac{p_1-p}{p_2'-p_1} \geq \frac{K_1}{K_2}\X(p)^{0.1}. \]
Note that $r^2\leq \X(p)\gamma^{(d-4)/(2d-4)}$. Also, by \eqref{eq:ratio_diff_pc} and the bounds $p_1-p\geq K_1/\X(p)$ and $\tilde p_c(p)-p_1\leq K_2/\X(p_1)$, we have
$(\tilde p_c(p)-p)/(\tilde p_c(p)-p'_2)\leq \X(p)^3$. In other words,
\begin{equation} \frac{p'_2-p}{\tilde p_c(p)-p}\leq 1-\X(p)^{-3}. \label{eq:check_hypo_lem_TCL}\end{equation}
 Thus, we may apply Lemma \ref{lem:Future_diff_trick} to get
\[ \Mc_{\leq r}(\tau_{p,p'_2}^\square) \leq (1+2\gamma^{-1/2})\Mc(\tau_{p,p_1}^\square)+(2\gamma^{-1/2}\X(p)^\eps+\X(p)^{-100}).\]
So, for some constant $K_3>0$ depending only on $d,L,C$,
\begin{equation} \Mc_{\leq r}(\tau_{p,p'_2}^\square) \leq \Mc(\tau_{p,p_1}^\square)+K_3\X(p)^{-0.05}.\label{eq:MC_leq_r}\end{equation}

On the other hand, when $\X(p)$ is large enough, for every $\eps>0$, Lemma \ref{lem:LTCL_partial_tau_pp'^k} gives
\begin{equation} \Mc_{>r}(\tau_{p,p'_2}^\square)\leq E_{p,\infty}C_{V(p)}^\Gauss+\X(p)^\eps (r^2/\X(p))^{-1/4}.\label{eq:norm_compare_cvg}\end{equation}
We then have to bound $E_{p,\infty}C_{V(p)}^\Gauss$ using $\Mc(\tau_{p,p_1}^\square)$. To this end, by \eqref{eq:check_hypo_lem_TCL} and $p_1<p'_2$, we may apply Lemma \ref{lem:LTCL_partial_tau_pp'^k} with $p,p_1$ and some $y\in \Z^d$ such that $\|y\|_2=\lfloor \X(p)^{0.53}\rfloor$. We get
\[
\left |\tau_{p,p_1}^\square(y) - E_{p,\infty}\Gauss_{V(p),(p_1-p)/(\tilde p_c(p)-p)}(y) \right | \leq \X(p)^\eps \left (\langle y\rangle^2/\X(p)\right )^{-1/4}\langle y\rangle^{2-d}.
\]
Thus 
\[ E_{p,\infty}\Gauss_{V(p),(p_1-p)/(\tilde p_c(p)-p)}(y) \leq (\Mc(\tau_{p,p_1}^\square)+\X(p)^{-0.01})\langle y\rangle^{2-d}.\]
Reusing $p_1-p\geq K_1/\X(p)$ and $\tilde p_c(p)-p_1\leq K_2/\X(p_1)$, we get, for some constant $K_3>0$,
\begin{equation} E_{p,\infty}\Gauss_{V(p),1-K_3\X(p)/\X(p_1)}(y) \leq (\Mc(\tau_{p,p_1}^\square)+\X(p)^{-0.01})\langle y\rangle^{2-d}. \label{eq:rough_minor_Cvg}\end{equation}
Moreover, writing $\beta=1-K_3\X(p)/\X(p_1)$, we have
\begin{equation} C_{V(p)}^\Gauss\langle y\rangle^{2-d}-\Gauss_{V(p),1-K_3\X(p)/\X(p_1)}(y)=(2\pi V(p)/d)^{-d/2}\int_0^\infty(1-\beta^t)t^{-d/2} \exp \left (-\frac{\|y\|^2_2 d}{2t V(p)} \right ) dt. \label{eq:integral_minor_Cvg} \end{equation}
Let $\eps=0.001$. Let $T=\X(p)^\eps \|y\|_2^2/V(p)$. The contribution from the case $t\leq T$ to the right-hand side is at most $(1-\beta^T)\Gauss_{V(p),1}(y)$. By Lemma \ref{lem:LowerBound_V(p)},
\[ (1-\beta^T)\leq T(1-\beta)\leq (\X(p)^\eps \|y\|_2^2/V(p)) K_3 \X(p)/\X(p_1)\leq \X(p)^{2\eps}\|y\|_2^2/\X(p_1)\leq \X(p)^{-0.03},\]
provided $\eps>0$ is small enough and $\X(p)$ is large enough. Also, the contribution of the case $t>T$ to the right-hand side of \eqref{eq:integral_minor_Cvg} is at most
\begin{align*} &(2\pi V(p)/d)^{-d/2}\int_T^\infty(1-\beta^t)t^{-d/2} \exp \left (-\frac{\|y\|^2_2 d}{2t V(p)} \right ) dt 
  \\ & \leq (2\pi V(p)/d)^{-d/2}\int_T^\infty t^{-d/2} dt \leq dT (2\pi T V(p)/d)^{-d/2}\leq  \langle y\rangle^{2-d}\X(p)^{-1.002},\end{align*}
using Lemma \ref{lem:LowerBound_V(p)} and the definition of $T$ for the last inequality.
Returning to \eqref{eq:rough_minor_Cvg} and using $E_{p,\infty}\leq K_4\X(p)$ from Lemma \ref{lem:Large_k_L1}, we deduce
\[ (1-\X(p)^{-0.03})E_{p,\infty}C_{V(p)}^\Gauss \leq \Mc(\tau_{p,p_1}^\square)+\X(p)^{-0.01}+\X(p)^{-0.002}.\]
Returning further to \eqref{eq:norm_compare_cvg}, we get
\[ \Mc_{>r}(\tau_{p,p'_2}^\square)\leq \Mc(\tau_{p,p_1}^\square)+\X(p)^{-0.001}. \]
Combining with \eqref{eq:MC_leq_r}, we finally get 
\begin{equation} \Mc(\tau_{p,p'_2}^\square)\leq \Mc(\tau_{p,p_1}^\square)+\X(p)^{-0.001}. \label{eq:MC_MC_b} \end{equation}

By \eqref{eq:initialization} at $p_1$, \eqref{eq:MC_MC_a}, \eqref{eq:MC_MC_b}, and \eqref{eq:X_of_p'_2}, we can now use Theorem \ref{thm:PointwiseError} and Lemma \ref{lem:AnnoyingXi} at $p'_2$ to get 
\begin{equation} \left |\Mc(\tau_{p'_2})-\Mc(\tau_{p,p'_2}^\square) \right |\leq K \frac{\|\tau_{p,p'_2}^\square\|_1\psi_d(\X(p))}{\X(p)^2}\leq 3K \X(p)^{-0.29}. \label{eq:MC_MC_c} \end{equation}
Combining \eqref{eq:MC_MC_a}, \eqref{eq:MC_MC_b}, and \eqref{eq:MC_MC_c} thus yields
\[ \Mc(\tau_{p'_2})\leq \Mc(\tau_{p_1})+3\X(p)^{-0.001}. \]
Finally, Theorem \ref{thm:PointwiseError} and \eqref{eq:X_of_p'_2} also give $p_2<p'_2$ and hence $\Mc(\tau_{p_2})\leq \Mc(\tau_{p'_2})$.
\end{proof}
\subsection{Proof of Theorem \ref{thm:Main} and \ref{thm:tau_local_CLT}} \label{sec:FinalMainProof}
By Theorem \ref{thm:MainWeaker}, to show Theorem \ref{thm:Main}, it suffices to prove that \eqref{eq:etaUpper} implies \eqref{eq:eta}. To this end, it is enough to prove that Theorem \ref{thm:tau_local_CLT} holds under the weaker condition \eqref{eq:etaUpper}. Indeed, by taking $p\to p_c$ much faster than $x\to \infty$ in Theorem \ref{thm:tau_local_CLT}, we get as $x\to \infty$,
\[ \tau_{p_c}(x)=C_\Psi \int_0^\infty e^{-C_\Psi'\|x\|^2_2/t}t^{-d/2}dt+o(\|x\|_2^{2-d-\eps})=A \|x\|_2^{2-d}+o(\|x\|_2^{2-d-\eps}).\]
We thus assume that \eqref{eq:etaUpper} holds and aim to show the LTCL in Theorem \ref{thm:tau_local_CLT}. Our proof is based on the LTCL in Lemma \ref{lem:LTCL_partial_tau_pp'^k}, which first requires precise estimates of $E_{p,\infty}$ and $V(p)$ as $p\to p_c$.
\begin{lemma} \label{lem:E_p,infty_accurate}
Assume \eqref{eq:etaUpper} holds. Let $C_\X$ be the constant of Theorem \ref{thm:susceptibility}. For every $\eps>0$, we have as $p\uparrow p_c$,
\[
E_{p,\infty}=C_\X(p_c-p)^{-1}+O((p_c-p)^{-1/2-\1_{d=7}/4-\eps}).
\]
\end{lemma}
\begin{proof} Let $p'=p'(p)<p_c$ be such that $\X(p')=\X(p)^{3/2-\1_{d=7}/4}$. We estimate $\|\tau_{p,p'}^\square\|_1$ in two different ways. First, using Theorems \ref{thm:susceptibility} (which we proved under \eqref{eq:etaUpper}) and \ref{thm:PointwiseError}, we get
\[ \|\tau_{p,p'}^\square\|_1=C_\X(p_c-p')^{-1}+O((\X(p')/\X(p))(p_c-p)^{-1/2-\1_{d=7}/4-\eps}).\]

Second, using \eqref{eq:Tau_X_Diff}, we have
\[ \|\tau_{p,p'}^\square\|_1=\sum_{k=0}^\infty \left (\frac{p'-p}{\tilde p_c(p)-p } \right )^k(E_{p,k}-E_{p,\infty}) +\frac{\tilde p_c(p)-p}{\tilde p_c(p)-p'}E_{p,\infty}. \]
The above sum over $k$ is $O(\X(p)\log(\X(p)))$ by Lemmas \ref{eq:L1Diff_After_Mixing}, \ref{lem:Small_k_L1}, and \ref{lem:Large_k_L1}. Then, using Theorem \ref{thm:p_c-p_c} and Lemmas \ref{lem:X'} and \ref{lem:Large_k_L1}, we get
\[ \|\tau_{p,p'}^\square\|_1=O(\X(p)\log(\X(p)))+\frac{p_c-p}{p_c-p'}E_{p,\infty}+O \left ( \frac{\X(p')}{\X(p)}(p_c-p)^{-1/2-\1_{d=7}/4-\eps} \right ). \]

Comparing the two estimates on $\|\tau_{p,p'}^\square\|_1$ and then using Lemmas \ref{lem:X'} and \ref{lem:Large_k_L1}, we get
\begin{align*} C_\X(p_c-p')^{-1}& =\frac{p_c-p}{p_c-p'}E_{p,\infty}+O \left ( \frac{\X(p')}{\X(p)}(p_c-p)^{-1/2-\1_{d=7}/4-\eps} \right )
  \\ & =\frac{p_c-p}{p_c-p'}E_{p,\infty}+O\left (\frac{p_c-p}{p_c-p'} (p_c-p)^{-1/2-\1_{d=7}/4-\eps}\right ).
\end{align*}
The desired result follows by dividing by $(p_c-p)/(p_c-p')$.
\end{proof}
\begin{lemma} \label{lem:V_p,infty_accurate}
Assume \eqref{eq:etaUpper} holds. Let $C_\X$ and $C'_\xi$ be the constants of Theorems \ref{thm:susceptibility} and \ref{thm:Gyration_NotRenormalized}. For every $\eps>0$, we have as $p\uparrow p_c$,
\[
V(p)=(C'_\xi/C_\X)(p_c-p)^{-1}+O((p_c-p)^{-1/2-\1_{d=7}/4-\eps}).
\]
\end{lemma}
\begin{proof} We directly adapt the previous proof. Take $p'$ as before.
We estimate $\|\tau_{p,p'}^\square\|_2^2$ in two ways.
First, Theorem \ref{thm:Gyration_NotRenormalized} and Proposition \ref{pro:L2Error}, followed by Lemma \ref{lem:X'}, give
\begin{align*}
\|\tau_{p,p'}^\square\|_2^2
& =C'_\xi(p_c-p')^{-2}
+O\left(
\frac{\X(p')^{3+\eps/2}\psi_d(\X(p))}{\X(p)^2}
+(p_c-p')^{-1-\eps/2}\psi_d(\X(p'))
\right)
\\ & =C'_\xi(p_c-p')^{-2}+O((\X(p')/\X(p))^2 (p_c-p)^{-3/2-\1_{d=7}/4-\eps}).
\end{align*}

On the other hand, proceeding as in the proof of Proposition \ref{pro:PolynomialApproximation_For_Xi2}, using Corollary \ref{lem:BoundVariance} and Lemmas \ref{eq:L1Diff_After_Mixing}, \ref{lem:Small_k_L1}, \ref{lem:CovarianceDecay}, \ref{lem:Bulk_L2}, \ref{lem:Start_L2}, and \ref{lem:End_L2}, we get
\[
\|\tau_{p,p'}^\square\|_2^2
=E_{p,\infty}V(p)
\frac{(\tilde p_c(p)-p)^2}{(\tilde p_c(p)-p')^2}
+O\left(\X(p)^{2+\eps}
\frac{\tilde p_c(p)-p}{\tilde p_c(p)-p'}\right).
\]
We omit the details as the argument does not fundamentally change. 

Theorem \ref{thm:p_c-p_c}, Lemma \ref{lem:X'}, and the choice of $p'$ show, by comparing the two estimates, that
\[
E_{p,\infty}V(p)
=C'_\xi(p_c-p)^{-2}
+O((p_c-p)^{-3/2-\1_{d=7}/4-\eps}).
\]
The result follows by dividing by $E_{p,\infty}$, which we estimated in Lemma \ref{lem:E_p,infty_accurate}.
\end{proof}
Before proving Theorem \ref{thm:tau_local_CLT}, let us first show the following technical lemma.
\begin{lemma} \label{lem:Compare_Green} Assume \eqref{eq:etaUpper} holds. For every $\eps>0$, uniformly as $p\to p_c$ over  $p<p'<p_c$ with $\X(p)^{1.1}\leq \X(p')\leq \X(p)^{17/12}$ and $x\neq 0\in \Z^d$, we have
\[ \left |E_{p,\infty}\Gauss_{V(p),(p'-p)/(\tilde p_c(p)-p)}(x)-\Gauss_{(C'_\xi/C_\X^2),e^{-1/\X(p')}}(x) \right |=o(\X(p)^{-1/12+\eps}\langle x\rangle^{2-d}).\]
\end{lemma}
\begin{proof} First, the change of variables $t\mapsto t/E_{p,\infty}$ gives
\[
E_{p,\infty}\Gauss_{V(p),(p'-p)/(\tilde p_c(p)-p)}(x)
=\Gauss_{V(p)/E_{p,\infty},
((p'-p)/(\tilde p_c(p)-p))^{1/E_{p,\infty}}}(x).
\]  
Lemmas \ref{lem:E_p,infty_accurate} and \ref{lem:V_p,infty_accurate} give
\[ V(p)/E_{p,\infty}
=C'_\xi/C_\X^2
+o(\X(p)^{-1/12+\eps}).
\]
By Theorem \ref{thm:p_c-p_c} and Lemma \ref{lem:X'}, we have $p'<\tilde p_c(p)$ and $(\tilde p_c(p)-p')/(\tilde p_c(p)-p)=O(\X(p)^{-1/10})$. Thus, $-\log(1-u)=u+O(u^2)$ together with Lemma \ref{lem:E_p,infty_accurate} and Theorems \ref{thm:p_c-p_c} and \ref{thm:susceptibility} yields
\[
-\frac{1}{E_{p,\infty}}
\log\left(\frac{p'-p}{\tilde p_c(p)-p}\right)
=\frac{1}{\X(p')}
\left(1+o(\X(p)^{-1/12+\eps})\right).
\]

It remains to compare the Green functions of two killed Brownian motions whose parameters are close.
To this end, we claim that for  every compact interval $I$ in $(0,\infty)$, there exists $C_I<\infty$ such that for every $v,w\in I$ and every $a,b>0$ with $|a/b-1|\leq 1/100$, we have
\begin{equation} |\Gauss_{v,e^{-a}}(x)-\Gauss_{w,e^{-b}}(x)|\leq C_I\left(|v-w|+|a/b-1|\right)\langle x\rangle^{2-d}. \label{eq:Compare_Green}\end{equation}
If the claim is true, then the desired result follows from the previous displays.  

To show \eqref{eq:Compare_Green}, note that the change of variables $t\mapsto \|x\|_2^2t$ gives
\[
\Gauss_{v,e^{-a}}(x)
=\langle x\rangle^{2-d}
\left(\frac{2\pi v}{d}\right)^{-d/2}
\int_0^\infty t^{-d/2}
\exp\left(-a\|x\|_2^2t-\frac{d}{2vt}\right)dt.
\]
Uniformly for $v\in I$, the integral on the right-hand side and its derivative with respect to $v$ are bounded.
Writing $z=a\|x\|_2^2$, differentiation with respect to $\log z$ produces, up to a bounded factor,
\[
z\int_0^\infty t^{1-d/2}
\exp\left(-zt-\frac{d}{2vt}\right)dt,
\]
which is bounded uniformly over $z>0$ and $v\in I$.
The claim \eqref{eq:Compare_Green} then follows from the mean value theorem.
\end{proof}
\begin{proof}[Proof of Theorem \ref{thm:tau_local_CLT} under \eqref{eq:etaUpper}]
Rewriting the estimate in Theorem \ref{thm:tau_local_CLT}, it suffices to prove that for every $\eps>0$, as $p'\to p_c$ and $x\to \infty$, we have
\begin{equation} \tau_{p'}(x)=\Gauss_{(C'_\xi/C_\X^2),e^{-1/\X(p')}}(x)+o(\|x\|_2^{2-d-1/8+\eps}). \label{eq:thm_final_rewritten}\end{equation}
To this end, for every $x\in \Z^d$, let $p_x$ be the probability such that $\X(p_x)=\|x\|_2^{3/2}$. We distinguish three cases according to the value of $p'$.
If $\X(p')\leq \X(p_x)^{1.1}$, then standard Gaussian deviation estimates and Theorem \ref{thm:ExponentialTail} give $\tau_{p'}(x)=o(\|x\|_2^{2-d-1/8+\eps})$ and $\Gauss_{(C'_\xi/C_\X^2),e^{-1/\X(p')}}(x)=o(\|x\|_2^{2-d-1/8+\eps})$. So \eqref{eq:thm_final_rewritten} follows in this case.

If $\X(p_x)^{1.1}\leq \X(p')\leq \X(p_x)^{17/12}$, then Lemmas \ref{lem:LTCL_partial_tau_pp'^k} and \ref{lem:Compare_Green} yield
\[ \tau_{p_x,p'}^\square(x)=\Gauss_{(C'_\xi/C_\X^2),e^{-1/\X(p')}}(x)+ o \left (\X(p_x)^{-1/12}+(\langle x\rangle^2/\X(p_x) )^{-1/4} \right)\X(p_x)^{\eps/2}\langle x\rangle^{2-d}).\]
We chosed $\X(p_x)$ such that the above error term is $\|x\|_2^{2-d-1/8+\eps}$. Also, by Theorem \ref{thm:PointwiseError} and Lemma \ref{lem:AnnoyingXi}, we have, for some constant $K>0$,
\[\left |\tau_{p_x,p'}^\square(x)-\tau_{p'}(x)\right |=O(\frac{\X(p')\psi_d(\X(p_x))}{\X(p_x)^2})\|x\|_2^{2-d}=O(\|x\|_2^{2-d-1/8}),  \]
since $\psi_d(\X(p_x))\leq \sqrt{\X(p_x)}$ and $\X(p')\leq \X(p_x)^{17/12}$.
Combining the last two displays gives \eqref{eq:thm_final_rewritten}.

Finally, when $\X(p')>\X(p_x)^{17/12}$, we let $p'_x$ be such that $\X(p'_x)=\X(p_x)^{17/12}$. The BK inequality then gives
\[ |\tau_{p'}(x)-\tau_{p'_x}(x)|\leq (p'-p'_x)\tau_{p'}\ast D\ast \tau_{p'_x}(x)=O((p'-p'_x)\langle x\rangle^{4-d}).\]
Indeed, if there exists a $p'$-open path between $0$ and $x$ that is not $p'_x$-open, then there exists an oriented edge $\vec e$ such that the following occur disjointly: $0\slr e^-$ at $p'$, the edge $\vec e$ is $p'$-open but $p'_x$-closed, and $e^+\slr x$ at $p'_x$. The last inequality follows from \eqref{eq:etaUpper} and Lemma \ref{lem:TruncatedPoly}. Hence, using Lemma \ref{lem:X'} and $\X(p'_x)=\langle x\rangle^{2+1/8}$, we deduce
\[ \tau_{p'}(x)=\tau_{p'_x}(x)+O(\langle x\rangle^{2-d-1/8}).\]
Then, using that \eqref{eq:thm_final_rewritten} holds at $p'_x$ from the previous case, we get
\[ \tau_{p'}(x)=\Gauss_{(C'_\xi/C_\X^2),e^{-1/\X(p'_x)}}(x)+O(\langle x\rangle^{2-d-1/8})=\Gauss_{(C'_\xi/C_\X^2),e^{-1/\X(p')}}(x)+O(\langle x\rangle^{2-d-1/8}),\]
where the last estimate follows from elementary integration and Lemma \ref{lem:X'}. Hence, \eqref{eq:thm_final_rewritten} also holds in this last case, which concludes the proof.
\end{proof}

For completeness, let us quickly show Proposition \ref{pro:Square} from Theorem \ref{thm:tau_local_CLT}.
\begin{proof}[Proof of Proposition \ref{pro:Square}]
We already derived the upper bound on $\square(p)$ in Lemma \ref{lem:TwoTriangleSquare} (c). For the lower bound, Theorem \ref{thm:tau_local_CLT} and a basic integration argument directly imply
that as $p\to p_c$ and $x\to \infty$ with $\langle x\rangle \leq \X(p)^{1/2}$, we have $\tau_p(x)=\Theta(\langle x\rangle^{2-d})$.
Consequently, by restricting the convolution to points $z$ for which both $\langle z\rangle$ and $\langle y-z\rangle$ are comparable to $\langle y\rangle$, we obtain
$(\tau_p\ast\tau_p)(y)\geq c\langle y\rangle^{4-d}$ whenever $\langle y\rangle\leq \X(p)^{1/2}/2$.
Using the symmetry of $\tau_p$, we conclude that
\[
\square(p)
\geq (\tau_p\ast\tau_p\ast\tau_p\ast\tau_p)(0)
=\sum_{y\in\Z^d}(\tau_p\ast\tau_p)(y)^2
\geq c\sum_{\langle y\rangle\leq \X(p)^{1/2}/2}\langle y\rangle^{8-2d}
=\Theta(\psi_d(\X(p))).
\]
where the last equality follows by summing over dyadic annuli.
\end{proof}
 \footnotesize{
  \bibliographystyle{abbrv}
  \bibliography{unimodularthesis.bib}
  }
 \end{document}